\documentclass[12pt,reqno]{amsart}

\usepackage[colorlinks=true,urlcolor=blue,
bookmarks=true,bookmarksopen=true,citecolor=blue
]{hyperref}

\usepackage{color,soul}
\setstcolor{red}
\usepackage{pdfsync}
\usepackage[all, cmtip]{xy}
\usepackage{graphicx}

\usepackage{epsfig}
\usepackage{amscd}
\usepackage[mathscr]{eucal}
\usepackage{amssymb}
\usepackage{bm}
\usepackage{amsxtra}
\usepackage{amsmath}
\usepackage{latexsym} 
\usepackage[all]{xy}
\usepackage{enumerate}

\usepackage{enumitem}

\usepackage{color}
\usepackage{xcolor, soul}
\sethlcolor{green}

\usepackage{bbm}

\theoremstyle{plain}

\newtheorem{mainthm}{Theorem}
\newtheorem{thm}{Theorem}[subsection]

\newtheorem{cor}[thm]{Corollary}

\newtheorem{lem}[thm]{Lemma}
\newtheorem{prop}[thm]{Proposition}
\newtheorem{cnj}{Conjecture/Question}

\theoremstyle{definition}
\newtheorem{dfn}[thm]{Definition}

\newtheorem*{claim-nonum}{Claim}

\newtheorem*{assumption}{Assumption}

\theoremstyle{remark}
\newtheorem{rem}[thm]{Remark}

\newtheorem{axio}{Axiom}

\theoremstyle{plain}

\newcommand{\cobto}{\leadsto}

\newcommand{\R}{\mathbb{R}}
\newcommand{\Z}{\mathbb{Z}}

\newcommand{\N}{\mathbb{N}}
\newcommand{\C}{\mathbb{C}}

\newcommand{\fuk}{\mathcal{F}uk}

\newcommand{\mor}{{\textnormal{Mor\/}}}

\newcommand{\tcn}{{\mathcal{C}one}}

\newcommand{\pbaddress}{biran@math.ethz.ch}
\newcommand{\ocaddress}{cornea@dms.umontreal.ca}

\begin{document}

\title{A LAGRANGIAN PICTIONARY}

\date{\today }

\thanks{The second author was supported by an individual {\em NSERC
  Discovery }grant a {\em Simons Fellowship} and an {\em Institute for Advanced Study} fellowship grant.}

\author{Paul Biran and Octav Cornea}

\address{Paul Biran, Department of Mathematics, ETH-Z\"{u}rich,
  R\"{a}mistrasse 101, 8092 Z\"{u}rich, Switzerland}
\email{\pbaddress}
 
\address{Octav Cornea, Department of Mathematics and Statistics,
  University of Montreal, C.P. 6128 Succ.  Centre-Ville Montreal, QC
  H3C 3J7, Canada} \email{\ocaddress}

\bibliographystyle{alphanum}

% ----------------------------------------------------------------------
%

% ----------------------------------------------------------------------
%
% Abstract

%\begin{abstract}
  %\end{abstract}

\maketitle

% ----------------------------------------------------------------------
%
% Beginning of text
%

\tableofcontents 

% !TEX root = ImmersedS.tex

\section{Introduction.}

The purpose of this paper is to describe a dictionary $$\mathit{ geometry}\ \longleftrightarrow\  \mathit{algebra}$$ in Lagrangian topology. As a by-product we obtain a tautological (in a sense to be explained) proof of a folklore conjecture (sometimes attributed to Kontsevich) claiming that the objects and structure of the derived Fukaya category can be represented through immersed Lagrangians. We also shed some new light on the dichotomy rigidity/flexibility that is at the heart of symplectic topology.
 
We focus on a symplectic manifold $(M^{2n}, \omega)$ with $M$ compact or convex
at infinity and we consider  classes $\mathcal{L}ag^{\ast}(M)$ of (possibly immersed) Lagrangian submanifolds $L\hookrightarrow M$ carrying also a variety of additional structures. In this paper
the Lagrangians are exact and the choices of primitives are part of the structure.

\subsection{Outline of the dictionary}
\subsubsection{Geometry}
The geometric side of our dictionary, which does not appeal at any point to $J$-holomorphic curves,
 consists of a class  $\mathcal{L}ag^{\ast}(M)$ of Lagrangians, possibly immersed and marked (and endowed with additional structures). 
For an immersed Lagrangian $j_{L}:L\to M$, in generic position,
a marking consists of a choice of a set of self intersection points $(P_{-},P_{+})\in L\times L$ with $j_{L}(P_{-})=j_{L}(P_{+})$.  The Lagrangians in $\mathcal{L}ag^{\ast}(M)$ are organized as a cobordism category $\mathsf{C}ob^{\ast}(M)$, whose morphisms are represented by Lagrangian cobordisms belonging
to a class $\mathcal{L}ag^{\ast}(\C\times M)$.
These are Lagrangians in $\C\times M$ whose projection onto $\C$ is as in Figure \ref{Fig:morph}.
To such a category we associate a number of geometric operations. 
We then list a set of axioms governing them. 
If  the category $\mathsf{C}ob^{\ast}(M)$  obeys these axioms we say that it
{\em has surgery models}.  The reason for this terminology is that the cobordisms associated to 
the classical operation of Lagrangian surgery \cite{La-Si:Lag}, \cite{Po:surgery}, \cite{Bi-Co:cob1}  are in some sense the simplest morphisms in our category, see also Figure \ref{Fig:surg3}.
There is a natural equivalence relation $\sim$, called {\em cabling equivalence}, on the morphisms of $\mathsf{C}ob^{\ast}(M)$.   The fact that the category has surgery models means that each morphism is equivalent to one coming from surgery (together with some other properties that we skip now).  To understand this condition through an analogy, 
consider the path category of a Riemannian  manifold, with objects the points of the manifold and with morphisms the Moore paths. In this case the equivalence relation is homotopy and the place of surgery morphisms is taken by geodesics. If $\mathsf{C}ob^{\ast}(M)$ has surgery models, then the quotient category $$\widehat{\mathsf{C}}ob^{\ast}(M)=\mathsf{C}ob^{\ast}(M)/\sim$$ is triangulated.  
The triangulated structure appears naturally in this context because a surgery cobordism, resulting from the surgery operation, has three ends.

Cobordims are endowed with a natural  measurement provided by their shadow - roughly the area of the connected ``filling'' of their projection onto $\C$ \cite{Co-She:metric}.  If $\mathsf{C}ob^{\ast}(M)$ has surgery models,
then the triangulated structure of $\widehat{\mathsf{C}}ob^{\ast}(M)$ together with the weight provided by
the shadow of cobordisms leads to the definition of a class of pseudo-metrics on 
$\mathcal{L}ag^{\ast}(M)$ called (shadow) {\em fragmentation  pseudo-metrics} and denoted by $d^{\mathcal{F}}$ \cite{Bi-Co-Sh:lshadows-long}. Here $\mathcal{F}$ is a family of objects in $\mathsf{C}ob^{\ast}(M)$ (it can be imagined as a family of generators with respect to the triangulated structure).
Intuitively, the pseudo-distance $d^{\mathcal{F}}(L,L')$ between two objects $L$ and $L'$ infimizes 
the weight of iterated cone-decompositions that express $L$ in terms of $L'$ and of objects 
from the family $\mathcal{F}$.
The category $\mathsf{C}ob^{\ast}(M)$ is called (strongly) {\em rigid} (or non-degenerate) in case, under some natural constraints on $\mathcal{F}$ and a second such family $\mathcal{F}'$ (that can be viewed as 
a generic perturbation of $\mathcal{F}$), the pseudo-metrics 
$d^{\mathcal{F},\mathcal{F}'}=d^{\mathcal{F}}+d^{\mathcal{F}'}$ are non-degenerate.

\subsubsection{Algebra}
The algebraic side of the dictionary is mainly given by the derived Fukaya category $D\fuk^{\ast}(M)$ associated to the {\em embedded} Lagrangians $\mathcal{L}ag^{\ast}_{e}(M)\subset \mathcal{L}ag^{\ast}(M)$.  The definition of this triangulated category \cite{Se:book-fukaya-categ} is an algebraically sophisticated application of $J$-holomorphic curves techniques. The construction is only possible for Lagrangian submanifolds that are {\em unobstructed} \cite{FO3:book-vol1}.  Starting with the definition of Floer theory \cite{Fl:Morse-theory}, a variety of constraints, often of topological type - such 
as exacteness, monotonicity etc -  have been identified that are sufficient to ensure unobstructedness. 
Further unobstructedness conditions are now understood for Lagrangian immersions following the work
of Akaho \cite{Akaho} and Akaho-Joyce \cite{Akaho-Joyce}.

\subsection{Main result} In this framework, the central point of our dictionary is: 
\begin{mainthm}\label{thm:BIG} There are certain classes $\mathcal{L}ag^{\ast}(L)$, $\mathcal{L}ag^{\ast}(\C\times M)$  of {\em exact, marked, unobstructed, immersed} Lagrangian submanifolds and, respectively, cobordisms such that the category $\mathsf{C}ob^{\ast}(M)$ has surgery models 
and is rigid. Moreover, the quotient category $\widehat{\mathsf{C}}ob^{\ast}(M)$ is triangulated and its subcategory
generated by the embedded Lagrangians is triangulated isomorphic to $D\fuk^{\ast}(M)$.  
\end{mainthm}

One way to look at this result is that if the class $\mathcal{L}ag^{\ast}(M)$ is small enough, the metric structure is non-degenerate but defining a triangulated structure  on $\widehat{\mathsf{C}}ob^{\ast}(M)$ requires more objects than those in $\mathcal{L}ag^{\ast}(M)$.  On the other hand, if the class $\mathcal{L}ag^{\ast}(M)$ is too big, the triangulated structure is well defined but the expected metric is degenerate. 
However, the unobstructed class in the statement satisfies both properties  - it is flexible enough so that triangulation is well-defined and, at the same time, it exhibits enough rigidity so that the metric is non-degenerate. Additionally, in this case the sub-category generated by embedded objects agrees with the derived Fukaya category and, thus, most algebraic structures typical in Lagrangian Floer theory can be reduced to understading how unobstructedness behaves with respect to the geometric operations appearing in 
our axioms.  

Among immediate consequences
of this result we see that the natural cobordism group generated by the embedded Lagrangians, modulo
the relations given by the (possibly immersed) cobordisms in $\mathcal{L}ag^{\ast}(\C\times M)$, is isomorphic to the
Grothendieck  group $K_{0}(D\fuk^{\ast}(M))$.  In a different direction, because the Lagrangians in $\mathcal{L}ag^{\ast}(M)$ are unobstructed, there is a Donaldson category $\mathcal{D}on(\mathcal{L}ag^{\ast}(M))$ associated to this family and this category is isomorphic to $\widehat{\mathsf{C}}ob ^{\ast}(M)$.

\

The proof of Theorem \ref{thm:BIG} relates geometric constructions in $\widehat{\mathsf{C}}ob^{\ast}(M)$ to corresponding operations in 
$D\fuk^{\ast}(M)$. All objects and morphisms in $D\fuk^{\ast}(M)$ are represented geometrically and all exact triangles in $D\fuk^{\ast}(M)$ are represented by surgery. There are also relations between the pseudo-metrics $d^{\mathcal{F},\mathcal{F}'}$ 
and corresponding pseudo-metrics defined on the algebraic side by making use of action-filterered  $A_{\infty}$
machinery.  

\

This language also allows to formulate what today seems a very strong conjecture (or, maybe, just an intriguing question) that is a formal version in this context of Yasha Eliashberg's claim that ``there is no rigidity beyond $J$-holomorphic curves'':
\begin{cnj}\label{conj:unobstr} If $\mathsf{C}ob^{\ast}(M)$ has surgery models and is rigid, then it consists of unobstructed Lagrangians
and, in particular, the morphisms in $\widehat{\mathsf{C}}ob^{\ast}(M)$ are Floer homology groups.
\end{cnj}

\subsection{Some details} We discuss briefly here some of the main points in our approach.

\subsubsection{Triangulation} As mentioned before, there is a simple reason why triangulated categories are the correct framework for organizing Lagrangian submanifolds. The simplest operation with such submanifolds, surgery, produces a Lagrangian cobordism with three ends, two being the Lagrangians that are inserted in the construction and the third being the result of the surgery. By ``shuffling'' appropriatedly the three ends (see Figure \ref{Fig:Extr}) one immediately extracts a triple of maps that can be expected to give rise, in the appropriate setting, to a distinguished triangle. This simple remark implies that, to aim towards building a triangulated category, the class of Lagrangians considered 
has to be closed under the particular type of surgery used. For two Lagrangians $L$ and $L'$ that 
intersect transversely, surgery is possible along any subset $c\subset L\cap L'$ of their intersection 
points.  However, it is clear that without some constraints on these sets $c$ the resulting category will not be 
rigid. So the question becomes how to constrain the choices of these $c$'s allowed in the surgeries 
so that the output is rigid. The notion of unobstructedness comes to the forefront at this point: if $L$ and $L'$ are 
unobstructed the Floer chain complex $CF(L,L')$ is defined and one can chose $c$ to be a cycle in this complex. One expects in this case that the output of the surgery is unobstructed and that the surgery
cobordism itself is unobstructed. This process can then be iterated. 
Further along the line, if all these expectations are met, comparison with the relevant Fukaya category becomes credible as well as the rigidity of the resulting category. 

\subsubsection{Immersed Lagrangians} There are a number of difficulties with this approach. The first is that surgery in only some of the intersection points of $L$ and $L'$ produces an immersed Lagrangian, even if both $L$ and $L'$ are embedded. As a result, one is forced to work with unobstructed immersed Lagrangians. 
Unobstructedness for immersed objects depends on choices of data and, therefore, this leads to the complication that 
the objects in the resulting category can not be simply immersed Lagrangian submanifolds $L$ but rather pairs
$(L,\mathcal{D}_{L})$ consisting of such $L$ together with these data 
choices $\mathcal{D}_{L}$.  This turns out to be just a formal complication but,
more significantly, to prove unobstructedness of the result of the surgery and of the surgery cobordism,  
one needs to compare moduli spaces of $J$-holomorphic curves before and after surgery in a way that is 
technically very delicate. 

\subsubsection{Marked Lagrangians}\label{subsubsec:bounding} In this paper we pursue a different approach that bypasses this comparison. 
We further enrich our objects (thus the immersed Lagrangians $L$, assumed having only transverse double points) with one additional structure, a marking $\mathbf{c}$, which consists of 
a set of double points of $L$. Therefore, now our objects are triples $(L,\mathbf{c}, \mathcal{D}_{L})$.
The advantage of markings is the following: given two Lagrangians $L$ and $L'$ (possibly immersed and endowed with markings) and a set of intersection points $c\subset L\cap L'$, we can define a new immersed, marked Lagrangian by the union $L\cup L'$ endowed with a marking consisting of $c$ together with the markings of $L$ and $L'$. Geometrically, this corresponds to a surgery at the points in $c$ but using $0$-{\em size handles} in the process.
 Once markings are introduced, all the usual definitions of Floer complexes and
the higher $\mu_{k}$ operations have to be adjusted. While usual operations for immersed Lagrangians
count $J$-holomorphic curves that are not allowed to jump branches at self-intersection points, in the marked
version such jumps are allowed but only it they occur at points that are part of the marking. Experts will probably recognize that these markings are, in fact, just particular examples of bounding chains as in \cite{Akaho-Joyce} - see \S\ref{subsubec:bound-inv} for more details.  Keeping track of the markings and the data $\mathcal{D}_{L}$ through all the relevant operations requires roughly the same type of choices and the same effort as needed  to build coherent, regular perturbations in the construction 
of the Fukaya category.  Indeed, the process produces geometric representatives of the modules in $\mathcal{D}\fuk^{\ast}(M)$ that
are unions of embedded Lagrangian  submanifolds  $L_{1}\cup\ldots \cup L_{k}$ (together with some marking given by relevant intersection points) and the data associated to such a union is similar to the choice of data associated to the family  $L_{1},\ldots, L_{k}$ in the construction of the Fukaya category $\fuk^{\ast}(M)$.
This is why we call {\em tautological} our solution to the conjecture mentioned at the beginning of the introduction. At the same time, the use of markings has also some conceptual, intrinsic
advantages (see also \S\ref{subsubsec:nonmarked-Lag}).

\subsubsection{Immersed cobordisms.} There are two other additional difficulties that need to be addressed. 
The first, is that cobordisms with immersed ends do not have isolated double points, at best they have clean self-intersections that contain some half-lines of double points. As a result, we need to consider some appropriate deformations that put them in generic position.  This leads to additional complications with respect to composing the morphisms in our category. In particular,
because the morphisms are represented by cobordisms together with the relevant deformations, the objects themselves have to be equipped with one additional piece of data that comes down to a germ of this deformation along the ends.

\subsubsection{Cabling.} The second difficulty has to do with the equivalence relation $\sim$
(cabling equivalence) that is defined on the morphisms of
$\mathsf{C}ob^{\ast}(M)$. This relation is defined through an operation
called cabling that is applied to two cobordisms $V:L\cobto (L_{1},\ldots, L_{m},L')$ and 
$V':L\cobto (L'_{1},\ldots, L'_{s},L')$  (see Figure \ref{Fig:morph}) and produces a new Lagrangian $\mathcal{C}$ in $\C\times M$ pictured in Figure \ref{Fig:cabling}. Both $V, V'\in \mathcal{L}ag^{\ast}(\C\times M)$ are viewed in our category as morphisms $\in\mor (L, L')$  and they are cabling equivalent if their cabling also belongs to $\mathcal{L}ag^{\ast}(\C\times M)$. It is easy to see that $\mathcal{C}$ is generally immersed (even if $V$ and $V'$ are both emebdded) and that there are, in general,  holomorphic curves  $u: D^{2}\backslash \{1\} \to \C\times M$ with boundary on $\mathcal{C}$ and with
the boundary puncture that is sent to one of the self-intersection points of $\mathcal{C}$. 
Such curves are called tear-drops.
 In the unobstructed setting, the counts of these tear-drops, at each self-intersection point, need to vanish.
 The difficulty is to ensure that the relevant moduli spaces are regular so that these counts are possible. We deal with this issue by using appropriate interior marked points (having to do with the planar
 configurations associated to cabling) and use standard arguments for regularity.

\subsection{Structure of the paper.}
In the second section we review the geometric cobordism category
and list the axioms mentioned above. We also prove that, when the axioms are satisfied, the resulting quotient category is triangulated. We also indtroduce the shadow fragmentation pseudo-metrics. All the constructions in this part are ``soft''. In the third section we discuss unobstructedness for both Lagrangians and cobordisms and we also define the classes $\mathcal{L}ag^{\ast}(M)$ and $\mathcal{L}ag^{\ast}(\C\times M)$ that 
appear in the statement of Theorem \ref{thm:BIG}, the constructions here are ``hard''. 
In the fourth section we pursue the  correspondence {\em geometry  $\leftrightarrow$ algebra} and  prove Theorem \ref{thm:BIG}. In the last section we discuss some further technical points and open questions. 

\

\noindent {\bf Acknowledgements.} This work started during the academic year 2015-2016 that the second author has spent at the IAS where the first author was also hosted during  two weeks in the Fall of 2015.  Both authors thank the {\em Institute for Advanced Study} and, in particular, Helmut Hofer for their hospitality. The second author also thanks the {\em FIM} in Z\"urich for its hospitality during frequent research visits.

% !TEX root = ImmersedS.tex

\section{Cobordism categories with surgery models.}\label{sec:cob}
\subsection{Objects, morphisms and basic operations.}\label{subsubsec:basic-def}
The definition of Lagrangian cobordism that we use, as in \cite{Bi-Co:cob1}, is a variant of a notion first 
introduced by Arnold \cite{Ar:cob-1,Ar:cob-2}. 

\

We fix on $\mathbb{R}^2$  the symplectic
structure $\omega_{0} = dx \wedge dy$, $(x,y) \in
\mathbb{R}^2$ and on $\mathbb{R}^2 \times M$  
the symplectic form $\omega_{0} \oplus \omega$. 
We denote by $\pi: \mathbb{R}^2 \times M \to \mathbb{R}^2$ be the projection. 
For a subset $V \subset \mathbb{R}^2 \times M$ and $S \subset \mathbb{R}^2$
we let $V|_{S} = V \cap \pi^{-1}(S)$.

\begin{dfn}\label{def:Lcobordism}
   Let $(L_{i})_{1\leq i\leq k_{-}}$ and $(L'_{j})_{1\leq j\leq
     k_{+}}$ be two tuples of closed Lagrangian submanifolds of
   $M$. We say that that these two tuples are Lagrangian
   cobordant, $(L_{i}) \simeq (L'_{j})$, if there exists a smooth
   compact cobordism $(V;\coprod_{i} L_{i}, \coprod_{j}L'_{j})$ and a
   Lagrangian embedding $V \subset ([0,1] \times \mathbb{R}) \times M$
   so that for some $\epsilon >0$ we have:
   \begin{equation} \label{eq:cob_ends}
      \begin{aligned}
         V|_{[0,\epsilon)\times \mathbb{R}} = & \coprod_{i} 
         ([0, \epsilon) \times \{i\})  \times L_i \\
         V|_{(1-\epsilon, 1] \times \mathbb{R}} = 
         & \coprod_{j} ( (1-\epsilon,1]\times \{j\}) \times L'_j~.~
      \end{aligned}
   \end{equation}
   The manifold $V$ is called a Lagrangian cobordism from the
   Lagrangian tuple $(L'_{j})$ to the tuple $(L_{i})$. We 
   denote such a cobordism by $V:(L'_{j}) \cobto (L_{i})$.
   \end{dfn}
   
We will allow below for both the cobordism $V$ and its ends $L_{i}, L'_{j}$ to be immersed manifolds.
The connected components of the ends $L_{i}$, $L'_{j}$ will be assumed to have only isolated, transverse double points. The connected components of the cobordism $V$ 
will be assumed to have double points that are of two types, either isolated transverse double points
or $1$-dimensional clean intersections (possibly diffeomorphic to semi-closed intervals). Such intervals
of double points are associated to the double points of the ends.

\

It is convenient to extend the ends of such a cobordism hotizontally in both directions and
thus to view $V$ as a (non-compact) sub-manifold of $\C\times M$.
A cobordism $V$ is called {\em simple} if it has a single positive end and a single negative end. Such a cobordism
$V:L\cobto L'$ can be rotated by using a $180^{o}$ - rotation in $\C$ and this provides a new simple cobordism
denoted $\bar{V}:L'\cobto L$.
The two most common examples of cobordisms are Lagrangian surgery - that will be discussed in more 
detail later in the paper - and Lagrangian suspension. This is a simple cobordism associated to a Hamiltonian isotopy $\phi$, $V^{\phi}:L\cobto \phi(L)$.  

\begin{figure}[htbp]
   \begin{center}
     \includegraphics[width=0.6\linewidth]{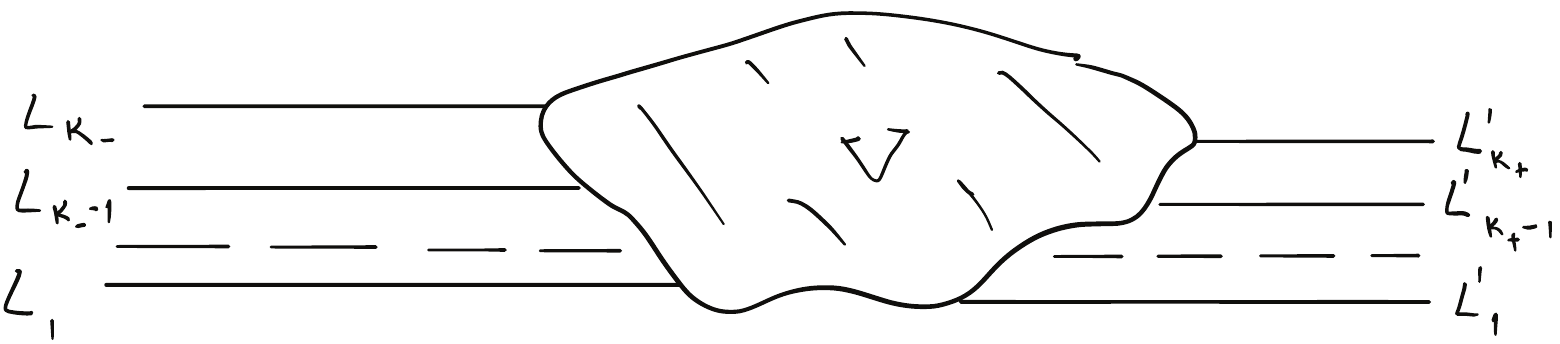}
   \end{center}
   \caption{A cobordism $V:(L'_{j})\cobto (L_{i})$ projected
     on $\mathbb{R}^2$.}
\end{figure}

The category $\mathsf{C}ob^{\ast}(M)$ has appeared before in \cite{Bi-Co:lcob-fuk}
with the notation $S\mathcal{C}ob$ that we shorten here for ease of notation (in that paper 
also appeared a slightly different category denoted $\mathcal{C}ob^{\ast}(M)$ that will not play a role here). 
There are a number of variants of this definition so we give a few details here.

\

The objects of this category are closed Lagrangian submanifolds in $M$, possibly immersed,
in a fixed class $\mathcal{L}ag^{\ast}(M)$, we allow $\emptyset \in \mathcal{L}ag^{\ast}(M)$.  
We also fix a class of cobordisms  $\mathcal{L}ag^{\ast}(\C\times M)$ with the property
that the ends of the cobordisms in this class belong to $\mathcal{L}ag^{\ast}(M)$. 
To fix ideas, the classes in $\mathcal{L}ag^{\ast}(M)$, $\mathcal{L}ag^{\ast}(\C\times M)$ consist of immersed Lagrangian submanifolds endowed with certain decorations (such as, possibly, an orientation, a spin structure or possibly other such choices) and subject to certain
contraints (such as exactness, monotonicity etc). An important type of less well-known decoration (for non-embedded Lagrangians) is a marking, as given in Definition \ref{def:marked}.

\

The category
$\mathcal{L}ag^{\ast}(\C\times M)$ is closed with respect to Hamiltonian isotopy horizontal at infinity
and it contains all cobordisms of the form $\gamma\times L$ where $L\in \mathcal{L}ag^{\ast}(M)$
and $\gamma$ is  a curve in $\C$ horizontal (and of non-negative integer heights) at $\pm \infty$.
This also implies that Lagrangian suspension cobordisms for Lagrangians in 
$\mathcal{L}ag^{\ast}(M)$ are in $\mathcal{L}ag^{\ast}(\C\times M)$ (the reason is that they
can be obtained by Hamiltonian isotopies of the trivial cobordism).
We will also assume that both classes $\mathcal{L}ag^{\ast}(M)$ and $\mathcal{L}ag^{\ast}(\C\times M)$
are closed under (disjoint) finite unions. 

\begin{rem} \label{rem:direct-sum} It is useful to be more explicit at this point concerning
the properties of the class of Lagrangians immersions $\mathcal{L}ag^{\ast}(M)$. 
We  assume that all the immersions $j_{L}: L\to M$ in $\mathcal{L}ag^{\ast}(M)$ are such that
the restriction of $j_{L}$ to any connected component of $L$ only has isolated, transverse double-points and that $L$ has a finite number of connected components.
Morevoer, as mentioned above, the class $\mathcal{L}ag^{\ast}(M)$
(as well as $\mathcal{L}ag^{\ast}(\C\times M)$) is required to be closed under disjoint union.
In other words, given two immersions $j_{L_{1}}:L_{1}\hookrightarrow M$, $j_{L_{2}}:L_{2}\hookrightarrow M$ in the class
$\mathcal{L}ag^{\ast}(M)$, the  obvious immersion defined on the disjoint union $j_{L_{1}}\sqcup j_{L_{2}}: L_{1}\sqcup L_{2}\hookrightarrow M$ is also in the class
$\mathcal{L}ag^{\ast}(M)$. This operation gives rise to Lagrangians with non-isolated double points
(for instance, by ``repetition'' of the same immersion $j_{L}:L\to M$) but, clearly, the restriction of the resulting immersion to the connected components of the domain only have isolated, transverse double-points if $j_{L_{i}}$, $i=1,2$ have this property. 

A similar, definition applies to the class $\mathcal{L}ag^{\ast}(\C\times M)$ that consists of 
immersed cobordisms such that each connected component of such an immersion has only isolated, transversal self-intersection points or, possibly, clean intersections of dimension $1$.
\end{rem}

\

The morphisms in $\mathsf{C}ob^{\ast}(M)$ between two objects $L$ and $L'$ are the cobordisms $V:L\cobto (L_{1},\ldots, L_{m}, L')$ (with $m$ arbitrary), $V\in \mathcal{L}ag^{\ast}(\C\times M)$ modulo Hamiltonian isotopy horizontal at infinity (see again \cite{Bi-Co:lcob-fuk}), as in Figure \ref{Fig:morph}
(and possibly modulo some additional identifications, having to do with the decorations included in $\ast$). When viewing a cobordism $V$ as before as a morphism $V:L\to L'$ we will refer to the
ends $L_{1},\ldots, L_{m}$ as the {\em secondary ends} of the morphism $V$.

\begin{figure}[htbp]
   \begin{center}
     \includegraphics[width=0.6\linewidth]{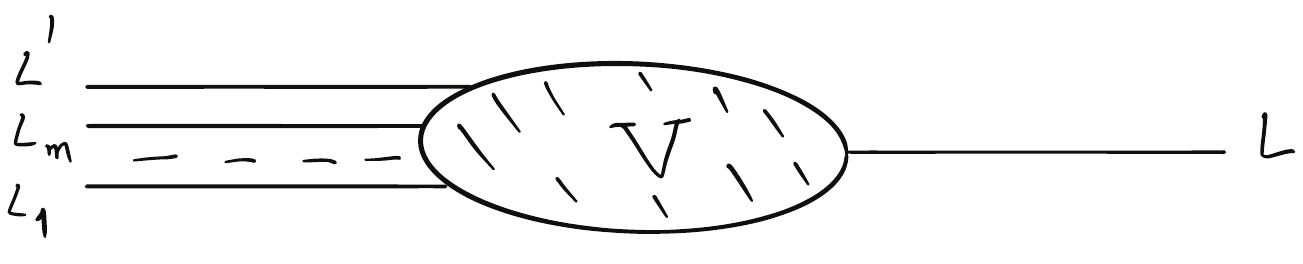}
   \end{center}
   \caption{A a morphism $V:L\to L'$. \label{Fig:morph}}
\end{figure}

\

The composition of two morphisms $V:L\to L'$, $V':L'\to L''$ is given by the composition of the cobordism
$V$ with $V'$ by gluing the upper ``leg'' of $V$ to the input of $V'$, as in Figure \ref{Fig:Comp}.
\begin{figure}[htbp]
   \begin{center}
   \includegraphics[width=0.6\linewidth]{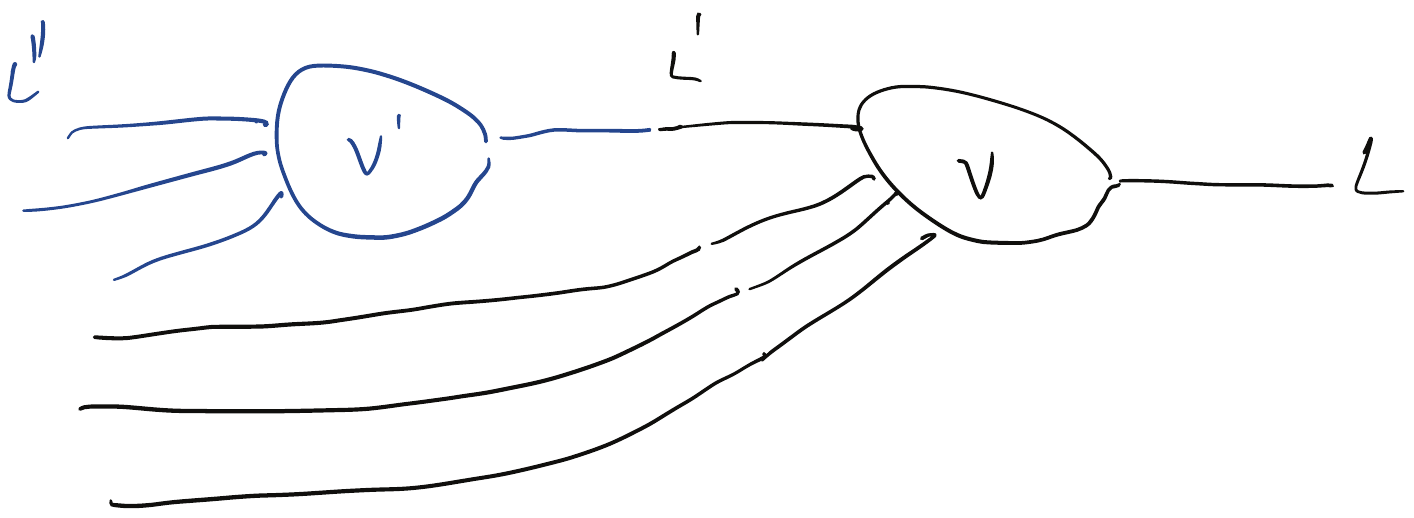}
   \end{center}
   \caption{\label{Fig:Comp} The composition $V'\circ V$ of $V:L\to L'$ and $V':L'\to L''$.}
\end{figure}
We assume that composition of the cobordisms in the class $\mathcal{L}ag^{\ast}(\C\times M)$ along {\em any} possible ``leg''  - not only on the top leg of $V$ - preserves the class $\mathcal{L}ag^{\ast}(\C\times M)$.

\

As mentioned before, the connected immersed objects in $\mathcal{L}ag^{\ast}(M)$ are supposed to only have (isolated) self-transverse 
double points and the connected immersed objects in $\mathcal{L}ag^{\ast}(\C\times M)$ have clean intersections with 
the double points along manifolds (possibly with boundary) of dimension at most $1$.

\

We next discuss a few basic notions related to this framework.

\subsubsection{Distinguished triangles.} \label{subsubsec:tri}
Cobordisms $V\in \mathcal{L}ag^{\ast}(\C\times M)$ with three ends $V:L\to (L'',L')$  play an 
important role here. Let $u:L\to  L'$ be the morphism represented by $V$. There are  two other
morphisms $Ru: L'\to L''$ and $R^{-1}u:L''\to L$ that are obtained by ``shuffling'' the ends of $V$ 
as in Figure \ref{Fig:Extr}. 

\begin{figure}[htbp]
   \begin{center}
   \includegraphics[width=0.6\linewidth]{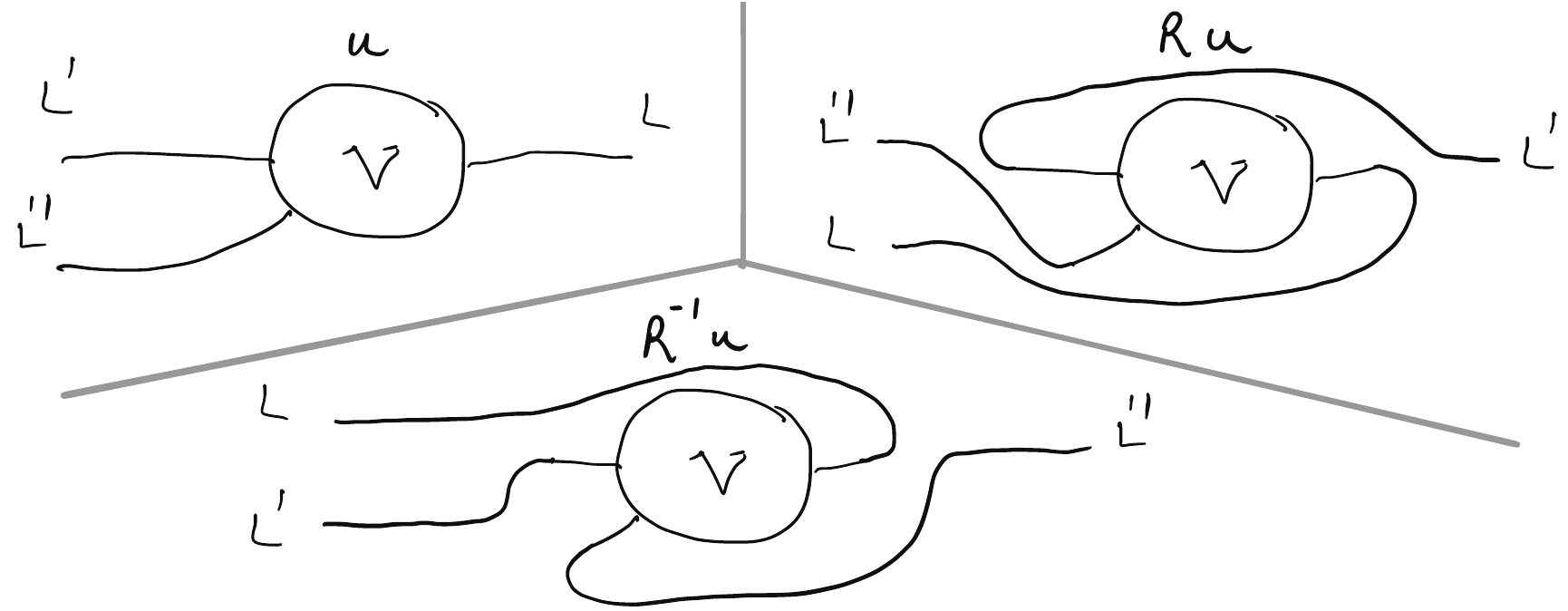}
   \end{center}
   \caption{\label{Fig:Extr} The three morphisms $(R^{-1}u,u, Ru)$ in a distinguished triangle.}
\end{figure}
The triple $(u,Ru,R^{-1}u)$ will be called a {\em distinguished} triangle in $\mathsf{C}ob^{\ast}(M)$
and we refer to the shuffling $R$ as a rotation (operator).
\begin{rem}\label{rem:comp}
For a cobordism $V:L\to (L'',L')$ that represents a morphism $u:L\to L'$ as above we have
$RR^{-1}u= u, R^{3}u=u$ (recall that two cobordisms induce the same morphism if they are 
horizontally Hamiltonian isotopic). In particular, an alternative writing of an exact 
triangle $(u,Ru, R^{-1}u)$ as above is $(u,Ru, R^{2}u)$
\end{rem}
For a general cobordism $V: L\to (L_{1},\ldots, L_{m})$ inducing a morphism $v:L\to L_{m}$
we let $Rv:L_{m}\to L_{m-1}$ be the morphism induced by the cobordism obtained from $V$ by shuffling the
$L$, $L_{m}$ and $L_{m-1}$ ends as in Figure \ref{Fig:GenRot} (after shuffling $L$ becoms
the bottom negative end; $L_{m}$ is the unique positive end; $L_{m-1}$ is the top negative end). 
\begin{figure}[htbp]
   \begin{center}
   \includegraphics[width=0.4\linewidth]{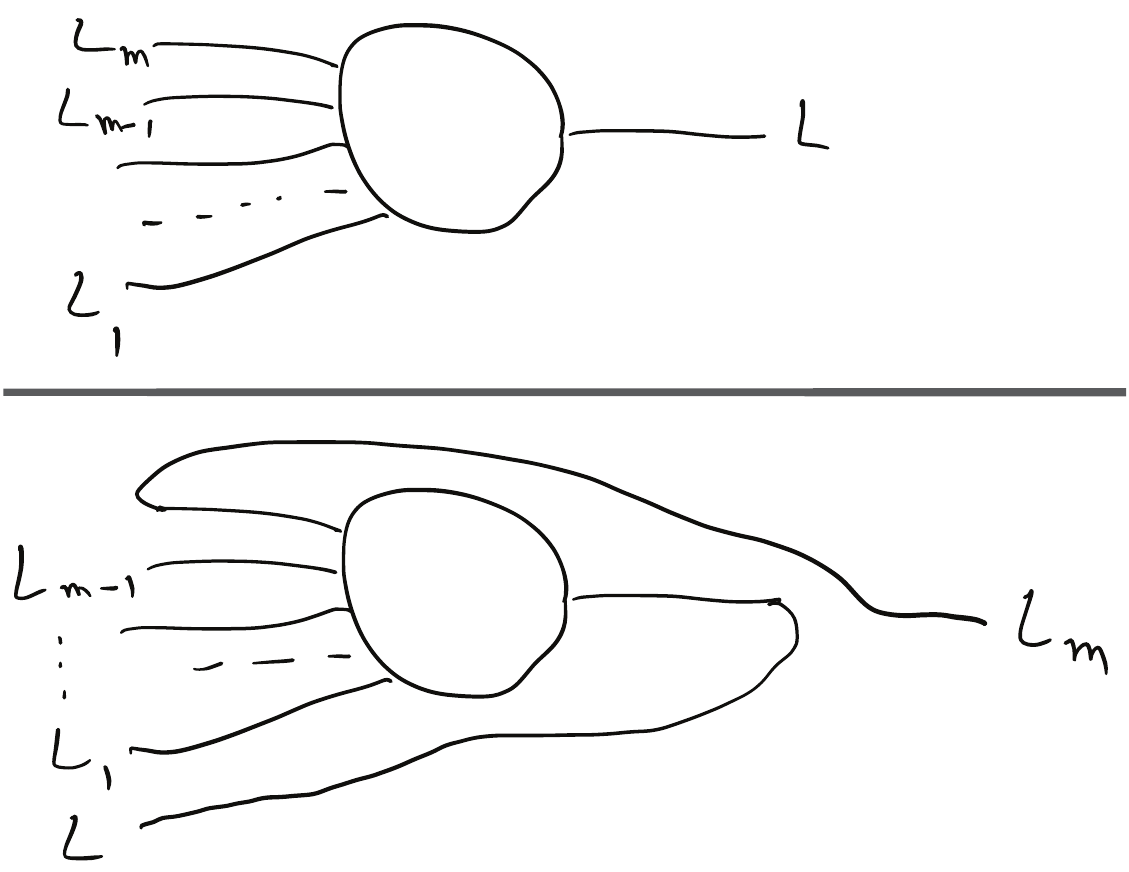}
   \end{center}
   \caption{\label{Fig:GenRot} The morphism $v$ on top and $Rv$ at the bottom.}
\end{figure}
We have obvious
relations similar to those in Remark \ref{rem:comp}. In this case too, for instance $R^{m+1}v=v$.
More generally, given a cobordism $V$ as above  there are unique morphisms (defined up to horizontal Hamiltonian isotopy) $L_{j+1}\to L_{j}$ and $L_{1}\to L$
that are associated to $V$ and defined as an appropriate iterated rotation of $v$. To simplify 
terminology, we will refer in the following to such a morphism $L_{j+1}\to L_{j}$ or $L_{1}\to L$
 associated to $V$ as the  {\em rotation of} $v$ {\em with domain} $L_{j+1}$ (respectively $L_{1}$). For instance, in Figure \ref{Fig:Extr}
 at the right we have the rotation of $v$ with domain $L'$ and, at the bottom, the rotation with domain $L''$. 

We asssume that the class $\mathcal{L}ag^{\ast}(\C\times M)$ is closed with respect to these rotations.

\subsubsection{Surgery and marked Lagrangians.}\label{subsec:surg}
Lagrangian surgery, as introduced by Lalonde-Sikorav \cite{La-Si:Lag} for surfaces and Polterovich \cite{Po:surgery} in full generality,
is one of the basic operations with Lagrangian submanifolds. It was remarked in \cite{Bi-Co:cob1} that the 
trace of a Lagrangian surgery is itself a Lagrangian cobordism. We will use the conventions in \cite{Bi-Co:cob1}
and will apply the construction in a situation a bit more general than \cite{Bi-Co:cob1}  to two - possibly immersed -
Lagrangians $L_{1}$ and $L_{2}$. We assume that $L_{i}$, $i=1,2$ has only self-transverse double points
and that $L_{1}$ and $L_{2}$ intersect transversely. The surgery will be performed on a subset $c$ of the 
intersection points of $L_{1}$ and $L_{2}$ and the outcome will be denoted by $L_{1}\#_{c,\epsilon}L_{2}$
and is again a (generally) immersed Lagrangian with self-transverse double points. 
The $\epsilon$ that appears in the formula is a parameter that is used to keep track of the size of the handle used in the surgery (it equals the area of the region encompassed by the handle
in $\C$). We consider here only surgeries using the same model handle for all double points 
in $c$.

The projection on $\C$  of the trace of a surgery is as in Figure \ref{Fig:surg}.
\begin{figure}[htbp]
   \begin{center}
    \includegraphics[width=0.5\linewidth]{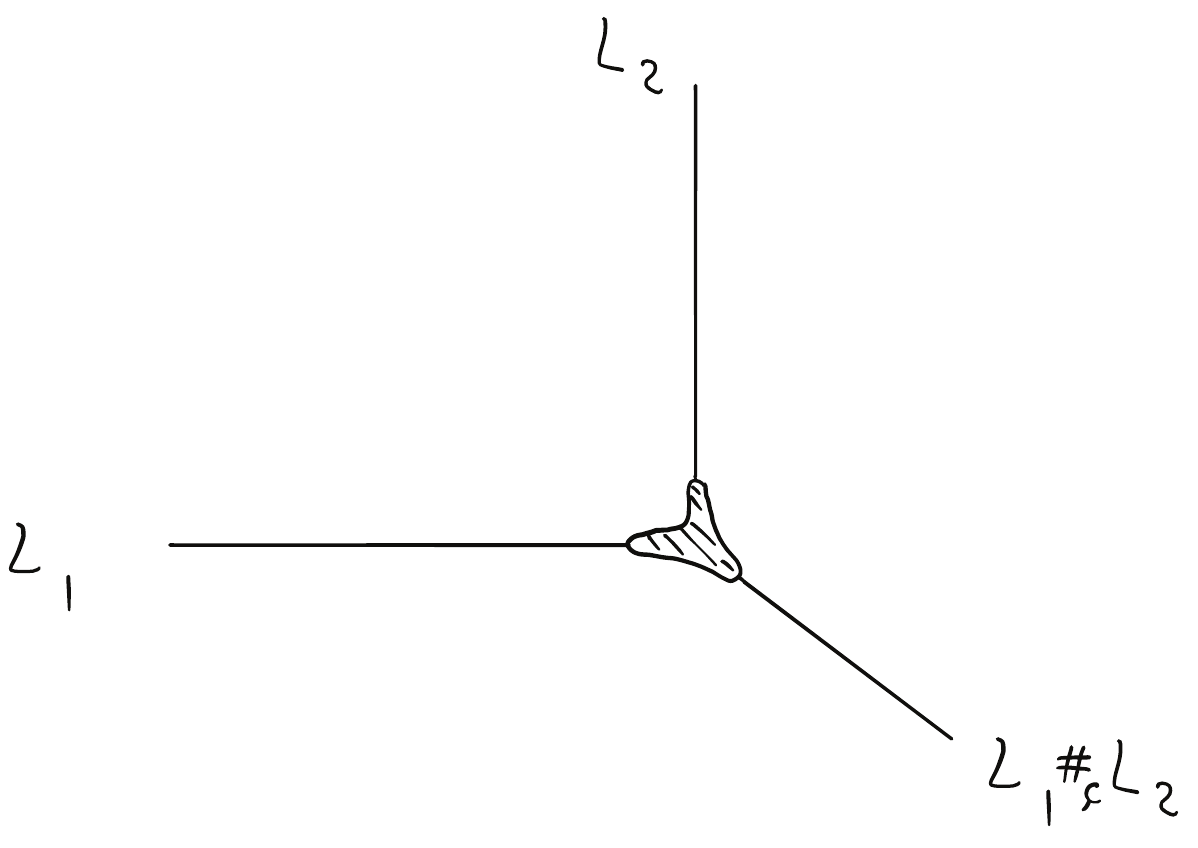}
   \end{center}
   \caption{\label{Fig:surg} The projection of the trace of a surgery onto $\C$.}
\end{figure}
The model for the surgery in the neighbourhoods of the points where the surgery handles are attached is as described
in \cite{Bi-Co:cob1}. The projection of this handle attachement is (up to a small smoothing) as the left drawing in  Figure \ref{Fig:surg2} below.  On regions away from the surgered intersection points the local model 
is as in the drawing at the right in Figure \ref{Fig:surg2}. To glue the two local models a further perturbation is  required in the direction of the horizontal and vertical axes which leads to a projection 
as in Figure \ref{Fig:surg}.

\begin{figure}[htbp]
   \begin{center}
   \includegraphics[width=0.7\linewidth]{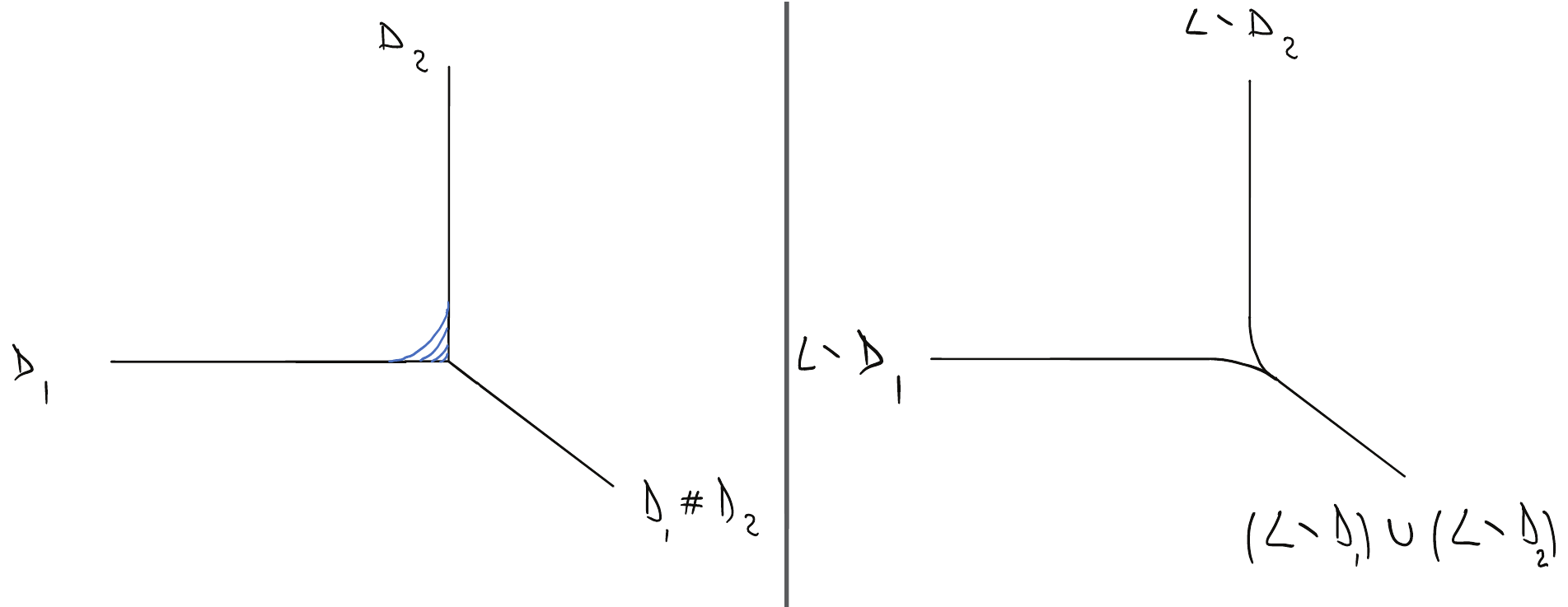}
   \end{center}
   \caption{\label{Fig:surg2} The local models for the trace of the surgery around points where the surgery 
   takes place, at the left, and
   where nothing happens, at the right.}
\end{figure}
In particular, it is clear that even if $L_{1}, L_{2}$ are embedded but $c$ does not consist of all intersection 
points $L_{1}\cap L_{2}$, then the trace of the surgery has double points along an interval of type $[0,\infty)$.

\

Given that we are working with immersed Lagrangians, it is natural to consider 
also ``formal'' $0$-size surgeries $L_{1}\#_{c} L_{2}:= L_{1}\#_{c,0} L_{2}$ that correspond to  $\epsilon=0$.
In this case $L_{1}\#_{c}L_{2}$ is, by definition, the immersed Lagrangian $L_{1}\cup L_{2}$ together with the set $c$ viewed as an additional decoration.   This 
seemingly trivial operation will play an important role later in the paper and we formalize the structure of $L_{1}\#_{c}L_{2}$ in the next definition.

For an immersed Lagrangian $j_{L}:L\to M$ with only transverse, isolated double points we denote by $I_{L}\subset L\times L$ the set of double points of $L$:
$$I_{L}= \{(x,y)\in L\times L : j_{L}(x)=j_{L}(y)\}~.~$$ 

\begin{dfn} \label{def:marked}
A marked Lagrangian immersion $(L,c)$ in $M$ consists of a  Lagrangian immersion $j_{L}:L\to M$ 
with isolated, transverse double points and a choice of a subset $c\subset I_{L}$ called a marking of $L$.
\end{dfn}

The $0$-size surgery $L_{1}\#_{c}L_{2}$ is a marked Lagrangian in this sense, the relevant immersion is $L_{1}\sqcup L_{2}\hookrightarrow M$ and the marking $c$ consists of the family of intersection points   $c\subset L_{1}\cap L_{2}$ with each point $P \in c$ lifted back to $L_{1}$ and $L_{2}$ as a couple $(P_{-},P_{+}) \in L_{1}\times L_{2}\subset (L_{1}\sqcup L_{2})\times (L_{1}\sqcup L_{2})$.  
Once the class of Lagrangians is extended as before to marked ones, we may also consider surgery
as an operation with marked Lagrangians.  The construction discussed above 
extends trivially to his case. For instance, given two marked Lagrangians $(L_{1},c_{1}), (L_{2},c_{2})$, and $c\subset (L_{1},L_{2})$
a subset of interesection points of $L_{1}$ and $L_{2}$,
the $0$-size surgery $(L_{1},c_{1})\#_{c} (L_{2},c_{2})$ is the union $L_{1}\cup L_{2}$
marked by $c_{1}\cup c\cup c_{2}$. We also allow for $\epsilon$-surgeries
$(L_{1},c_{1})\#_{c,\epsilon} (L_{2},c_{2})$ with $\epsilon >0$ in which case the new marking is 
simply $c_{1}\cup c_{2}$. A non-marked Lagrangian is a special case of a marked Lagrangian whose marking is the void set.

\

Recall from Remark \ref{rem:direct-sum} that the elements
of $\mathcal{L}ag^{\ast}(M)$ are unions of immersions with double isolated, transverse double points.
A marking of such an element $L\in \mathcal{L}ag^{\ast}(M)$ consists of a set $c=\{(P_{-},P_{+})\}\subset L\times L$ such that if $L_{\pm}\subset L$ is respectively the component of $L$ containing
respectively $P_{\pm}$, then $L_{-}$ intersects $L_{+}$ transversely (or self-transversely in case $L_{-}=L_{+}$).  
In other words, the class $\mathcal{L}ag^{\ast}(M)$ consists of (possibly) marked Lagrangians in the sense of Definition \ref{def:marked} and their disjoint unions.

\

To define the surgery operation for two elements $(L_{1}, c_{1}), (L_{2},c_{2})\in \mathcal{L}ag^{\ast}(M)$, we impose the condition that the set  $c\not=\emptyset$ (that is used to define $\epsilon$-size surgery, $\epsilon\geq 0$) consists of a finite set of double points  $(P_{-},P_{+}) \in (L_{1},L_{2})$ such that the connected component of $L_{1}$ containing $P_{-}$ intersects transversely (and away from the double-points of $L_{i}$, $i=1,2$) the connected component of $L_{2}$ containing $P_{+}$. Notice that with this convention the disjoint union $j_{L_{1}}\sqcup j_{L_{2}}$ can be
viewed as marked $0$-size surgery with $c=\emptyset$ and it does not require any transversality among 
the components of $L_{1}$ and $L_{2}$.

\

The trace of a $0$-size surgery projects globally as the drawing on the right in Figure \ref{Fig:surg2}. This trace  is obviously an immersed Lagrangian such that each connected component has clean intersections 
having $1$-dimensional double-points manifolds corresponding to the intersection points $L_{1}\cap L_{2}$. In the following, we will need to use 
also marked cobordisms. Due to the presence of these $1$-dimensional double point sets, the definition 
of this notion is a bit more complicated, 
and we postpone it to \S\ref{subsubec:marked-cob}.  The trace of the $0$-size surgery
$L_{1}\#_{c}L_{2}$,  is marked in this sense in a particularly simple way, see 
 \S\ref{subsubsec:surg-mor}.  Both $L_{1}\#_{c} L_{2}$ as well as the trace of this surgery can obviously be viewed as limits of $L_{1}\#_{c,\epsilon} L_{2}$ and, respectively, of the corresponding surgery trace, when $\epsilon \to 0$. 
 
 \begin{rem}\label{rem:formal-surg} a. To include the $0$-size surgery and its trace inside a cobordism 
 category $\mathsf{C}ag^{\ast}(M)$ a bit more ellaboration is needed.
  In essence,  in this case, the class $\mathcal{L}ag^{\ast}(M)$ needs to contain
  marked Lagrangians (as  always, subject to further constraints included in the condition $\ast$),
 similarly the elements of $\mathcal{L}ag^{\ast}(\C\times M)$ are marked cobordisms.
A marked cobordism $(V,v)$ between two tuples of marked 
 Lagrangians $\{(L_{i},c_{i})\}_{i}$ and $\{(L'_{j},c'_{j})\}_{j}$ is a cobordism 
 as in Definition \ref{def:Lcobordism} with the additional property that the restriction of the
 marking $v$ to the ends $L_{i}$, $L'_{j}$ coincides, respectively, with the markings
 $c_{i}$, $c'_{j}$. Composition of marked cobordisms $V$, $V'$ is again a marked cobordism in the obvious way, by concatenating the markings to obtain a marking on  the gluing of $V$ and $V'$. Full details appear in \S\ref{subsubec:marked-cob}. Obviously, the class of usual ``non-marked'' Lagrangians and cobordisms is included in this set-up by simply associating to them the empty marking.
 Whenever discussing/using $0$-size surgeries below we will always assume that the classes of
 Lagrangians involved are as described above.
 
 b. While markings behave in a way similar to other structures  that are suitable to 
incorporation in cobordism categories such as choices of orientation or spin structure, 
there is a significant distinction: immersed marked Lagrangians are 
{\em more general} objects compared to  non marked ones. It is useful to imagine 
them as the result of a two stage extension process {\em  embedded} $\to $ {\em immersed} $\to$ {\em marked}.
 \end{rem}

 The trace of a surgery as in Figure \ref{Fig:surg} above will be viewed as a morphism in a cobordism  category by twisting the ends as in Figure \ref{Fig:surg3} below. This cobordism
(and the associated morphism) will be denoted by $S_{L_{2},L_{1};c,\epsilon}$. We will refer to $S_{L_{2},L_{1};c,\epsilon}$
as the {\em surgery morphism} from $L_{2}$ to $L_{1}$ relative to $c$ and to a handle of area $\epsilon$. The same conventions and notation apply to $0$-size surgeries, under the assumptions in Remark \ref{rem:formal-surg}. The morphism associated to a $0$-size surgery associated to $c$ is denoted by $S_{L_{2},L_{1}; c}$. In the figures  below we will not distinguish the case $\epsilon=0$ and $\epsilon > 0$.
\begin{figure}[htbp]
   \begin{center}
   \includegraphics[width=0.6\linewidth]{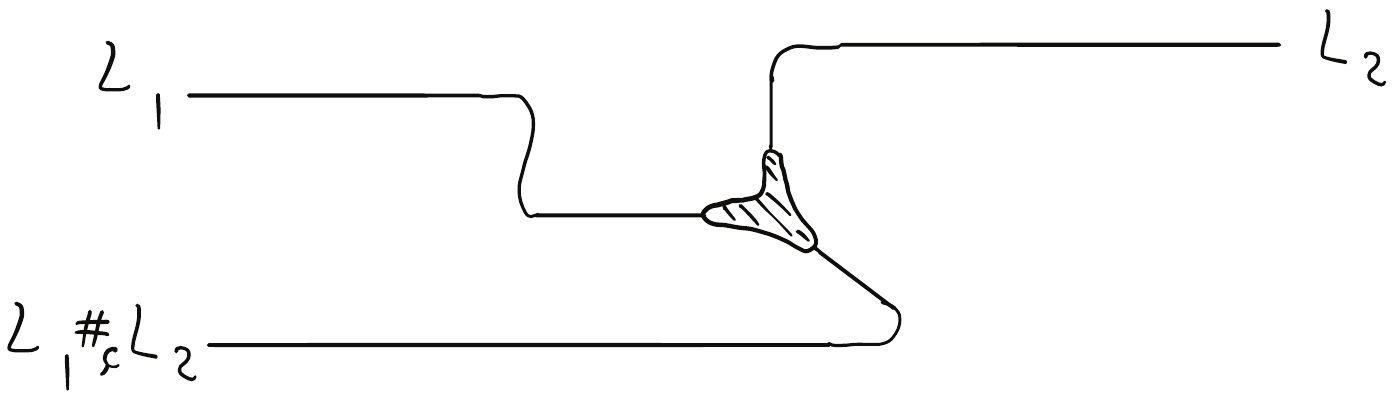}
   \end{center}
   \caption{\label{Fig:surg3} The morphism $S_{L_{2},L_{1};c,\epsilon}$ associated to the trace of the
  surgery of $L_{1}$ and $L_{2}$ along $c\subset L_{1}\cap L_{2}$.}
\end{figure}
   
Notice that the set of all $0$-size surgery morphisms from $L_{2}$ to $L_{1}$ as before are in bijection with the choices of the subset $c$ of intersection points of the two Lagrangians. In turn, we can view $c$ as an element in the $\Z/2$-vector space
 $\Z/2 <L_{1}\cap L_{2}>$. 
In short, we will identify the set of all $0$-size surgery morphisms from $L_{2}$ to $L_{1}$, $\bar{S}_{L_{2},L_{1}}$,  
with the vector space:

\begin{equation}\label{eq:surg-lin}
\bar{S}_{L_{2},L_{1}}\equiv \Z/2<L_{1}\cap L_{2}>~.~
\end{equation}
This identification endows $\bar{S}_{L_{2},L_{1}}$ with a linear structure. Assuming
fixed choices of handle models, small $\epsilon$, and fixed choices of Darboux charts around the intersection points of $L_{1}$ and $L_{2}$, the same applies to the set of $\epsilon$-size
surgery morphisms $S_{L_{2},L_{1}; c,\epsilon}$.  

\

In some of the constructions below we will consider Lagrangians in $\C\times M$ whose 
projection onto $\C$ will contain a picture as in Figure \ref{Fig:crossing} below. In other words,
the projection will contain two segments $\gamma_{1}$ and $\gamma_{2}$ that intersect transversely at one point $P\in \C$.
Their preimage in $\C\times M$ coincides with $\gamma_{1}\times L_{1}$, respectively 
$\gamma_{2}\times L_{2}$, with $L_{1},L_{2}$ two Lagrangian submanifolds in $M$ that intersect transversely.
For any point $x\in L_{1}\cap L_{2}$ we can operate two types of surgeries in $\C\times M$
 between $\gamma_{1}\times L_{1}$ and $\gamma_{2}\times L_{2}$ at the point $P\times x$. A {\em positive} surgery whose handle has a projection like the region in blue in the figure or a {\em negative} surgery 
 whose projection is as the region in red. More generally, we will talk about a $+$ surgery at $P$ (respectively,
 a $-$ surgery at $P$) to refer to simultaneous positive (respectively negative) surgeries along a subset $c\subset L_{1}\cap L_{2}$. Such a surgery will be denoted by $\#^{+}_{c,\epsilon}$ (respectively, $\#^{-}_{c,\epsilon}$) when the size of the handles is $\epsilon$. Again, we allow $\epsilon=0$ and the notation in that case is $\#^{\pm}_{c}$. In terms of the relevant marking the difference between a positive $0$-size surgery and a negative one at some point
 $P\in (\gamma_{2}\times L_{2})\cap (\gamma_{1}\times L_{1})$ is simply the 
 order of the two points $P_{-}\in \gamma_{1}\times L_{1}$, $P_{+}\in \gamma_{2}\times L_{2}$: the marking is $(P_{-},P_{+})$ for a positive surgery and
 $(P_{+},P_{-})$ for a negative one.
For positive surgeries we will generally ommit the superscript $^{+}$.

\begin{figure}[htbp]
   \begin{center}
      \includegraphics[width=0.3\linewidth]{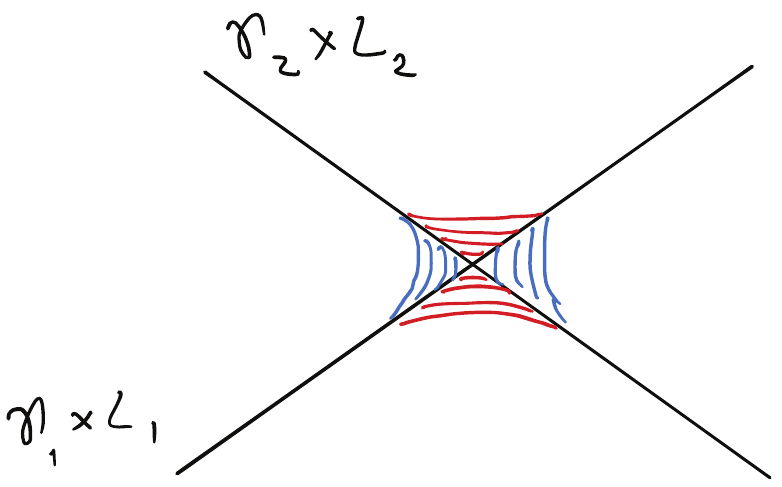}
   \end{center}
   \caption{\label{Fig:crossing} The crossing of $\gamma_{1}\times L_{1}$ and $\gamma_{2}\times L_{2}$
   and a positive surgery - in blue - and a negative surgery - in red - at the crossing point.}
\end{figure}

\subsubsection{Cabling.} \label{subsubsec:cables}
We now consider a simple construction, based on Lagrangian surgery,
that will play an important role in our framework.

Consider two cobordisms $V:L\cobto (L_{1},\ldots, L_{m}, L')$ and $V': L\cobto (L'_{1},\ldots, L'_{s},L')$ 
as in Figure \ref{Fig:twocob}  and 
assume that $L$ and $L'$ intersect transversely. 
\begin{figure}[htbp]
   \begin{center}
  \includegraphics[width=0.7\linewidth]{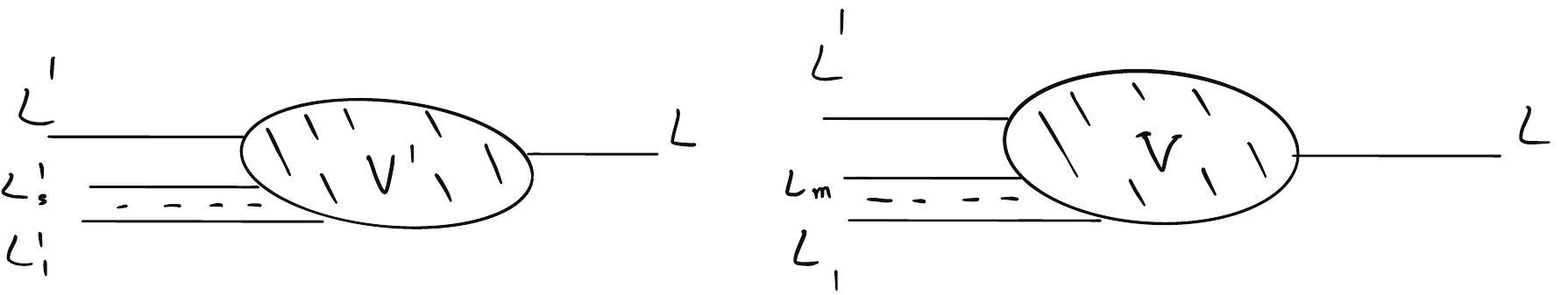}
   \end{center}
   \caption{\label{Fig:twocob} Two cobordisms $V:L\cobto (L_{1},\ldots, L_{m}, L')$, $V': L\cobto (L'_{1},\ldots, L'_{s},L')$.}
\end{figure}

A {\em cabling} $\mathcal{C}(V,V';c,\epsilon)$ of $V$ and $V'$ relative to $c\in L\cap L'$
(and relative to a handle of size $\epsilon$) is a cobordism
$\mathcal{C}(V,V';c,\epsilon) : L_{1}\to (L_{2},\ldots, L_{m},L'_{1},\ldots, L'_{s})$ whose projection is as in Figure
\ref{Fig:cabling}.
\begin{figure}[htbp]
   \begin{center}
   \includegraphics[width=0.7\linewidth]{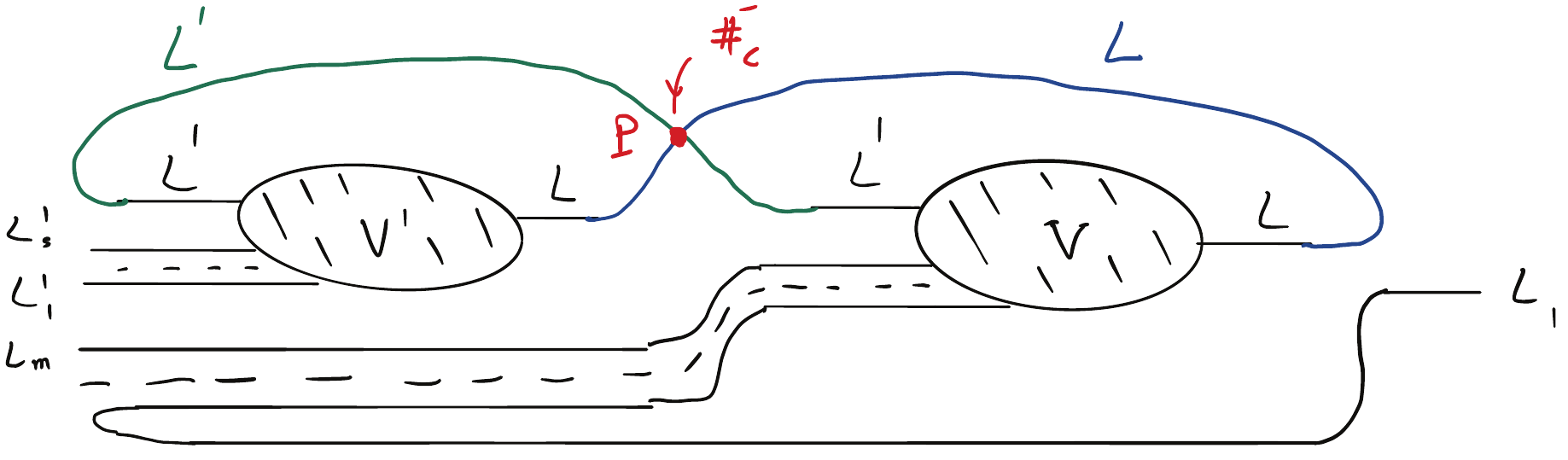}
   \end{center}
   \caption{\label{Fig:cabling} The cabling $\mathcal{C}(V,V';c)$ of $V$ and $V'$ relative to $c$.}
\end{figure}
More explicitely, the positive end of $V$ is bent, above $V$, towards the negative direction and is glued to the
positive end of $V'$. Similarly, the top negative end of $V'$ is bent above $V'$ towards the positive direction
and is glued to the top negative end of $V$. The extensions of these two ends intersect above a point $P$
in the plane in a configuration similar to that in Figure \ref{Fig:crossing}. At this point $P$ are performed
negative surgeries along the subset $c\subset L\cap L'$ (more precisely, $c\subset L\times L'\subset (V\sqcup V')\times (V\sqcup V')$).  Finally, the $L_{1}$ end is bent below $V$ 
in the positive direction. Again we allow for $\epsilon=0$ which leads to a marked cobordism
with the marking associated to negative surgeries of size $0$ along $c$. The notation in this case is $\mathcal{C}(V,V';c)$

\begin{rem}\label{rem:transversality}
a. As mentioned before, the surgery operation (and thus also cabling) depends not only on the choice of the subset $c$ of intersection points but also on the choice and size of surgery handles, on the choice of Darboux charts around the point of surgery and, in the case of cabling, on the choice of the bounded  regions inside the bent curves, on one side and the other of $P$, in Figure \ref{Fig:cabling}.

b. Surgery as well as cabling have been defined before only if the relevant Lagrangians are intersecting transversely.
However, if this is not the case one can use a small Hamiltonian isotopy to achieve transversality.  For instance, in the   cabling case, assume that $\phi$ is a small Hamiltonian isotopy such  that $\phi(L)$ intersects
transversely $L'$. In this case the region around $P$ in Figure \ref{Fig:cabling} is replaced with the configuration
on the left in Figure \ref{Fig:perturbed-cable}. The regions in grey are projections of small Lagrangian suspension cobordisms,
the one to the left of $P$ associated to $\phi$ and the one to right of $P$ associated to $\phi^{-1}$. In this case
$c\subset \phi(L)\cap L'$. We still use the notation $\mathcal{C}(V,V';c)$ to denote a cabling in this 
case even if it obviously also depends on the choice of $\phi$. A similar perturbation can be used for the surgery morphisms, as represented in the drawing on the right in Figure \ref{Fig:perturbed-cable}. Obviously, in both these perturbed cases $c\subset \phi(L)\cap L'$.
\end{rem}

\begin{figure}[htbp]
   \begin{center}
   \includegraphics[width=0.75\linewidth]{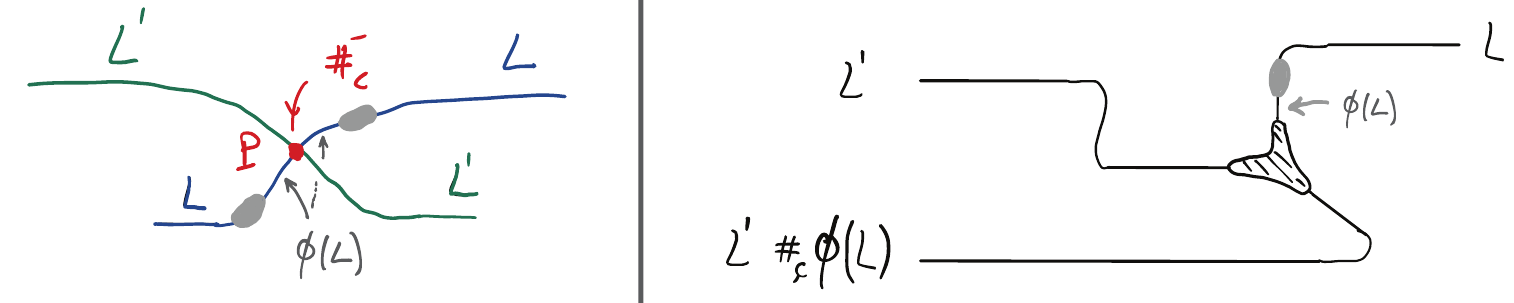}
   \end{center}
   \caption{\label{Fig:perturbed-cable} The perturbed cabling around $P$ using an isotopy $\phi$ such that 
   $\phi(L)$ intersects $L'$ transversely. Similarly, using a perturbation to define the surgery morphism $S_{L,L';c}$.}
\end{figure}

The cabling construction leads to the following definition that is central for us.

\begin{dfn}
\label{dfn:cabling-rel}
Let $V,V'$ be two cobordisms in $\mathcal{L}ag^{\ast}(\C\times M)$. We say that $V$ and $V'$ are related by cabling and write $V\sim V'$, if there exists a cabling $\mathcal{C}(V,V';c,\epsilon)\in \mathcal{L}ag^{\ast}(\C\times M)$. 
\end{dfn}
In case we need to indicate the precise cabling in question we will write
$$\mathcal{C}(V,V';c,\epsilon): V\sim V'~.~$$
We allow for $\epsilon$ to be $0$, and adopt the notation $$\mathcal{C}(V,V';c): V\sim V'$$ 
in this case. As discussed in Remark \ref{rem:formal-surg} - this requires that all the classes of Lagrangians and cobordisms under consideration be marked.  In case the context allows for it, we sometimes refer to a cabling simply by $\mathcal{C}(V,V'): V\sim V'$ with the understanding that this represents a cabling of $V$ and $V'$  associated to
an $\epsilon\geq 0$ and a $c$ as above.

\subsection{Axioms and surgery models for $\mathsf{C}ob^{\ast}(M)$ .} \label{subsec:surg-models}
We will say that the cobordism category $\mathsf{C}ob^{\ast}(M)$
{\em has surgery models} if it satisfies the following five axioms:

\begin{axio} (\underline{homotopy}) The cabling relation descends to $\mor_{\mathsf{C}ob^{\ast}(M)}$ and
it gives an equivalence relation $\sim$ that is preserved by the composition of cobordisms along any end, by the union of cobordisms as  well as by the insertion of the void end among the negative ends of a cobordism.\end{axio}

The first part of this axiom means that two cobordisms that are Hamiltonian isotopic by an isotopy 
horizontal at infinity are cabling equivalent and that the resulting relation on 
$\mor_{\mathsf{C}ob^{\ast}(M)}$ is an equivalence. 
We indend to make more explicit the further properties of this relation as implied by this axiom.
To simplify notation, from now on we will no longer distinguish between a morphism $u\in \mor_{\mathsf{C}ob^{\ast}(M)}(L,L')$ and the underlying cobordism $V: L\cobto (L_{1},\ldots, L_{m},L')$.

The relation $\sim$ is preserved by the composition of cobordisms in the following sense.
First, if $V\sim V'$ and $U$, $W$ are morphisms in $\mor_{\mathsf{C}ob^{\ast}(M)}$ such that the relevant compositions are defined, then $U\circ V \sim U\circ V'$ and $V\circ W\sim V'\circ W$.
Moreover, let $V:L\cobto (L_{1},\ldots, L_{m}, L')$
and consider another cobordism $K:L_{i}\cobto (L'_{1},\ldots, L'_{s})$ so that the composition
$V\circ_{i}K:L\to (L_{1},\ldots, L_{i-1},L'_{1},\ldots, L'_{s}, L_{i+1},\ldots, L_{m},L')$ is defined by
gluing $V$ to $K$ along the (secondary) $L_{i}$ end of $V$, then  $V\circ_{i}K$ viewed as morphism $:L\to L'$
has the property that $V\sim V\circ_{i}K$.  The union statement means that if
$V:L\cobto (L_{1},\ldots, L_{m}, L')$ can be written as a union of cobordisms
$V':L\cobto (L_{j_{1}},\ldots, L')$ and a second  cobordism $V'':\emptyset \cobto (L_{s_{1}},\ldots, L_{s_{r}})$ where the ends $L_{j_{i}}$, $L_{s_{k}}$ put together recover all the ends $L_{1},\ldots, L_{m}$, 
then $V\sim V'$. Finally, the last property claims that two cobordisms $V:L\cobto (L_{1},\ldots, L_{k}, L')$ and $V':L\cobto (L_{1},\ldots, L_{i},\emptyset, L_{i+1},\ldots, L_{k},L)$ that differ only by the insertion of a void end (and the corresponding rearrangement of the other ends), are equivalent.

\begin{axio}(\underline{cabling reduction}) Suppose that $V$ and $V'$ are two cobordisms
such that the cabling $\mathcal{C}(V,V';c,\epsilon)$ is defined and belongs to $\mathcal{L}ag^{\ast}(\C \times M)$.
Then $V,V'\in \mathcal{L}ag^{\ast}(\C \times M)$.
\end{axio}
By inspecting Figure \ref{Fig:cabling}, this means that, assuming $\mathcal{C}(V,V';c,\epsilon)$ is in 
$\mathcal{L}ag^{\ast}(\C \times M)$,  by erasing in that picture the point $P$ as well as the blue and
green arcs, we are left with a configuration that remains in $\mathcal{L}ag^{\ast}(\C \times M)$. 

\

For any $L\in \mathcal{L}ag^{\ast}(M)$ we denote by $id_{L}:L\to L$ the cobordism given 
by $\R\times L\subset \C\times M$.

\begin{axio}(\underline{existence of inverse}) If $V\in\mor_{\mathsf{C}ob^{\ast}(M)}(L,L')$ is 
a simple cobordism, then $$V\circ \bar{V}\sim id_{L}~.~$$
\end{axio}

Recall that $\bar{V}$ is  the cobordism obtained by rotating $V$ using a $180^{o}$ rotation 
in the plane.

\begin{axio}(\underline{surgeries}) 
For any $V\in \mor_{\mathsf{C}ob^{\ast}(M)}(L,L')$ there exists a surgery morphism $S_{L,L';c,\epsilon}\in\mathcal{L}ag^{\ast}(\C\times M)$ such that
$V\sim S_{L,L';c,\epsilon}$. Moreover, the sum of surgeries preserves $\mathcal{L}ag^{\ast}(\C\times M)$ and is compatible with $\sim$.
\end{axio}

As in Remark \ref{rem:transversality} b small Hamiltonian perturbations may be required for the surgery morphisms used here. This is particularly the case to associate to $id_{L}$ an equivalent surgery morphism.       
The second part of this axiom claims that the set $\bar{S}^{\ast}_{L,L'}= \bar{S}_{L,L'}\cap \mathcal{L}ag^{\ast}(\C\times M)$ of those surgeries that are in the class $\mathcal{L}ag^{\ast}(\C\times M)$ is a vector sub-space of $\bar{S}_{L,L'}$.  In particular, as $0=2b$ for any $b\in \bar{S}_{L,L'}$ this means that the $0$ surgery (that is the ``absence'' of surgery) is in $\bar{S}^{\ast}_{L,L'}$. The compatibility with $\sim$ means that the set of equivalence classes $\bar{S}^{\ast}_{L,L'}/\sim$ is a $\Z/2$-vector space.

\begin{rem}
In the main example discussed later in the paper $\epsilon$ in Axiom 4 can be taken to be $0$. However, we 
do not impose this condition in the axiom because this requirement would automatically imply that any 
category with surgery models necessarily consists of marked Lagrangians. 
\end{rem}
\

The next axiom is considerably more difficult to unwrap.

\begin{axio} (\underline{naturality})
Fix $V:L\to L'$, $V':K\to L'$, $V'': L\to K'$, $S:K\to L$, $S':L'\to K'$ in $\mor_{\mathsf{C}ob^{\ast}(M)}$ such that $S$ and $S'$ have exactly three ends.
If $\mathcal{C}(V\circ S,V'): V\circ S\sim V'$ belongs to $\mathcal{L}ag^{\ast}(\C\times M)$, then
 \begin{equation}\label{eq:cancellation1}
 \mathcal{C}(V\circ S,V')\circ RS\sim RV'\circ V~.~
 \end{equation}
 Similarly, if $\mathcal{C}(V'', S'\circ V): V''\sim S'\circ V$ belongs to $\mathcal{L}ag^{\ast}(\C\times M)$, then 
 \begin{equation}\label{eq:cancellation2}
 R^{-1}S'\circ \mathcal{C}(V'',S'\circ V)\sim V\circ R^{-1}V'' ~.~
 \end{equation}
\end{axio}

We explain in a geometric way this Axiom in  Figures \ref{Fig:ax5fig1}, \ref{Fig:ax5fig2}, \ref{Fig:ax5fig3} and \ref{Fig:ax5fig4}.

\begin{figure}[htbp]
   \begin{center}
    \includegraphics[width=0.65\linewidth]{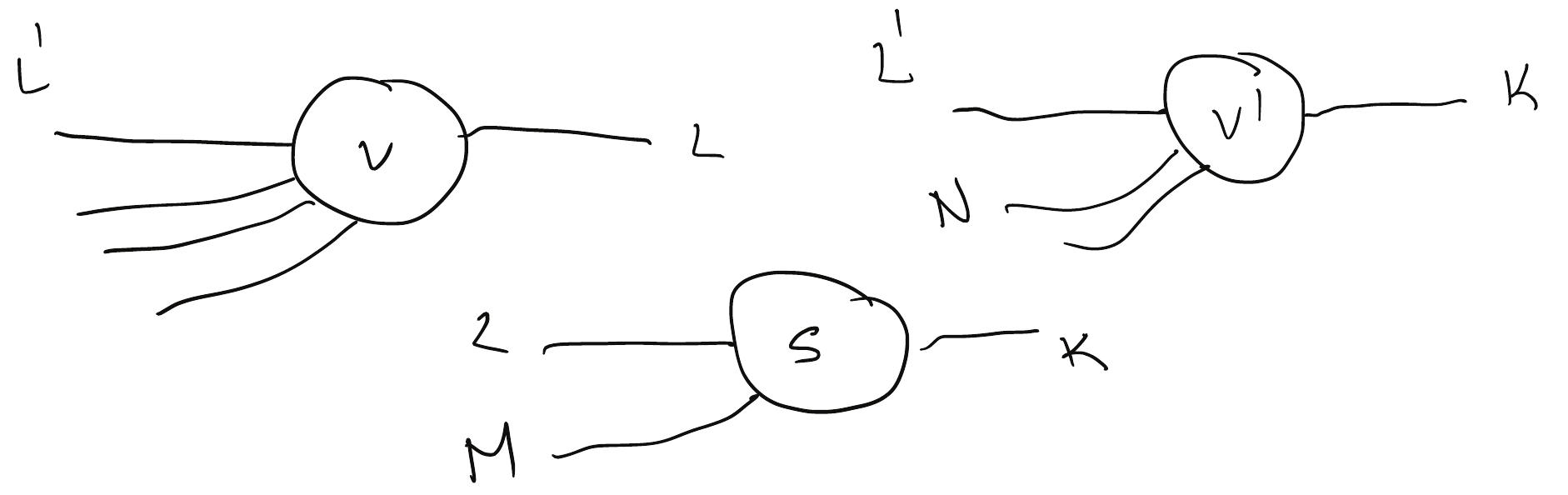}
   \end{center}
   \caption{\label{Fig:ax5fig1} The three cobordisms $V:L\to L'$, $V':K\to L'$ and $S:K\to L$.}
\end{figure}

\begin{figure}[htbp]
   \begin{center}
    \includegraphics[width=0.6\linewidth]{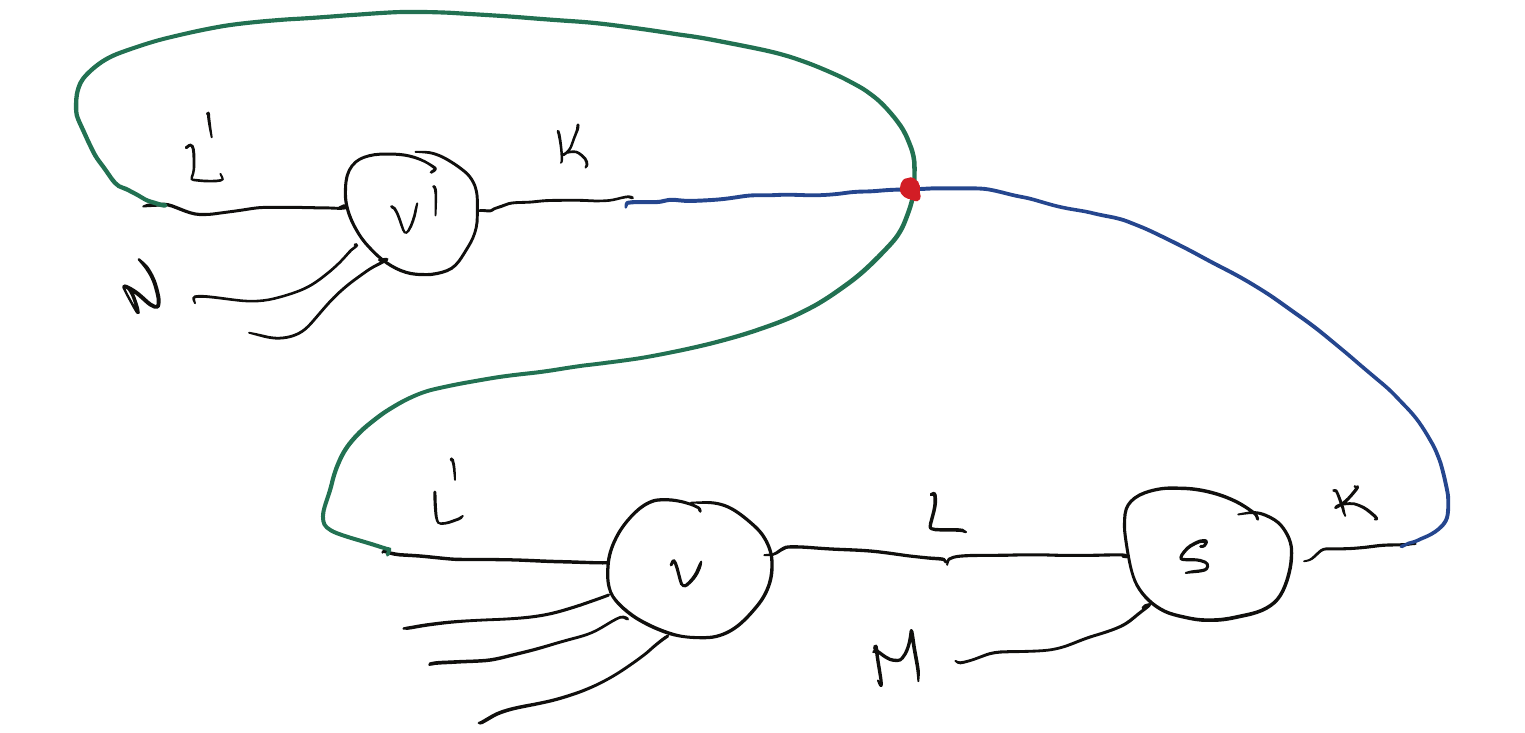}
   \end{center}
   \caption{\label{Fig:ax5fig2} The cabling $\mathcal{C}(V\circ S,V')$.}
\end{figure}

\begin{figure}[htbp]
   \begin{center}
    \includegraphics[width=0.65\linewidth]{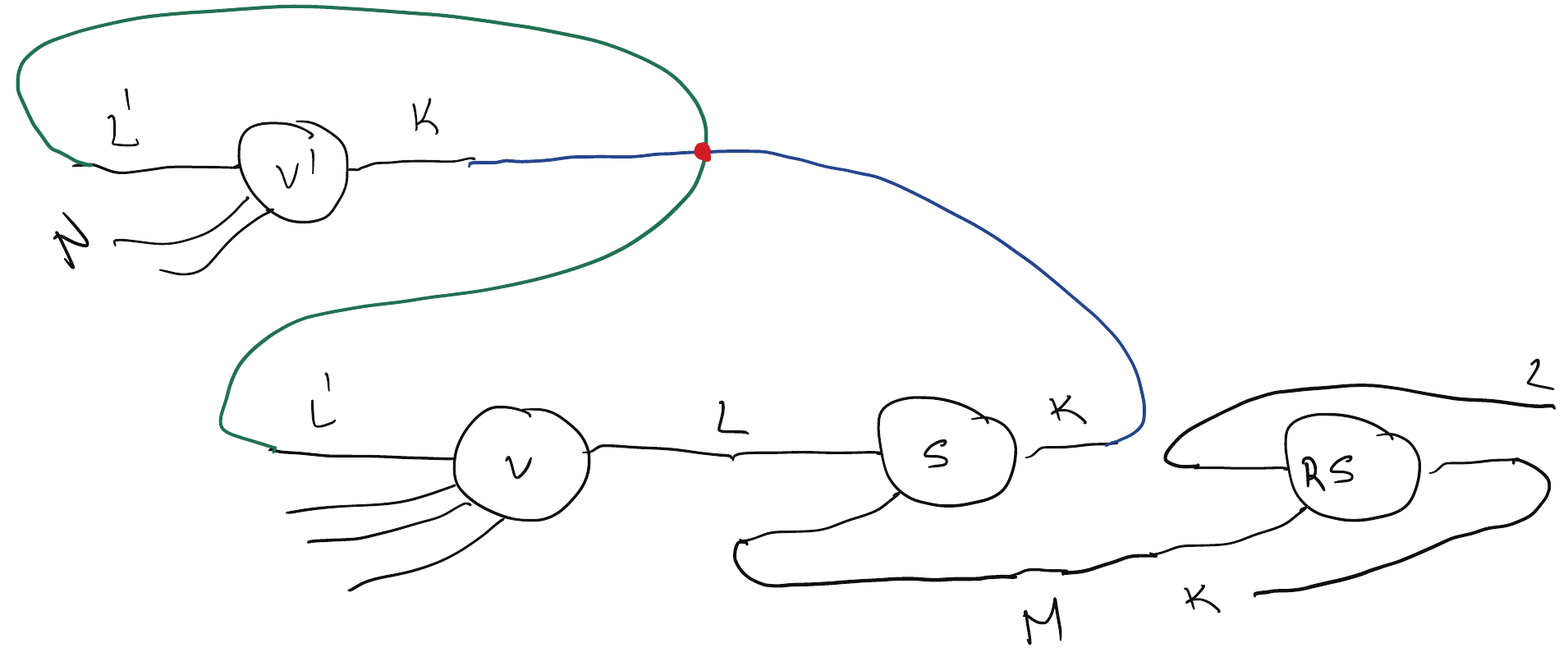}
   \end{center}
   \caption{\label{Fig:ax5fig3} The composition $C=\mathcal{C}(V\circ S,V')\circ RS$ viewed
   as morphism $C:L\to N$.}
\end{figure}

\begin{figure}[htbp]
   \begin{center}
    \includegraphics[width=0.5\linewidth]{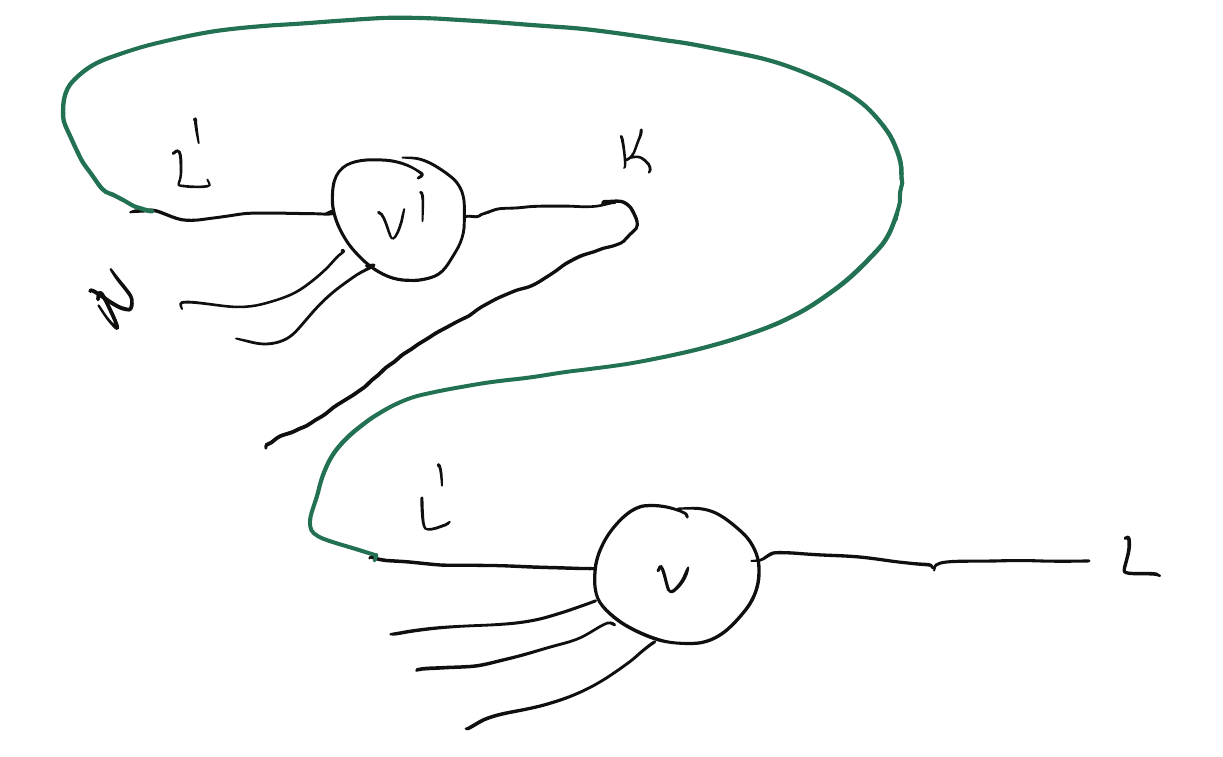}
   \end{center}
   \caption{\label{Fig:ax5fig4} The cobordism $C'=RV'\circ V$ can be seen as obtained from $C$ 
   by erasing from $C$ the blue arc as well as $S$ and $RS$. Axiom 5 claims that $C$ is 
   equivalent to $C'$.}
\end{figure}

To better understand the meaning of this axiom see Remark \ref{rem:axiom5} b.

In the axioms above we allow $0$-size surgeries and cablings as long as the respective 
classes $\mathcal{L}ag^{\ast}(M)$ and $\mathcal{L}ag^{\ast}(\C\times M)$
are enlarged to include marked Lagrangians and cobordisms (see also Remark \ref{rem:formal-surg}).

\subsection{The quotient category $\widehat{\mathsf{C}}ob^{\ast}(M)$ is triangulated.}
We assume that $\mathsf{C}ob^{\ast}(M)$ has surgery models so that, in particular, by Axiom 1, the
cabling relation $\sim$ is an equivalence.
We define $\widehat{\mathsf{C}}ob^{\ast}(M)$ to be the quotient category of $\mathsf{C}ob^{\ast}(M)$
by this equivlence relation. As usual, the objects of $\widehat{\mathsf{C}}ob^{\ast}(M)$ are the same as the objects
of $\mathsf{C}ob^{\ast}(M)$ and the morphisms are the classes of equivalence of the morphisms in
$\mathsf{C}ob^{\ast}(M)$.

Here is the main property of the structures described before.
\begin{thm}\label{thm:tri}
If the cobordism category $\mathsf{C}ob^{\ast}(M)$ has surgery models, then the quotient category
$\widehat{\mathsf{C}}ob^{\ast}(M)$ is triangulated with the exact triangles represented by surgery distinguished
triangles as in \S\ref{subsubsec:tri}.
\end{thm}
\begin{proof}
The proof is elementary. To start we remark that in the somewhat naive setting here the 
suspension functor $T$ is taken to be the identity. 

\

\noindent \underline{a. $\widehat{\mathsf{C}}ob^{\ast}(M)$ is additive.} We start by noting that in view of Axioms  4, the set of morphism between $L$ and $L'$ in $\widehat{\mathsf{C}}ob^{\ast}(M)$ has the structure of a $\Z/2$-vector space. The sum of objects is given by the union $L\oplus L'=L\sqcup L'$,
endowed with the obvious immersion $L\sqcup L'\hookrightarrow M$ that restricts to the immersions of $L$ and $L'$ (see also Remark \ref{rem:direct-sum}).

\

\noindent \underline{b. Exact triangles.} We define an exact triangle in $\widehat{\mathsf{C}}ob^{\ast}(M)$
as a triangle isomorphic to a triangle represented by a distinguished triangle
 - as defined in \S\ref{subsubsec:tri} -  in  $\mathsf{C}ob^{\ast}(M)$.
Clearly, as any morphism in  $\mathsf{C}ob^{\ast}(M)$ is equivalent to a surgery morphism which is itself part of an obvious distinguished triangle, we deduce that any morphism in $\widehat{\mathsf{C}}ob^{\ast}(M)$ can be completed to an exact triangle.  

\begin{rem} \label{rem:var-braiding}  To understand a bit the interplay of the different axioms 
we shall describe next a different way to complete a morphism $V:L\to L'$ associated to a cobordism 
$V:L\cobto (L_{1},\ldots, L_{m}, L')$  to an exact triangle. We will assume that $m=3$ and discuss the construction
by following Figures \ref{Fig:tri-completion1}, \ref{Fig:tri-completion2} and \ref{Fig:tri-completion3}.
\begin{figure}[htbp]
   \begin{center}
    \includegraphics[width=0.4\linewidth]{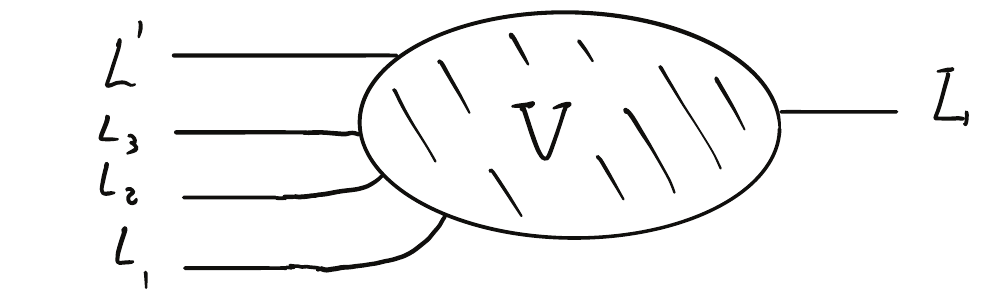}
   \end{center}
   \caption{\label{Fig:tri-completion1} The initial cobordism $V$.}
\end{figure}
We start with the cobordism $V$ in Figure \ref{Fig:tri-completion1}.
By Axiom 4 the rotation $R^{3}V$ with domain $L_{2}$ is equivalent to a surgery morphism $S_{L_{2},L_{1};c}$. 
We use this surgery morphism to construct the cabling of $R^{3}V$ and  $S_{L_{2}, L_{1};c}$ and this
produces (after another rotation) the cobordism $V'$ in Figure \ref{Fig:tri-completion2}.
Here and below, we neglect $\epsilon\geq 0$ from the notation.

\begin{figure}[htbp]
   \begin{center}
    \includegraphics[width=0.45\linewidth]{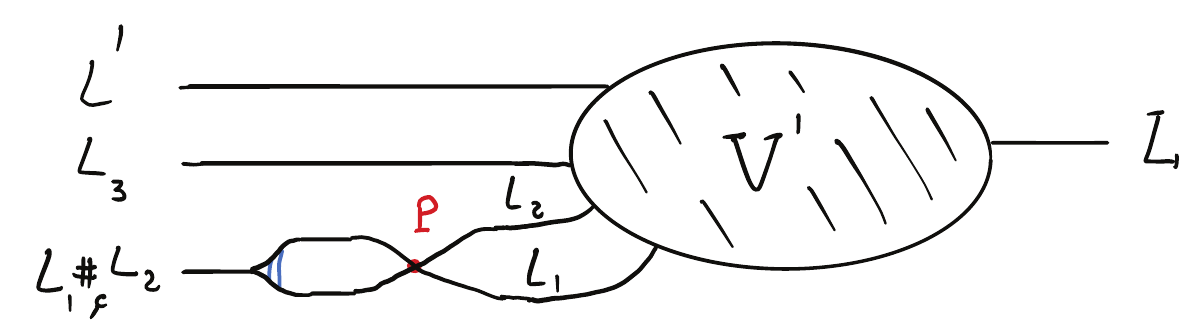}
   \end{center}
   \caption{\label{Fig:tri-completion2} The cobordism $V'$ resulting from a cabling of a rotation of $V$ with a surgery of the ends $L_{2}$ and $L_{1}$.}
\end{figure}
We now reapply the same argument but this time for the bottom two negative ends of $V'$ and we obtain this 
time the cobordism $V''$ in Figure \ref{Fig:tri-completion3} that represents a morphism $V'':L\to L'$
and has three ends and thus fits directly into an exact triangle. 
\begin{figure}[htbp]
   \begin{center}
    \includegraphics[width=0.65\linewidth]{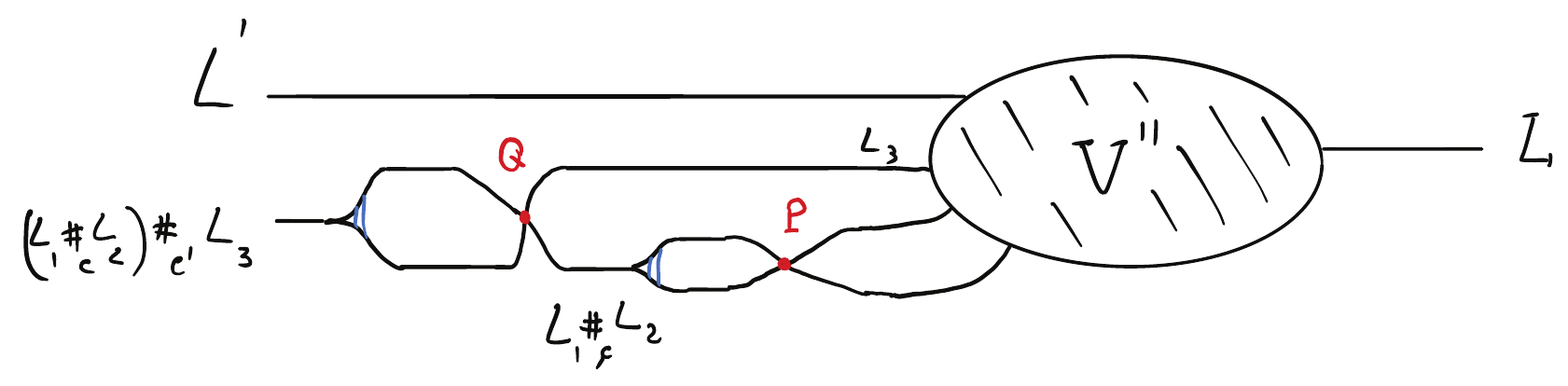}
   \end{center}
   \caption{\label{Fig:tri-completion3} The final cobordism  $V''$ obtained by iterating the operation before one
   more step.}
\end{figure}

To end this discussion we want to remark that the morphism $V'':L\to L'$ is equivalent
to $V$.  Given that the relation $\sim$ is an equivalence this follows if we can show that $V\sim V'$
where $V'$ is viewed also as a morphism $V': L\to L'$. To show this we proceed as follows. By symmetry,
$V'\sim V'$ so that a cabling as in Figure \ref{Fig:tri-completion4} belongs 
to $\mathcal{L}ag^{\ast}(\C\times M)$. In this cabling we denote by $P'$ the crossing point corresponding
to one of the copies of $V'$ and by $P$ the other such point.

\begin{figure}[htbp]
   \begin{center}
    \includegraphics[width=0.65\linewidth]{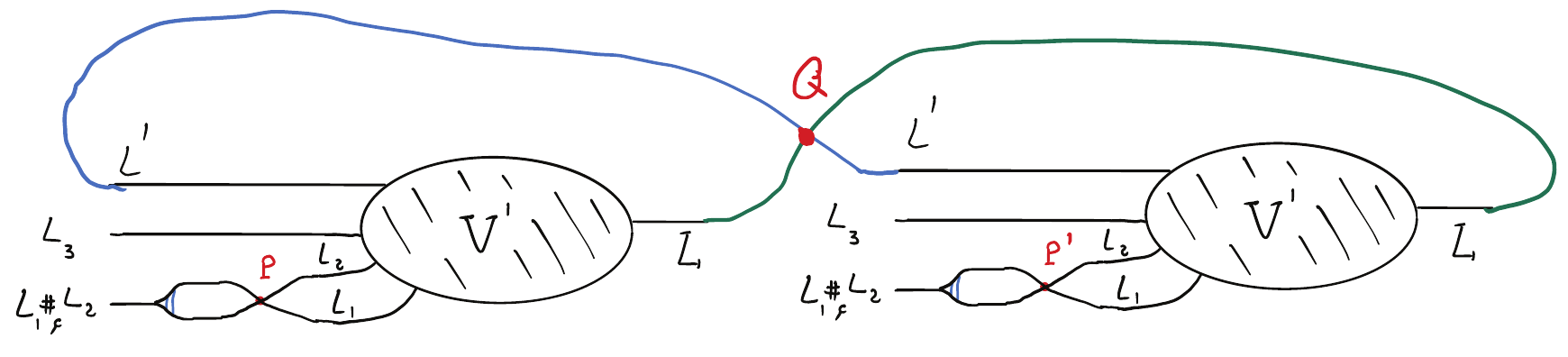}
   \end{center}
   \caption{\label{Fig:tri-completion4} The cabling of $V'$ and $V'$.}
\end{figure}

We now use Axiom 2 to ``erase'' the point $P'$ and remove  the part of the picture at the left of $P'$ -
as in Figure \ref{Fig:tri-completion5} - and to deduce that the remaining portion, which is precisely the
 cabling of $V'$ and $V$, belongs
to $\mathcal{L}ag^{\ast}(\C\times M)$ and thus $V\sim V'$.

\begin{figure}[htbp]
   \begin{center}
   \includegraphics[width=0.65\linewidth]{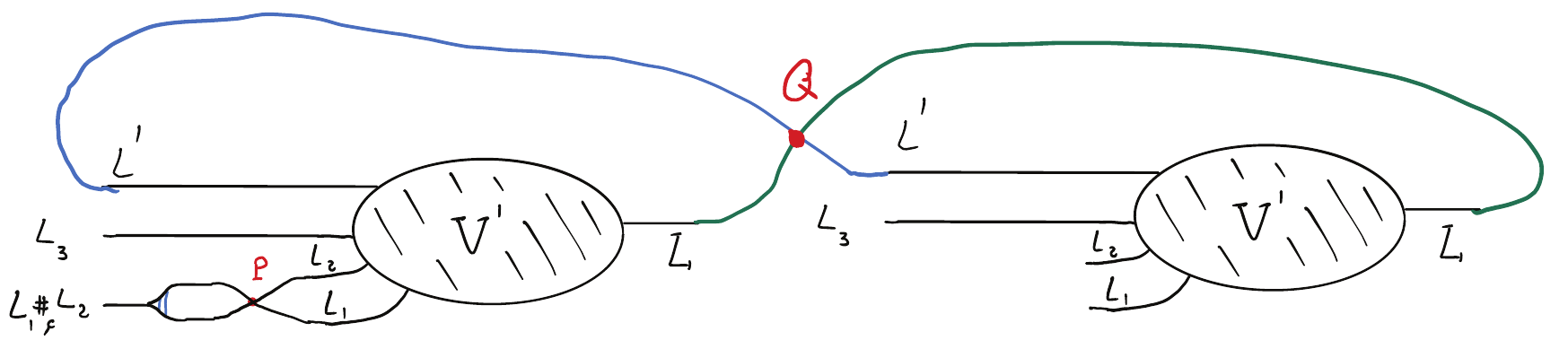}
   \end{center}
   \caption{\label{Fig:tri-completion5} The cabling of $V'$ and $V$.}
\end{figure}
For further reference, notice that the type of braiding of two consecutive ends as in Figures \ref{Fig:tri-completion2}
and \ref{Fig:tri-completion3} produces a Lagrangian representing the cone of the morphism
relating the two ends. For instance, the cone of the morphism $R^{3}V$ is (isomorphic) to the Lagrangian $L_{1}\#_{c} L_{2}$ in Figure \ref{Fig:tri-completion2}. Using also Axiom 3, it is easy to see that, up to isomorphism, the cone 
only depends of the equivalence class of a morphism. 
\end{rem}

\underline{c. Naturality of exact triangles} 
We now consider the following diagram in $\mathsf{C}ob^{\ast}(M)$.

\begin{eqnarray}\label{eq:comm-sq}
   \begin{aligned}
\xymatrix@=0.5in{
  L \ar[r]^{S} \ar[d]_{V} & M \ar[d]_{V'}\ar[r]^{T} & K\ar@{.>}[d]_{V''}\ar[r]^{U}& L \ar[d]_{V}\\
  L' \ar[r]^{S'} %\ar[d]_{W} 
  & M'\ar[r]^{T'}%\ar[d]_{W'}
  & K'\ar[r]^{U'}& L'
  %\ar@{.>}[d]_{W''}
  %\\
  %L''\ar@{.>}[r]^{S''}       & M''\ar@{.>}[r]^{T''} & K'' 
   }  
\end{aligned}
\end{eqnarray}
The morphisms $S$, $V$, $V'$ and $S'$ are given and $V'\circ S\sim S'\circ V$. These morphisms can be
each completed to exact triangles corresponding to the horizontal solid lines. Explicitely, this means 
that the cobordisms $S$, $V$, $V'$ and $S'$ can each be replaced with cobordisms with exactly three ends
that represent the same respective morphisms in $\widehat{\mathsf{C}}ob^{\ast}(M)$. Notice that there 
are many ways to do this replacement (we indicated two such ways, equivalence with a surgery morphism and
braiding, above).  Assuming from now on that $S$, $V$, $V'$ and $S'$ are each given by cobordisms 
with three ends we have $T=RS$, $T'=RS'$ and $U=R^{2}S$, $U'=R^{2}S'$.

We want to show that there exists  a morphism $V'':K\to K'$ such that $V''\circ T\sim T'\circ V'$ and 
$V\circ U\sim U'\circ V''$. We will again use a series of figures to pursue this construction.
We start with Figure \ref{Fig:fix-morph} where we fix the four morphisms $V,S,V',S'$ as above.
\begin{figure}[htbp]
   \begin{center}
       \includegraphics[width=0.65\linewidth]{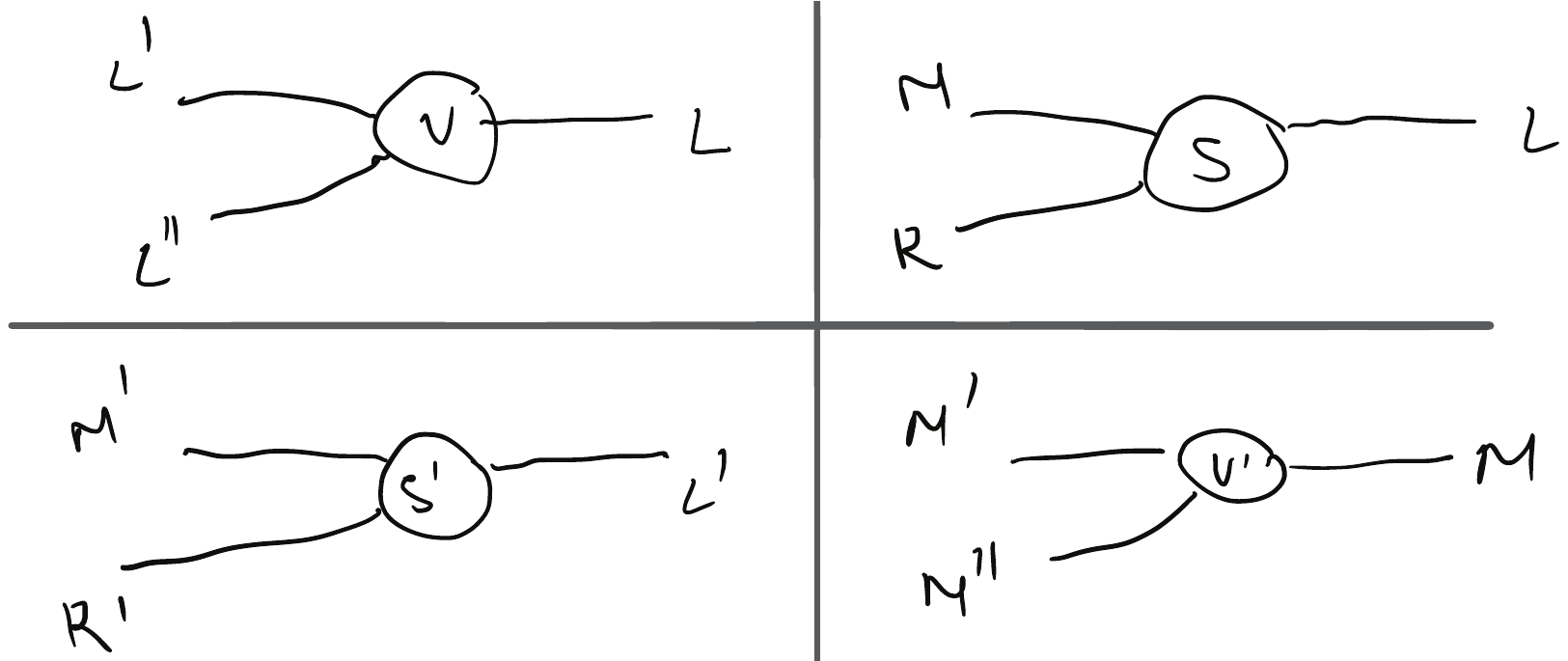}
   \end{center}
   \caption{\label{Fig:fix-morph} The morphisms in the left square of (\ref{eq:comm-sq}).}
\end{figure}

The next step is in Figure \ref{Fig:fix-morph2}. This describes the cabling of $S'\circ V$ and $V'\circ S$ - this exists
in $\mathcal{L}ag^{\ast}(\C\times M)$ because $V'\circ S\sim S'\circ V$ - and views it as a morphism $V'':K\to K'$.

\begin{figure}[htbp]
   \begin{center}
    \includegraphics[width=0.65\linewidth]{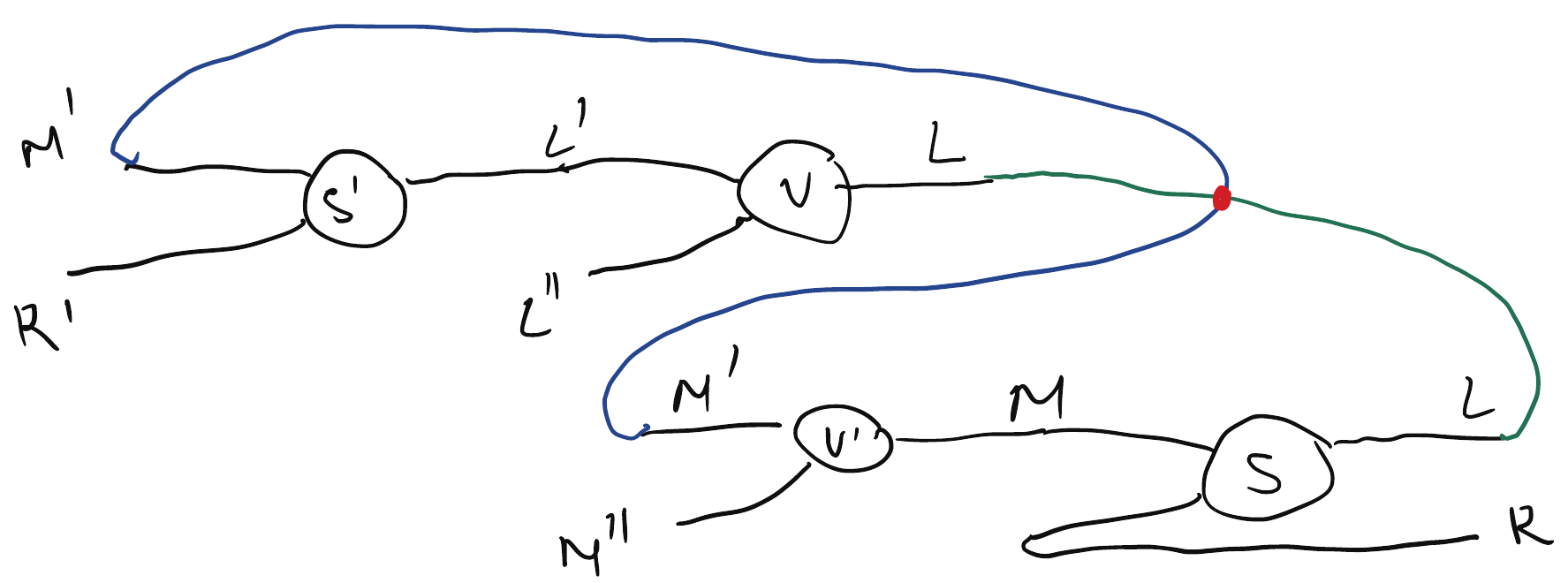}
   \end{center}
   \caption{\label{Fig:fix-morph2} The cabling of $S'\circ V$ and $V'\circ S$ viewed as a morphism $V'':K\to K'$.}
\end{figure}

\

For the next step we will make use of Axiom 5 to show the commutativity $V''\circ T\sim T'\circ V'$.
We start in Figure \ref{Fig:fix-morph3} by representing the composition $V''\circ T$.

\begin{figure}[htbp]
   \begin{center}
    \includegraphics[width=0.65\linewidth]{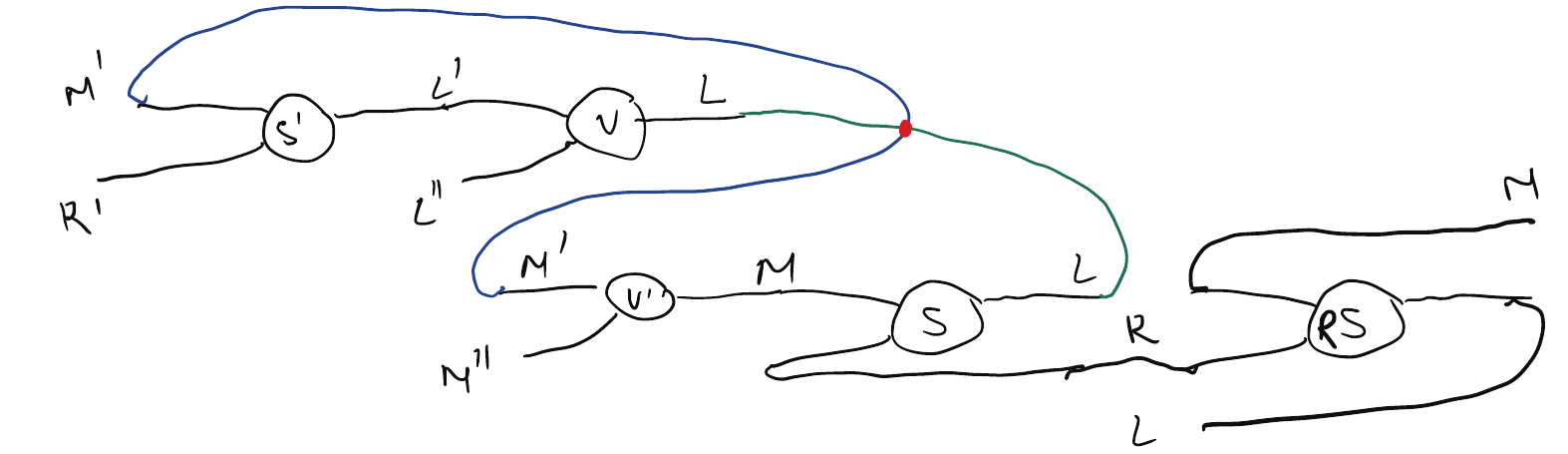}
   \end{center}
   \caption{\label{Fig:fix-morph3} The composition of $V'':K\to K'$ and $T=RS:M\to K$.}
\end{figure}

We now apply Axiom 5 to ``erase'' the green arc and the cobordisms $S$ and $RS$ in the picture thus
getting a new morphism $V'''\sim V''\circ T$ as in Figure \ref{Fig:fix-morph4}.

\begin{figure}[htbp]
   \begin{center}
    \includegraphics[width=0.65\linewidth]{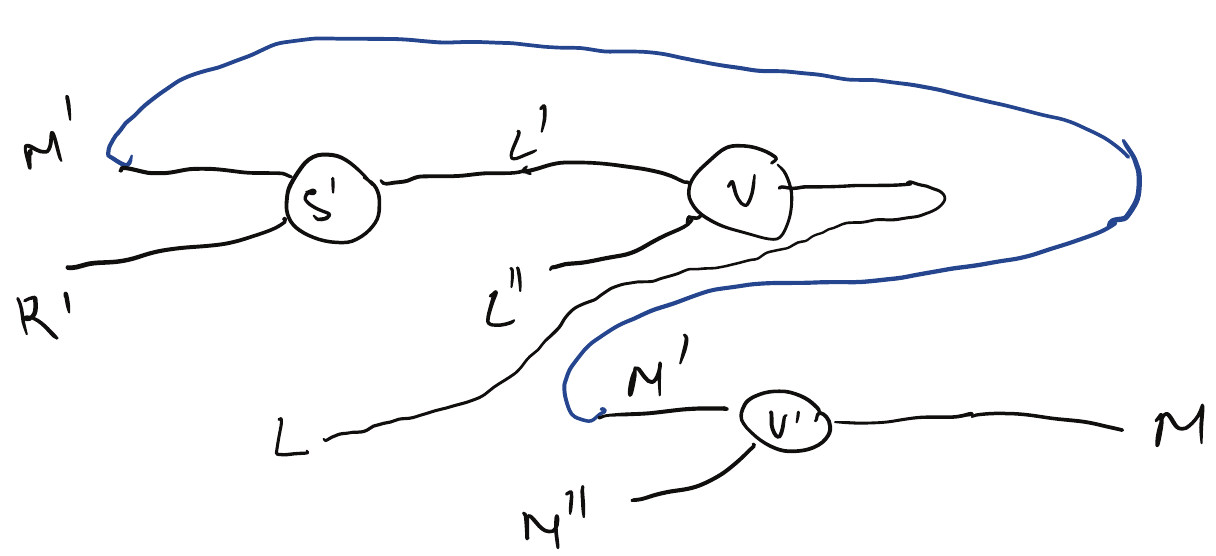}
   \end{center}
   \caption{\label{Fig:fix-morph4} Here is the cobordism $V'''$ obtained from $V''\circ T$ after erasing the green arc as well as $S$ and $RS$ (and bending the $L$ end in the positive direction).}
\end{figure}

The composition of $T'\circ V'$ is represented in Figure \ref{Fig:fix-morph5}. We 
notice that $V'''$ is equal to the composition of $T'\circ V'$ with a rotation of $V$ along a secondary leg (with end
$L'$). In view of Axiom 1, this means that $V'''$ is equivalent to the cobordism $T'\circ V'$ and concludes the argument.
\begin{figure}[htbp]
   \begin{center}
    \includegraphics[width=0.65\linewidth]{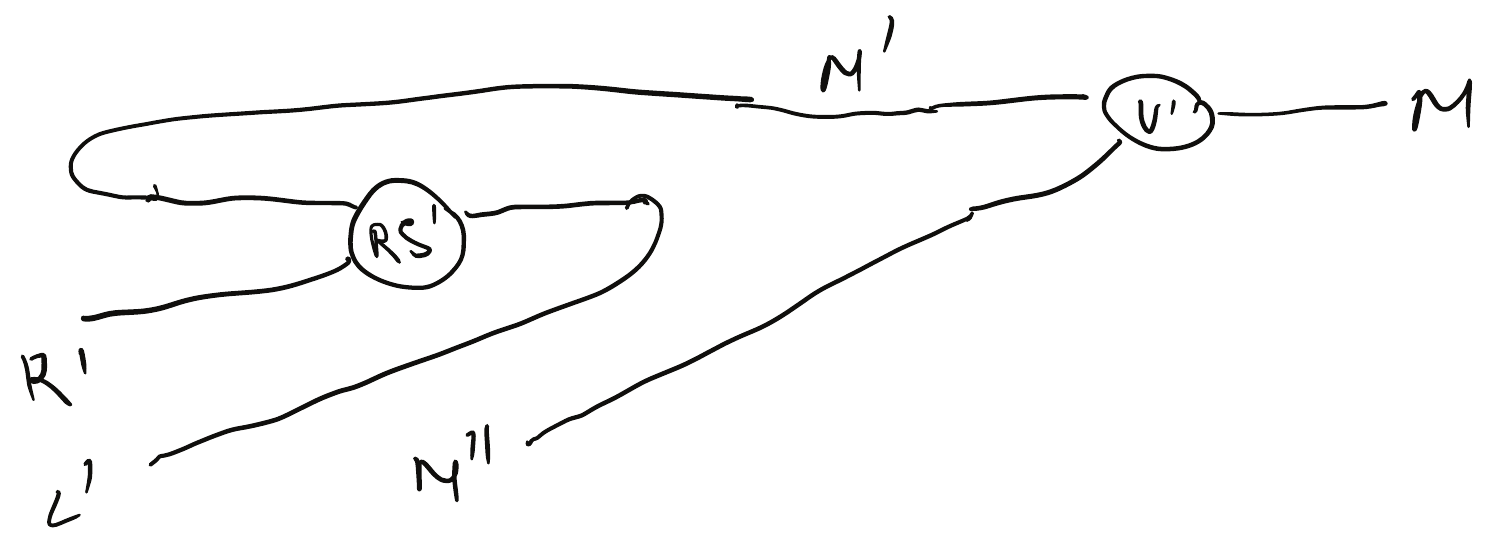}
   \end{center}
   \caption{\label{Fig:fix-morph5} The morphism $T'\circ V'$.}
\end{figure}
The argument required to show $V\circ U\sim U'\circ V''$ is perfectly similar, it appeals to the second identity
in Axiom 5.

\begin{rem}\label{rem:axiom5} a. It follows from the naturality of exact triangles that if $V:L\cobto (K,L')$, then $RV\circ V\sim 0$.

b.   Some special choices
in  diagram (\ref{eq:comm-sq}), as in diagram (\ref{eq:comm-sq3}) below,
indicate that {\em naturality} is an adequate name for Axiom 5.

\begin{eqnarray}\label{eq:comm-sq3}
   \begin{aligned}
\xymatrix@=0.5in{
  K \ar[r]^{S} \ar[d]_{id} & L \ar[d]_{V}\ar[r]^{RS} & M\ar@{.>}[d]_{V''}\\
  K \ar[r]^{V'} %\ar[d]_{W} 
  & L'\ar[r]^{RV'}%\ar[d]_{W'}
  & N
  %\ar@{.>}[d]_{W''}
  %\\
  %L''\ar@{.>}[r]^{S''}       & M''\ar@{.>}[r]^{T''} & K'' 
   }  
\end{aligned}
\end{eqnarray}
In this setting, notice that the definition of $V''$ given above is precisely the cabling $\mathcal{C}(V\circ S, V')$
and thus the first identity in Axiom 5 simply claims that the square on the right in diagram (\ref{eq:comm-sq3}) commutes in $\widehat{\mathsf{C}}ob^{\ast}(M)$. The second identity has a similar interpretation.
We consider the  next diagram (we assume that all cobordisms here have only three ends) which is again obtained 
from (\ref{eq:comm-sq}) by adjusting the notation so that it fits with the statement of Axiom 5. \begin{eqnarray}\label{eq:comm-sq4}
   \begin{aligned}
\xymatrix@=0.5in{
  L \ar[r]^{V''} \ar[d]_{V} & K' \ar[d]_{id}\ar[r]^{RV''} & M\ar@{.>}[d]_{\mathcal{C}}\ar[r]^{R^{-1}V''}& L \ar[d]_{V}\\
  L' \ar[r]^{S'} %\ar[d]_{W} 
  & K'\ar[r]^{RS'}%\ar[d]_{W'}
  & N\ar[r]^{R^{-1}S'}& L'
  %\ar@{.>}[d]_{W''}
  %\\
  %L''\ar@{.>}[r]^{S''}       & M''\ar@{.>}[r]^{T''} & K'' 
   }  
\end{aligned}
\end{eqnarray}
The map $\mathcal{C}$ is the cabling coming from $V''\sim S'\circ V$. With this notation, the second statement from 
axiom 5 claims the commutativity of the square on the right in (\ref{eq:comm-sq4}).

\end{rem}

\underline{d. The octahedral axiom.}
We consider the following diagram:

\

\begin{eqnarray}\label{eq:comm-sq2}
   \begin{aligned}
\xymatrix@=0.5in{
  L \ar[r]^{S} \ar[d]_{id} & M \ar[d]_{V'}\ar[r]^{T} & K\ar[d]_{V''}\ar[r]^{U}& L\ar[d]^{id}\ar[r]^{S}& M\\
  L \ar[r]^{S'}   & M'\ar[r]^{T'}\ar[d]_{W'}
  & K'\ar@{.>}[d]_{W''}\ar[r]^{U'}& L&\\
     & M''\ar[r]^{id}\ar[d]_{R^{2}V'} & M''\ar@{.>}[d]_{R^{2}V''} & &\\
  &  M\ar[r]^{T} & K &   &
   }  
\end{aligned}
\end{eqnarray}
Notice that compared to (\ref{eq:comm-sq}) the place of the morphism $V$ is taken by $id:L\to L$.
The assumption is that the two horizontal rows are exact triangles and that the left vertical column
is also an exact triangle. Moreover, as before, $W'=RV'$, $T=RS$, $T'=RS'$ (and we assume that $S,S',V'$ have only three ends). We first need to show the existence of an exact triangle like the column on the right
and that it makes the bottom square on the right commutative. We will take $V''$ to be defined 
as at the point c. before. This implies the commutativity of $V''$ with the connectants $U=R^{2}S$ and $U'=R^{2}S'$. Compared to Figure \ref{Fig:fix-morph2} the morphism $V''$ is simpler in our case because
instead of $V$ we can insert the identity cobordism of $L$. We thus get the morphism in the Figure \ref{Fig:fix-morph6}  below:

\begin{figure}[htbp]
   \begin{center}
    \includegraphics[width=0.65\linewidth]{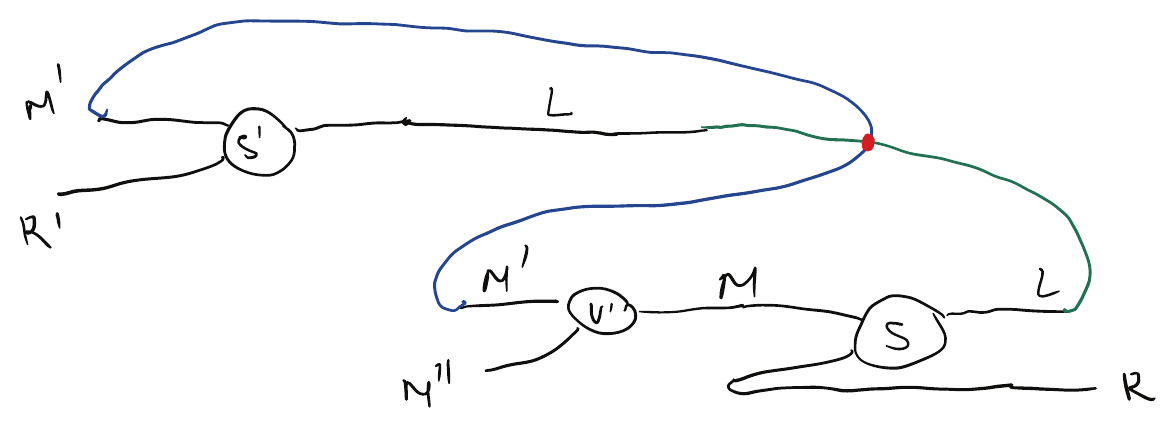}
   \end{center}
   \caption{\label{Fig:fix-morph6} The morphism $V'':K\to K'$ when $V=id|_{L}$.}
\end{figure}
This morphism has three ends and thus gives directly rise to an exact triangle. We put $W''=RV'':K'\to M''$.
We now check that $W''\circ T'\sim W'$. 
\begin{figure}[htbp]
   \begin{center}
    \includegraphics[width=0.8\linewidth]{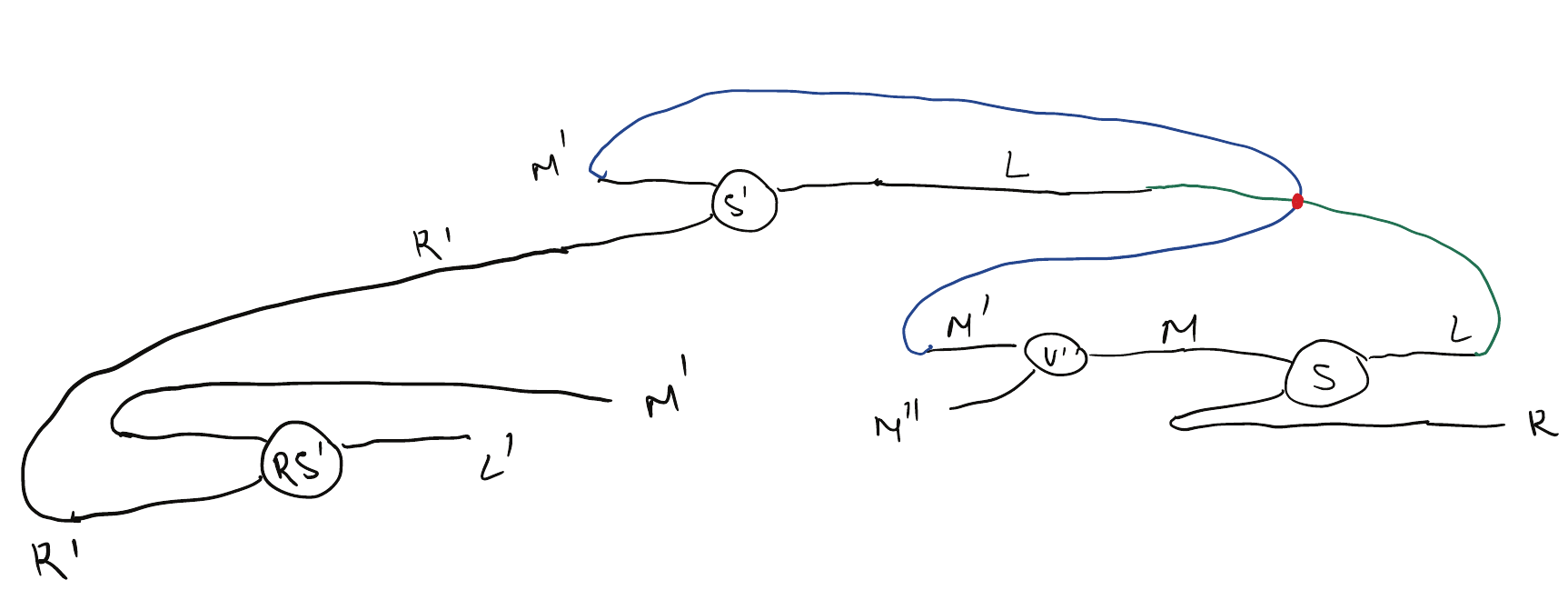}
   \end{center}
   \caption{\label{Fig:fix-morph7} The composition $W''\circ RS'$ with $W''=RV'':K'\to M''$ and $T'=RS'$.}
\end{figure}

This is done again in two Figures \ref{Fig:fix-morph7} and \ref{Fig:fix-morph8}: the first represents the composition $RV''\circ RS'$ (recall $T'=RS'$) and the second reflects again an application of Axiom 5
to conclude that $W'\sim V''\circ RS'$.

\begin{figure}[htbp]
   \begin{center}
    \includegraphics[width=0.65\linewidth]{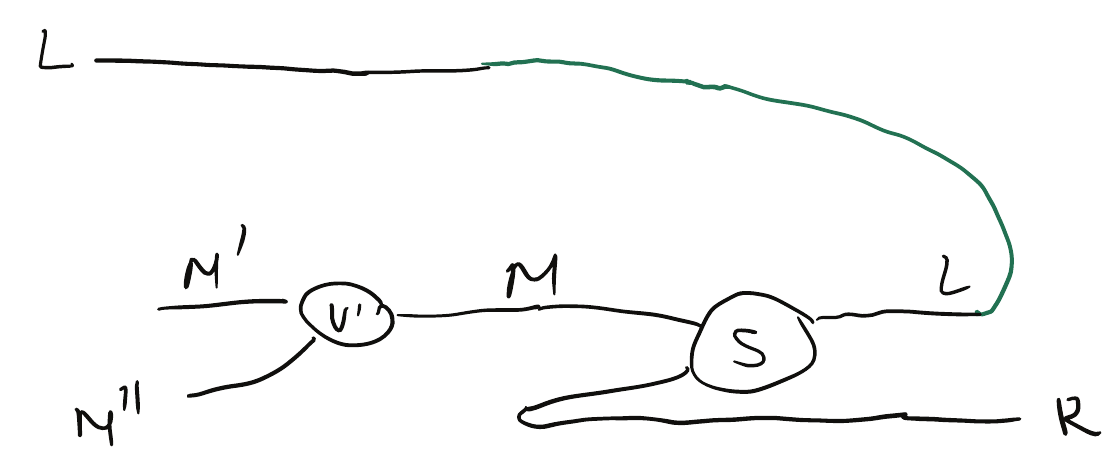}
   \end{center}
   \caption{\label{Fig:fix-morph8} Erasing the blue arc as well as $RS'$ and  $S'$ from Figure \ref{Fig:fix-morph7} leads, by Axiom 5, to the equivalent morphism above which, in turn, is
   equivalent to $RV'=W'$.}
\end{figure}

To conclude establishing the octahedral axiom we need to also remark that 
$R^{2}V'\circ  W''\sim S\circ U'$ or, in other words, $R^{2}V'\circ RV''\sim S\circ R^{2}S'$. 
This follows from a similar argument, by making use of the 
fact that the cobordism $V'$ has three ends and applying the second identity in Axiom 5. 
\end{proof}

\begin{rem}\label{rem:dec} Fix a cobordism $V:L\cobto (L_{1},\ldots, L_{m})$ 
in $\mathcal{L}ag^{\ast}(\C\times M)$. By succesively braiding the ends $L_{1},\ldots, L_{m}$,
as in Remark \ref{rem:var-braiding} (but including in the braiding all the negative ends), 
we replace $V$ by a new cobordism $V': L\cobto L'$ where $L'=L_{m}\#_{c_{m}}(L_{m-1}\#\ldots (L_{2}\#_{c_{1}}L_{1})\ldots) )$. In view of Axiom 3,
$V'$ is an isomorphism and $L'$ is given in $\widehat{\mathsf{C}}ob^{\ast}(M)$ as an iterated cone
$L'=\tcn(L_{m}\to \tcn(L_{m-1}\to\ldots \tcn(L_{2}\to L_{1})\ldots))$.
\end{rem}

\subsection{Shadows and fragmentation metrics}\label{subsec:shadow}
Given a Lagrangian cobordism $V : L\cobto (L_{1},\ldots, L_{m})$
we define its shadow  \cite{Co-She:metric},
$\mathcal{S}(V)$, by:
\begin{equation} \label{eq:shadow-intro} \mathcal{S}(V) = Area \bigl(
  \mathbb{R}^2 \setminus \mathcal{U} \bigr),
\end{equation}
where $\mathcal{U} \subset \mathbb{R}^2 \setminus \pi(V)$ is the union
of all the {\em unbounded} connected components of
$\mathbb{R}^2 \setminus \pi(V)$. Here
$\pi: \mathbb{R}^2 \times M \longrightarrow \mathbb{R}^2$ is the
projection.

 The shadow of a surgery morphism is bounded by the ``size'' of the handle used in the surgery. Therefore,
 assuming  that  $\mathsf{C}ob^{\ast}(M)$ has surgery models, a natural additional requirement to 
 impose is to ask that $\mathsf{C}ob^{\ast}(M)$ has  {\em small} surgery models - in the sense that each
morphism is equivalent to a surgery cobordism of arbitrarily small shadow.  Similarly, in the definition of cabling equivalence, Definition \ref{dfn:cabling-rel}, the requirement for $V\sim V'$ becomes that for each $\epsilon$ sufficiently small, there exists a cabling $\mathcal{C}(V,V';c,\epsilon)\in\mathcal{L}ag^{\ast}(M)$.  In the marked context, this includes the case 
$\epsilon=0$ and, in fact, $V\sim V'$ in this setting if and only if there exists exists a  cabling 
$\mathcal{C}(V,V';c)\in \mathcal{L}ag^{\ast}(M)$.

We will say that the category $\mathsf{C}ob^{\ast}(M)$ is {\em rigid with surgery models}
if it has small surgery models in the sense above and if additionally it satisfies the next axiom.

\begin{axio}(\underline{rigidity}) For any two Lagrangians $L$ and $L'$ there exists a constant 
$\delta(L,L')\geq 0$ that  vanishes only if $L=L'$, such that for any simple cobordism $V:L\cobto L'$, 
$V\in\mor_{\mathsf{C}ob^{\ast}(M)}$ we have \begin{equation}\label{eq:ineq}
\mathcal{S}(V)\geq \delta(L,L')
~.~\end{equation} 
\end{axio}

We now discuss the meaning of this axiom. We start by noting that the notion of shadow
 easily leads to a family of pseudo-metrics, called {\em fragmentation
shadow pseudo metrics}, as introduced in \cite{Bi-Co-Sh:lshadows-long}. 
They are defined as follows. 
Fix  a family of objects  $\mathcal{F}$ in $\mathcal{L}ag^{\ast}(M)$
and, for two objects $L,L'\in \mathcal{L}ag^{\ast}(M)$,
put:
\begin{align}\label{eq:pesudo-m}
d^{\mathcal{F}}(L,L') = &\inf\{\mathcal{S}(V) \ :  \ \ V:L\cobto (F_{1},\ldots,F_{i}, L', F_{i+1},\ldots , F_{m}),\\   & \ \ \ \ F_{i}\in \mathcal{F}, \ V \in\mathcal{L}ag^{\ast}(\C\times M) \  \} \notag 
\end{align}

This is clearly symmetric and satisfies the triangle inequality and thus defines a pseudo-metric, possibly infinite.
If $\mathcal{F}'$ is another family of objects in $\mathcal{L}ag^{\ast}(M)$ we can consider
the average of the two pseudo-metrics which is again a pseudo-metric:
$$d^{\mathcal{F},\mathcal{F}'}(L,L')=\frac{d^{\mathcal{F}}(L,L')+d^{\mathcal{F}'}(L,L')}{2}~.~$$
There is an obvious order among fragmentation  pseudo-metrics induced by the order on pairs  of  families $\mathcal{F},\mathcal{F}'$ given by inclusion. In particular, $d^{\emptyset,\emptyset}\geq d^{\mathcal{F},\mathcal{F}'}$. 

Apriori, none of the pseudo-metrics above is non-degenerate so that it is natural to define the rigidity
of a cobordism category $\mathsf{C}ob^{\ast}(M)$ by imposing this requirement, as below.
We will only be concerned with some special pairs $\mathcal{F},\mathcal{F}'$ - see Corollary \ref{cor:rig}.
\begin{dfn}\label{def:rig}
A cobordism category $\mathsf{C}ob^{\ast}(M)$ is called {\em strongly rigid} if, for any two 
families $\mathcal{F}$ and $\mathcal{F}'$ such that the intersection
$$\overline{\bigcup_{F\in\mathcal{F}}  F}\cap \overline{\bigcup_{F'\in\mathcal{F}'}  F'}$$ is totally discrete,  
the pseudo-metric $d^{\mathcal{F},\mathcal{F}'}$ is non-degenerate.
\end{dfn}

This definition of strong rigidity does not require $\mathsf{C}ob^{\ast}(M)$ to have surgery models, it is a purely geometric
constraint. Similarly, the condition in Axiom 6 is also purely geometric and can be formulated independently of the existence of surgery models. With our definitions, the inequality (\ref{eq:ineq}) is equivalent to 
\begin{equation}\label{eq:ineq2}
d^{\emptyset, \emptyset}(L,L')\geq \delta(L;L')~,~
\end{equation} and thus (\ref{eq:ineq}) is equivalent to the fact that $d^{\emptyset, \emptyset}$
is non-degenerate. In view of displacement energy - width inequalities,  relations such as (\ref{eq:ineq}), are natural to expect in non flexible settings. 

We will see that  unobstructed classes of Lagrangians, in the sense to be made precise later in the paper,  have small surgery models and are strongly rigid. In fact, as will be noticed in Corollary \ref{cor:rig}, 
in this case the non-degeneracy of $d^{\emptyset, \emptyset}$ implies the non-degeneracy  of all the other pseudo-metrics from Definition \ref{def:rig} (for this result, our objects need to be immersed Lagrangians).

\begin{rem}\label{rem:weights1}a. To further unwrap the meaning  of Axiom 6, notice that,
because Lagrangian suspensions of objects in $\mathcal{L}ag^{\ast}(M)$ are in 
$\mathcal{L}ag^{\ast}(\C\times M)$, the pseudo-metric $d^{\emptyset,\emptyset}$ is bounded from above by the 
Hofer distance on Lagrangian submanifolds. Thus, Axiom 6 implies the non-degeneracy of this distance too.
In a different direction, the (pseudo)-metrics $d^{\mathcal{F},\mathcal{F}'}$ are finite
much more often compared to $d^{\emptyset,\emptyset}$ and thus much larger classes of Lagrangians are
endowed with a meaningful geometric structure.

b. In case the category $\mathsf{C}ob^{\ast}(M)$
has surgery models (and thus  $\widehat{\mathsf{C}}ob^{\ast}(M)$ is triangulated) 
there is a more general context that fits the construction of this shadow pseudo-metrics that we
briefly recall from  \cite{Bi-Co-Sh:lshadows-long}.
Let $\mathcal{X}$ be a
triangulated category and recall that 
there is a category denoted by $T^{S}\mathcal{X}$ that was introduced
in \cite{Bi-Co:cob1,Bi-Co:lcob-fuk}. This category is monoidal and its
objects are finite ordered famillies $(K_{1},\ldots, K_{r})$ with
$K_{i}\in \mathcal{O}b(\mathcal{X})$ with the operation given by
concatenation. In essence, the morphisms in
$T^{S}\mathcal{X}$ parametrize all the cone-decompositions of the
objects in $\mathcal{X}$. Composition in $T^{S}\mathcal{X}$ comes down
to refinement of cone-decompositions.  Assume given a weight
$w: \mor_{T^{S}\mathcal{X}}\to [0,\infty]$ such that
\begin{equation}\label{eq:weight-tr}
  w(\bar{\phi}\circ \bar{\psi})\leq 
  w(\bar{\phi})+w(\bar{\psi}) \ , \ w(id_{X})=0 \ , \forall \ X,
\end{equation}
Fix also a
family $\mathcal{F}\subset \mathcal{X}$.  We then can define a measurement on the objects of $\mathcal{X}$:
\begin{equation}\label{eq:pseudo-metric2}
  s^{\mathcal{F}}(K',K)=
  \inf \{w(\bar{\phi}) 
  \mid \ \ \bar{\phi} : K'\to 
  (F_{1},\ldots, K,\ldots, F_{r}),\ \ F_{i}\in \mathcal{F}, \forall i \}~.~
\end{equation}
This satisfies the triangle inequality but is generally non-symmetric. However,
as noted in \cite{Bi-Co-Sh:lshadows-long}, $s^{\mathcal{F}}$ can be symmetrized and it leads to what is called in 
 \cite{Bi-Co-Sh:lshadows-long} a fragmentation pseudo-metric. Our  remark here is that
 if  $\mathcal{X}=\widehat{\mathsf{C}}ob^{\ast}(M)$,
 then the category $T^{S}\mathcal{X}$ can be identified with (a quotient of) the general cobordism
 category $\mathcal{C}ob^{\ast}(M)$ as introduced in \cite{Bi-Co:cob1}. This category has as objects
 families of Lagangians $(L_{1},\ldots, L_{k})$ and as morphisms families of cobordisms $V:L\cobto (L_{1},\ldots, L_{k})$.  The difference compared to the category $\mathsf{C}ob^{\ast}(M)$ is that such a $V$ is viewed in $\mathcal{C}ob^{\ast}(M)$ as a morphism from the family formed by a single element $L$ to the family $(L_{1},\ldots, L_{k})$ and in $\mathsf{C}ob^{\ast}(M)$ the same $V$ is viewed as a morphism from $L$ to $L_{k}$. 
 The relation between $T^{S}\widehat{\mathsf{C}}ob^{\ast}(M)$ and $\mathcal{C}ob^{\ast}(M)$ is not at all surprising because, 
 as noted in Remark \ref{rem:dec}, each cobordism $V:L\cobto (L_{1},\ldots, L_{k})$  induces an iterated cone-decomposition  in $\widehat{\mathsf{C}}ob^{\ast}(M)$.  Finally, the fragmentation pseudo-metric $s^{\mathcal{F}}$ defined on the objects of $\widehat{\mathsf{C}}ob^{\ast}(M)$ by using as weight the shadow of cobordisms coincides with $d^{\mathcal{F}}$.
 \end{rem}
 
\

\subsection{Cobordism groups.}
Given a cobordism category $\mathsf{C}ob^{\ast}(M)$ there is a natural associated cobordism group,
$\Omega^{\ast}(M)$, given as the free abelian group generated by the Lagrangians in $\mathcal{L}ag^{\ast}(M)$
modulo the subgroup of relations generated by the expressions $L_{1}+\ldots  + L_{m}=0$ whenever 
a cobordism $V:\emptyset \cobto (L_{1},\ldots, L_{m})$ exists and $V\in \mathcal{L}ag^{\ast}(\C\times M)$.

In our setting, where orientations are neglected, this group is a $\Z_{2}$-vector space.

\begin{cor}\label{cor:K}
If the category $\mathsf{C}ob^{\ast}(M)$ has surgery models, then there is an isomorphism
$$\Omega^{\ast}(M)\cong K_{0}(\widehat{\mathsf{C}}ob^{\ast}(M))~.~$$
\end{cor}
\begin{proof} Recall that $K_{0}\mathcal{C}$ is the Grothendieck group of the triangulated category
$\mathcal{C}$. It is generated by the objects of $\mathcal{C}$ subject to the relations
$B=A+C$ for each exact triangle $A\to B\to C$. It immediately follows from Remark \ref{rem:dec}
that the relations defining the cobordism group $\Omega^{\ast}(M)$ belong to the relations subgroup
giving $K_{0}$. Conversely, as all exact triangles in  $\widehat{\mathsf{C}}ob^{\ast}(M)$ are represented by 
surgeries it follows that all relations defining $K_{0}$ are also cobordism relations. Thus, we have the isomorphism
claimed.
\end{proof}

\

% !TEX root = ImmersedS.tex

\section{Unobstructed classes of Lagrangians and the category $\mathsf{C}ob^{\ast}(M)$}
\label{sec:unobstr-tech}

The purpose of this section is to discuss in detail the classes of Lagrangians $\mathcal{L}ag^{\ast}(M)$ and cobordisms  $\mathcal{L}ag^{\ast}(\C\times M)$  that 
are used to define the category $\mathsf{C}ob^{\ast}(M)$. This category will be shown later in the 
paper to have rigid surgery models, leading to a proof of Theorem \ref{thm:BIG}.

To fix ideas, the Lagrangians $L$
belonging to the class $\mathcal{L}ag^{\ast}(M)$ are exact, generic immersions $i_{L}:L\to M$, endowed with some additional structures. In short, these additional decorations consist of a primitive, a marking $c$ (in the sense of Definition \ref{def:marked} - thus, a choice of double points of $i_{L}$) 
and some perturbation data associated to $L$, $\mathcal{D}_{L}$, providing good control on the pseudo-holomorpic (marked) curves with boundary on $L$ and such that, with respect to these data,
 $L$ is unobstructed.  The cobordisms $V\in \mathcal{L}ag^{\ast}(\C\times M)$ satisfy similar properties: they 
are exact, marked, immersions $i_{V}:V\to \C\times M$, they also carry perturbation data of a similar nature
as $\mathcal{D}_{L}$ in a way that extends the data of the ends. An additional complication is that 
cobordisms come with an additional decoration consisting of a choice of perturbation required to 
transform such a cobordism with immersed ends, therefore having non-isolated double points along the ends, 
into generic immersions.  

The section is structured as follows: in \S \ref{subsec:unobstr} we first discuss the types of curves with 
boundary that appear in our constructions as well as the notion of unobstructed, marked Lagrangian in $M$;  
in \S \ref{subsubec:marked-cob} we discuss the similar notions for cobordisms; we then 
specialize in \S\ref{subsec:exact} the discussion to the exact case; finally, in \S\ref{subsubsec:cob-cat-def},
after some other adjustments, we define the category $\mathsf{C}ob^{\ast}(M)$.

\

\subsection{$J$-holomorphic curves and unobstructed Lagrangians.}\label{subsec:unobstr}

Immersed Lagrangians have been considered from the point of view of Floer
theory starting with the work of  Akaho \cite{Akaho} by Akaho-Joyce \cite{Akaho-Joyce} as well as other authors such as Alston-Bao \cite{Alston-Bao}, Fukaya \cite{Fukaya-immersed}, Palmer-Woodward \cite{Pal-Wood-imm} and others. A variety of unobstructedness type conditions appear in all these works.
We discuss here only the aspects that are relevant for our approach.

\

While in the embedded case unobstructedness can be deduced from certain topological constraints (such as
exactness or monotonicity) and thus it is independent of choices of almost complex structures and other such data, for immersed marked  Lagrangians this condition is much more delicate as it requires certain counts of $J$-holomorphic curves to vanish ({\em mod $2$}). For these counts to be well defined, the relevant moduli spaces need to be regular and, further, the counts themselves depend on the choices of data.  This makes the definition of an unobstructed immersed, marked, Lagrangian considerably more complicated as one can see in Definition \ref{def:unobstr} below. For immersed marked cobordisms the relevant definitions are even more complex (as seen in Definition \ref{def:unobstrcob}) because one needs to deal with the additional problem that double points of  cobordisms with immersed ends are not isolated. 
\

%Our approach to the proof of Theorem \ref{thm:surg-models} is that once the appropriate definitions of %unobstructedness for both $\mathcal{L}ag^{\ast}(M)$ and $\mathcal{L}ag^{\ast}(\C\times M)$  are in %place, we may consider the subcategory of unobstructed {\em embedded} Lagrangians and complete it by %surgery to a category with surgery models.  Apriori, this does not recover all the objects in $\mathcal{L}%ag^{\ast}(M)$ (even up to isomorphism) but does provide a category isomorphic to $D\fuk^{\ast}(M)$.

\

\subsubsection{$J$-holomorphic curves with boundary along marked Lagrangians}\label{subsubsec:curves-marked}
We consider here marked, immersed  Lagrangians $(L, c)$ with $j_{L}:L\to M$ an immersion with only transverse double points, as in Definition \ref{def:marked}. Recall that the set $c\subset L\times L$ is a collection of (ordered) double points of $j_{L}$,  in particular, for each $(P_{-},P_{+})\in c$ we have $j_{L}(P_{-})=j_{L}(P_{+})$.  

Consider a finite $n$-tuple of such Lagrangians $(L_{i}, c_{i})$, $1\leq i\leq n$ 
and assume for the moment that for all $i\not=j$ the intersections of $L_{i}$ and $L_{j}$ 
are transverse and distinct from all self intersection points of the $L_{i}$'s.  We denote by $\mathbf{c}$
the family $\mathbf{c}=\{c_{i}\}_{i}$.

A $J$-holomorphic polygon with boundary on the $(L_{i},c_{i})$'s is a map $u:D^{2}\to M$ with
$u(\partial D^{2})\subset \bigcup_{i} j_{L_{i}}(L_{i})$ such that (see Figure \ref{fig:polygon})
:
\begin{itemize}
\item[i.] $u$ is continuous and smooth on 
$D^{2}\backslash\ \bigcup_{i}\{a^{i}_{1},\ldots a^{i}_{s_{i}}\}$ 
and satisfies $\bar{\partial}_{J}u=0$ inside $D^{2}$. Here
$a^{i}_{k}\in S^{1}$, $1\leq k \leq s_{i}$, and this family $\mathbf{a}=\{a^{i}_{k}\}$ is 
ordered $$(a^{1}_{1}, a^{1}_{2},\ldots, a^{1}_{s_{1}}, a^{2}_{1},a^{2}_{2},\ldots, a^{2}_{s_{2}},
\ldots , a^{n}_{s_{n}})~.~$$ With this order, the points in $\mathbf{a}$ are placed around the circle 
$S^{1}$ in clockwise order. We denote  by $C^{i}_{k}$ the (closed) arc of $S^{1}$ that starts at $a^{i}_{k}$ and stops at the next point in $\mathbf{a}$ (in cyclic order). 
\item[ii.] For each $i,k$, the restriction  $u|_{C^{i}_{k}}$ has a continuous 
lift $\hat{u}^{i}_{k}: C^{i}_{k}\to L_{i}$.
\item[iii.] $u$ has asymptotic corners  (in the usual sense \cite{Se:book-fukaya-categ}) at each of the 
$a^{i}_{k}$'s such that each of the $a^{i}_{k}$ is mapped to a self intersection point of $L_{i}$ 
for $1<k\leq s_{i}$, and to an intersection point of $L_{i}\cap L_{i+1}$ for $a^{i+1}_{1}$, $i\geq 1$
and, for $a^{1}_{1}$, to an intersection point in $L_{n}\cap L_{1}$. If the boundary conditions consist of a single Lagrangian $L_{1}$, we assume that $a_{1}^{1}$ is also mapped to a self intersection point of $L_{1}$. 
\end{itemize}
With the usual conventions for the orientation of $D^{2}\subset \C$, the punctures $a^{i}_{k}$ correspond
to ``entry points'' in the disk. To be more precise, viewing a puncture $a^{i}_{k}$ as an entry or an exit point is equivalent to making a choice of a class of strip-like end coordinates around the puncture: $0$-side of the strip before the $1$-side in clock-wise order corresponds to an entry and the other way around for an exit. 
The convention is that each time we use strip-like coordinates around a puncture, they are assumed to be of the corresponding type as soon as we fix the type of the puncture to be either an entrance or 
an exit.  To relate these configurations to the operations typical in Floer theory, it is often useful to consider all 
punctures as entries except possibly for one for which the strip-like coordinates are 
reversed so that it becomes an exit point.  We will fix the convention that the exit point, if 
it exists, is associated to  $a^{1}_{1}$.

\

We denote the moduli spaces of curves as above by $\mathcal{M}_{J;L_{1},\ldots, L_{n}}(x_{1},\ldots , x_{m}; y)$
where each $x_{i}$ is either one of the self intersection points of an $L_{j}$ viewed as a couple in $((x_{i})_{-}, (x_{i})_{+})\in  L_{j}\times L_{j}$, or an intersection point of 
$L_{j}\cap L_{j+1}$  that can also be viewed as a pair of points   
$((x_{i})_{-}, (x_{i})_{+})\in (L_{j}, L_{j+1})$,  in such a way that the punctures
$a^{i}_{k}$ are sent in order to the $x_{j}$'s, starting with the exit  $a^{1}_{1}$ that is sent to $y$.  Moreover, we fix conventions such that the path $\hat{u}^{i}_{k}(C^{i}_{k})$,  followed  clockwise  around the circle,  starts at a point $(x_{j})_{+}$ (or at $y_{-}$ in the case of $C^{1}_{1}$) and ends at 
$(x_{j+1})_{-}$ (or at $y_{+}$ for $C^{n}_{s_{n}}$).

It is also useful to consider the case when there is no exit point, thus even $a^{1}_{1}$ is an entry. 
The corresponding moduli space is denoted by $\mathcal{M}_{J;L_{1},\ldots, L_{n}}(x_{1},\ldots , x_{m}; \emptyset)$.
In the case of a single boundary condition $L_{1}$, we also allow for the case of no punctures ($\mathbf{a}=\emptyset$) and notice that in this case $u$ is a $J$-holomorphic disk with boundary on $L_{1}$ with the 
moduli space denoted $\mathcal{M}_{J,L_{1}}$. We omit some of the subscripts, such as the boundary conditions $L_{1},\ldots, L_{n}$, if they are clear from the context. 

\

Moduli spaces of this type appear often in Floer type machinery for instance
\cite{FO3:book-vol1}, \cite{Se:book-fukaya-categ} and, in the case of immersed Lagrangians, in
 \cite{Alston-Bao},\cite{Alston-Bao:imm2} as well as \cite{Akaho-Joyce}. 

\

A special class of moduli spaces as above plays a particular role for us. 
We will say that $u$ is a $\mathbf{c}$-marked $J$-holomorphic curve if it is a curve as above
with the additional constraint that:
\begin{itemize}
\item[iv.] If $n\geq 2$, for each $i\in \{1,\ldots, n\}$ and $1<k < s_{i}$  we have $(\hat{u}^{i}_{k}(a^{i}_{k+1}), \hat{u}^{i}_{k+1}(a^{i}_{k+1}))\in c_{i}$. In case $n=1$ and if $u$ has no exit, then we 
also assume   $(\hat{u}^{1}_{s_{1}}(a^{1}_{1}), \hat{u}^{1}_{1}(a^{1}_{1}))\in c_{1}$.
\end{itemize}
In other words, a $\mathbf{c}$-marked curve $u$ is a $J$-holomorphic polygon with asymptotic corners at the punctures 
given by the points $a^{i}_{k}$, such that it switches branches along the Lagrangian $L_{i}$ {\em only} at points belonging to $c_{i}$ and it switches from  $L_{i}$ to $L_{i+1}$ (or, respectively, from 
$L_{n}$ to $L_{1}$) at the puncture point $a^{i+1}_{1}$ (respectively, $a^{1}_{1}$). Notice that, 
if there is a single boundary condition $L_{1}$, and $a^{1}_{1}$ is an exit, then
the switch at the point $a^{1}_{1}$ is not required to belong to $c_{1}$. On the other hand, if there is a single boundary condition $L_{1}$, and all punctures 
are entries, then  we require for $a^{1}_{1}$ to also correspond to a point in $c_{1}$.
We denote the moduli spaces of $\mathbf{c}$-marked polygons by $\mathcal{M}_{J,\mathbf{c}; L_{1},\ldots, L_{n}}(x_{1},\ldots, x_{m};y)$. With respect to this notation, the condition relevant to $\mathbf{c}$ is, in summary, that all the switching of branches at entry points corresponds to $x_{k}$'s that belong to some $c_{i}$.

\begin{figure}[htbp]
   \begin{center}
    \includegraphics[width=0.5\linewidth]{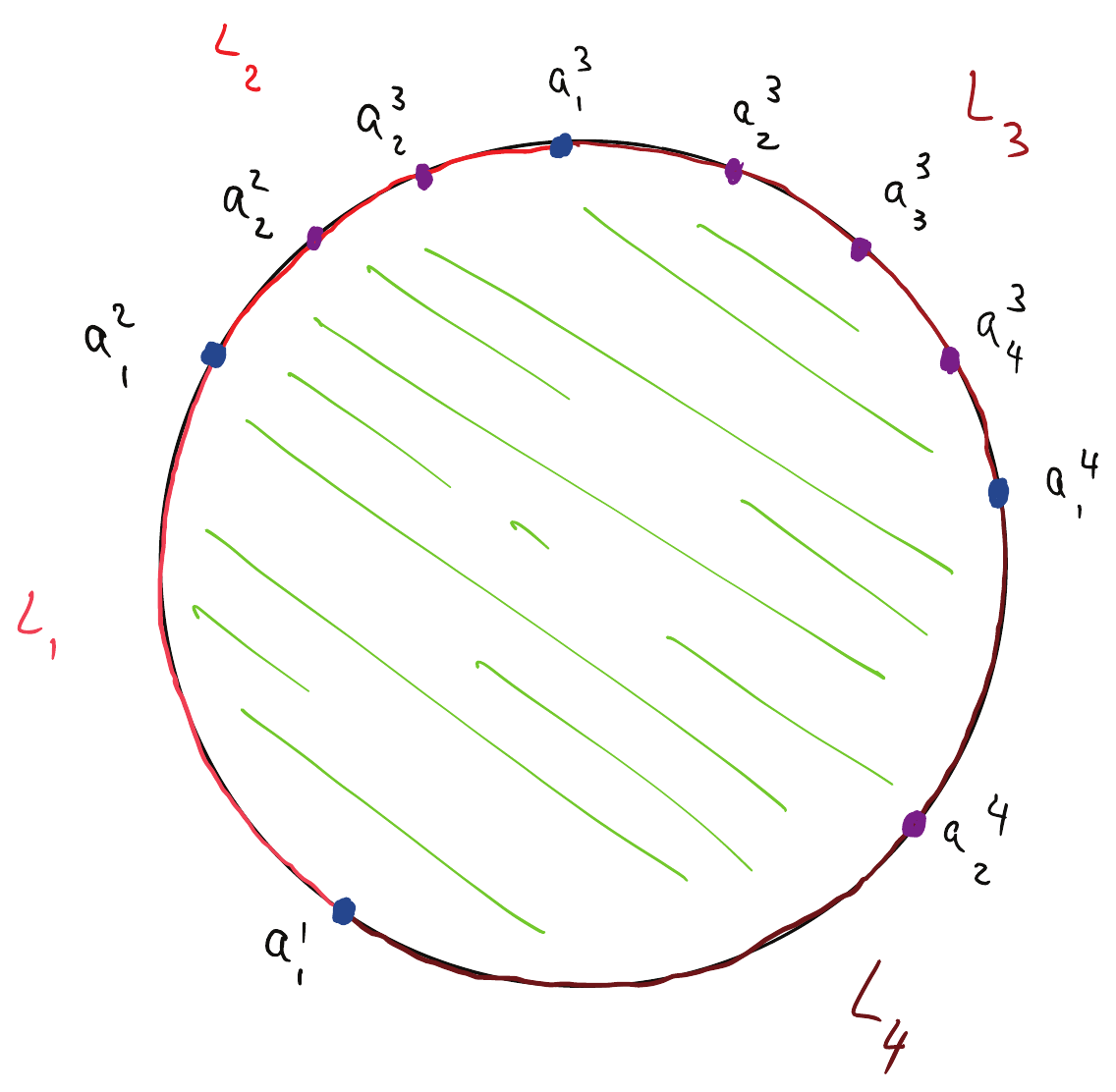}
   \end{center}
   \caption{\label{fig:polygon} An example of a curve $u$  with boundary conditions along $L_{1}$, $L_{2}$, $L_{3}$, $L_{4}$.}
\end{figure}

\

In the literature, most Floer type algebraic constructions associated to immersed Lagrangians  are defined through counts of curves that correspond, in our language, to the case $\mathbf{c}=\emptyset$. In other words, these are curves that do no switch branches along the immersed Lagrangians. In our case, the invariants discussed later in the paper are associated to counts of $\mathbf{c}$-marked curves, thus {\em switching of branches along $L_{i}$  is allowed as long as it  takes place at self intersection points
belonging  to $c_{i}$}.  
 
 \
 
A few special cases are worth mentioning: Floer strips, in the usual sense, correspond to the 
case when there are just two boundary conditions $L_{1}, L_{2}$ and only two punctures
$a^{1}_{1}$ - the exit -  and $a^{2}_{1}$ - the entry (in other words, $s_{1}=1$, $s_{2}=1$); curves
with boundary conditions again along $L_{1}$ and $L_{2}$ but with $s_{1}, s_{2}$ not necessarily
equal to $1$ but still with an exit at $a^{1}_{1}$ will be referred to as $\mathbf{c}$-marked Floer strips; another important special case appears when 
there is a single boundary condition and one puncture $a^{1}_{1}$ which is viewed as an exit, 
these curves are called tear-drops. Such a
curve with boundary condition along $L_{2}$ appears in Figure \ref{fig:tear-drop}. A more general case
of a $\mathbf{c}$-marked curve is also of interest to us: in case there is a single boundary condition $L_{1}$ and $s_{1}>1$, with $a_{1}^{1}$ an exit
and $a^{i}_{k}$ entries for all $k>1$ and, as discussed above, all entries are associated to elements in $c_{1}$.  
This type of curves will be referred to as marked $\mathbf{c}$-marked tear-drops (of course, we include
the usual, non-marked, teardrops among the marked ones). Finally, one last case: again there is just one boundary condition $L_{1}$ but in this case all $a^{1}_{k}$'s are viewed as entries. In case
all these entries (including $a^{1}_{1}$) are included in $c_{1}$ we call the curve a $\mathbf{c}$-marked $J$-disk with boundary on $L_{1}$.

\begin{rem}
a. With our conventions some geometric curves can appear in more than a single category. 
For instance, a ear-drop with an exit $(P_{-},P_{+})$ such that  $(P_{+},P_{-})$ belongs to the marking
$c_{1}$ can also be viwed  as a $\mathbf{c}$-marked disk. However, once the nature of the punctures is fixed there is no such ambiguity.

b. Recall from Remark \ref{rem:direct-sum} that the objects in $\mathcal{L}ag^{\ast}(M)$ are 
unions of immersions with isolated, transverse double points. The definition of the moduli spaces before extends 
naturally to this situation because of condition ii at the beginning of this section, and because the markings are self-intersection points of immersions with isolated, transverse double points.
\end{rem}

\

For a curve $u$ as above we denote by $|u|$ the number of punctures of $u$, $|u|=s_{1}+s_{2}+\ldots + s_{n}$. 
 
\begin{figure}[htbp]
   \begin{center}
    \includegraphics[width=0.4\linewidth]{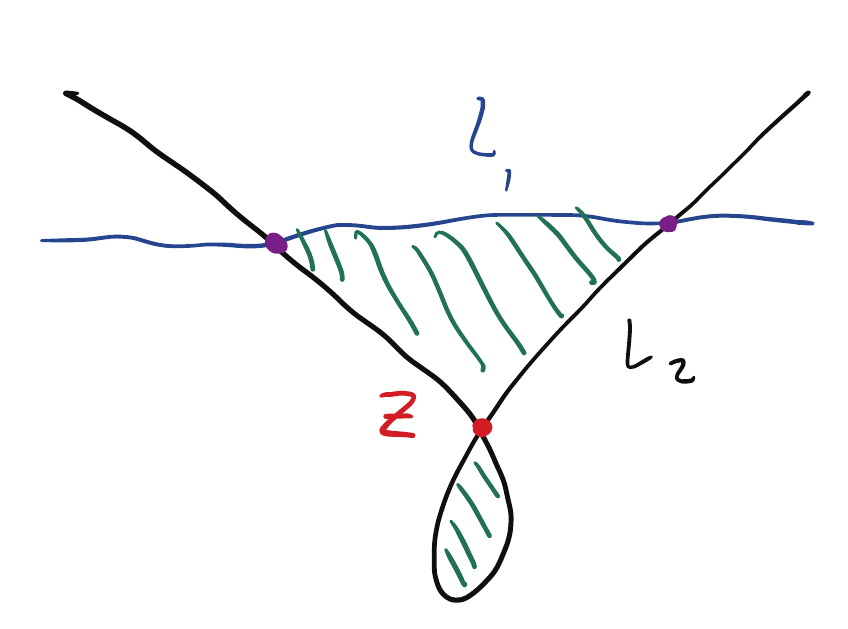}
   \end{center}
   \caption{\label{fig:tear-drop} A tear-drop at $Z$ (along $L_{2}$) and a triangle - a curve with two boundary conditions $L_{1}$, $L_{2}$ and $s_{1}=1$, $s_{2}=2$.}
\end{figure}

\

While in the discussion above we have only considered $J$-holomorphic curves, in practice we need
to also include in our consideration perturbed such curves. We will discuss the system of perturbations
in some more detail below but, to fix ideas, we mention that for all moduli spaces of curves $u$ with $|u|\geq 2$, the choice of these perturbations follows Seidel's scheme in his construction of the Fukaya category \cite{Se:book-fukaya-categ}  (we assume here familiarity with his approach). 

\

Fix one marked, immersed Lagrangian $(L, c)$. 

\begin{dfn}\label{defi:perturbD}A regular, coherent system of perturbation data $\mathcal{D}_{L}$ for $(L,c)$ consists of the following:
\begin{itemize}
\item[i.]A time dependent almost complex structure $J=J_{t}$, $t\in [0,1]$, with $J_{0}=J_{1}$,
 such that:
\begin{itemize}
\item[a.] all moduli spaces of 
$J_{0}$-holomorphic {\em $\mathbf{c}$- marked} disks with boundary on  $L$ are void.
\item[b.]  the moduli spaces of 
non-marked $J_{0}$-holomorphic tear-drops are all void.
\end{itemize}
\item[ii.] A system of perturbations parametrized by the universal-associahedron, as described in \cite{Se:book-fukaya-categ}, such that all the resulting moduli spaces of polygons $u$ with boundary on $L$, with $|u|\geq 2$ and with an exit %and virtual dimension at most $1$ 
are regular
and satisfy the following additional properties:
\begin{itemize}
\item[a.] On the boundary of each polygon $u$ (with $|u|\geq 1$) the almost complex structure coincides 
with $J_{0}$ and the Hamiltonian perturbation is trivial. 
\item[b.] For each polygon $u$ with $|u|=2$, with one entrance and one exit, 
the respective perturbation has a trivial Hamiltonian term and its almost complex structure part coincides with $J_{t}$ in strip-like coordinates (where $(s,t)\in \R\times [0,1]$).
\item[c.] For each polygon $u$ with $|u|\geq 3$, there are strip-like end coordinates around the punctures
 and in these coordinates the perturbation has vanishing Hamiltonian term and its almost complex part coincides
 with $J_{t}$.
 \item[d.] The choices of perturbations are coherent with respect to gluing and splitting for all curves $u$ with $|u|\geq 2$.
\end{itemize}
\end{itemize}
\end{dfn}

We refer to the almost complex structure $J_{0}$ as the {\em base} almost complex
structure of the coherent system of perturbations $\mathcal{D}_{L}$. In some of the considerations
needed later in the paper we will also need to use curves with some interior marked points (see also point b
in Remark \ref{rem:coherent-perturb} below) but for the moment we will limit our discussion to the 
types of curves introduced above.

\begin{rem} \label{rem:coherent-perturb}
a. Condition i.a in the definition above requires not only that there are no $J_{0}$-holomorphic disks with boundary on $L$ but also that there are no $\mathbf{c}$-marked polygons with boundary on 
$L$ and without an exit.  

b. It is easy to formally relax condition i.b to only require that moduli spaces of (non-marked), possibly perturbed, tear-drops are regular and most further definitions and constructions work under this weaker  
assumption. However, this regularity is hard to achieve in practice in full generality. One of the reasons is that tear-drops carry a single puncture and thus their domains are unstable. As a result, Seidel's recursive scheme in choosing perturbations for polygons with more and more corners does not automatically work. 
 Alternative ways to address this in some settings appear in \cite{Alston-Bao}, another possibility is to implement the Kuranishi method as in \cite{Akaho-Joyce}. Yet another different approach, initiated by Lazarinni \cite{Laz:decomp} for disks and pursued for polygons by Perrier \cite{Per1}, is to attempt to use more special but still generic classes of autonomous almost complex structures in the whole construction. 

In this paper we will mostly 
assume as at point i.b in Definition \ref{defi:perturbD} that {\em all moduli spaces of (non-marked) tear-drops are void}.  Later in the paper, when dealing with the moduli spaces of curves
with boundaries along cobordisms, there are special situations when tear-drops can be stabilized by using an interior marked point and in that case Seidel's method to pick coherent perturbations continues to apply.  

c. Notice that we require regularity of moduli spaces of all polygons with boundary on $L$ and not only
of the $\mathbf{c}$-marked polygons. The reason is that, even if the various $\mu_{k}$ operations  defined later in the paper only use $\mathbf{c}$-marked polygons, the proofs of relations of type $\mu\circ \mu=0$ make also use of non $\mathbf{c}$-marked polygons as these  appear through an application of Gromov compactness. However, the regularity of all polygons can be relaxed at point {\em ii} of Definition
\ref{defi:perturbD} to the regularity of all the moduli spaces of $\mathbf{c}$-marked polygons as well as that of the moduli spaces appearing in compactifications of these.
\end{rem}

\subsubsection{Moduli spaces of marked polygons.} \label{subsubsec:marked-poly}
Given systems of coherent perturbations $\mathcal{D}_{L}$ for each $(L,c)$ in some
 class $Lag^{\ast}(M)$, these can be extended in 
the sense in \cite{Se:book-fukaya-categ} to a coherent system of perturbations for curves with boundary conditions along all the $(L,c,\mathcal{D}_{L})$'s.  In full generality, this requires picking for each pair $(L,c), (L',c')$ Floer data $H_{L,L'}, J_{L,L'}$ where, in general, $H_{L,L'}$ is not vanishing so that the construction can handle Lagrangians with non-transverse intersections. Moreover, $J_{L,L'}$ is non-autonomous and picked in such a way as to agree with $(J_{0})_{L}$ and, respectively, with $(J_{0})_{L'}$ for $t$ close to $0$ (respectively,  close to $1$).

The resulting moduli spaces involving these systems of coherent perturbations will be denoted
by $\mathcal{M}_{\mathcal{D};L_{1},\ldots, L_{n}}(x_{1},\ldots , x_{m}; y)$ for boundary
conditions along $L_{1},\ldots, L_{m}$.
Notice that, in the non-transverse intersection case the punctures of a curve $u$ in these moduli space
correspond to two types of ends: self intersection points $x_{i}\in I_{L_{i}}$ and time-$1$  Hamiltonian chords associated to the Hamiltonians $H_{L_{i},L_{i+1}}$. We denote these chords by $\mathcal{P}(L_{i},L_{i+1})$.  To simplify the discussion we focus below on the transverse case with the understading that the adjustements of the relevant notions to the non-transverse situation are immediate.
Similarly, in the marked case, we have the  moduli spaces  $\mathcal{M}_{\mathcal{D},\mathbf{c};L_{1},\ldots, L_{n}}(x_{1},\ldots , x_{m}; y)$
\

It is useful to regroup various components of these moduli spaces in the following way. 
We fix intersection points (or, more generally, Hamiltonian chords) $h_{i}\in\mathcal{P}( L_{i}, L_{i+1})$ and $y\in \mathcal{P}(L_{1}, L_{n})$. We denote by $\overline{\mathcal{M}}_{\mathcal{D};L_{1},\ldots, L_{n}}(h_{1},\ldots , h_{n-1}; y)$ be the union
of the curves $u\in\mathcal{M}_{\mathcal{D};L_{1},\ldots, L_{n}}(x_{1},\ldots , x_{m}; y)$
where, in order, $x_{1},\ldots, x_{s_{1}}$ are self intersection points of $L_{1}$, 
$x_{s_{1}+1}=h_{1}$, the next $s_{2}$ - points $x_{i}$ are self intersections points of $L_{2}$,
the next point equals $h_{2}$ and so forth around the circle (see again Figure \ref{fig:polygon}). 
In essence, the spaces $\overline{\mathcal{M}}$ group together curves with fixed ``corners''
at the chords $h_{i}\in \mathcal{P}(L_{i}, L_{i+1})$, but with variable numbers of additional corners 
at self intersection points along each one of the $L_{i}$'s. We will also use the notation 
$\overline{\mathcal{M}}_{\mathcal{D},\mathbf{c};L_{1},\ldots, L_{n}}(h_{1},\ldots , h_{n-1}; y)$
for the corresponding $\mathbf{c}$-marked moduli spaces which are defined, as usual, by requiring
that all self intersection corners belong to $c_{i}\subset I_{L_{i}}$. For instance, with this notation, the moduli
space of all $\mathbf{c}$-marked tear-drops at a self intersection point $y\in I_{L_{1}}$
is written as $\overline{\mathcal{M}}_{\mathcal{D}_{L_{1}},c_{1};L_{1}}(\emptyset;y)$

\subsubsection{Unobstructed marked Lagrangians in $M$} 
The purpose of this subsection is to define the notion of unobstructed marked Lagrangians that appears
in the proof of Theorem \ref{thm:surg-models}. 

\begin{dfn}\label{def:unobstr}
A triple $(L,c, \mathcal{D}_{L})$ with $(L,c)$ a marked Lagrangian and $\mathcal{D}_{L}$ a coherent
system of regular perturbations as in Definition \ref{defi:perturbD} is called {\em unobstructed} if, 
for each self intersection point $y\in I_{L}$, the $0$-dimensional part of the moduli space of $\mathbf{c}$-marked tear-drops $\overline{\mathcal{M}}_{\mathcal{D}_{L},c_{1};L}(\emptyset;y)$ is compact and the $mod\ 2$ number of its elements vanishes.
\end{dfn}
\begin{rem}\label{rem:bding-chain}
In the terminology of \cite{FO3:book-vol1} and \cite{Akaho-Joyce} the choice of a marking $c$ as in the
definition above is a particular case of a {\em bounding chain}. 
\end{rem}

Over $\Z/2$ the definitions of Floer and Fukaya category type algebraic structures 
generally require that the relevant moduli spaces satisfy certain compactness conditions ensuring that the 
relevant $0$-dimensional spaces are compact and that the $1$-dimensional moduli spaces have appropriate compactifications. This follows from Gromov compactness as long as apriori energy bounds are available.
We integrate this type of condition in our definition of an unobstructed class of Lagrangians.

\begin{dfn}\label{def:finite-energy}
A class $Lag^{\ast}(M)$ of marked, immersed
Lagrangians $(L,c, \mathcal{D}_{L})$, each of them unobstructed in the sense of Definition \ref{def:unobstr}, is {\em unobstructed}
 with respect to a system 
of coherent perturbations $\mathcal{D}$ that extends the $\mathcal{D}_{L}$'s if:
\begin{itemize}
\item[i.] $\mathcal{D}$ is regular.
\item[ii.] for any finite family $(L_{i},c_{i})$, $1\leq i\leq m$, there exists a constant $E_{L_{1},\ldots, L_{n};\mathcal{D}}$
such that we have  $E(u)\leq E_{L_{1},\ldots, L_{n};\mathcal{D}}$ for any $h_{i}\in \mathcal{P}(L_{i}, L_{i+1})$, $y\in \mathcal{P}(L_{1},L_{n})$ and $u\in\overline{\mathcal{M}}_{\mathcal{D},\mathbf{c};L_{1},\ldots, L_{n}}(h_{1},\ldots , h_{n-1}; y)$.
\end{itemize}
\end{dfn}

Here $E(u)$ is the energy of the curve $u$, $E(u)=\frac{1}{2}\int_{S}||du-Y||^{2}$ (where $S$ is the 
punctured disk, domain of $u$, and $Y$ is the $1$-form with values in Hamiltonian vector fields used in the 
perturbation term). Regularity of $\mathcal{D}$
means that all the moduli spaces of $\mathbf{c}$-marked (and non-marked) polygons with boundaries along families picked from $Lag^{\ast}(M)$ are regular.
We refer to condition ii. as the {\em energy bounds} condition relative to $\mathcal{D}$.

\begin{rem}
a. Through an application of Gromov compactness,  the energy bounds condition implies that the number of corners at self intersection points of all  $\mathbf{c}$-marked polygons with boundaries along the family $(L_{1},\ldots, L_{n})$ is uniformly bounded.

b. Stardard methods show that, starting with a family $\{(L,c,\mathcal{D}_{L})\}$  of 
unobstructed marked Lagrangians, it is possible to extend the perturbations $\mathcal{D}_{L}$ to a coherent, regular, system of perturbations $\mathcal{D}$, as required at point i. in the definition. 
However, in the absence of some {\em apriori} method to bound
the energy of the curves in the resulting moduli spaces - such as estimates involving primitives of the Lagrangians involved or monotonicity arguments etc - defining invariants over $\Z/2$ is not possible even for this regular $\mathcal{D}$. Moreover, the energy bounds condition itself is not stable with respect to 
small perturbations and constraints such as exacteness or monotonicity are lost through surgery. As a consequence,  we made here the unusual choice to integrate the
energy bounds condition in the definition.

c. When working over the universal Novikov field the energy bounds condition is no longer required.
\end{rem}

\subsection{Unobstructed marked cobordisms.}\label{subsubec:marked-cob}

The next step is to explain the sense in which a class of marked, immersed, cobordisms 
$Lag^{\ast}(\C\times M)$ is unobstructed. In essence, the condition is 
the same as in Definition \ref{def:unobstr} but there are additional subtelties in this case that have to do with the fact that a cobordism with immersed ends has double points that are not 
isolated. 

\

We start by making more precise  the type of immersions included in the class 
$Lag^{\ast}(\C\times M)$. These are unions of 
Lagrangian cobordism immersions $j_{V}:V \to  \C\times M$
such that the immersion 
$j_{V}$ has singular points of (potentially) two 
types: the first type consists of isolated, transverse double points; the second are clean intersections along  a sub-manifold $\Sigma_{V}$ of dimension $1$ in $\C\times M$ (possibly with boundary, and not necessarily compact or connected) such that the projection $$\pi:\C\times M\to \C$$ restricts to an embedding on each connected component of $\Sigma_{V}$. Moreover, each non-compact component of $\Sigma_{V}$ corresponds
to an end of the cobordism.  More precisely, if a cobordism $V$ has an end $L$ (say positive) that is immersed with a double point $(P_{-},P_{+})$, then, by definition,  a component of $\Sigma_{V}$ 
includes an infinite semi-axis of the form $[a,+\infty)\times \{k\}\times \{P\}$ where
$P=j_{L}(P_{-})=j_{L}(P_{+})$. We assume that all non-compact components of $\Sigma_{V}$ correspond to a double point from the fiber in this sense.  We continue to denote by $I_{V}$ the 
double point set of such an immersion: $$I_{V}=\{(P_{-},P_{+})\subset V\times V \ : \ j_{V}(P_{-})=j_{V}(P_{+})\}~.~$$ 
We denote the $i$-dimensional subspace of $I_{V}$ by $I^{i}_{V}$, $i\in\{0,1\}$. In particular,
$\Sigma_{V}$ is the image of $I^{1}_{V}$ through $j_{V}$.

As in the case of the class $Lag^{\ast}(M)$, is is sufficient to focus here on the cobordisms $V$ as above
as the various arguments extend trivially to finite disjoint unions of such objects.

%Recall from \S\ref{subsec:surg}
%that a marked Lagrangian cobordism $(V,c)$ is a Lagrangian immersion as before together
%with a subset $c$ of the path connected components of $I_{V}$.

\subsubsection{A class of perturbations for cobordisms with $1$-dimensional clean intersections.}\label{subsubsec:preturb}

Dealing analytically with clean self intersections along intervals with boundary is probably possible but it requires some new ingredients (due to the boundary points) that we prefer to avoid (for Floer theory
for Lagrangians with clean intersections see \cite{Schmaschke-imm}). 
For this purpose we will  describe below a class of perturbations that transforms such a cobordism $V$ into
a Lagrangian $V_{h}$ immersed in $\C\times M$ which is no longer a cobordism but has isolated,
generic double points and whose behaviour at $\infty$ presents ``bottlenecks'' in the sense 
of \cite{Bi-Co:lcob-fuk}. These perturbations will also be needed to even define the notion of marked cobordism. 

\

Fix a cobordism $V$ immersed in $\C\times M$, as discussed before, with clean intersections
along a one-dimensional manifold $\Sigma_{V}$. All perturbations to be considered here 
can be described as follows. Consider $U\subset T^{\ast}V$ a small neighbourhood of the $0$-section.
We can take $U$ small enough so that the immersion $i_{V}:V\to \C\times M$ extends to an 
immersion (still denoted by $i_{V}$) of $U$ which is symplectic and such that $i_{V}$ is an embedding 
on each $U_{x}=U\cap T^{\ast}_{x}V$. We then consider a small Hamiltonian perturbation (that will be made more precise below) of the $0$-section, $V_{h}\subset U$. The perturbed immersion that we are looking for is given by the restriction of $i_{V}$ to $V_{h}$. 

We now describe more precisely the perturbation $V_{h}$ of the $0$-section. It will be written as the 
time one image of $V$ through a Hamiltonian isotopy induced by a Hamiltonian $h$. To make this $h$
explicit, notice that the 
$1$-dimensional manifold $\Sigma_{V}$ has some compact components as well as some non-compact ones. We focus now on one of the non-compact components, $\Sigma'$. This corresponds to one 
of the ends $L$ of the cobordism and, more precisely, to a double point $(P_{-},P_{+})\in I_{L}$ of the 
immersion $j_{L}$. We consider two small disks $D_{-}$ and $D_{+}$ in $L$ around respectively
the points $P_{-}$ and $P_{+}$. To fix ideas, we imagine these two disks to be of unitary radius.
Inside these disks we consider smaller disks $D'_{-}$ and $D'_{+}$ of half radius as well as closed 
rings $C_{-}$, $C_{+}$ that are, respectively, the closure of the complements of the $D'_{\pm}$'s inside the
disks $D_{\pm}$.

By assumption, outside of a compact set, $\Sigma'$ is of the form $[a,+\infty) \times \{k\}\times \{P\}$
and we have the inclusions $[a,+\infty) \times \{k\}\times (D_{-}\sqcup D_{+})\subset [a,+\infty) \times \{k\}\times L\subset V\subset U$. 
Outside of $[a, +\infty)\times\{k\}\times (D_{-}\cup D_{+})$ we take the pertubation $h$ to be the identity and thus $V_{h}$ is equal to $V$ there.  We now describe the construction on the region $[a,+\infty) \times \{k\}\times (D'_{-}\cup D'_{+})$.  
We consider a Hamiltonians $h_{\pm}$ defined in a small neighbourhood of the zero section of $T^{\ast}V$ given as a composition $h_{\pm}=\hat{h}_{\pm}\circ Re( \pi\circ j_{V})$
where $\hat{h}_{\pm}: [a,\infty)\times \to \R$ has a nondegenerate critical point at $a+2$, is linear 
for $t\geq a+4$ (increasing for $h_{+}$ and decreasing for $h_{-}$)  and vanishes for $t\in [a, a+1]$,
see Figure \ref{fig:bottle}. 

The perturbation $h$ is required to agree with $h_{\pm}$ on $[a, +\infty)\times\{k\}\times (D'_{-}\cup D'_{+})$. We now describe $h$ on $[a, +\infty)\times\{k\}\times (C_{-}\cup C_{+})$. We write
each $C_{\pm}$ as a cylinder $S_{\pm}\times [0,1]$ with the $0$ end corresponding to the inner
boundary of $C_{\pm}$ and the $1$ end corresponding to the outer boundary. Denote by $s$ the
variable associated to the height in these cylinders.
We define $h:[a, +\infty)\times\{k\}\times C_{\pm}\to \R$, to be an interpolation $h^{s}_{\pm}$, $s\in [0,1]$
between $h_{\pm}$  and $0$, equal to $h_{\pm}$ for $s$ close to $0$ and equal to $0$ for $s$
close to $1$. More explicitely, $h^{s}_{\pm}=\hat{h}^{s}_{\pm}\circ Re( \pi\circ j_{V})$ 
with $\hat{h}^{s}_{\pm}=(1-s)\hat{h}_{\pm}$. The image of $V$ through the time one
hamiltonian diffeomorphism associated to $h$, $V_{h}$, has a projection onto $\C$ (along the end discussed here) as in Figure \ref{fig:bottle}. Assuming $h_{\pm}$ sufficiently small we see that,
inside the set of double points of $j_{V}|_{V_{h}}$, the set $\Sigma'$ has been replaced by a union
of a single, generic, double point that lies over $\{a+2\}\times \{k\}$ (and, in the fibre, it is close to $P\in M$)  and a
clean intersection, compact component along $\Sigma'_{h}=[a,a+1]\times \{k\}\times\{P\}$.
We will say that the double point over $\{a+2\}\times \{k\}$ is a bottleneck at $a+2$.
This construction can be repeated for 
all non-compact ends. There is an obvious adjustment for the negative ends where the bottlenecks will be over
$-a-2$. Notice, that in this case too, the crossing pattern of the two curves $\gamma_{-}$ and 
$\gamma_{+}$ over the bottleneck is assumed to be the same as in Figure \ref{fig:bottle}.
Once this construction completed, we are left 
with double points that are either generic and isolated or along clean intersections along compact $1$-manifolds.  The construction of $V_{h}$ ends by a generic perturbation, 
supported in a neighbourhood of the preimage through $j_{V}$ of these compact components,
such that these components are also replaced by a union of isolated double points.
\begin{figure}[htbp]
   \begin{center}
    \includegraphics[width=0.55\linewidth]{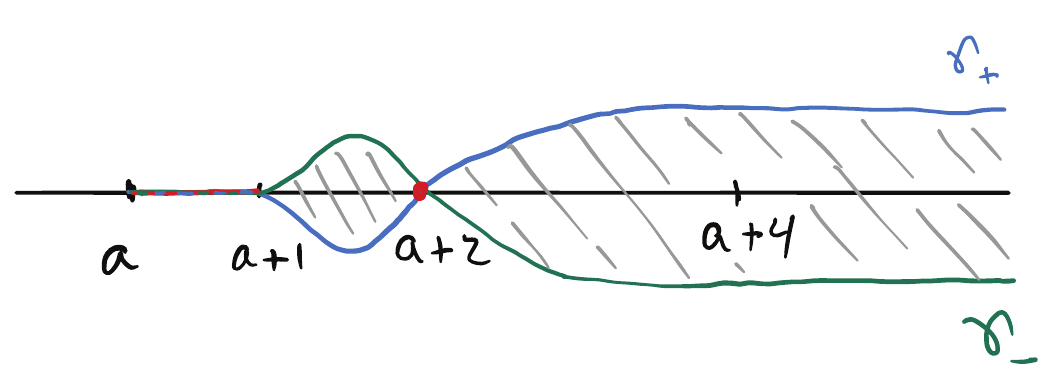}
   \end{center}
   \caption{\label{fig:bottle} The projection onto $\C$ of the $L$-end of $V_{h}$, 
   after replacing the semi-axis of double points
   $\Sigma'=[a,\infty)\times \{k\}\times \{P\}$ with a compact clean intersection along $[a,a+1]$ and a double point over $\{a+2\}\times\{k\}$.}
\end{figure}

\begin{rem} For each double point $P$ of an end $L$ of the cobordism $V$, the perturbation above
 requires a choice of a lift of $P$ to $I_{L}$. We will refer to a perturbation $h$ as before to be {\em positive} relative
to a point $(P_{-},P_{+})\in I_{L}$ if $h$ has the profile in Figure \ref{fig:bottle}. In other words, if
 $V_{h}$ contains $\gamma_{+}\times P_{+}\cap \gamma_{-}\times P_{-}$. It is negative relative 
 to $(P_{-},P_{+})$ in the opposite case. We denote all the intersection points where $h$ 
 is positive by $(I_{L})^{+}_{h}$.
\end{rem}

\subsubsection{Unobstructed cobordisms} \label{subsubsec:marked-cob}
The definition of unobstructed marked cobordism below is quite complicated. 
This being said, it is a minimal list of conditions to attain the following  aims: first, such a cobordisms $V$ should allow the definition of Floer homology $HF(V',V)$ for $V'=\gamma\times L$  with $\gamma\subset \C$ an appropriate planar curve and $L\in Lag^{\ast}_{e}(M)$; secondly, the data associated to $V$ should restrict appropriately to the ends; finally, cabling in the sense of \S\ref{sec:cob} should be possible inside this class.
In practice, these conditions will become much simpler in the exact setting that we will mainly focus on,
as indicated in \S\ref{subsec:exact-cob}.

\

Assume that $\{(L'_{j},c'_{j})\}$ and $\{(L_{i},c_{i})\}$ are respectively marked Lagrangians in $M$.
\begin{dfn}\label{def:marked-cob}
A {\em marked cobordism} $(V,h,c)$ between these two families of marked Lagrangians is an immersed
Lagrangian cobordism
$V: (L'_{j})\cobto (L_{i})$  as in \S\ref{subsubsec:preturb}, with double points either transverse and isolated or along clean $1$-dimensional intersections, together with a perturbation $V_{h}$ as above and a choice of $c\in I_{V_{h}}$ such that:
\begin{itemize}
\item[i.] The restriction of $c$ over the positive bottlenecks coincides with $c'_{j}$ for the end $L'_{j}$. Similarly, the restriction of $c$ over the negative bottlenecks coincides with $c_{i}$ for the end $L_{i}$.
\item[ii.] The perturbation $h$ is positive relative to each point of $c$ that  corresponds to an end, as at point \emph{i} above.
\end{itemize} 
\end{dfn}

Defining unobstructedness for cobordisms requires consideration of some additional moduli spaces of (marked) $J$-holomorphic curves. The reason is that cabling produces naturally tear-drops as one can see from Figure \ref{Fig:cabling}.  In turn, this leads to difficulties with regularity, as mentioned before in Remark \ref{rem:coherent-perturb} b.  

\

Consider a marked cobordism $(V,h,c)$ as above. To define the new moduli spaces we first fix some points in the plane $P_{1},\ldots, P_{k}$ such that
$(\{P_{i}\}\times M)\cap V_{h}=\emptyset$. The new moduli spaces consist of marked, perturbed $J$-holmorphic polygons $u$ (just as in the definitions in \S\ref{subsubsec:curves-marked}) with boundary along   $V_{h}$  but with additional distinct (moving) interior marked points  $b_{1},b_{2},\ldots, b_{r}\in int(D^{2})$ with the property that $u(b_{i})\in \{P_{j}\}\times M$ for some $j$. Around each point $P_{j}$ we fix  closed disks $D_{j}$, $D'_{j}$ such that 
$D_{j}\subset D'_{j}\subset \C\backslash  \pi(V_{h})$ and $P_{j}\in int(D_{j})\subset D_{j}\subset int (D'_{j})$. We will assume that each bounded, connected component of $\C\backslash \pi(V_{h})$ contains {\em at most} one point $P_{j}$.  Moreover, we assume that, for each $P_{j}$ there is {\em at most one} of the marked points $b_{j}$ that is sent to $P_{j}\times M$.

\begin{rem} \label{rem:interior-mark} The condition that each marked point $b_{i}$ is sent to a different
$P_{j}\times M$ is quite strong. The reason it is needed is that we are working here over $\Z/2$ and
 counting meaningfully configurations with more marked points being sent to the same hypersurface $P_{j}\times M$ is delicate.
\end{rem} 

We will consider moduli spaces of curves $u$ with marked points along the boundary and with 
interior marked points $b_{i}$ that satisfy a Cauchy-Riemann type equation
such that the allowable perturbations appearing in the  (perturbed) Cauchy-Riemann type equation associated to these $u$ contain a perturbation supported 
 in small neighbourhoods of the points $b_{i}$, in the domain of $u$, and defined 
in terms of almost complex structures and  Hamiltonian perturbations supported inside $D_{i}\times M$.
Over the region $D'_{i}\backslash D_{i}$ the Hamiltonian perturbation term is trivial
and the almost complex structure is of the form $i\times J_{0}$ (outside of the region $D'_{i}\times M$
we allow for the possibility of other perturbation terms; in short, the support of our perturbations is 
in $(\C\backslash D'_{i})\cup D_{i}$).  The resulting moduli spaces will be denoted
by  $\mathcal{M}_{(\mathcal{D}, P_{1},\ldots, P_{r});\mathbf{c};V}(x_{1},\ldots , x_{m}; y)$. Notice that the choices available for a marked point $b_{i}$ as before add two parameters  to the relevant moduli space and  the condition $u(b_{i})\in \{P_{i}\}\times M$ substracts two parameters.

\begin{rem}
In principle we would like to avoid domain dependent perturbations. With the degree constraints given here this is probably possible but we will not pursue this appoach.\end{rem}

\begin{figure}[htbp]
   \begin{center}
      \includegraphics[width=0.87
      \linewidth]{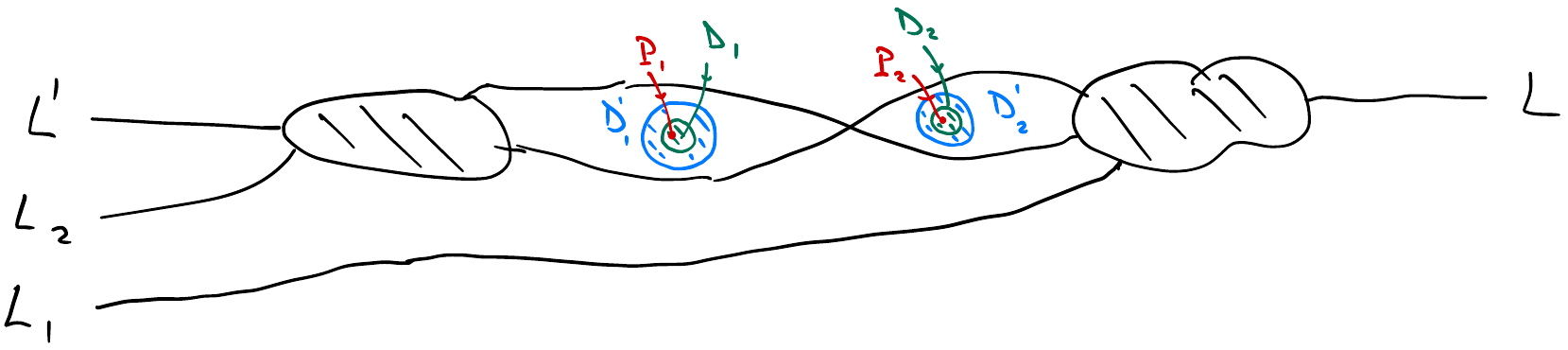}
   \end{center}
   \caption{\label{fig:markedpts} The projections of a perturbed cobordism $V_{h}$ onto $\C$
   together with the poinst $P_{1},P_{2}\in\C$ used to stabilize tear-drops.}
\end{figure}

Our choices mean that, for any point $P_{i}$ and any curve $u$ in one of these moduli spaces, there 
is a well defined degree of the map $\pi\circ u$ at any of the points $P_{i}$ and, because the composition
$\pi\circ u$ is holomorphic over $D'_{i}\backslash D_{i}$ this degree is positive  or null.
Moreover, the degree is null only if the curve $u$ does not go through $P_{i}$. 

We will say that a curve $u$ is of {\em positive degree} if one of these degrees is non vanishing and is {\em plane-simple} if each of these degrees is $0$ or $1$.
All the curves that we will need to use here will be plane-simple. 

We denote by $||u||$ the total valence of $u$ - the number given as $|u|+$ twice the number
of internal punctures of $u$ (recall that $|u|$ is the number of boundary punctures of $u$).   
By convention, we will assume that an element  $u\in\mathcal{M}_{(\mathcal{D}, P_{1},\ldots, P_{s});\mathbf{c};V}(x_{1},\ldots , x_{m}; y)$ has a positive degree at each of the $P_{i}$'s that appear
as subscripts.  As mentioned before, in our examples, the degree of the curves $u$ at each 
point $P_{j}$ is at most $1$ and there is precisely a single marked point $b_{j}$ associated 
to each $P_{j}$ for which this degree is non-vanishing.
 
 We will view below the choice of points $P_{1},\ldots, P_{k}$ and the relevant perturbations to be 
included in the perturbation data $\mathcal{D}_{V_{h}}$. We call the points $P_{j}$ the {\em pivots} associated to $V$ and we will denote the associated hypersurface
$\cup_{j} P_{j}\times M$ by $H_{V}$. In terms of terminology, we will refer
to  $\mathbf{c}$-marked polygons, tear-drops etc without generally distinguishing between the corresponding objects with interior marked points and those without them: thus tear-drops can be curves with $|u|=||u||=1$ or  curves with $|u|=1$, $||u||\geq 3$.  

In this setting, we define unobstructed cobordisms as follows:

\begin{dfn}\label{def:unobstrcob}
A cobordism $(V,h,c)$ as in Definition \ref{def:marked-cob}, is {\em unobstructed} if:
\begin{itemize}
\item[i.]   there is a system of coherent perturbation data
$\mathcal{D}_{V_{h}}$  containing a base almost complex structure, denoted $\bar{J}_{0}$, 
defined on $\C\times M$ and such that:

\begin{itemize} \item[a.] the moduli spaces of $\bar{J}_{0}$-marked disks are void and the moduli spaces of non-marked $\bar{J}_{0}$-holomorphic tear-drops are also void except for those moduli spaces of
tear-drops that are of positive degree and plane-simple.
\item[b.] the regularity and coherence conditions in Definition \ref{defi:perturbD} apply but only to those moduli spaces of curves with $||u||\geq 2$ consisting of $\mathbf{c}$-marked polygons 
and the moduli spaces appearing in their compactification (as in Remark \ref{rem:coherent-perturb} c).
\item[c.] the $0$-dimensional moduli space of teardrops with $||u||\geq 3$ associated to  $(V_{h}, c, \mathcal{D}_{V_{h}})$ at any point $y\in I(V_{h})$ is compact and the number of its elements {\em mod $2$} vanishes (as in Definition \ref{def:unobstr}). 
\end{itemize}
\item[ii.] the data $\mathcal{D}_{V_{h}}$  extends  the corresponding data for the ends in the following sense: 
\begin{itemize}
\item[a.] There are coherent systems of perturbations $\mathcal{D}_{L_{i}}$ and $\mathcal{D}_{L'_{j}}$
for all the ends of $V$ such that the base almost complex structure $J_{0}$ (as  in Definition \ref{defi:perturbD} \emph{i}) is the same for all $L_{i},L'_{j}$.
\item[b.] The triples $(L_{i},c_{i},\mathcal{D}_{L_{i}})$ and $(L'_{j}, c'_{j},\mathcal{D}_{L'_{j}})$ are
unobstructed. 
\item[c.] The  base almost complex structure $\bar{J}_{0}=\bar{J}_{V_{h}}$  has the property that $\bar{J}_{0}=i\times J_{0}$ over a region of the form $(-\infty, -a-\frac{3}{2}]\times \R\ \bigcup\  [a+\frac{3}{2},\infty)\times \R$.
\item[d.] The coherent system of perturbations $\mathcal{D}_{V_{h}}$ has the property that it restricts to the
$\mathcal{D}_{L_{i}}$, respectively, $\mathcal{D}_{L'_{j}}$  over all negative, respectively, positive
bottlenecks.
\item[e.] The perturbation $h$ and the data $\mathcal{D}_{L_{i}}$, $\forall i$, 
have the property that the moduli spaces of polygons in $M$ with boundary
on $L_{i}$ and with  entries all belonging to $(I_{L_{i}})^{+}_{h}$  are non-void only if the exit  of the polygons also belongs to $(I_{L_{i}})^{+}_{h}$.
Similarly for $\mathcal{D}_{L'_{j}}$ and $L'_{j}$, $\forall j$.  
\end{itemize}
\end{itemize}
\end{dfn}

We will denote an unobstructed cobordism in this sense by $(V,h,c,\mathcal{D}_{V_{h}})$.

\begin{rem}\label{rem:gluing-pert}
a. Assumptions {\em i, ii} in Definition \ref{def:marked-cob} and {\em ii.}\emph{a},\emph{b},\emph{c},\emph{d}
in Definition \ref{def:unobstrcob}
are basically what is expected from compatibility with the ends. We have explained before the 
reason to allow the existence of tear-drops as well as the use of interior marked points in achieving 
regularity.

The remaining parts of Definition \ref{def:unobstrcob}
that warrant explanation are assumptions {\em i}.{\em b}  and \emph{ii.}\emph{e}.
These two requirements are a reflection of the following remark. Consider a moduli space $\mathcal{M}$ of polygons in $M$ with boundary on an end $L_{i}$ and assume that it is regular in $M$.
The same moduli space can be viewed as a moduli space $\widehat{\mathcal{M}}$ of polygons in $\C\times M$, with boundary along $V_{h}$, but included in the fibre over a  bottleneck. In general, $\widehat{\mathcal{M}}$ is  regular only if $h$ is positive along all the corners of the polygons. There might exist non-void moduli spaces of polygons in $M$ with corners that are not all positive for $h$. When these moduli spaces are viewed as moduli spaces in $\C\times M$, they are no longer regular and further perturbations to remedy this situation interferes with condition \emph{ii.c}.
We deal with this difficulty by using condition \emph{ii.e} to 
ensure that moduli spaces of $\mathbf{c}$-marked polygons, included in the fiber over a bottleneck, remain regular in $\C\times M$ and that in the compactification of such moduli spaces only appear
moduli spaces with corners where $h$ is positive, and thus their regularity is also achieved.  
This is sufficient for our purposes as all algebraic structures of interest to us
are defined in terms of $\mathbf{c}$-marked polygons.

Obviously, 
given fixed ends for a cobordism together with associated coherent systems of data, showing that a perturbation $h$ as before exists is non-trivial.  However, this turns out to be rather easy to do under 
the exactness assumptions that we will use in our actual constructions.

b. Inserting the data $h$ in the definition of a marked cobordism has the disadvantage that one needs to be 
precise on how gluing of cobordisms behaves with respect to these perturbations. In particular, if $(V,h)$ and 
$(V',h')$ need to be glued along a common end $L$, then the perturbations $h$ and $h'$ have to be positive 
exactly for the same double points of $L$. On the other hand, there is no obvious way to avoid including this $h$
in the definition because associating the marking $c$ directly to connected components of $I_{V}$ is not sufficient to canonically induce a marking to a deformation $V_{h}$ and, at the same time, for Floer homology type considerations the actual object in use is precisely the marked Lagrangian $V_{h}$. 

c. In our arguments, whenever non-marked tear-drops will contribute to the definition of algebraic structures they will carry interior punctures $b_{j}$ and the relevant moduli spaces will be regular (condition i.b). At the same time, if such tear-drops exist there are also non-perturbed tear-drops, with the same degree relative to each pivot  $P_{i}$, for which regularity is not necessarily true (they can be seen to exist by making perturbations $\to 0$). This explains the formulation at i.a: these non-marked $\bar{J}_{0}$ (non-perturbed) tear-drops  are allowed to exist even if in all arguments
only perturbed curves $u$ with $||u||\geq 3$ will be used.

\end{rem}
Unobstructed, marked, Lagrangian cobordisms in the sense above can be assembled
in a class $Lag^{\ast}(\C\times M)$ as in Definition \ref{def:finite-energy}. We will denote 
by $\bar{\mathcal{D}}$ the perturbation data for this family.  More details on how to pick this perturbation
data will be given further below, in \S\ref{subsubsec:morph}.

\subsection{Exact marked Lagrangians and cobordisms.}\label{subsec:exact} In this subsection we
specialize the discussion in \S \ref{subsec:unobstr},\ref{subsubec:marked-cob} to the case of 
exact Lagrangians and cobordisms. Our purpose is to prepare the ground to set up the category $\mathsf{C}ob^{\ast}(M)$
that appears in Theorem \ref{thm:BIG}.

\

We fix the context for the various considerations to follow. We  assume that the manifold $(M,\omega)$ is exact, $\omega=d\lambda$, tame at infinity.  We will consider Lagrangian immersions  $j_{L}:L\to M$ that are exact and {\em endowed with a primitive} $f_{L}$,  $j_{L}^{\ast}\lambda=d f_{L}$. We will also assume two
additional assumptions that are generic: 
\begin{itemize}
\item[i.]$j_{L}$ only has double points that are isolated, self-transverse 
intersections.
\item[ii.] if $(P_{-},P_{+})\in I_{L}$, then $f_{L}(P_{-})\not= f_{L}(P_{+})$.
\end{itemize} 
For such an exact Lagrangian we denote by $I_{L}^{<0}$ the set of all {\em action negative} double points
$(P_{-},P_{+})\in I_{L}$. By definition, these are those double points $(P_{-},P_{+})$  such 
that $f_{L}(P_{+}) <  f_{L}(P_{-})$. Similarly, we denote by $I_{L}^{> 0}$ the complement of $I_{L}^{< 0}$ in $I_{L}$.
We will only consider marked Lagrangians of type $(L,c)$ where $L$ is exact with a fixed primitive as above 
and:
\begin{itemize}
\item[iii.]
 $c \subset I_{L}^{ < 0}$.
\end{itemize}
Notice that, in general, the Lagrangian $L$ is not assumed to be connected. 
Condition {\em iii.} will play an important role in a number of places in our construction. In particular, 
in relation to unobstructedness and when composing cobordisms together with their perturbations. Condition {\em ii.}
is of use, among other places, to properly define the composition of cobordisms. Notice that
if $j_{L}$ is an embedding, then $c=\emptyset$ and $L$ trivially satisfies the unobstructedness condition 
in Definition \ref{def:unobstr}. Moreover, the class of all such embedded Lagrangians  is easily seen to satisfy condition {\em ii} in Definition \ref{def:finite-energy}. 

\

Finally, we consider an unobstructed class $Lag^{\ast}(M)$ of {\em exact}, marked Lagrangians $(L,c,\mathcal{D}_{L})$ (as above) in the sense of Definition \ref{def:finite-energy}, in particular they
satisfy the condition \emph{i.b} in Definition \ref{defi:perturbD} (in other words, all moduli spaces of non-marked tear-drops are void). 

\

We denote by $\mathcal{D}$ the perturbation data for this class and by $\mathcal{D}_{L}$ the relevant perturbation data for each individual Lagrangian $L$, as in Definition \ref{def:finite-energy}. 
 For embedded Lagrangians $L$, the data $\mathcal{D}_{L}$ reduces to fixing one almost complex structure $J$ on $M$. We will assume:
 \begin{itemize}
 \item[iv.] for a certain almost complex structure $J$ - called the {\em ground} a.c.s. - on $M$, 
 all exact embedded Lagrangians $L$ (with all choices of primitives)  and with $\mathcal{D}_{L}=J$
 belong to $Lag^{\ast}(M)$,
 \begin{equation}\label{eq:embed} (L,\emptyset, J)\in Lag^{\ast}(M)~.~ \end{equation} 
 \end{itemize}

\begin{rem} \label{rem:geo-rep}
The same geometric immersion $j_{L}:L\to M$ may appear numerous times
inside the class $Lag^{\ast}(M)$: first, it can appear with more than a single primitive $f_{L}$; with various markings $c$; even if both the primitive and the marking are fixed, there are possibly more  choices of data $\mathcal{D}_{L}$ leading to different triples $(L,c,\mathcal{D}_{L})$. At the same time,
the class $Lag^{\ast}(M)$ does not necessarily contain all the unobstructed triples as before, but it does contain all embedded Lagrangians in the sense described above.
\end{rem}

\subsubsection{Fukaya category of unobstructed marked Lagrangians.}\label{subsubsec:modules}
In view of Definition \ref{def:finite-energy}, we apply Seidel's method to construct an $A_{\infty}$-category associated with objects in the class $Lag^{\ast}(M)$ defined as above with perturbation data $\mathcal{D}$. We use homological notation here but, in matters of substance, there is a single difference  
in our setting compared to the construction in \cite{Se:book-fukaya-categ}. In our case, as we deal with {\em marked} Lagrangians $(L,c)$, the operations $\mu_{k}$, $k\geq 1$,
$$\mu_{k}: CF((L_{1},c_{1}),(L_{2},c_{2}))\otimes \ldots \otimes CF((L_{k},c_{k}),(L_{k+1},c_{k+1}))\longrightarrow CF((L_{1},c_{1}),(L_{k+1},c_{k+1}))$$
are defined by counting elements $u$ in $0$-dimensional moduli spaces of $\mathbf{c}$-{\em marked} polygons
$$u\in\overline{\mathcal{M}}_{\mathcal{D};L_{1},\ldots, L_{k+1}}(h_{1},\ldots , h_{k}; y)$$ as defined
in \S\ref{subsubsec:marked-poly}. The conditions in the unobstructedness Definitions \ref{def:unobstr} and \ref{def:finite-energy} imply that the sums involved in the definitions of the $\mu^{k}$'s are well defined
and finite and that the usual $A_{\infty}$ relations are true in this setting. 

\begin{rem}
A slightly delicate point is to show that this category is (homologically) unital. One way to do this is to
use a Morse-Bott type description of the complex $CF(L,L)$, where $j_{L}:L\to M$ is an unobstructed, marked immersion in our class (with $L$ connected), 
through a model that consists of a complex having as generators the critical points of a Morse function $f$ on $L$ together with two copies of each of the self intersection points of $L$. The differential has a part that is given by the usual Morse trajectories of $f$, another part that
consists of marked $J$-holomorphic strips joining self intersection points of $L$ as well as an additional
 part  that consists of (marked) tear-drops starting (or arriving) at a self-intersection point joined by a Morse flow line to a critical point of $f$ (this model is in
the spirit of \cite{Cor-La:Cluster-2} and appears in the non-marked case in \cite{Alston-Bao:imm2}). 
Of course, one also needs to construct the operations $\mu_{k}$ in this setting. We omit the details here but mention that, with this model, it is easily seen that the unit is represented by the class of the maximum, as in the embedded case. 
\end{rem}

This leads to an $A_{\infty}$- category, $\fuk_{i}^{\ast}(M;\mathcal{D})$, where $i$ indicates that the objects considered here are in general immersed, marked Lagrangians. This category contains the subcategory  $\fuk^{\ast}(M;\mathcal{D})$ whose objects are the embedded Lagrangians 
$\mathcal{L}ag^{\ast}_{e}(M)\subset Lag^{\ast}(M)$.  There are also associated derived
versions, $D\fuk^{\ast}_{i}(M;\mathcal{D})$ and $D\fuk^{\ast}(M;\mathcal{D})$.
The Yoneda functor $$\mathcal{Y}:\fuk^{\ast}_{i}(M;\mathcal{D})\to mod (\fuk^{\ast}_{i}(M;\mathcal{D}))$$ continues to be well defined and the Yoneda lemma from \cite{Se:book-fukaya-categ} still applies. In particular, there is a quasi-isomorphism:
\begin{equation}\label{eq:yoneda-imm}
\mor_{mod} (\mathcal{Y}(L),\mathcal{Y}(L'))\simeq \mor_{\fuk} (L,L')~.~
\end{equation}

\subsubsection{The class $Lag^{\ast}(\C\times M)$} \label{subsec:exact-cob}
In this subsection we make precise the type of cobordisms that will appear
in the class $Lag^{\ast}(\C\times M)$. This class can be viewed as a close approximation to the actual
class $\mathcal{L}ag^{\ast}(\C\times M)$ from Theorem \ref{thm:BIG}. Our conventions are that $(\C, \omega_{0})$ is exact with primitive $\lambda_{0}=xdy$, $(x,y)\in \R\oplus \R\cong \C$ and, similarly, the primitive of $\omega_{0}\oplus \omega$ is $\lambda_{0}\oplus\lambda\in \Omega^{1}(\C\times M)$. 

Consider a cobordism $$V: (L_{i})\cobto (L'_{j})~.~$$ 
We fix some additional notation. Given some data defined
for the Lagrangian $V_{h}$, obtained by perturbing $V$ as described in \S\ref{subsubsec:preturb}, we 
denote by $( - ) |_{L_{i}}$ the restriction of this data to the end $L_{i}$ in the fibre (of $\pi:\C\times M\to \C$)
that lies over the bottleneck corresponding to $L_{i}$, and similarly for the negative ends $L'_{j}$. 

\

We consider an unobstructed family $Lag^{\ast}(\C\times M)$ of cobordisms $(V,c,h, \mathcal{D}_{V_{h}})$,
$V: (L_{i})\cobto (L'_{j})$,  as in \S\ref{subsubsec:marked-cob} 
%(in this case we allow for 
%moduli spaces of tear-drops to be non-void, see Remark \ref{rem:coherent-perturb} b) 
subject to the additional conditions that follow:
\begin{itemize}
\item[i.]  $L_{i}, L'_{j}\in \mathcal{L}ag^{\ast}(M)$, $\forall i,j$.
\item[ii.] $V$ is exact and
the primitive $f_{V}$ on $V_{h}$ has the property that $f_{V}|_{L_{i}}= f_{L_{i}}$,
 $f_{V}|_{L'_{j}}= f_{L'_{j}}$, $\forall i,j$. 
\item[iii.] The double points of $V_{h}$ have the property that $f_{V}(P_{-})\not= f_{V}(P_{+})$
whenever $(P_{-},P_{+})\in I_{V}$ and, moreover, $c\subset I_{V_{h}}^{ < 0}$.
\item[iv.]The perturbation $h$ is positive exactly for those points $(P_{-},P_{+})\in I_{V_{h}}|_{L_{i},L'_{j}}$ such that $(P_{-},P_{+})\in I_{V_{h}}^{ < 0}$.
 \end{itemize}
 \begin{rem}\label{rem:var1}
 a. Notice that if $V$ is exact, then its deformation $V_{h}$ is also exact. In case $V$ is embedded,
 the perturbation $h$ is taken to be $0$ and all restrictions over bottlenecks simply correspond to restrictions to the fiber over the corresponding point (where $V$ is horizontal).
 
 b. Conditions {\em i, ii, iii} are the immediate analogues of the conditions already imposed on $Lag^{\ast}(M)$. Condition {\em iv} ties the perturbation $h$ to the ``action negativity'' of the self intersection points of the ends $L_{i}, L_{j}$ of $V$ in a natural way. The impact of this condition is further discussed 
below. 
 
 c. In the setting above, we can revisit the regularity constraint {\em i} in Definition \ref{def:unobstrcob}  and condition {\em ii.e} in the same definition. The key point is that,
 at least when the Hamiltonian terms in the perturbation data $\bar{\mathcal{D}}$ are sufficiently small (which we will implicitely assume from now on),
 a non-constant polygon $u\in \mathcal{M}_{\mathcal{D}_{V_{h}};V_{h}}(x_{1},\ldots, x_{m};y)$ has 
 energy $$0 < E(u) \approx \sum_{i=1}^{m} (f_{V}((x_{i})_{+})-f_{V}((x_{i})_{-}) - (f_{V}(y_{+})-f_{V}(y_{-}))$$ where $x_{i}\in I_{V_{h}}$ is written as $x_{i}=((x_{i})_{-}, (x_{i})_{+})\in V_{h}\times V_{h}$ (the sign $\approx$ is due to the curvature terms that we do not make explicit here, but assume to be very small). Thus, if $x_{i}\in I_{V_{h}}^{<0}$  (which, we recall means, 
 $f_{V}((x_{i})_{-}) > f_{V}((x_{i})_{+})$, then we also have $y\in I_{V_{h}}^{ < 0}$. In particular, this means that condition {\em ii.e}  in Definition  \ref{def:unobstrcob} is automatically satisfied in our context because condition \emph{iv} above requires that the perturbation $h$ is positive precisely at those self-interesection points (along the ends) that are action negative. 
Moreover, recall that we assume (at point {\em iii} above) 
that the marking $c\subset I_{V_{h}}$ is also action negative
 in  the sense  that $c\subset I_{V_{h}}^{< 0}$. This implies that a $c$-marked polygon with boundary on $V_{h}$ has an exit $y\in I_{V_{h}}^{ < 0}$ and thus there are no $\mathbf{c}$-marked $J$-disks in this setting, thus the disk part  in condition \emph{i.a} in  Definition \ref{def:unobstrcob}  is automatically satisfied.  Further, it is easy to see that all moduli spaces appearing in the compactification of moduli spaces of $\mathbf{c}$-marked polygons only have action negative inputs and an action negative exit.
In short, the regularity constraint  in Definition \ref{def:unobstrcob} can be ensured by requiring that  all moduli spaces $\mathcal{M}_{\mathcal{D}_{V_{h}};V_{h}}(x_{1},\ldots, x_{m};y)$ with 
$x_{i}, y\in I_{V_{h}}^{<0},  \forall i $ are regular - see also Remark \ref{rem:gluing-pert}.
 \end{rem}
 
In view of the remark above, the regularity requirements become simpler in our exact setting and we adjust them slightly 
here to  a form that is easier to use in practice.  We first notice that the discussion in Remark \ref{rem:var1} \emph{c}
above also applies in an obvious way to the Lagrangians $L\in Lag^{\ast}(M)$
and to their systems of coherent perturbations $\mathcal{D}_{L}$. Thus, in our context, 
regularity of moduli spaces of polygons with boundary on $L\subset M$, respectively, $V\subset \C\times M$  is only necessary for those moduli spaces  such that all inputs are action negative
(see also Remark \ref{rem:coherent-perturb} \emph{c}), as long as the 
Hamiltonian perturbation terms in the respective systems of perturbations are sufficiently small.
We will work from now with this understanding of regularity. We formulate it more precisely below. 

\

Given all the various definitions and conditions contained in the previous subsections, it is useful to
list here all the types of assumptions that operate in our setting. 
First, we recall that we work
with Lagrangians $(L,c, \mathcal{D}_{L})$ and cobordisms $(V, c, h,\mathcal{D}_{V_{h}})$ that are exact,
marked, endowed with fixed primitives such that each self-intersection point of $L$ and $V_{h}$ are
action positive or negative (but non-zero) and also with fixed coherent perturbation data $\mathcal{D}_{L}$
and, respectively, $\mathcal{D}_{V_{h}}$. The perturbations $\mathcal{D}_{V_{h}}$ associated to cobordisms also include choices of points in the plane and associated domain dependent perturbations that are used to deal with tear-drops of positive degree (see \S\ref{subsubsec:marked-cob}).

\

In view of the discussion above, the classes $Lag^{\ast}(M)$, $Lag^{\ast}(\C\times M)$ satisfy:
\begin{assumption}I. (regularity)
\begin{itemize}
\item[i.]The Hamiltonian perturbation terms for $\mathcal{D}_{L}$, $\mathcal{D}_{V_{h}}$ are 
sufficiently small so that moduli spaces with action negative inputs and an action positive output 
(or no output) are void.
\item[ii.] condition {\em ii} in Definition \ref{defi:perturbD}  applies to only those moduli spaces with
action negative inputs.
\item[iii.] condition {\em i.b} in Definition \ref{def:unobstrcob} applies to all moduli spaces with action negative inputs.
\end{itemize}
\end{assumption}
There are six other assumptions (each detailed in the various points in Definitions \ref{defi:perturbD}, \ref{def:unobstr}, \ref{def:finite-energy}, \ref{def:marked-cob}, \ref{def:unobstrcob}):
\begin{itemize}
\item[II.] (negativity of markings) All markings for Lagrangians and cobordisms are action negative.
\item[III.] (restriction to ends) Cobordism data restricts to end data over the bottlenecks. In particular,
the ends of $V\in Lag^{\ast}(\C\times M)$ are in $Lag^{\ast}(M)$.
\item[IV.] (positivity of perturbations) For cobordisms $(V,h,c)$ the perturbation $h$ is positive
precisely at the action negative self-intersection points of the ends.
\item[V.] (unobstructedness) The number of $\mathbf{c}$-marked tear-drops through any self-intersection point for Lagrangians as well as cobordisms is $0$ \emph{mod $2$}. The moduli spaces of tear-drops (non-marked) are void for the elements of $Lag^{\ast}(M)$.  The moduli spaces of (non-marked) tear-drops for the elements of $Lag^{\ast}(\C\times M)$ are void except, possibly, for those of positive degree and plane-simple.
\item[VI.] (regularity of data for families) The data $\mathcal{D}$ 
for $Lag^{\ast}(M)$, respectively, $\bar{\mathcal{D}}$ for $Lag^{\ast}(\C\times M)$ extends the data $\mathcal{D}_{L}$, respectively $\mathcal{D}_{V_{h}}$, and is regular.
\item[VII.] (nontriviality) all embedded exact Lagrangians are contained in $Lag^{\ast}(M)$ and
all embedded exact cobordisms are contained in $Lag^{\ast}(\C\times M)$.
\end{itemize}
In our context, as mentioned above, condition {\em ii.e} in Definition \ref{def:unobstrcob} is redundant
and the energy bounds condition in Definition \ref{def:finite-energy} {\em ii} is automatically
satisfied. Moreover, because all the markings are action negative, there are no $\mathbf{c}$-marked $J$-holomorphic disks. The choice of perturbation data $\bar{\mathcal{D}}$ will be made explicit below.
The part of condition VII that applies to cobordisms means that we assume that the class $Lag^{\ast}(\C\times M)$ contains all exact {\em embedded} Lagrangian cobordisms  $V$ as quadruples $(V,\emptyset, 0, i\times J)$ where $J$ is the ground  almost complex structure from (\ref{eq:embed}). An analogue of Remark \ref{rem:geo-rep} remains valid in the case of immersed cobordisms $j_{V}:V\to \C\times M$: such a cobordism can appear multiple times in the class
$Lag^{\ast}(\C\times M)$ due to different primitives $f_{V}$, different perturbations $h$, different 
markings $c$ and different perturbation data $\mathcal{D}_{V_{h}}$.

In reference to point V. above  if  some moduli space of non-marked tear-drops associated to an element $V\in Lag^{\ast}(\C\times M)$ is non-void, we say  that  $V$ {\em carries non-marked tear-drops}. 
If all these moduli spaces are void we say that $V$ is {\em without non-marked tear-drops}.
The purpose of the distinction is that the elements of $Lag^{\ast}(\C\times M)$ without non-marked
tear-drops form a Fukaya category while the other unobstructed cobordisms can be used in Floer type
arguments but with more care due to the presence of the pivots in the plane (see also \S\ref{subsubsec:marked-cob}) . We denote the 
subset of  $Lag^{\ast}(\C\times M)$ consisting of cobordisms without non-marked tear-drops by $Lag^{\ast}_{0}(\C \times M)$ (of course, this class depends on the choice of almost complex structure $J$).

\

 \subsubsection{The Fukaya category of cobordisms.}\label{subsubsec:morph}

 We  denote the class of embedded cobordisms by $Lag^{\ast}_{e}(\C\times M)$.
Embedded  cobordisms are the objects of a Fukaya category, $\fuk^{\ast}(\C\times M)$,
 as constructed in \cite{Bi-Co:lcob-fuk}.  We extend this construction here, namely, we define a Fukaya category of immersed, marked, unobstructed cobordisms $\fuk^{\ast}_{i}(\C\times M; \bar{\mathcal{D}})$ that contains $\fuk^{\ast}(\C\times M)$ as a full subcategory. Along the way, we make more explicit the choice of the perturbation data $\bar{\mathcal{D}}$. The objects of  $\fuk^{\ast}_{i}(\C\times M; \bar{\mathcal{D}})$ are the elements of $Lag^{\ast}_{0}(\C\times M)$, in other words those unobstructed cobordisms
 without non-marked tear-drops.
 
 \
 
The construction is  an almost identical replica of the construction of the embedded cobordism category 
$\fuk^{\ast}(\C\times M)$ in \cite{Bi-Co:lcob-fuk} with a few modifications that we indicate now. 

\

\noindent \emph{a. Profile functions and cobordism perturbations.}\\
Recall from \S\ref{subsubsec:preturb} that the immersed marked cobordisms $(V,c,h)$ have bottlenecks at points of real coordinates $a+2$ for
the positive ends and $-a-2$ for the negative ends (here $a\in (0,\infty)$ is some large constant).
We will make the additional assumption that all the perturbation $h$ associated to marked immersed cobordisms
 - see Figure \ref{fig:bottle} - have the property that the projection of the perturbed Lagrangians 
 along one end is included inside the strip of width $2\epsilon$ (for a fixed small $\epsilon$) 
 along the respective horizontal axis and, more precisely, inside the region drawn in Figure \ref{fig:strict-bottle} below.
\begin{figure}[htbp]
   \begin{center}
      \includegraphics[width=0.68\linewidth]{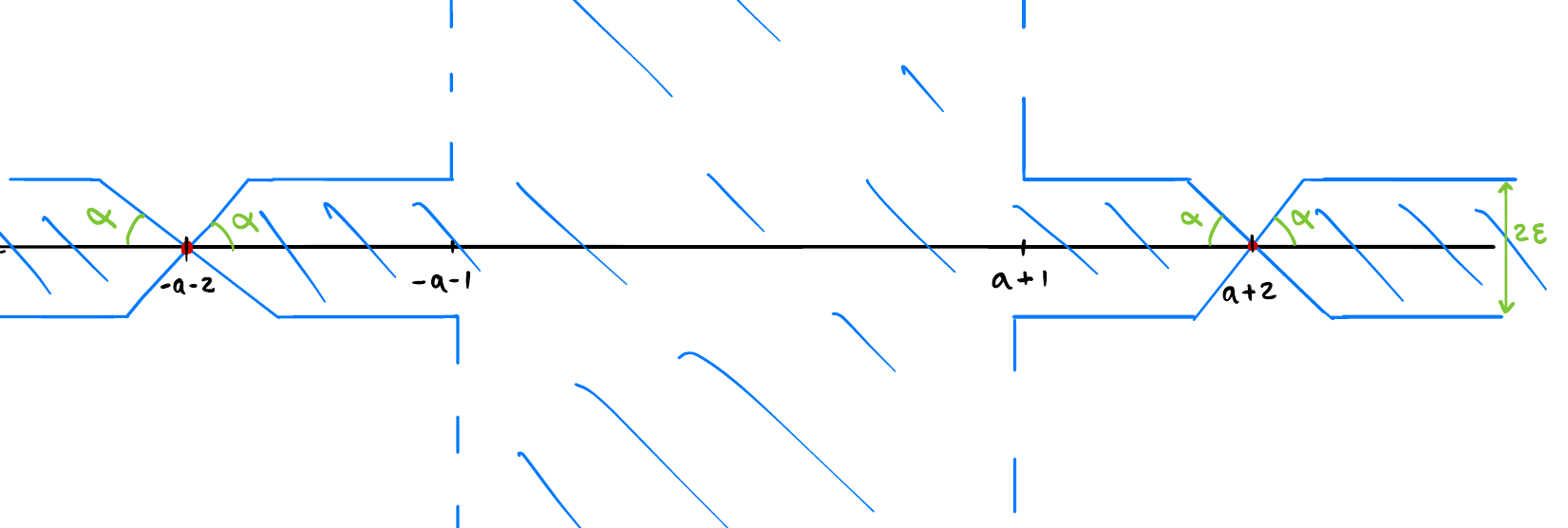}
   \end{center}
   \caption{\label{fig:strict-bottle} The region containing the projection of all perturbations $h$ associated
   to marked, immersed cabordisms $(V,c,h)$.}
\end{figure}

We assume that the angle $\alpha$ in Figure \ref{fig:strict-bottle} is equal to $\frac{\pi}{8}$.
We will assume that all embedded cobordisms in $\mathcal{L}ag^{\ast}_{e}(\C\times M)$ are cylindrical outside the region $[-a-1,a+1]\times \R$. The construction of $\fuk^{\ast}(\C\times M)$ makes use
of a profile function that we will denote here by $\hbar$ (this appears on page 1760-1761 of \cite{Bi-Co:lcob-fuk} with the notation $h$, we use $\hbar$ here to distinguish this profile function from the perturbations associated to immersed marked cobordisms). We recall that, in essence,
 the role of the profile function $\hbar:\R^{2}\to \R$ is to disjoin the ends of cobordisms at $\infty$ by using the associated Hamiltonian flow $\phi_{t}^{\hbar}$.  This profile function also has some bottlenecks, as they appear in  Figures 8 and 9 in  \cite{Bi-Co:lcob-fuk}. The real coordinate of the bottlenecks is fixed in \cite{Bi-Co:lcob-fuk} at $\frac{5}{2}$ for the positive ones and $-\frac{3}{2}$ for the negative (of course, it is later shown that this is just a matter of convention). Our first assumption here 
 is that the positive bottlenecks of $\hbar$ have real coordinate $a+2$ and that the negative one has real coordinate $-a-2$. Moreover, we will assume that the (inverse) hamiltonian isotopy $(\phi_{1}^{\hbar})^{-1}$ associated to $\hbar$ transforms
 the region in Figure \ref{fig:strict-bottle} as in Figure \ref{fig:moved-str-bottle} below.

\begin{figure}[htbp]
   \begin{center}
      \includegraphics[width=0.8\linewidth]{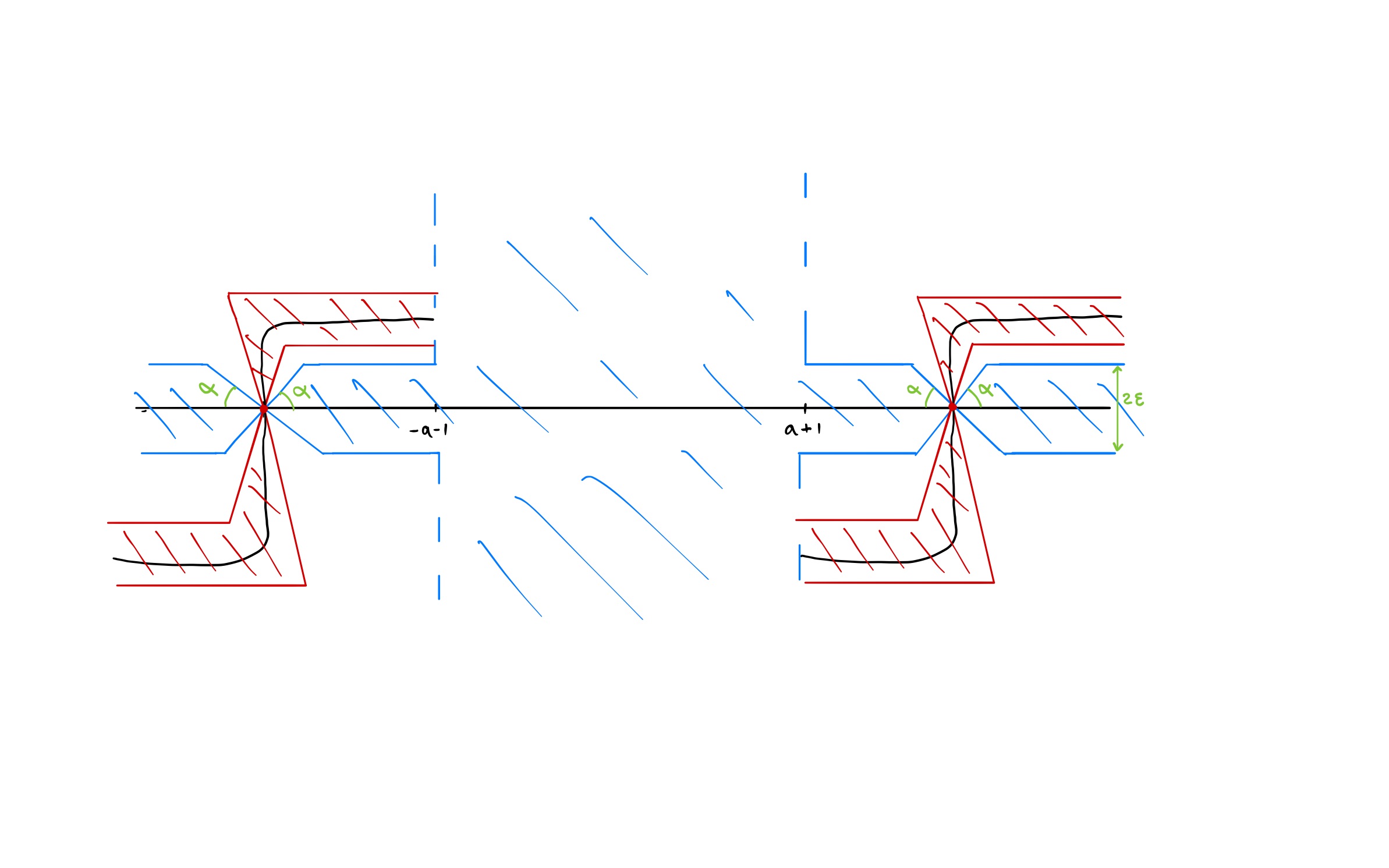}
   \end{center}
   \caption{\label{fig:moved-str-bottle} The red region containing the transformation through 
   $(\phi^{\hbar}_{1})^{-1}$ of the blue region.}
\end{figure}

\noindent \emph{b. Choices of almost complex structures.}\\
Once the profile functions chosen as above, 
the construction proceeds as in \S 3 in \cite{Bi-Co:lcob-fuk} with an additional requirement having to
do with the choice of almost complex structures (see the points ii.p.1763 and iii*. on p.1764 of
\cite{Bi-Co:lcob-fuk}) associated to a family of marked, unobstructed cobordisms $\{(V_{i},c_{i},h_{i},\mathcal{D}_{V_{h_{i}}})\}_{1\leq i\leq k+1}$ (we denote $V_{h_{i}}=(V_{i})_{h_{i}}$).  
The requirement is that the choice of (domain dependent) almost complex structure $\mathbf{J}=\{J_{z}\}_{z}$ on $\C\times M$, where $z\in  S_{r}$ with $r$ varying in the relevant family of pointed disks, has the property that it agrees with the $\bar{J}_{V_{i}}$ - this is the base almost structure of $\mathcal{D}_{V_{h_{i}}}$, see {\em ii.c} in Definition \ref{def:unobstrcob} - along each boundary of $S_{r}$ that lies on the cobordism $V_{h_{i}}$. This condition is added to the point iii*. on p.1764 of \cite{Bi-Co:lcob-fuk}.   All the relevant perturbation data is collected in the set $\bar{\mathcal{D}}$.

\noindent \emph{c. Definition of the $\mu_{k}$ operations.}\\
With the data fixed as above, the operations $\mu_{k}$ are defined as in \S\ref{subsubsec:modules} (see also 
\S\ref{subsubsec:marked-poly} for the definition of the relevant moduli spaces).
Namely, they are given through counts of $\mathbf{c}$-marked polygons:
$$u\in\overline{\mathcal{M}}_{\bar{\mathcal{D}};V_{h_{1}},\ldots, V_{h_{k+1}}}(h_{1},\ldots , h_{k}; y)$$

The moduli space $\overline{\mathcal{M}}_{\bar{\mathcal{D}};V_{h_{1}},\ldots, V_{h_{k+1}}}(h_{1},\ldots , h_{k}; y)$ contains in this case all the $\mathbf{c}$-marked polygons $u$ with $||u||\geq 2$, without distinguishing the interior marked points (see also \S\ref{subsubsec:marked-cob})
Recall form Definition \ref{def:unobstrcob} that the almost complex structures $\bar{J}_{V_{h}}$ are split outside
a finite width vertical band in $\C$ that does not contain the bottlenecks. As a consequence, all the arguments 
related to naturality, and compactness of moduli spaces in \S 3.3 \cite{Bi-Co:lcob-fuk} remain valid. It follows that the operations $\mu_{k}$ are well defined and that they satisfy the $A_{\infty}$ relations. The key point
here is that bubbling off of tear-drops can only produce tear-drops of positive degree (because, by Assumption V,
there are no tear-drops of vanishing degree) and, due to the choice of perturbation data, these are regular and 
in even numbers at each self intersection point of any (perturbed) cobordism $V_{h}$.

\

From the discussion above we conclude that 
 the $A_{\infty}$ category $\fuk^{\ast}_{i}(\C\times M;\bar{\mathcal{D}})$ is now well-defined. Moreover,
 this category contains as a full subcategory $\fuk^{\ast}(\C\times M)$. In particular, to each
 unobstructed, marked cobordism without non-marked tear-drops $(V,c,h,\mathcal{D}_{V_{h}})$ we may associate an $A_{\infty}$
 Yoneda module $\mathcal{Y}(V)$ over $\fuk^{\ast}_{i}(\C\times M;\bar{\mathcal{D}})$ and, by restriction, also an $A_{\infty}$ module over $\fuk^{\ast}(\C\times M)$. 
 
 \subsubsection{Algebraic structures associated to cobordisms that do carry non-marked tear-drops.}\label{subsubec:tear}
Given a cobordism $V\in Lag^{\ast}(\C\times M)$ that carries non-marked teardrops, it is still possible
to define certain Floer type invariants associated to it.

Consider for instance a
full subcategory $\fuk^{\ast}_{i,V}(\C\times M)$ of $\fuk^{\ast}_{i}(\C\times M)$ that 
has as objects those elements of $\fuk^{\ast}_{i}(\C\times M)$ that do not intersect the complex hypersurface
$H_{V}=\cup_{i}P_{i}\times M$ where $P_{i}$ are pivots associated to $V$
and such that any object $W\in  \fuk^{\ast}_{i,V}(\C\times M)$ only intersects $V$ transversely, 
only at {\em two} consecutive  ends of $V$.

With appropriate choices of perturbations (which we fix from now on), the 
cobordism $V=(V,h,c, \mathcal{D}_{V_{h}})$ induces a module, denoted $\mathcal{Y}(V)$
over  $\fuk^{\ast}_{i,V}(\C\times M)$. In the definition of the different operations we only 
count curves that are plane-simple (in other words such that the degree of these curves at each pivot $P_{j}\in \C$  is at most one) and, moreover, such that for each point $P_{j}$ where the degree is one,
there is a marked point $b_{j}$ in the domain of the curve.  The choice of relevant regular perturbations
is possible by Seidel's scheme, by induction on $||u||$.

Compared to the construction of $\fuk^{\ast}_{i}(\C\times M)$ there is one additional delicate point
to note: whenever two plane-simple curves can be glued the outcome needs again to be plane-simple. 
Indeed, in case this condition is not satisfied, we would need to allow counts of curves where more than a single interior marked point is sent to the same $P_{j}\times M$ and this creates problems with our counting over $\Z/2$. We want to avoid this additional complication here and this is the purpose of our assumption on the 
intersection between the objects of  $\fuk^{\ast}_{i,V}(\C\times M)$ and $V$ as it ensures that this 
type of plane-simple coherence is satisfied. 

More generally, the same construction works for any other full subcategory of $\fuk^{\ast}_{i}(\C\times M)$, as long as each object continues to be disjoint from $H_{V}$ and if this plane coherence condition continues to be satisfied. The resulting module will again be denoted by $\mathcal{Y}(V)$.

%Some care is needed in choosing perturbations due to 
%possible collapsing of the interior marked points, in other words tangencies to the hypersurface
%$H_{V}$) but this technique has been set-up by Cieliebak-Mohnke \cite{Ci-Mo}.
%We will still refer to this module as the Yoneda module of $V$ but some caution is needed when such $V$'s
%appear to avoid evaluating this module on objects that intersect $H_{V}$.

\subsection{The category $\mathsf{C}ob^{\ast}(M)$ revisited.}\label{subsubsec:cob-cat-def}
We now consider unobstructed classes of Lagrangians $Lag^{\ast}(M)$ and cobordisms $Lag^{\ast}(\C\times M)$ subject to Assumptions I- VII
in \S\ref{subsec:exact-cob}. To define $\mathsf{C}ob^{\ast}(M)$ additional modifications of these classes
are needed. These modifications have to do with the composition of morphisms in  
$\mathsf{C}ob^{\ast}(M)$ as well as with the identity morphisms and, as will be seen later on, 
with the definition of cabling.

\

Concerning composition, recall that for two immersed cobordisms $V: L\cobto (L_{1},\ldots, L_{k},L')$ and $V':L'\cobto (L'_{1},\ldots, L'_{s}, L'')$, the composition $V'\circ V$ is defined as 
in Figure \ref{Fig:Comp}. However, the actual objects in $Lag^{\ast}(\C\times M)$
are immersed cobordisms together with perturbations (as well as additional decorations). This means that,
to properly define the composition in this setting, we need to actually describe a specific perturbation of $V'\circ V$ (together with its other decorations) and we also need to show that the resulting object is
unobstructed. 

\

A similar issue has to do with the identity morphism: assuming that $L$ is immersed we can consider
the cobordism $V=\R\times L$ that should play the role of the identity. However, in our formalism
we need to indicate a specific perturbation of this $V$.

\subsubsection{An additional decoration for objects.}\label{subsubsec:dec-bis}
Given an immmersed marked Lagrangian $$(L,c,\mathcal{D}_{L})\in Lag^{\ast}(M)$$ the additional decoration we require is an {\em $\epsilon$-deformation germ} for $L$. To make this notion precise, assume that $L$ is the $k$-th positive end  of a marked, immersed cobordism $(V, c, h,\mathcal{D}_{V_{h}})$. The restriction to the band $[a +2-\epsilon, a +2+\epsilon]\times \R$
of the part of $V_{h}$ corresponding to the end $L$ can be viewed as a specific perturbation $h_{L}$
of the (trivial) cobordism $[-\epsilon,\epsilon]\times \{1\}\times L$ 
with a bottleneck at $0$, as defined in \S\ref{subsubsec:preturb},  with the profile of the picture in  Figure \ref{fig:bottle} (around the point $a+2$), but translated
so that the bottleneck is moved at the point $\{a+2\}\times \{k\}$.
The same notion also applies for negative ends and, because of our constraints involving the positivity
of perturbations (Assumption IV), the same perturbation $h_{L}$ can appear as a restriction to a negative end
(in the band $[-a -2-\epsilon, -a -2+\epsilon]\times \R$).

\

From now on, we will view the objects in $Lag^{\ast}(M)$ as quadruples
$(L,c,h_{L},\mathcal{D}_{L})$ where $h_{L}$ is a perturbation of $[-\epsilon,\epsilon]\times \{1\}\times L$ with a bottleneck at $0$, in the sense above.  Two such 
objects $(L,c, \mathcal{D}_{L}, h_{L}), (L,c, \mathcal{D}_{L}, h'_{L})$ will be identified
if for a sufficiently small $\epsilon$ the perturbations $h_{L}$ and $h'_{L}$ agree on 
$[-\epsilon,\epsilon]$. We call $h_{L}$ {\em a perturbation germ} associated 
to $L$.  If $L$ is embedded, the perturbation germ $h_{L}$ is null.

\

Given this notion, Assumption III in \S\ref{subsec:exact-cob} is strenghtened to require that if an immersed marked cobordism $(V,c,h, \mathcal{D}_{V_{h}})\in Lag^{\ast}(\C\times M)$ has an
end $(L,c,h_{L},\mathcal{D}_{L})$, then $h_{L}$ is the restriction of $h$.

\begin{rem}
The same triple $(L,c,\mathcal{D}_{L})$ can appear multiple times in the class $Lag^{\ast}(M)$
with different perturbation  germs $h_{L}$. However, we will see below that we will restrict the 
class of allowable perturbation germs to only sufficiently small perturbations  (in a sense to be 
explained) and with this restriction any two perturbation germs will become equivalent from our 
perspective. 
\end{rem}

\subsubsection{Definition of composition of immersed marked cobordisms.}\label{subsubsec:comp}
The composition of two cobordisms $V,V'\in Lag^{\ast}(\C\times M)$ along an 
end $L'\in Lag^{\ast}(M)$ (with the decorations listed in \S\ref{subsubsec:dec-bis}) follows
the scheme in Figure \ref{Fig:Comp} but by taking into account the perturbations $h$ and respectively
$h'$ associated to $V$, respectively, $V'$ as in Figure \ref{fig:comp2}.

\begin{figure}[htbp]
   \begin{center}
      \includegraphics[width=0.71\linewidth]{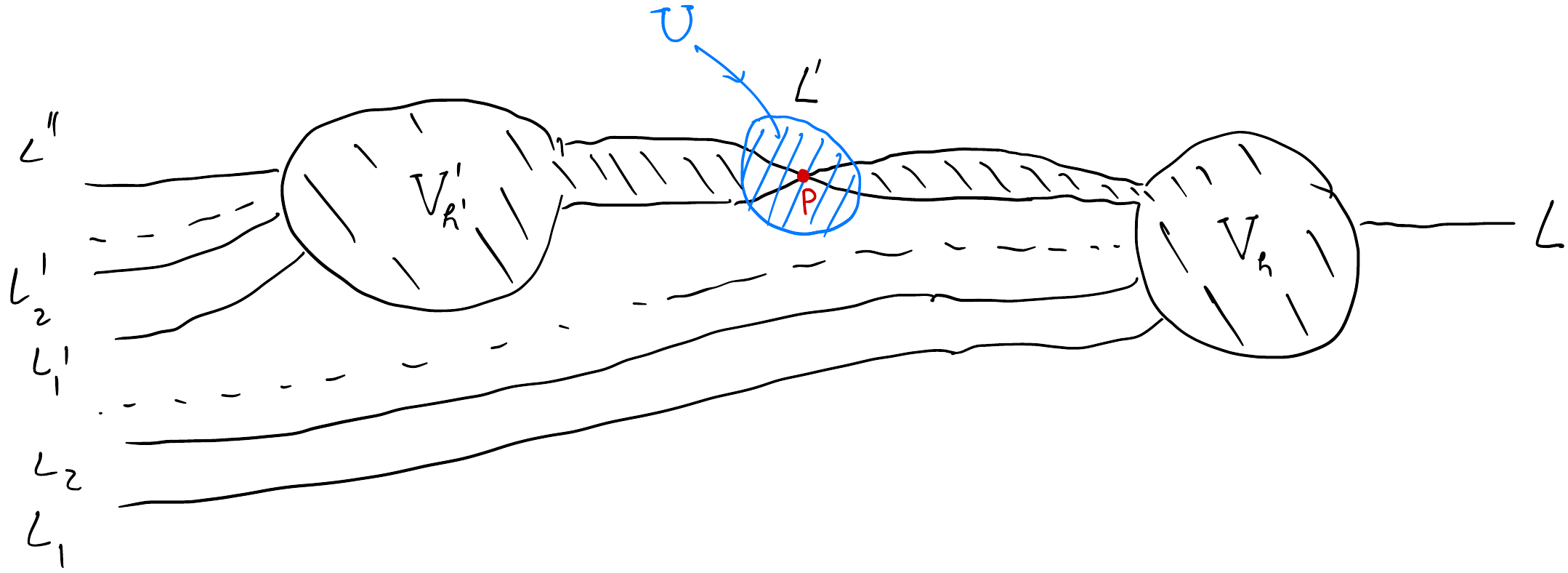}
   \end{center}
   \caption{\label{fig:comp2} The composition of $V,V'\in Lag^{\ast}(\C\times M)$.
 This composition restricts to the perturbation $h_{L'}$ over the set $U$. The bottleneck is the point $P$. Are not drawn the perturbations of $V$ and $V'$ along the ends different from $L'$.}
\end{figure}
In other words, if the two objects to be composed are $(V,c,h,\mathcal{D}_{V_{h}})$, $(V',c',h',\mathcal{D}_{V'_{h'}})$ we define the resulting object $(V'\circ V, c'',h'',\mathcal{D}'')$ by letting $h''$ be the perturbation
of the cobordism $V'\circ V$ (as defined in Figure \ref{Fig:Comp}) obtained by splicing together $h'$ and $h$
in a neighbourhood of type $[b-\epsilon,b+\epsilon]\times [k-\delta, k+\delta]\times L'$ (with convenient $b\in \R$, $k\in \N$, $\delta>0$). We can also splice together the perturbation data $\mathcal{D}_{V_{h}}$ and $\mathcal{D}_{V'_{h'}}$ and extend the result over $\C$ to get the perturbation data $\mathcal{D}''$.
In particular, we assume that the almost complex structure in a band  $[b-\epsilon,b+\epsilon]\times \R$ is a product and we also assume that the same is true in a region under the ``bulk'' of $V'_{h'}$, that contains
the extended ends of $V_{h}$ ($L_{1}, L_{2},\ldots$ in the picture). More details on the definition of $\mathcal{D}''$ are given below, in \S\ref{subsubsec:unobstr-comp}.
Finally, the marking $c''$ is simply the union of $c$ and $c'$ where the marking of $L'$ - in the fiber over the bottleneck $P$ - is common to both $c$ and $c'$. 

\subsubsection{Unobstructedness of the composition.}\label{subsubsec:unobstr-comp}
To discuss this point we start by addressing in more detail the following aspect related to the composition of two cobordisms. This concerns the fact that when operating the splicing described above
it is convenient to extend the end $L'$ as well as the  secondary ends $L_{1}, L_{2},\ldots$ in Figure \ref{fig:comp2}. This type of operation needs to be implemented with some care as the specific aspect of 
the bottleneck needs to be preserved, even if the bottleneck itself is translated to the left or right. 
The argument is simple and makes use of a horizontal Hamiltonian isotopy as in Figure \ref{fig:transl}.

\begin{figure}[htbp]
   \begin{center}
      \includegraphics[scale=0.45]{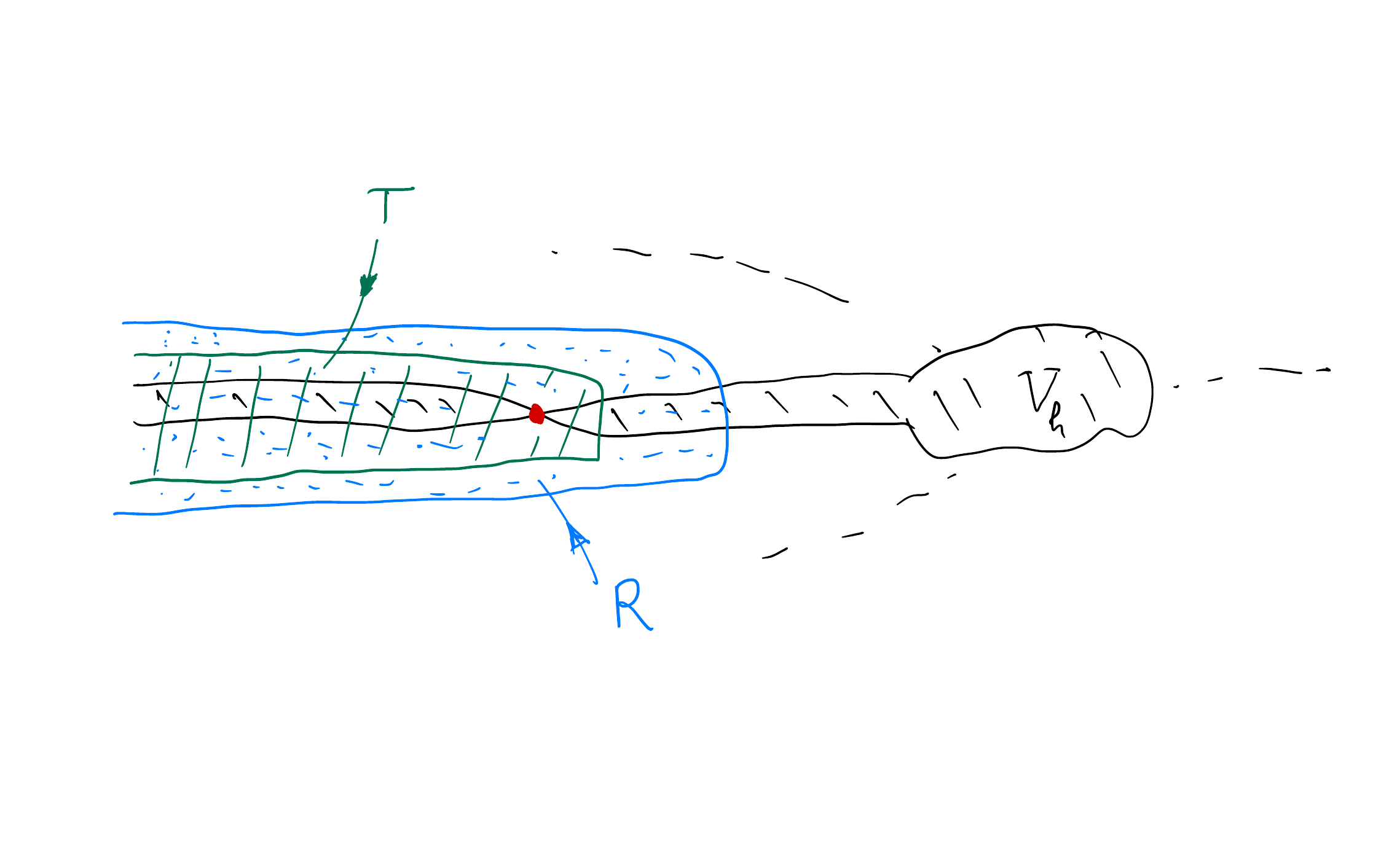}
   \end{center}
   \caption{\label{fig:transl} The horizontal Hamiltonian isotopy that acts as a translation (to the left) inside the region $T$ and is supported inside $R$.}
\end{figure}
Such a hamiltonian isotopy also acts on the relevant additional data, $c,h,\mathcal{D}_{V_{h}}$ and it preserves all the properties of the initial cobordism $V_{h}$. 

With this understanding of the extension to the left of the ends $L_{1}, L_{2},\ldots$ in Figure \ref{fig:comp2} the argument to show that the composition $V'\circ V$ has the required properties reduces to
a couple of observations. First, due to a simple open mapping argument, no $J$-holomorphic polygons can pass from one side of $P$ to the other (the two sides are to the right and left of the vertical axis through $P$): 
this is clear  (as in \cite{Bi-Co:lcob-fuk}) if the polygon does not jump branches at $P$
but, even if it does,  such a curve would still be forced to pass through a non-bounded quadrant
at the point $P$ which is not possible.  As a result of this remark, the relevant moduli spaces are regular.

It remains to check that the number of $\mathbf{c}$-marked tear-drops through any self-intersection point is $0$ \emph{mod} $2$. The key remark here is that due to our assumption on the action negativity of the marking
$\mathbf{c}$, a $\mathbf{c}$-marked tear-drop can only have an exit puncture that is action negative. At the same time, in view of the positivity of the perturbation $h_{L'}$ at  the action negative points (Assumption IV) it follows that all the tear-drops with an exit point belonging to the fiber over the the bottleneck $P$
are completely included in this fiber. Thus they are tear-drops with boundary along $L'$ and as $L'$ is unobstructed their count \emph{mod} $2$ vanishes. All the other tear-drops are only associated either
with $V_{h}$ or with $V'_{h'}$ and their counts also vanish because $V$, $V'$ are unobstructed.

\subsubsection{Small perturbation germs and identity morphisms.}\label{subsubsec:smallpert}
To define the category $\mathsf{C}ob^{\ast}(M)$ some further modifications of the classes $Lag^{\ast}(M)$
and $Lag^{\ast}(\C\times M)$ are needed. In essence, they 
come down to a notion of ``sufficiently small'' perturbation germs $h_{L}$.  Such a notion will be used
to define identity morphisms in our category. Additionally,  it will serve to view any two objects 
$(L, c, h_{L},\mathcal{D}_{L})$ and $(L,c, h'_{L},\mathcal{D}_{L})$ as equivalent as soon as $h_{L}$ 
and $h'_{L}$ are sufficiently small.

\

First notice that the perturbation germs - as defined through the perturbations in \S\ref{subsubsec:preturb} - 
can be compared to the $0$ perturbation in $C^{k}$-norm ($k=2$ is sufficient for us). 
Given two perturbation germs $h_{L}$ and 
$h'_{L}$ for the same marked, exact, Lagrangian $(L,c,\mathcal{D}_{L})$ we may consider the 
immersed Lagrangian $V_{h}$ drawn in Figure \ref{fig:smallperturb} below which is a perturbation of $\gamma\times L$ with $\gamma=\R$
such that $h$ restricts to $h_{L}$ close to the positive bottleneck and to $h'_{L}$ close to
the negative bottleneck. The definition of the actual perturbations is as in \S\ref{subsubsec:preturb}, Figure \ref{fig:bottle} but one uses instead of $\gamma_{+}$ there the upper blue curve 
exiting $P$ in the positive direction, on top in Figure  \ref{fig:smallperturb} and instead of 
$\gamma_{-}$ the blue curve exiting $P$ in the positive direction, at the bottom. 
In particular, there is no perturbation
of $\gamma \times L$ away from some disks around the self-intersection points of $L$ and, along
$\R\times (D_{-}\cup D_{+})$ where $D_{-}$, $D_{+}$ are respectively small disks around 
the components of a self-intersection point $(P_{-},P_{+})\in I^{<0}_{L}\subset L\times L$, $V_{h}$ equals $\gamma_{-}\times D_{-}\bigcup\gamma_{+}\times D_{+}$.

There is an {\em inverted} bottleneck at $0$ (this is inverted in the sense
that this time the points in $\{0\}\times I^{< 0}_{L}$ are points where the perturbation $h$ is negative; we will refer to the usual type of bottleneck as being {\em regular}).
The marked points of $V_{h}$ are over the  positive and negative bottlenecks and they
coincide there with the markings on $L$; the self intersection points over $\{0\}\in \C$ are 
not marked. We consider an additional curve $\gamma'$ as in the picture and a similarly defined
perturbation $V'_{h}$ of $\gamma'\times L$ so that again the germs over the positive/negative bottleneck are $h_{L}$ and $h'_{L}$, respectively. The cobordism $V'_{h}$ is obtained through 
a horizontal Hamiltonian  isotopy (that is independent of $L$, $h_{L}$, $h'_{L}$) from the cobordism $V_{h}-i$ (this is simply $V_{h}$ translated by $-i$). The bottlenecks of $V'_{h}$ and $V_{h}$ are positioned as in  the figure. 

\begin{figure}[htbp]
   \begin{center}
      \includegraphics[scale=0.18]{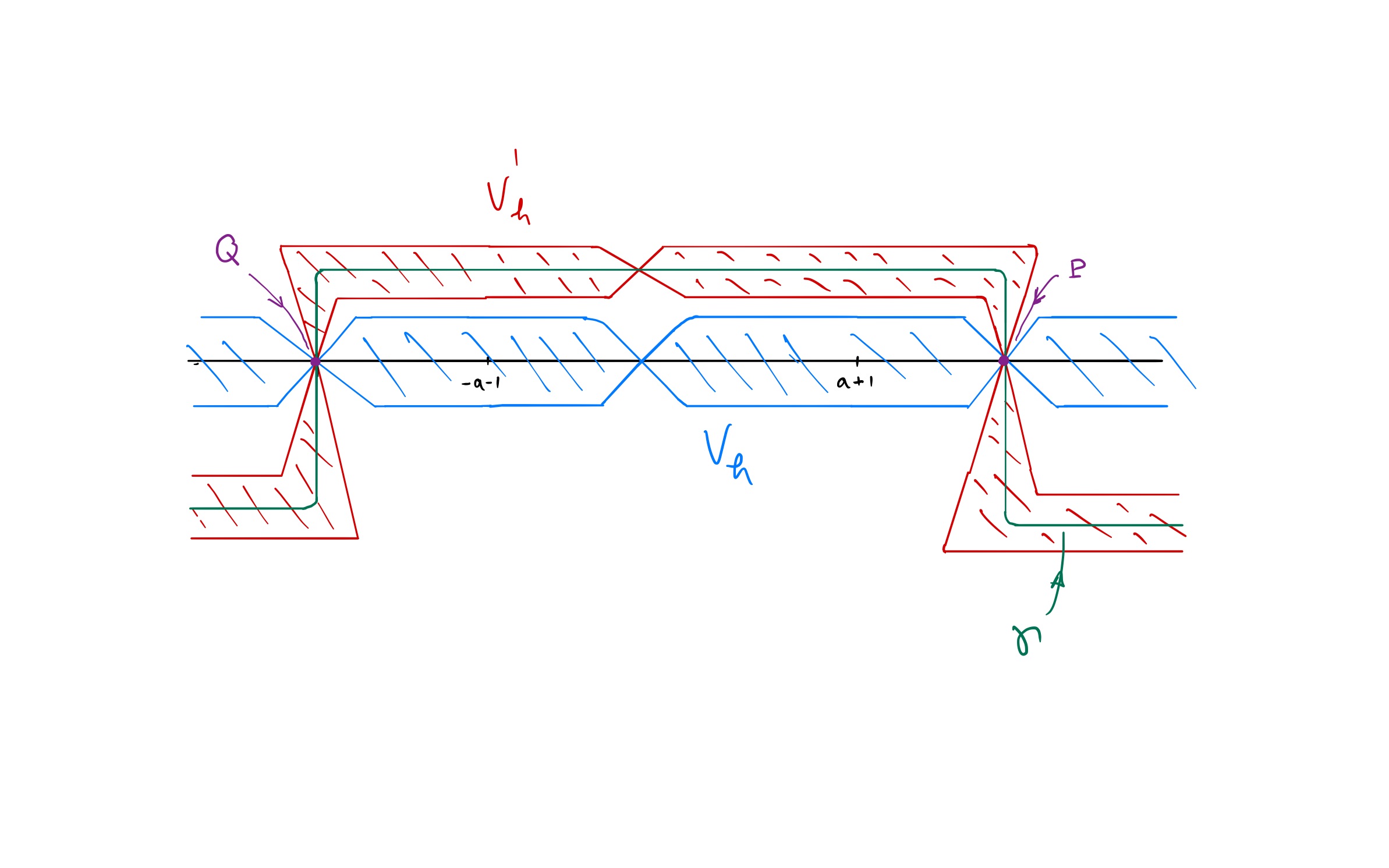}
   \end{center}
   \caption{\label{fig:smallperturb} Two small perturbations $V_{h}$ and $V'_{h}$ of 
   $\R\times L$ and $\gamma'\times L$, respectively.}
\end{figure}

We fix the perturbation data $\mathcal{D}_{L}$ on $L$ (as well as the marking $c$). We consider
perturbation data $\mathcal{D}_{V_{h}}$ on $V_{h}$ satisfying the Assumptions I-VII and that is a small 
perturbation of the perturbation $\mathcal{D}_{L}$. ``Smalleness'' here is understood in the sense that 
if $h_{L}$, $h'_{L}$ tend to $0$, then $V_{h}$ tends to $\R\times L$ and  this perturbation data tends to the product data $i \times D_{L}$. Moreover, the data $\mathcal{D}_{V_{h}}$ is defined in such a way that on 
the boundary of polygons the perturbation vanishes (in other, words the almost complex structure is the product one $i\times J_{0}$, and the Hamiltonian term is $0$ - as in Definition \ref{defi:perturbD} ii. a). 
We define the data $\mathcal{D}_{V'_{h}}$ by transporting appropriately the data of $V_{h}$ through a Hamiltonian isotopy induced from an isotopy in $\C$.

In view of the fact that there are no tear-drops with boundary on $L$, we will see below that, if 
$V_{h}$ is sufficiently close to the product $\R\times L$ and the data $\mathcal{D}_{V_{h}}$ is sufficiently close to $i\times \mathcal{D}_{L}$, then $V_{h}$ is unobstructed (and similarly for $V'_{h}$). 

\begin{lem}\label{lem:ident}
 If $V_{h}$ is sufficiently close to $\R\times L$ (and thus $h_{L}$ is small too) 
 and if $\mathcal{D}_{V_{h}}$
is sufficiently close to the product $i\times \mathcal{D}_{L}$, then $V_{h}$ is unobstructed.
\end{lem} 
\begin{proof} Let $X$ denote the inverted bottleneck (at $0$).  Given that $L$ is unobstructed, and in view of the shape of the curves $\gamma_{+},\gamma_{-}$, the only (marked) tear-drops with boundary on $V_{h}$ can have inputs in points either all in the fiber $\{Q\}\times M$ or  all in $\{P\}\times M$ and  the output has to be in a point in $\{X\}\times M$. A tear-drop can not pass from one side to the other of $X$ (due to the fact that the perturbation of the data chosen here vanishes on the boundary
of the perturbed $J$-holomorphic polygons). The key remark is that for $V_{h}$ sufficiently close to $\R\times L$ and  $\mathcal{D}_{V_{h}}$ sufficiently
close to the product $i\times \mathcal{D}_{L}$, the only marked tear-drops in $0$-dimensional moduli spaces are of a very special type: they
only have one input and one output ($|u|=2$) with the input some  point of the form
$\{Q\}\times \{x\}$ or $\{P\}\times \{x\}$) and the output $\{X\}\times \{x\}$ with $x\in c$. If this would not be 
the case, there would exist marked tear-drops $u_{n}$ (with the same asymptotic conditions) for $V_{h}$ approaching $\R\times L$  and data  approaching $ i\times \mathcal{D}_{L}$, and by taking the limit we would get a curve $u$ whose projection $v$ onto $M$ will also be a marked teardrop with boundary on $L$. This curves are regular by the assumption on $\mathcal{D}_{L}$. But it is easy to see, 
due to the fact that $X$ is an inverted bottleneck, that the dimension of the moduli space containing $v$  drops by at least one compared to the dimension of the moduli spaces containing the curves $u_{n}$. Thus, such curves $v$ do not exist, except if they are constant, in which case there is a single output and a single input and they coincide which implies the claim. 

We now fix $V_{h}$ and $\mathcal{D}_{V_{h}}$ so that the only tear-drops in $0$-dimensional moduli 
spaces are as explained above.  We want to notice that for each point $x\in c$ both possibilities occur, in the sense that both types of
 marked tear-drops, the ones starting over $P$ and the ones starting over $Q$ appear, and in each case their number (mod $2$) is $1$. This follows by a cobordism argument,  based on the fact that for the product structure $i\times \mathcal{D}_{L}$ these two types of tear-drops exist (they project to the  point $x$ on $M$), are regular and there is precisely one starting over $P$ and one starting over $Q$. 

%Alternatively, because we are only interested in moduli spaces of polygons with entries that 
%are part of the marking it is possible to directly see that $i\times \mathcal{D}_{L}$ is already
%regular.

The conclusion is that (mod $2$) the number of tear-drops at each such point $\{X\}\times \{x\}$ vanishes and this implies that $V_{h}$ is unobstructed.
\end{proof}

We now select Floer data for the couple $(V'_{h}, V_{h})$ so that it also
is close to $i \times \mathcal{D}_{L}$ in the sense above. 

We consider the resulting chain map $\phi:CF(L,L)\to CF(L,L)$ given by counting elements in $0$-dimensional moduli spaces of  (marked) Floer strips with boundary on $V'_{h}$ and $V_{h}$ that go from hamiltonian orbits  lying over the point $P$ to orbits in the fiber over  $Q$
(the definition of this map is standard in embedded cobordism type arguments as 
in \cite{Bi-Co:cob1} and the fact that it is a chain map is a simple application of the unobstructedness conditions).
Keeping the curves $\gamma$ and $\gamma'$ fixed as well as the Lagragian $L$, 
another application of a Gromov compactness argument shows that there exist $V_{h}$, $V'_{h}$ possibly even closer to products, as well as data also possibly closer 
 to $i\times \mathcal{D}_{L}$ such that the map $\phi$ is the identity.

We will call a cobordism $V_{h}$ endowed with the data $\mathcal{D}_{V_{h}}$ as before 
a {\em sufficiently small} deformation of the cobordism $\gamma\times L$ if there exists $V_{h}'$
 as above  (thus, in particular, the associated map $\phi$ is the identity).  Similarly,
the respective perturbation germs $h_{L}$ (that appear at the ends of $V_{h}$) 
will be called {\em sufficiently small} perturbation germs.

In summary, the deformation $V_{h}$ of the trivial cobordism $\gamma\times L$ is sufficiently
small if it is unobstructed and it admits itself a deformation $V'_{h}$, positioned as 
in Figure \ref{fig:smallperturb}, also unobstructed and such that the morphism $\phi$ relating 
the Floer homologies over $P$ and $Q$ is the identity.

\subsubsection{Final definition of the objects and morphisms in $\mathsf{C}ob^{\ast}(M)$.}
\label{subsubsec:objmorph} 

We first adjust the classes $Lag^{\ast}(M)$ and $Lag^{\ast}(\C\times M)$ (that satisfy the Assumptions
 I- VII from \S\ref{subsec:exact-cob})
in accordance with the discussion in \S\ref{subsubsec:smallpert}. Thus, we will only consider the
unobstructed, marked Lagrangians $(L,c, \mathcal{D}_{L}, h_{L})\in Lag^{\ast}(M)$  with $h_{L}$ a sufficiently small perturbation germ, in the sense of \S\ref{subsubsec:smallpert}.  
An additional important constraint is that we will assume that the base almost complex structure
included in the data $\mathcal{D}_{L}$ is {\em  the same} $J_{0}$ for all Lagrangians $L$. This constraint is particularly important to be able to define cabling properly, in this context. 

We denote the resulting class by $\mathcal{L}ag^{\ast}(M)$. In particular, all Lagrangians in
$\mathcal{L}ag^{\ast}(M)$ are unobstructed with respect to the same base almost complex 
structure $J_{0}$ on $M$, that we fix from now on  (see Definition \ref{defi:perturbD}). In our setting
this means that there are no $J_{0}$ (non-marked) tear-drops with boundary on any 
$L\in \mathcal{L}ag^{\ast}(M)$ and that, for the  perturbation data $\mathcal{D}_{L}$, the number
of $c$-marked tear-drops at any self-intersection point of $L$ are $0$.

%Two elements $(L,c, \mathcal{D}_{L}, h_{L}), (L,c, \mathcal{D}_{L}, h'_{L})$ will be identified
% as objects in $\mathsf{C}ob^{\ast}(M)$.

The class of cobordisms $\mathcal{L}ag^{\ast}(\C\times M)$ will refer from now on to those 
cobordisms $$(V,c,h,\mathcal{D}_{V_{h}})\in Lag^{\ast}(\C\times M)$$ with ends among the 
elements  of $\mathcal{L}ag^{\ast}(M)$.  To recall notation, the immersions associated to these 
Lagrangians is denoted by $j_{L}, j_{V}$ and the primitives are $f_{L}$, $f_{V}$.

%and subject to the additional condition that for some small $\epsilon_{V}$ {\em all} arbitrarily small %perturbations $h$ that are generic and with $||h||_{C^{2}}\leq \epsilon_{V}$ are such that $(V,c,h,%\mathcal{D}_{V_{h}})\in Lag^{\ast}(\C\times M)$ 
%(in particular,  for these arbitrarily small perturbations $V_{h}$ is unobstructed). 
%\ocnote{This condition has to be adjusted a bit to take into account the data $\mathcal{D}_{V_{h}}$}
%This is a strong condition
%whose aim is to diminish the dependence of the various constructions involving these cobordisms on the 
%perturbations $h$.

Finally, we define the cobordism category $\mathsf{C}ob^{\ast}(M)$ with
 objects the elements in $\mathcal{L}ag^{\ast}(M)$.  The morphisms are the elements
 in $\mathcal{L}ag^{\ast}(C\times M)$
modulo an equivalence relation that identifies  cobordisms up to:
\begin{itemize}
\item[i.] horizontal Hamiltonian isotopy  (of course, these isotopies act not only on $V,c$ and $h$ but, in an obvious way, on $\mathcal{D}_{V_{h}}$ as well).
\item[ii.] any two small enough deformations $V_{L}$, $V_{L}'$, 
of the cobordism $\R\times L$, $\forall L$,  are equivalent and, for any other cobordism $V$, 
the compositions $V\circ V_{L}$ and $V'_{L}\circ V$ (if defined) are both equivalent to $V$. 
\end{itemize}

Property ii allows for the definition of an identity morphism for each object $L$ of our category: this is the equivalence class of any sufficiently small deformation of the identity cobordism $\R\times L$.

\

We complete the set of elements of $\mathcal{L}ag^{\ast}(\C\times M)$ with respect to composition, as defined in \S\ref{subsubsec:comp}.  We denote by $\mathcal{L}ag^{\ast}_{0}(\C\times M)$ the class
of objects in $\mathcal{L}ag^{\ast}(\C\times M)$ that do not carry tear-drops (see also \S\ref{subsubsec:morph}).

\begin{rem} In the definition of the composition it is possible to make choices such that the 
resulting composition operation defined on the morphisms of $\mathsf{C}ob^{\ast}(M)$ is associative
but we will not further detail this point here.
\end{rem}

To avoid further complicating the notation, we will continue to denote by $\fuk^{\ast}_{i}(M)$ and, respectively, $\fuk^{\ast}_{i}(\C\times M)$, the Fukaya categories associated to the classes $\mathcal{L}ag^{\ast}(M)$ and $\mathcal{L}ag^{\ast}_{0}(\C\times M)$ even if, properly speaking, these are full subcategories
of the Fukaya categories introduced in  \S\ref{subsubsec:modules} and \S\ref{subsubsec:morph}.

% !TEX root = ImmersedS.tex

\section{From geometry to algebra and back.} 

\subsection{The main result.} 

With the preparation in \S\ref{sec:unobstr-tech} and, in particular, with the category $\mathsf{C}ob^{\ast}(M)$ defined in \S\ref{subsubsec:cob-cat-def} and with the machinery of
surgery models introduced in \S\ref{subsec:surg-models}, the purpose of this section is to prove the main result, Theorem \ref{thm:BIG}. The constructions in this proof can be viewed as establishing a number of ``lines'' in 
the dictionary  \emph{geometry} $\leftrightarrow$ \emph{algebra} mentioned at the beginning of the paper.
The passage from geometry to algebra is in essence a direct extension of the results concerning 
Lagrangian cobordism from \cite{Bi-Co:cob1}, \cite{Bi-Co:lcob-fuk}. 

We reformulate Theorem \ref{thm:BIG} in a more precise fashion just below. 

\begin{thm}\label{thm:surg-models}
The category $\mathsf{C}ob^{\ast}(M)$ as defined in \S\ref{subsubsec:cob-cat-def} has the following 
properties:
\begin{itemize}
\item[i.] $\mathsf{C}ob^{\ast}(M)$ has rigid surgery models.
\item[ii.] There exists a triangulated functor 
$$\widehat{\Theta}:\widehat{\mathsf{C}}ob^{\ast}(M)\to  H(mod(\fuk^{\ast}(M)))$$
that restricts to a triangulated isomorphism 
$\widehat{\Theta}_{e}:\widehat{\mathsf{C}}ob^{\ast}_{e}(M)\to  D\fuk^{\ast}(M)$
where $\widehat{\mathsf{C}}ob^{\ast}_{e}(M)$ is the triangulated subcategory of 
$\widehat{\mathsf{C}}ob^{\ast}(M)$ generated by $\mathcal{L}ag^{\ast}_{e}(M)$.
\item[iii] The category $\widehat{\mathsf{C}}ob^{\ast}(M)$ is isomorphic to the Donaldson 
category associated to the Lagrangians in $\mathcal{L}ag^{\ast}(M)$.
\item[iv.] For any family $\mathcal{F}\subset \mathcal{L}ag^{\ast}_{e}(M)$ the 
morphism $\widehat{\Theta}_{e}$ is non-contracting with respect to the shadow pseudo-metric $d^{\mathcal{F}}$ on $\widehat{\mathsf{C}}ob^{\ast}_{e}(M)$ and the pseudo-metric $s^{\mathcal{F}}_{a}$ on $D\fuk^{\ast}(M)$.
\end{itemize} 
\end{thm}

\begin{rem}
a. Allowing Lagrangians to be immersed is necessary, in general, for surgery models to exist.
However, our proof also makes essential use of the fact that we allow for {\em marked} immersed Lagrangians. It is possible that this extension to marked Lagrangians is not necessary to construct categories with surgery models but, as we will explain later in the paper, restricting to non-marked Lagrangians adds significant complications. It is an interesting open question at this time whether these can be overcome. 
 
 b. As mentioned before, in Remark \ref{rem:bding-chain}, the specific type of marked immersed Lagrangians that appears in our construction is a variant of the notion of an immersed Lagrangian endowed with a bounding chain in the sense of \cite{FO3:book-vol1}\cite{Akaho-Joyce}.
 \end{rem}

The proof of the theorem is contained in the next subsections. Here is the outline of the argument.
We first move from geometry to algebra. There are a couple main
steps in this direction.  First,  in \S\ref{subsubsec:funct-th} we show the existence of a functor
\begin{equation}\label{eq:funct-Theta}
\Theta: \mathsf{C}ob^{\ast}(M)\to H(mod(\fuk^{\ast}(M)))~.~
\end{equation} 
We then revisit in our setting the cabling construction from \S\ref{subsubsec:cables} and show in
\S\ref{subsubsec:theta-td} that 
the cabling equivalence of two cobordisms $V$, $V'$ in $\mathcal{L}ag^{\ast}(\C\times M)$
comes down to the equality $\Theta (V)=\Theta (V')$.
We then proceed in the reverse direction, from algebra to geometry.  We start by associating in \S\ref{subsubsec:surg-mor}
to each module morphism $\phi$ an appropriate $0$-surgery cobordism 
$V$ so that $\Theta(V)\simeq\phi$. We pursue in \S\ref{subsubsec:surg-model} by remarking that, as a consequence of the properties of the functor $\Theta$,  $\mathsf{C}ob^{\ast}(M)$ has surgery models
and that $\Theta$ induces $\widehat{\Theta}$ with the properties in the statement.
The various parts of the proof are put together in \S\ref{subsec:proof-surg-mod} where
we also address the rigidity part of the statement.

\

\subsection{From geometry to algebra.}\label{subsec:geo-to-alg}

\subsubsection{The functor $\Theta$.}\label{subsubsec:funct-th}
The definition and properties  of this functor are similar to those of the functor
$\Theta$ constructed for embedded Lagrangians and cobordisms in 
\cite{Bi-Co:lcob-fuk}.  In this subsection we first define $\Theta$ and pursue to describe 
some of its main properties. Namely, we give a homological expression for $\Theta$ in Lemmas \ref{lem:hlgy-cl} and \ref{lem:inv-elem}. We then show  in Lemma \ref{lem:comp-theta}  that $\Theta$ behaves functorially.  We then provide in Lemma \ref{lem:ham-isot} an alternative description of $\Theta$ involving Hamiltonian isotopies, at least for simple cobordisms and deduce the behaviour of $\Theta$ with respect to inverting simple cobordisms in Corollary \ref{cor:inv-cob}
Finally, we discuss the impact of changes of perturbation data for the same geometric Lagrangian $L$ in
Lemma \ref{lem:data-change}.

Recall from \S\ref{subsubsec:morph} that there is a category
$\fuk^{\ast}_{i}(\C\times M)$ (we neglect in this writing the data $\bar{\mathcal{D}}$) with objects the cobordisms in $\mathcal{L}ag^{\ast}_{0}(\C\times M)$.

The first step is  to consider the so-called inclusion functor  associated to a curve $\gamma :\R\to \C$ that
is horizontal at infinity (we view such a curve as a cobordism between two points, embedded in $\C$ itself):
$$i_{\gamma}:\fuk^{\ast}(M)\to \fuk^{\ast}(\C\times M)\to \fuk^{\ast}_{i}(\C\times M)~.~$$
This is defined as a composition, with the second map being the embedding of the category of embedded cobordisms in the category of immersed, marked cobordisms:
$$\fuk^{\ast}(\C\times M)\hookrightarrow \fuk^{\ast}_{i}(\C\times M)~.~$$
The first map $\fuk^{\ast}(M)\to \fuk^{\ast}(\C\times M)$ is the inclusion
functor constructed in \cite{Bi-Co:lcob-fuk} (and denoted there also by $i_{\gamma}$).
To fix ideas we recall that, on objects, $i_{\gamma}(L)=\gamma\times L$. 

Thus, to a cobordism $V\in \mathcal{O}b(\fuk^{\ast}_{i}(\C\times M))$ we can associate a module over $\fuk^{\ast}(M)$ defined as the pullback $i_{\gamma}^{\ast}(\mathcal{Y}(V))$ where $\mathcal{Y}(V)$
is the Yoneda module of $V$. This module is invariant, up to quasi-isomorphism, with respect to the usual horizontal homotopies of the curve $\gamma$ (horizontal homotopies of $\gamma$ are assumed here to leave the curve $\gamma$ horizontal outside of the region $[-a-1,a+1]\times \R$ - see Figure \ref{fig:strict-bottle}).

The second step is to consider a particular curve $\gamma$ that intersects $V_{h}$ (the deformation of the cobordism $V$) only over the bottlenecks (in, particular, it is away from the point $P_{i}$ in case $V$ carries
non-marked tear-drops), at points of real coordinates $-a-2$, for the negative ends and $a+2$ for the positive ones, and give a description of the module $i_{\gamma}^{\ast}(\mathcal{Y}(V))$ as an iterated cone of the Yoneda modules of the ends.  Assuming the curve $\gamma$ intersects the bottlenecks as in Figure \ref{fig:moved-str-bottle} (in other words, the curve $\gamma$ intersects the blue horizontal region only at the bottlenecks themselves) the arguments from \cite{Bi-Co:lcob-fuk} adjust to this 
case without difficulty. For instance, consider $V:L\cobto (L_{1},\ldots, L_{k},L')$
and $\gamma$ as in Figure \ref{fig:it-cone} below, such that $\gamma$ starts below the positive end then crosses it to pass over the non-horizontal part of $V$ and descends  by crossing the negative ends $L', L_{k},\ldots, L_{s+1}$, becoming again horizontal between the ends $L_{s}$ and $L_{s-1}$. In this case, 
the iterated cone $i^{\ast}_{\gamma}(\mathcal{Y}_{e}(V))$ is quasi-isomorphic to:
 \begin{equation}\label{eq:cones}
 Cone( \mathcal{Y}_{e}(L)\to Cone (\mathcal{Y}_{e}(L')\to Cone (\mathcal{Y}_{e}(L_{k})\to \ldots \to Cone(\mathcal{Y}_{e}(L_{s+1})\to \mathcal{Y}_{e}(L_{s})))..)
 \end{equation}
with the module morphisms indicated by the arrows above associated in a geometric fashion 
to the cobordism $V$ and where $\mathcal{Y}_{e}(N)$ is the pull-back to $\fuk^{\ast}(M)$ (the category of embedded objects) of the Yoneda module  $\mathcal{Y}(N)$ which is defined 
over the category of immersed Lagrangians, $\fuk^{\ast}_{i}(M)$.
\begin{figure}[htbp]
   \begin{center}
      \includegraphics[scale=0.18]{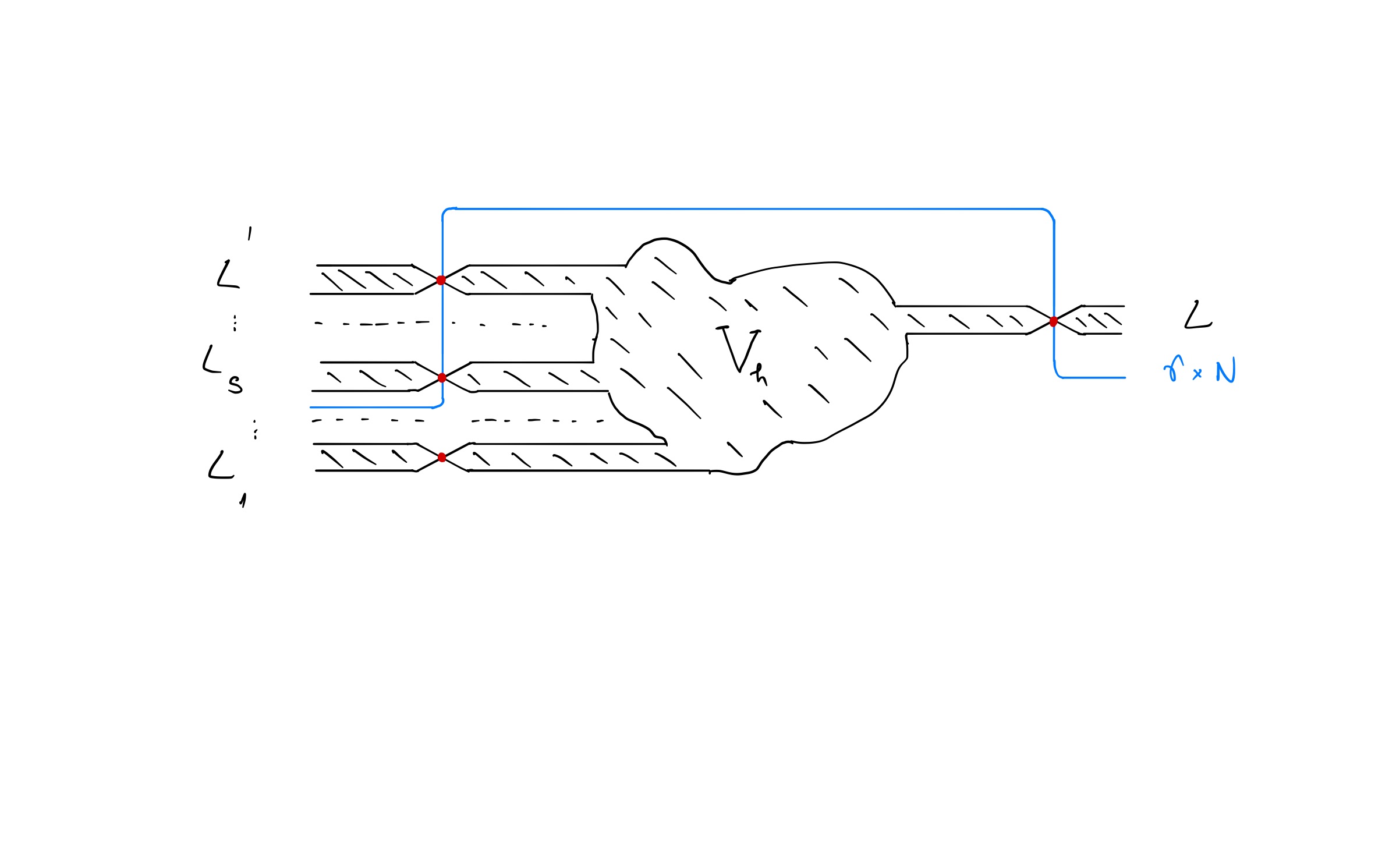}
   \end{center}
   \caption{\label{fig:it-cone} Position of the curve $\gamma$ relative to the (deformed) 
   cobordism $V_{h}$.}
\end{figure}

Of course, the $J$-holomorphic curves being counted in the different structures
considered here are marked. The use of positive action markings and 
the fact that the perturbation $h$ is itself positive at these intersection points of the ends 
implies that precisely the correct Yoneda module is associated to the intersection of $\gamma$ with each 
end of $V_{h}$.  

To conclude,  the construction of the inclusion functors
$i_{\gamma}$ and the description of the Yoneda modules $i^{\ast}_{\gamma}(\mathcal{Y}(V))$ 
carry over to the more general setting here without significant modifications when restricting to modules over the category of embedded objects. The definition of $\Theta$
is a particular application of this construction, as follows.  
On objects $\Theta$ associates to a Lagrangian $L\in \mathcal{L}ag^{\ast}(M)$ the Yoneda module 
$\mathcal{Y}_{e}(L)$. To define $\Theta$ on morphisms, consider first a cobordism $V\in \mathcal{L}ag^{\ast}_{0}(\C\times M)$ (thus $V$ does not carry (non-marked) tear-drops) and a curve $\gamma$ as in Figure \ref{fig:it-cone} such that $\gamma$ only crosses the ends $L$ and $L'$. In this case
\begin{equation}\label{eq:theta}
i^{\ast}_{\gamma}(\mathcal{Y}(V))\simeq Cone
\ (\mathcal{Y}_{e}(L) \stackrel{\phi_{V}}{\longrightarrow}\mathcal{Y}_{e}(L'))
\end{equation}
and we put $\Theta(V)=[\phi_{V}]\in \mor (H(mod(\fuk^{\ast}(M))))$.

For the particular curve $\gamma$ used to define $\Theta$ all the discussion also applies (with obvious
small modifications)
to $V$'s that carry non-marked tear-drops because $\gamma$ avoids
the pivots $P_{i}\in \C$ associated to $V$ and plane-simple coherence (see \S\ref{subsubec:tear}) 
is easy to check in this case. The resulting module is invariant under horizontal homotopies
of the curve $\gamma$ that again avoid the points $P_{i}$ and continue to intersect $V$ only along its 
two ends $L$ and $L'$.

To show that  $\Theta$ so defined is indeed a functor we need to check that $\Theta (V)$ does not 
depend on the choice of curve $\gamma$ (with the fixed behaviour at the ends); that it
is left invariant through horizontal Hamiltonian isotopies of $V$; that it associates the identity
to each sufficiently small deformation of the identity (recall these notions from \S\ref{subsubsec:smallpert}); and that it behaves functorially with respect to composition. The key points are the last two ones. They follow from an alternative description of $\Theta(V)$ that will play an important
role for us later on and thus we will first give this description. 

\begin{lem}\label{lem:hlgy-cl} Let $W$ be a sufficiently small deformation of $\gamma\times L$ where $\gamma\subset \C$ is as in Figure \ref{fig:cob-id} below. There is an associated morphism $\phi_{V}^{W}:CF(L,L)\to CF(L,L')$
and let $a^{W}_{V}=\phi_{V}^{W}([L])$ where $[L]$ is the unit in $HF(L,L)$.
We have the identity: 
\begin{equation}\label{eq:module-map}[\phi_{V}]=\mu_{2}(-,a^{W}_{V})
\end{equation}
where we use the identification $H_{\ast}(\hom_{mod}(\mathcal{Y}_{e}(L),\mathcal{Y}_{e}(L')))\cong HF(L,L')$.
\end{lem}
The morphism $\phi_{V}$ in (\ref{eq:module-map}) is defined through equation (\ref{eq:theta}) and the identity
 (\ref{eq:module-map}) is written, as will be explained below, in the homological category of modules
 over the immersed Fukaya category $\fuk_{i}^{\ast}(M)$.

\begin{figure}[htbp]
   \begin{center}
      \includegraphics[scale=0.18]{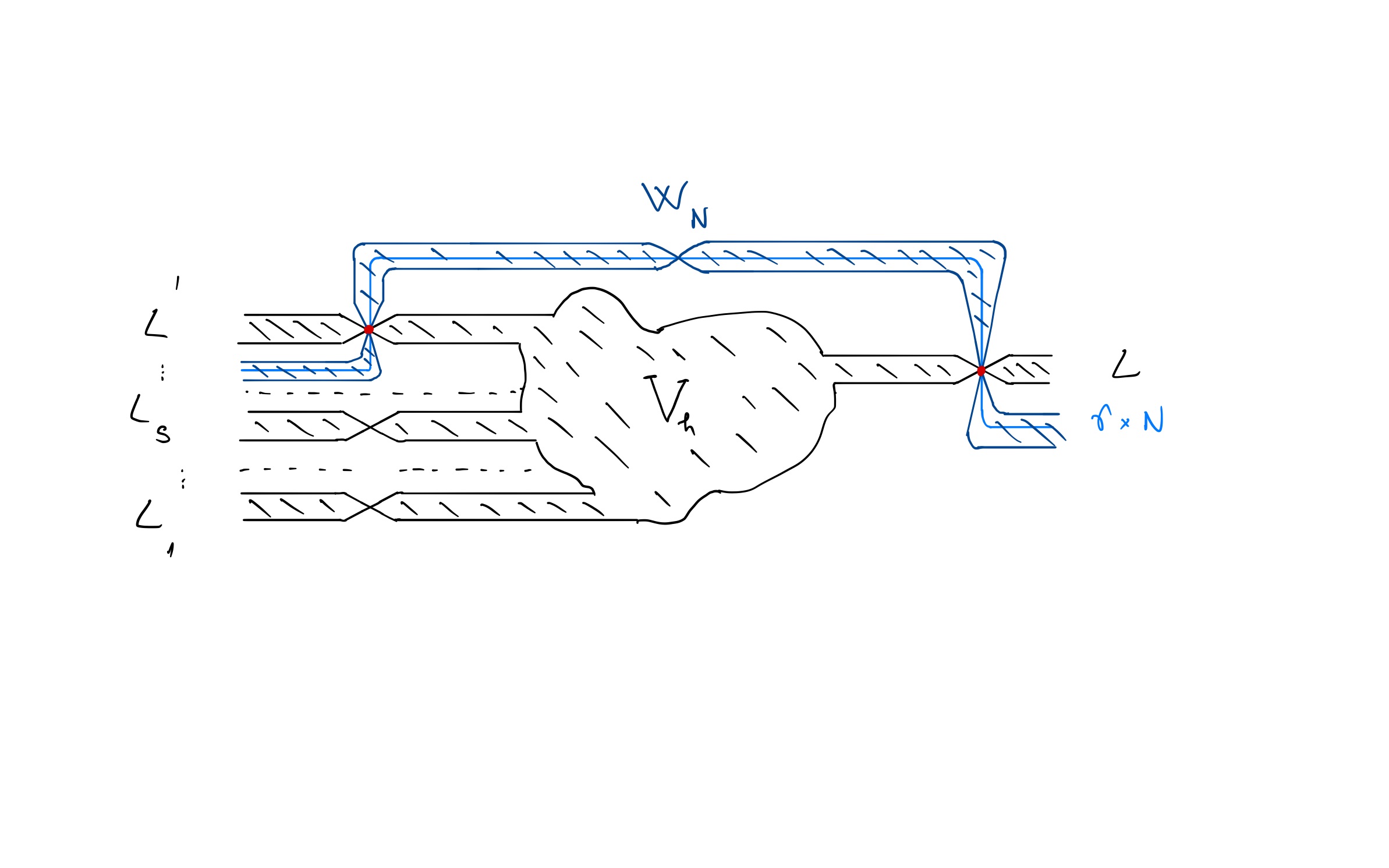}
   \end{center}
   \caption{\label{fig:cob-id}The perturbed cobordism $V_{h}$ and the small enough deformation $W_{N}$ of $\gamma\times N$.}
\end{figure}

\begin{proof}  Recall the existence of the immersed Fukaya  categories $\fuk_{i}^{\ast}(M)$, 
$\fuk_{i}^{\ast}(\C\times M)$.  We start by noting that there is an extension of the morphism
$\phi_{V}$ as a module morphism $\mathcal{Y}(L)\to \mathcal{Y}(L')$ over the category  $\fuk_{i}^{\ast}(M)$. To define this more general
 $\bar{\phi}_{V}$ we consider for each Lagrangian $N\in \mathcal{L}ag^{\ast}(M)$ a certain sufficiently small deformation, $W_{N}$, of  $\gamma \times N$  where $\gamma$ is the fixed curve in Figure \ref{fig:cob-id} ($W$ is one such example for $N=L$). Once these choices fixed, the same method as in \cite{Bi-Co:lcob-fuk} shows that this data leads to the extension of the  morphism $\phi_{V}$
 to $mod_{\fuk^{\ast}_{i}(M)}$. We denote this extension by $\bar{\phi}_{V}$.  Because the category $\fuk^{\ast}_{i}(M)$ is homologically unital it results that there exists 
an element $a^{W}_{V}$ as in the statement and this element is associated to $W=W_{L}$.
In fact, it follows that formula (\ref{eq:module-map}) is true over $\fuk^{\ast}_{i}(M)$.
\end{proof}

\begin{rem}\label{rem:ext-phi}
 It is useful to notice
 that the extension $\bar{\phi}_{V}$ of $\phi_{V}$ 
 depends on our choices of perturbations $W_{N}$ and, implicitely, also 
 depends on the choice of the curve $\gamma$. However, the properties of $\phi_{V}$
 can be deduced from properties of any extension  $\bar{\phi}_{V}$ because, up to a boundary,
 the morphisms $\phi_{V}$ are robust: they do not depend on the curve $\gamma$ (assuming the correct
 behaviour at the ends) or on perturbations.
\end{rem}
 
 We next consider a different curve  $\gamma'$ with the same ends at infinity as $\gamma$  in Figure \ref{fig:cob-id} and that intersects the  bottlenecks of $V_{h}$ in the same pattern
 as $\gamma$. 
  
  \begin{lem}\label{lem:inv-elem}
  In the setting above, there is a sufficiently small deformation $W'$
  of $\gamma'\times L$ such that the associated element $a^{W'}_{V}$ is defined and 
  $[a^{W}_{V}]=[a^{W'}_{V}]$.
  \end{lem}
  
  The argument to prove this is to use a horizontal Hamiltonian deformation from 
  $\gamma$ to $\gamma'$ and slide one of the bottlenecks along $\gamma$. To ensure unobstructedness
  all along this process the perturbation associated to $\gamma'$ (and along the transformation) might need
  to be reduced compared to that associated to $\gamma$.

 \begin{rem}
 A somewhat delicate point is that the size of the perturbation germs of $L$, as it appears at the two ends of $\gamma'\times L$,  depends potentially on $\gamma'$.  At the same time, the category $\fuk^{\ast}_{i}(M)$
 does not ``see'' perturbation germs associated to its objects and, as mentioned in Remark \ref{rem:ext-phi}, it follows that whenever we use various auxiliary
 curves $\gamma$, $\gamma'$, sufficiently small deformations of associated trivial cobordisms are sufficient
 for our arguments even if they depend on the curves $\gamma$, $\gamma'$.  
 
% b. While the class $[a_{V}^{W}]$ does not depend on $W$, in the sense of Lemma \ref{lem:inv-elem}, it is not clear that different choices of $\gamma'$, $W'$ (of size unrelated to the initial
%  $W$) might not produce different such classes. Nonetheless, all of them will induce the same morphism
 % $\phi_{V}$ (up to homotopy).
 \end{rem}

 Returning to the properties of $\Theta$, it follows from the definition of the sufficiently small 
 deformations of trivial cobordisms that the class $a_{V_{h}}^{V'_{h}}$ is the unit and thus $\Theta$
 applied to a trivial cobordism of the form $\gamma\times N$ is the identity.  
Showing that $\Theta$ respects the composition of cobordisms requires a slightly more involved
argument.

\begin{lem} \label{lem:comp-theta}Consider $V,V'\in \mathcal{L}ag^{\ast}(M)$ scuh that $V''=V'\circ V$ is defined
with the gluing over the end $L'$.
Then $\Theta (V'')=\Theta(V')\circ \Theta (V)$.
\end{lem} 
\begin{proof} We sketch
 $V$, $V'$ and their composition $V''$,  (after deformation) in Figure \ref{fig:comp-theta}. 
 
 \begin{figure}[htbp]
   \begin{center}
      \includegraphics[scale=0.18]{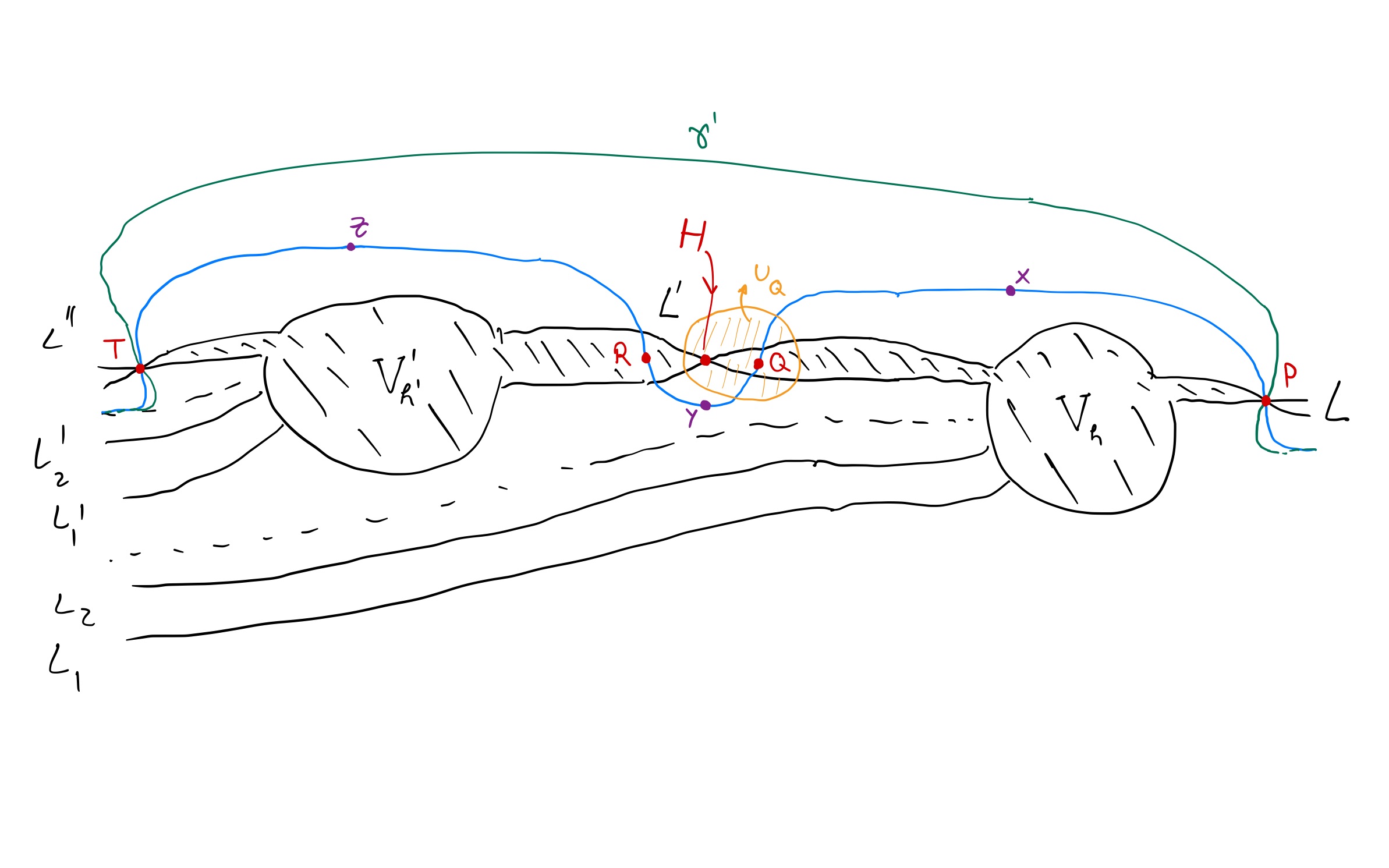}
   \end{center}
   \caption{\label{fig:comp-theta}The composition $V''=V'\circ V$ and the curves
   $\gamma$ and $\gamma'$. A sufficiently small deformation $W$ of $\gamma\times L$ has regular bottlenecks at $P,Q,R,T$ and inverted ones at $X,Y,Z$.}
\end{figure}

To fix ideas we start under the assumption that both $V$ and $V'$ do not carry (non-marked) tear-drops.
 We also consider two curves $\gamma$ and $\gamma'$ ($\gamma$ and $\gamma'$ are homotopic in the complement of the marked points $P_{i},P'_{j}\in\C$ associated to $V,V'$ in case these carry non-marked tear-drops) as in the
 picture. We consider a sufficiently small deformation $W$ of $\gamma \times L$ with regular
 bottlencks at the points $P,Q,R,T$ and with inverted bottlenecks at $X,Y, Z$. We assume that 
 $\gamma'$ is Hamiltonian isotopic to $\gamma$ with the ends $P$ and $T$ fixed and we also consider a
 sufficiently small deformation $W'$ of $\gamma'\times L$ which is obtained by the associated Hamiltonian deformation of $W$. We claim that there exists such a $\gamma$ and 
 $W$ (in particular $W$ is unobstructed)
 with the following additional porperties:  there is a complex $C_{Q}$ with generators the intersections
  $W\cap V''\cap \pi^{-1}(U_{Q})$, with $U_{Q}$ as in the picture, and whose differential
  counts curves with an image that does not get out of $U_{Q}$; under the translation bringing $Q$ over
   the bottleneck $H$ (that corresponds to the gluing of the end of $V$ to the corresponding end of $V'$) of$V''$ (by an isotopy coming from the plane).
   This complex is identified with $CF(L, L')$; there is a similar complex $C_{R}$
   associated to $R$ with similar properties; the strips going from the elements
    in $C_{R}$ to those in $C_{Q}$ give a morphism homotopic to the identity. 
    The existence of such a $W$ results from a Gromov compactness argument 
   applied by diminishing simultaneously the ``loop'' that $\gamma$ does around $H$ and the perturbation
giving $W$. Using the identifications of the complexes $C_{R}$ and $C_{Q}$ with the 
   complex $CF(L,L')$ over the bottleneck $H$, the complex $CF(W, V'')$ can be rewritten as a sum of four        
   complexes corresponding to the four regular bottlenecks of $W$ and with a differential that involves both 
   $\phi^{W}_{V}$, defined by strips going from $P$ to $C_{Q}$, $\phi^{W}_{V'}$ defined by strips
   going from $C_{Q}$ to $T$, and a morphism homotopic to the identity from $C_{R}$ to $C_{Q}$.
   As $W'$ is Hamiltonian isotopic to $W$, we finally deduce that $\phi^{W'}_{V''}\simeq
   \phi ^{W}_{V''}\simeq \phi^{W}_{V'}\circ \phi^{W}_{V}$  and, in view of Lemma \ref{lem:hlgy-cl},
  this ends the proof in case $V$ and $V'$ do not carry tear-drops. 
  
  If $V$ and $V'$ carry tear-drops, we notice that the curve
  $\gamma$ is away from the pivots of both $V$ and $V'$, that $\gamma$ and $\gamma'$ are homotopic in the complement of the union of  the pivots. Moreover, for the Floer complexes and maps that appear in this proof,   plane-simple coherence is easy to check. Thus, it follows that the argument remains true in this case too and this concludes the proof.
   \end{proof} 
   
\begin{rem} By using methods similar to the above it is also possible to show, following \cite{Bi-Co:lcob-fuk}, that the iterated cone-decomposition in (\ref{eq:cones}) remains valid also in terms of modules over the 
immersed Fukaya category $\fuk^{\ast}_{i}(M)$ at least for $V\in \mathcal{L}ag^{\ast}_{0}(\C\times M)$. 
\end{rem}

To summarize what we proved till now in this sub-section, we have shown that  $\Theta (V)$, as defined in (\ref{eq:theta}),   is characterized by the class $[a_{V}^{W}]\in HF(L,L')$ through equation  (\ref{eq:module-map}). Moreover, it also follows that the definition of $\Theta$ is independent of the choice of the curves $\gamma$ and that $\Theta$ is an actual functor, as it associates the identity to any small enough perturbation of a curve $\gamma\times L$. Finally, it follows that $\Theta$ does not ``see'' perturbation germs (as long as they are small enough).

\

We end the section with three other important features of the functor $\Theta$. 

\

Before going on, we simplify our drawing conventions to avoid drawing explicitly the projections onto the plane of the perturbations of the cylindrical ends of cobordisms. 
\begin{figure}[htbp]
   \begin{center}
      \includegraphics[scale=0.53]{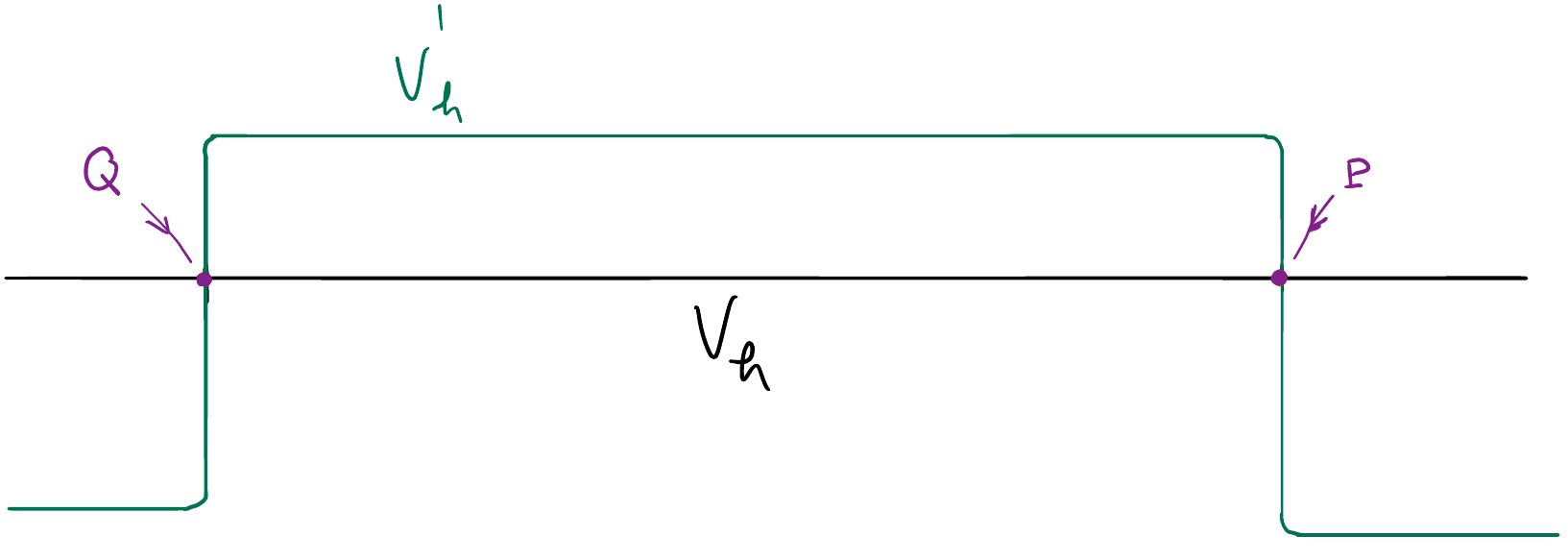}
   \end{center}
   \caption{\label{fig:smallperturb2} The schematic version of Figure \ref{fig:smallperturb}. The bottlenecks $P$ and $Q$ are regular for both $V'_{h}$ and $V_{h}$.}
\end{figure}

Instead, from now on we will represent these configurations schematically by using the following conventions that apply to the cylindrical ends (or parts) of immersed cobordisms:
we indicate the regular bottlenecks by a  dot (such as $P$, $Q$, $R$ and $T$ in Figure \ref{fig:comp-theta}); we will no longer indicate the inverted  bottlenecks (such as $X$, $Y$, $Z$ in Figure \ref{fig:comp-theta}) 
 but we assume that a unique inverted bottleneck always exists in between two regular ones; when two cylindrical components intersect at bottlenecks of both components, then the two bottlenecks are regular and positioned as in Figure \ref{fig:smallperturb}; there could also be pivots in $\C$, away from the 
 projection of the cobordisms considered, such as $P_{1}$, $P_{2}$ in Figure \ref{fig:markedpts}, that will be used to ``stabilize'' tear-drops. 
Of course, more details will be given if needed. To give an example of our conventions, Figure \ref{fig:smallperturb2} is the schematic version of Figure \ref{fig:smallperturb}.

The next result provides an alternative description for $\Theta(V)$ when $V$ is simple.

\begin{lem} \label{lem:ham-isot} Assume that $V\in \mathcal{L}ag^{\ast}_{0}(\C\times M)$ (which means
that $V$ does not carry any $J_{0}$-tear-drops) and is simple (which means that it has just two ends, one  positive and one  negative). Then $\Theta(V)$ is induced by the Hamiltonian 
translation moving a small enough perturbation of $\gamma \times L$ to $\gamma'\times L$ in Figure \ref{fig:ham-isotop}.
\end{lem}

 \begin{figure}[htbp]
   \begin{center}
      \includegraphics[scale=0.78]{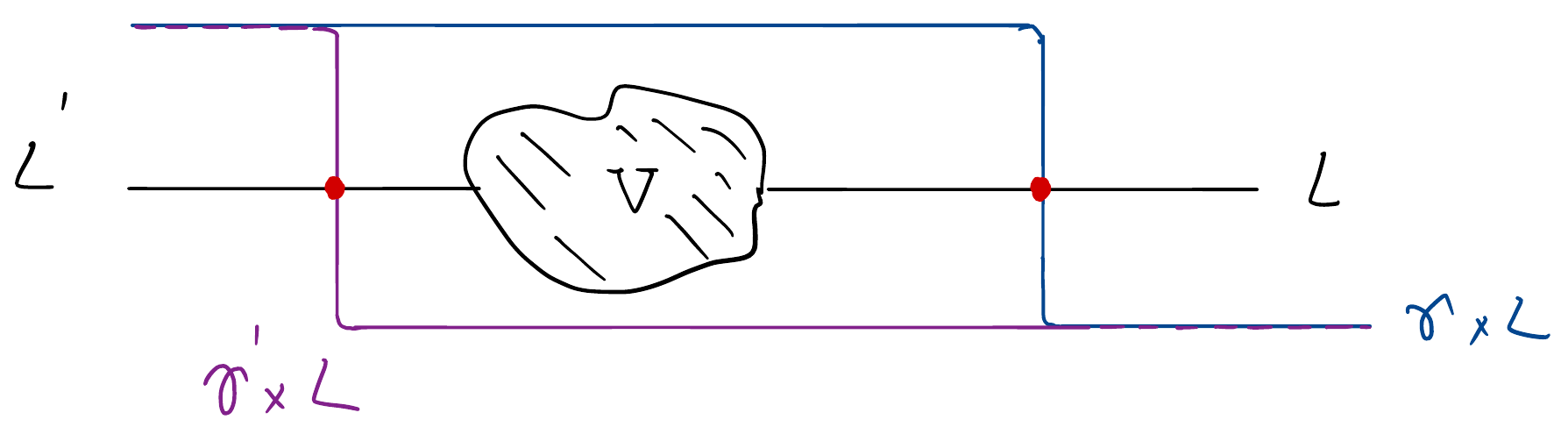}
   \end{center}
   \caption{\label{fig:ham-isotop}The Lagrangians $V$, $\gamma\times L$ and $\gamma'\times L$ drawn without perturbations. The intersections at the bottlenecks drawn in red is as in Figure \ref{fig:smallperturb} (at the point $P$). }
\end{figure}

\begin{rem}
There is a more general version of this statement when $V$ is not simple but we will not need it here.
\end{rem}

\begin{proof} The proof makes use of Figure \ref{fig:disp} below. 
 \begin{figure}[htbp]
   \begin{center}
      \includegraphics[scale=0.18]{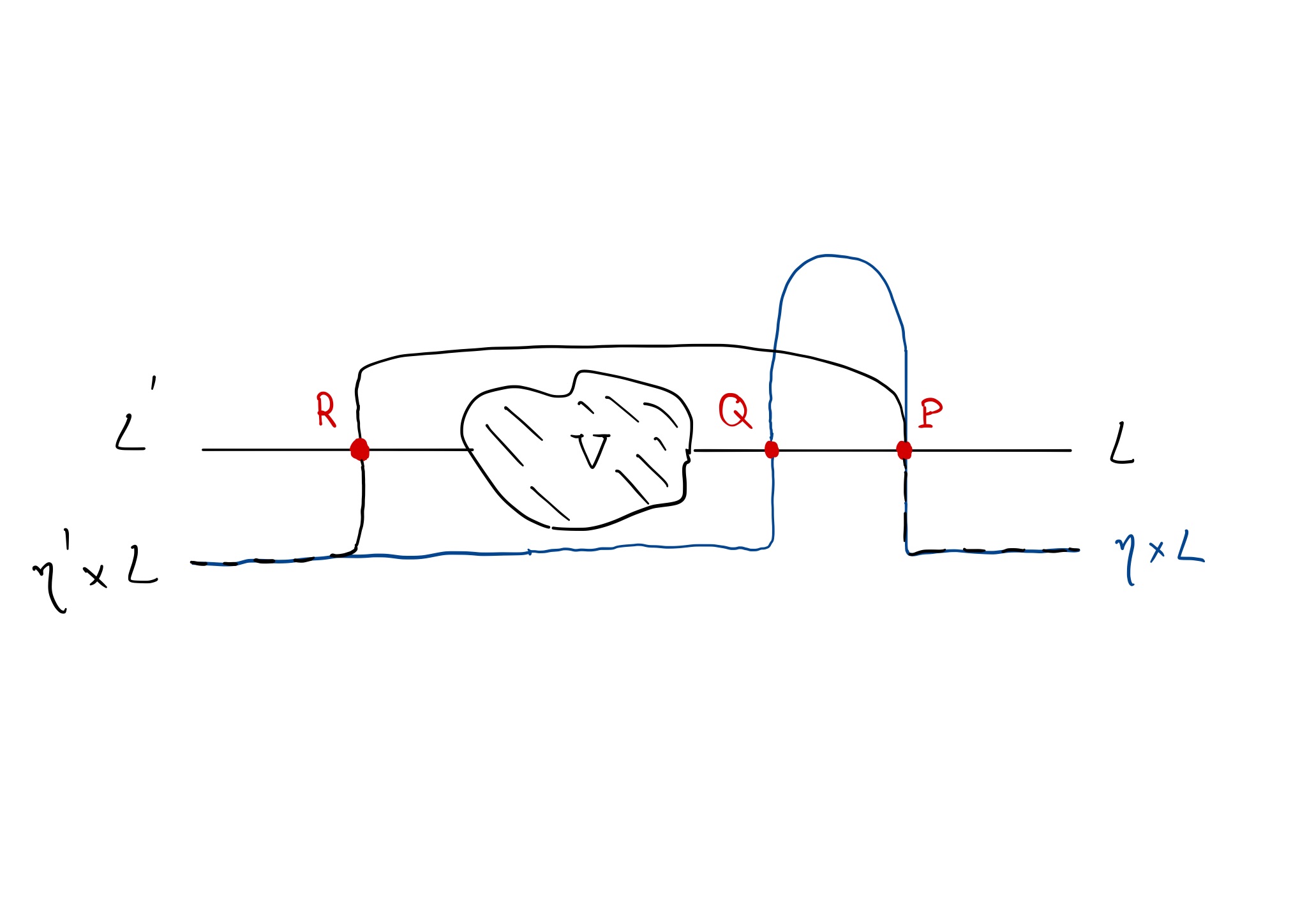}
   \end{center}
   \caption{\label{fig:disp} Schematic representation of $V$, $W=\eta\times L$ and $W'=\eta'\times L$
   obtained from $W$ by the obvious Hamiltonian isotopy carrying the bottleneck $P$ to $Q$ (and maintaining $P$ fixed.}
\end{figure}
We consider the two cobordisms 
$W=\eta\times L$ and $W'=\eta'\times L$ where $\eta, \eta'$ are the curves as in Figure \ref{fig:disp}.
There is an obvious horizontal Hamiltonian diffeomorphism $\varphi$ (induced from the plane $\C$)
 that carries $W$ to $W'$. Put $C=CF(W,V)$ and $C'= CF(W',V)$ and deduce that there is a chain
 quasi-isomorphism  $\psi : C\to C'$  induced by this Hamiltonian diffeomorphism. Of course, in the definition of these Floer complexes we need to use appropriate perturbations for $W$ and $W'$ (and we also consider a sufficiently small perturbation of $V$ in between $Q$ and $P$). In view of this, $C$ can be written
 as a cone over the identity morphism identifying the complex $C_{P}$ over $P$
 to the complex over $Q$, $C_{Q}$, both being identified to $CF(L,L)$. The complex $C'$ is the cone over the map $\phi_{V}^{W'}$ relating the complex $C_{P}$ and $C_{R}=CF(L,L')$.  The chain morphism
 $\psi$, induced by the Hamiltonian isotopy $\varphi$ moving $\eta$ to $\eta'$ (and that  moves 
 $Q$ to $R$ and keeps $P$ fixed), restricts
 to the chain map  $\psi^{W}_{V}: C_{Q}\to C_{R}$ that is easily seen to be homotopic
 to the chain map induced by the translation in the statement (to see this it is enough to notice that $\varphi$
 can be taken to be an appropriate translation moving $Q$ to $R$).  As $\psi$ is a chain map it 
 follows from a simple algebraic calculation that $\psi^{W'}_{V}$ is chain homotopic 
 to $\phi^{W}_{V}$ and this concludes  the proof.
 \end{proof}

Given a simple cobordism $V$ recall that $\bar{V}$ is the inverse cobordism, as defined in \S\ref{subsubsec:basic-def} by a $180^{o}$ roatation in the plane. 
We can also transform the data associated to 
$V$ in such a way as to get corresponding data for $\bar{V}$ and this is how we interpret
$\bar{V}$ as an element of $\mathcal{L}ag^{\ast}(\C\times M)$. It is also possible to adjust the deformation
of the ends, say $h_{L'}$, so that $\bar{V}\circ V$ is defined. Using Lemma \ref{lem:ham-isot}
and the functoriality of $\Theta$, it is a simple exercise to deduce the next statement.

\begin{cor} \label{cor:inv-cob} If $V$ is a simple cobordism $\in \mathcal{L}ag^{\ast}_{0}(\C\times M)$
and $\bar{V}$ is its inverse, then $\Theta(V)\circ \Theta (\bar{V})=id$.
\end{cor}

\begin{rem}
For simple cobordisms $V\in \mathcal{L}ag^{\ast}(\C\times M)$ that might carry (non-marked) tear-drops
the same result is true but the proof is more involved and is postponed to Lemma \ref{lem:inv-cob2}.
\end{rem}

We now fix an object $\in \mathcal{L}ag^{\ast}(M)$. Thus we are fixing a quadruple
$(L,c, \mathcal{D}_{L}, h_{L})$.  Recall that $L$ does not carry any $J_{0}$ (non-marked) tear-drops.
By using Seidel's iterative process, it is easy to construct coherent perturbations with $J_{0}$
as a base almost complex structure, that make all moduli spaces with action negative inputs (and output)
regular (as described in Definition \ref{defi:perturbD} and in Assumption I  in \S\ref{subsec:exact-cob}).
Let $\mathcal{D}'_{L}$ be such a choice of coherent perturbation data (the initial $\mathcal{D}_{L}$ is
another such choice). 

\begin{lem} \label{lem:data-change} There exist a marking $c'$, a perturbation $h'_{L}$ and a
simple unobstructed cobordism $V\in \mathcal{L}ag^{\ast}_{0}(\C\times M)$ with ends
$(L,c,\mathcal{D}_{L}, h_{L})$ and $(L,c',\mathcal{D}'_{L}, h'_{L})$. Geometrically, $V$ is 
a small deformation of the trivial cobordism $\R\times L$ with a projection such as $V_{h}$ in Figure \ref{fig:smallperturb}. 
\end{lem}

\begin{proof}
We will make use of Figure \ref{fig:data-change} below. In short, we consider a deformation
$W$ of the Lagrangian immersion $\R\times L$ with three bottlenecks, $P$ and $Q$ regular and
one inverted, $X$, in between $P$ and $Q$.  The element $V$ in the proof will be of the form 
$V=(\R\times L, \mathbf{c}, \overline{\mathcal{D}},h_{V})$ with $W=V_{h_{V}}$.

 \begin{figure}[htbp]
   \begin{center}
      \includegraphics[scale=0.18]{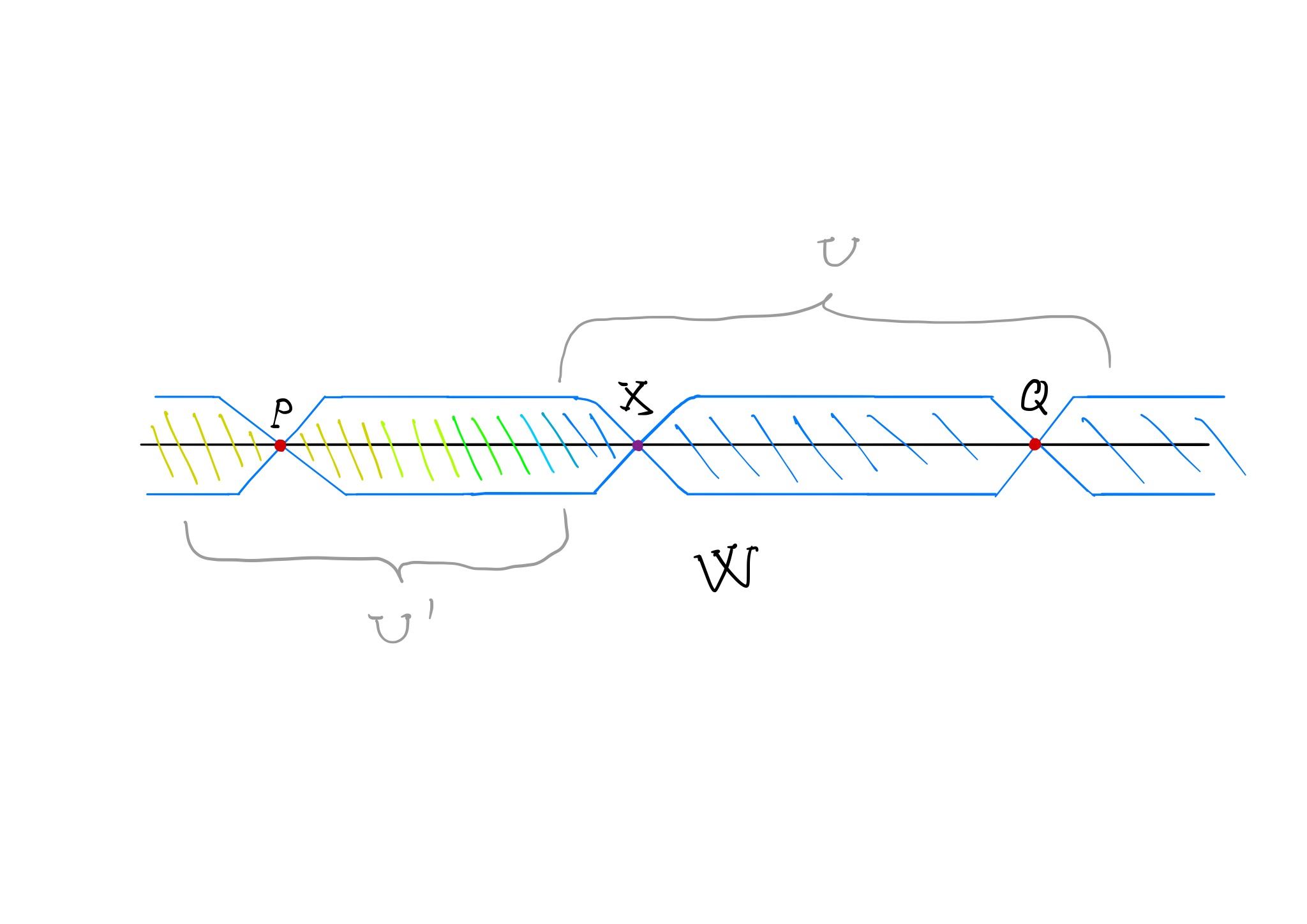}
   \end{center}
   \caption{\label{fig:data-change} The Lagrangian $W$. The data $\overline{\mathcal{D}}$ is a small deformation of the identity in the region $U$ and is an interpolation between the data $\mathcal{D}_{L}$
   and $\mathcal{D}'_{L}$ in the region $U'$.}
\end{figure}

We first discuss the data $\overline{\mathcal{D}}$ associated to $W$. It is a coherent choice of perturbations
such that $\overline{\mathcal{D}}$ restricts to $\mathcal{D}_{L}$
at $P$ and to $\mathcal{D}'_{L}$ at both $X$ and $Q$. Moreover, over the cobordism going from $X$ to $Q$, $(W, \overline{\mathcal{D}})$ is a small deformation of $(\R\times L, i\times \mathcal{D}'_{L})$. Such  $\overline{\mathcal{D}}$ exists by the usual construction of coherent perturbations because the base almost 
complex structure is the same $J_{0}$ for both $\mathcal{D}_{L}$ and $\mathcal{D}'_{L}$ and there are no
$J_{0}$ (non-marked) tear drops. The marking of $W$ coincides with that of $(L,c,\mathcal{D}_{L},h_{L})$
over $P$ and there are no markings over the bottleneck $X$. The purpose of the proof is to show that 
there is a marking $c'$ over $Q$ that consists only of self-intersection points that are action-negative
and such that the marking $\mathbf{c}=\{P\}\times c \cup \{Q\}\times c'$ makes $V$ unobstructed.

We denote by $J_{X}=\{X\}\times I_{L}$ the self intersection points of $L$ viewed in the 
fibre over the bottleneck $X$, $J_{X}\subset \{X\}\times L\subset \{X\}\times M$. We let $\Z/2<J_{X}>$ be the $\Z/2$ vector space of basis the elements of $J_{X}$ and consider 
the element $z\in \Z/2 < J_{X}>$ given by $z=\sum_{y\in I_{X}} \#_{2} \overline{\mathcal{M}}_{c}(\emptyset; y)y$. Here $\overline{\mathcal{M}}_{c}(\emptyset; y)$ is the moduli space of $c$-marked tear-drops
with boundary on $W$, defined with respect to the data $\overline{\mathcal{D}}$ and with output at $y$. 
As $c$ is the marking over $P$, all these tear-drops have marked inputs over $P$.
Of course, are only counted (with coefficients in $\Z_{2}$) the elements in those moduli spaces that are $0$ dimensional. By taking the deformation $W$ sufficiently close to $\R\times L$ and considering the 
action inequalities as in Remark \ref{rem:var1} c, we see that whenever 
$\overline{\mathcal{M}}_{c}(\emptyset; y)$  is not void the point $y$ is action negative.
We now define $c'$ to be the union of all the self-intersection points $y\in I_{L}$ that appear in $z$ and $\mathbf{c}=\{P\}\times c \cup \{Q\}\times c'$. Because $W$ is a sufficiently small deformation of the identity over the segment from $X$ to $Q$ , we deduce that at each point $y\in J_{X}$
the number of $\mathbf{c}$ marked tear drops is now $0$, as the number of tear-drops coming from $P$
agrees with the number (mod $2$) of tear-drops coming from $Q$ (see also  the proof of
 Lemma \ref{lem:ident}). It follows that $V$ is unobstructed and this concludes the proof.
\end{proof}

 %There is a need for some more precise description about behaviour of Floer complexes with respects to 
% special Morse homotopies.

\subsubsection{Cabling and counting tear-drops.}\label{subsubsec:theta-td}

In this section we revisit the cabling construction from \S\ref{subsubsec:cables} with the aim to show:

\begin{prop}\label{prop:theta-cabl}
Let $V,V'\in\mor_{\mathcal{C}ob^{\ast}(M)}(L,L')$. The two morphisms $V$ and $V'$ are cabling equivalent
(see \S\ref{subsubsec:cables} and Definition \ref{dfn:cabling-rel}) if and only if $\Theta(V)=\Theta(V')$.
\end{prop}

Along the way towards proving this result we will also need to adjust the cabling construction
to our setting, in a way that takes into account the deformations of the immersed cobordisms $V$ and $V'$.

\begin{proof}
Using  Lemma \ref{lem:hlgy-cl} we see that it is sufficient to show that $V$ and $V'$ are cabling equivalent iff
 $[a_{V}^{W}]=[a_{V'}^{W}]$ where $W$ is a small enough deformation of $\gamma\times L$ for an appropriate curve $\gamma$.  There are three steps to show this.
 
 \ \\
 \underline{Step 1. Cabling taking into account markings and perturbations.}

\noindent We now consider two cobordisms $V,V'\in \mor_{\mathcal{C}ob^{\ast}(M)}(L,L')$. We define the cabling
of $V$ and $V'$, $V''=\mathcal{C}(V,V';c)$, by revisiting the construction in \S\ref{subsubsec:cables} and, in particular, Figure \ref{Fig:cabling}.  We consider the Lagrangian represented in Figure \ref{fig:cabling2}
below (the components, $V$ and $V'$, appear in Figure \ref{Fig:twocob}). To fix ideas, we will first 
assume that $V, V'\in \mathcal{L}ag^{\ast}_{0}(M)$ and that $L$ and $L'$ intersect transversely.

\begin{figure}[htbp]
   \begin{center}
      \includegraphics[scale=0.2]{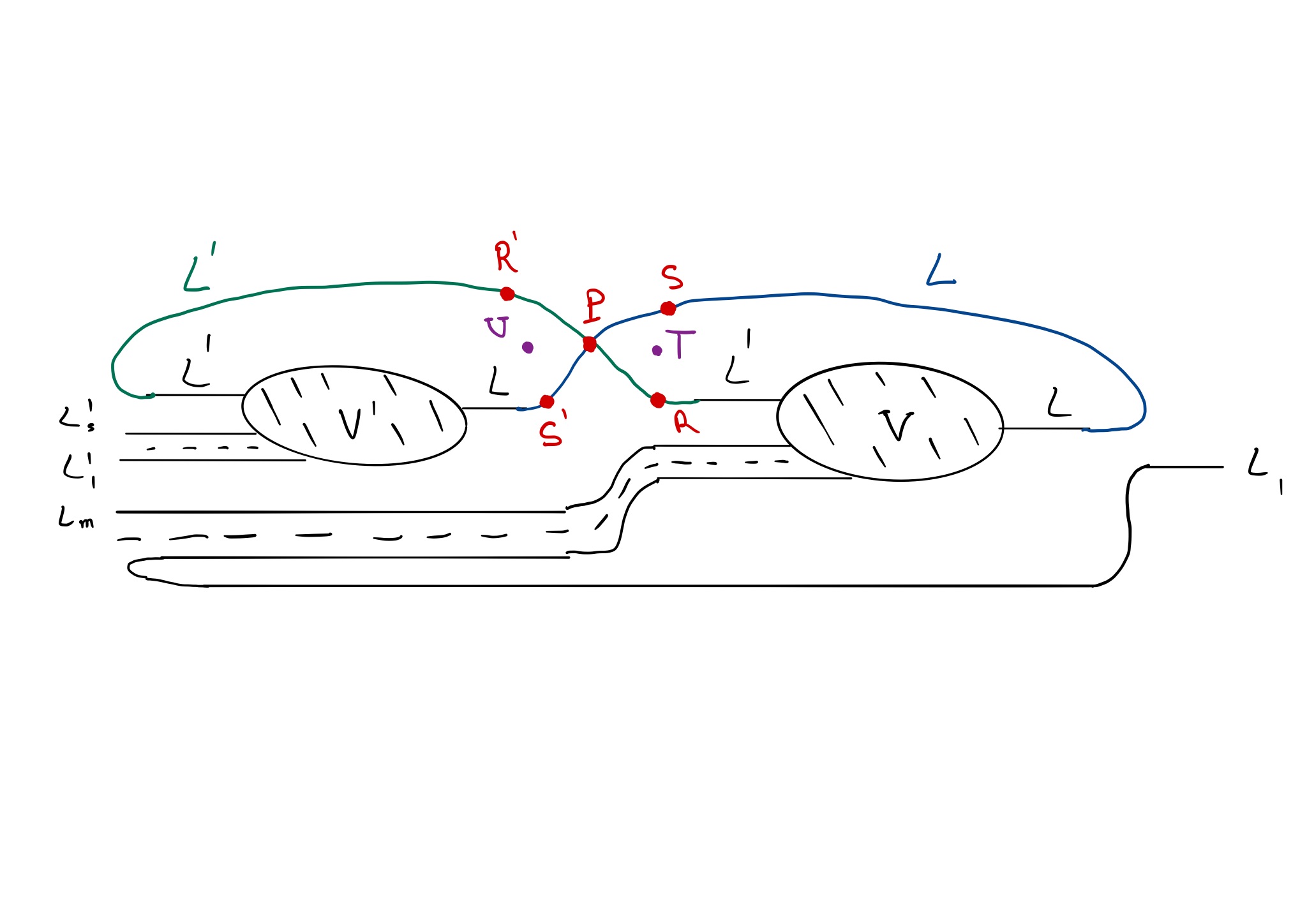}
   \end{center}
   \caption{\label{fig:cabling2} Schematic representation of cabling, in the immersed perturbed, marked, exact setting. The points $R,P,S,S',R'$ are regular bottlenecks, $U$ and $T$ are points in $\C$ used
   to stabilize tear-drops.}
\end{figure}
We now list details of the construction. The data for $V$ is $(V,h, c, \mathcal{D}_{V_{h}})$
and for $V'$ it is $(V',h',c',\mathcal{D}_{V'_{h'}})$ and we have the two primitives
$f_{V}$ and $f_{V'}$ (that are viewed as defined on $V_{h}$ and $V'_{h'}$).
\begin{itemize}
\item[i.] The bottlenecks $R,S$ and $R',S'$ are associated to the respective ends of $V$ and,
respectively, $V'$ with the structure induced from the deformed cobordisms $V_{h}$ and $V'_{h'}$.
\item[ii.] Postponing for the moment discussion of the relevant perturbation data, 
the  segments  $SP$, $RP$, and respectively $S'P$, $R'P$, represent 
sufficiently small deformations of products of the form $A \times L$ and, respectively, 
$A\times L'$ with $A$ the respective segment along the curves in green and blue in Figure \ref{fig:cabling2}. 
\item[iii.]
The profile of the gluing at $P$ for each of the two branches, one associated to $L'$ - with a restriction at $P$ written as $\{P_{+}\}\times L'$ and the other
associated to $L$, with restriction at $P$ written as $\{P_{-}\}\times L$, is as in Figure \ref{fig:comp2}. 
\item[iv.] The points $U$ and $V$ are pivots associated to $V''$.
\item[v.] The primitives $f_{V}$ and $f_{V'}$ are extended to a primitive $f_{V''}$ (this is defined
on the domain of the immersion and the ``sizes'' of the blue and green curves in Figure \ref{fig:cabling2} might need to be adjusted to ensure the existence of the primitive $f_{V''}$). The restriction of $f_{V''}$ to the submanifold $\{P_{+}\}\times L'$ is denoted $f_{V''}|_{L',P}$ and similarly for the restriction to $\{P_{-}\}\times L$ which is denoted by $f_{V''}|_{L,P}$. The curves in blue and green in Figure \ref{fig:cabling2} are picked in such a way
that $\min [f_{V''}|_{L,P}] \geq \max [f_{V''}|_{L',P}]$.
\end{itemize}
We further discuss the markings $c''$ and data $\mathcal{D}_{V''}$ associated to the cabling $V''$. 
To proceed, we first fix
the perturbation data associated to the ends we are interested in. Thus the data for $L$ is $(L, h_{L}, c_{1},
\mathcal{D}_{L})$ and that for $L'$ is $(L',h_{L'},c'_{1},\mathcal{D}_{L'})$. A key point here
is that for  the perturbation data for the cabling to make sense it has to be chosen in a coherent way for both
$L$ and for $L'$ when viewed as Lagrangians $\subset \{P\}\times M$.  To achieve this we will make use of Lemma \ref{lem:data-change}.  We first pick perturbation data $\mathcal{D}'$ that is coherent (and unobstructed)
for both $L$ and $L'$. By Lemma \ref{lem:data-change}, there are unobstructed cobordisms $V_{1}$ from 
$(L, h_{L}, c_{1}, \mathcal{D}_{L})$ to $(L, h_{L}, c_{2},
\mathcal{D}')$ and similarly $V'_{1}$ from  $ (L',h_{L'},c'_{1},\mathcal{D}_{L'})$ to $(L',h_{L'},c'_{2},\mathcal{D}')$. Here $c_{2}$, $c'_{2}$ are markings determined as in Lemma \ref{lem:data-change}.
We pick the perturbations in question here so that the cobordisms $\overline{V_{1}}\circ V_{1}$
 and $\overline{V'_{1}}\circ V'_{1}$ are defined (see Corollary \ref{cor:inv-cob}).

With these cobordisms $V_{1}$, $V'_{1}$ given with the associated perturbation data, the additional
conditions on the cabling $V''$ are as follows:

\begin{itemize}
\item[vi.] The data $\mathcal{D_{V''}}$ of $V''$ restricted to $V'$, viewed from a neighbourhood of $S'$ and to a neighbourhood
of $R'$ coincides with  the data of $V'$, including the markings. Similarly, the data of $V''$
restricted to $V$ viewed from a neighbourhood of $S$ to a neighbourhood of $R$ coincides with the data of $V$.
\item[vii.] The cobordism $V''$ together with its perturbation data restricts to $V'_{1}$ along the segment
$RP$ and to the cobordism $\overline{V'_{1}}$ along $PR'$. Similarly, $V''$ restricts to $
V_{1}$ along $SP$ and to $\overline{V_{1}}$ along $RS'$.

\item[viii.] The marking $c''$ of $V''$ contains the marking $c$ and $c'$ for $V$ and $V'$ as described
at point vi, the markings $\{P\}\times c_{2}$ and $\{P\}\times c'_{2}$ over $P$
together with a set (possibly void) of new points of the form $((\{P_{-}\}\times \{x\}), (\{P_{+}\}\times \{y\}))$
 where $(x,y)\in L\times L'$ such that $j_{L}(x)=j_{L'}(y)$. 

\item[ix.] The pivots of $V''$ are the union of the pivots of $V$ and $V'$ and of $U$ and $T$. Finally, the data for $V''$ is coherent and regular also with respect to the pivots $U,T$.
\end{itemize}

Two cobordisms $V,V'\in \mathcal{L}ag^{\ast}(\C\times M)$ are, by definition, cabling equivalent
if there exists a cobordism $V''$, with a choice of coherent perturbation data as described above, such 
that $V''$ is unobstructed.

\

 \ \\
 \underline{Step 2. Tear-drops and the classes $a_{V}^{W}$.} 
 
 \noindent The purpose of this paragraph is to show that 
 the homology class $[a_{V}^{W}]$ (from Lemma \ref{lem:hlgy-cl}) can be determined by counting
 the tear-drops in Figure \ref{fig:cabling2} with output at $P$ and with one interior marked point that is sent to $T$.  More precisely, let $\mathcal{M}^{T}_{V'', c''}(\emptyset; y)$ be the moduli space of marked tear-drops with one interior marked point sent to $T$ and output at the point $y\in \{P\}\times (L\cap L')$.  Notice that the condition on the interior marked point implies that the marking of the potential
 entries of such a tear-drop are only at points in $c\subset c''$.
 Consider the element $w=\sum_{y}\#_{2} \mathcal{M}^{T}_{V'',c''}(\emptyset; y)y\in CF(L,L'; \mathcal{D}')$ (again the count is only for those moduli spaces of dimension $0$). 
 We want to notice that $w$ is a cycle and that  $[w]$ is identified (in a precise sense to be made  explicit) 
 with $[a_{V}^{W}]$. We will make use of Figure \ref{fig:tear-drop-ident} below.
 
 \begin{figure}[htbp]
   \begin{center}
      \includegraphics[scale=0.2]{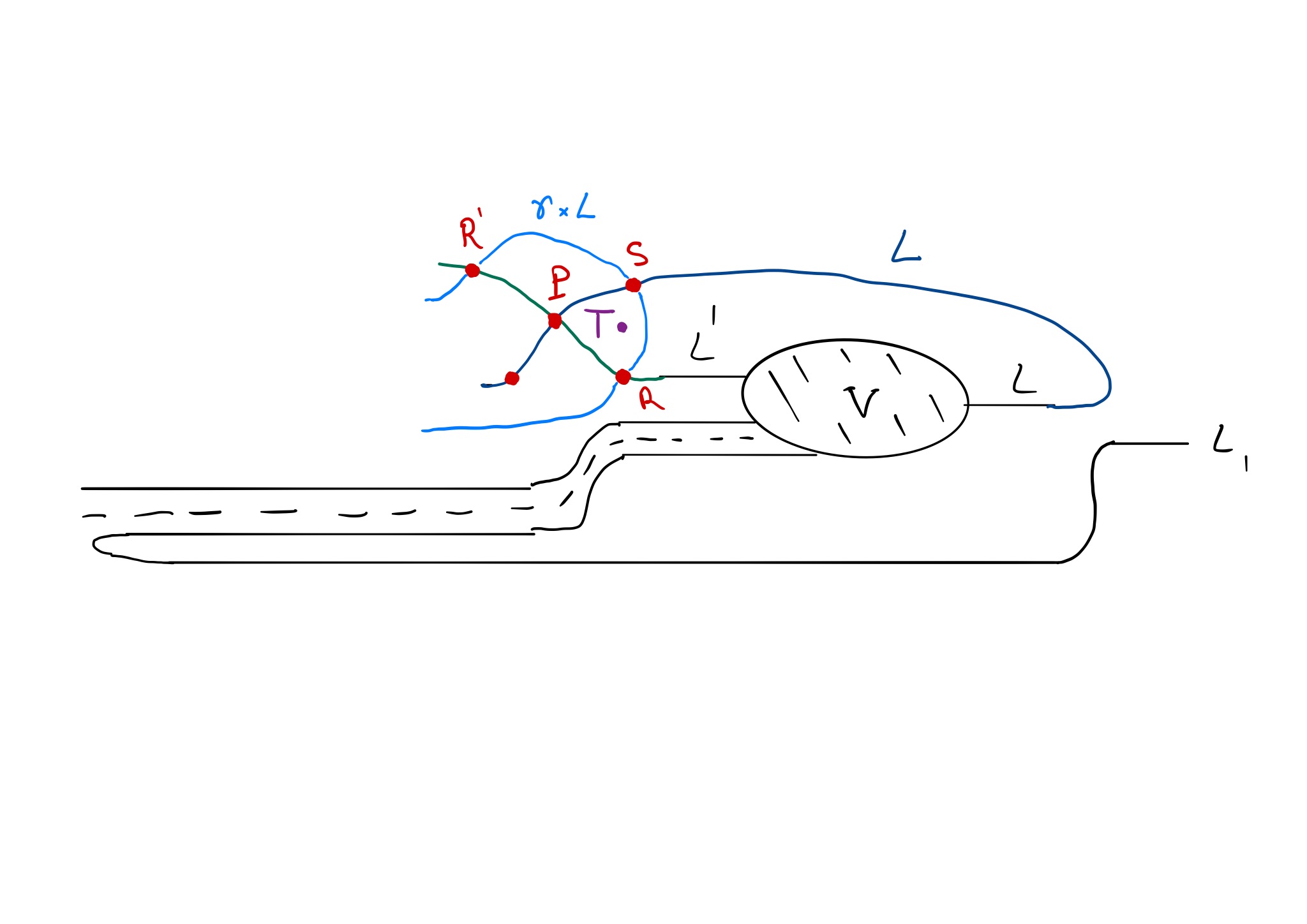}
   \end{center}
   \caption{\label{fig:tear-drop-ident} Comparing tear-drops and $\Theta(V)$ by using a small 
   perturbation $W$ of $\gamma\times L$.}
\end{figure}
 In this figure appears a sufficiently small perturbation $W$ of $\gamma\times L$. This has again 
 three regular bottlenecks at the points $R'$, $S$ and $R$ (as well as two inverted bottlenecks that are no
 longer pictured). By considering the $1$-dimensional moduli spaces of the form $\mathcal{M}^{T}_{V'',c''}(\emptyset; y)$ and their boundary, it is easy to see that $w$ is indeed a cycle. We then consider the $1$-dimensional moduli space of marked Floer strips with one interior marked point (again sent to $T$) that have an input at the point $S$ and with output at $R'$ and with the  $\{0\}\times \R$ part of the boundary along $W$ 
and the $\{1\}\times \R$ part of their boundary along $V''$. From the structure of the boundary of this moduli space we see that there are two maps  $\Psi_{1}= \phi_{1}\circ \phi^{W}_{V}$  and  $\Psi_{2}=\mu^{2}(-, w)$ that are chain homotopic.  Here $\phi_{1}$ is the chain map given by the marked Floer strips
(with one interior marked point that is sent to $T$) that exit the point $R$ and enter in $R'$. The operation $\mu^{2}$ is given by (marked) triangles with entries at $S$ and at $P$ and output at $R'$.  As $W$ is a small deformation
of $\gamma\times L$ and $V''$ is the composition of two inverse cobordisms over the interval $RR'$ it follows
that $\phi_{1}$ is chain homotopic to the identity. We now apply $\Psi_{1}$ and $\Psi_{2}$ to the unit
$[L]\in CF(L,L)$ viewed in the (marked) Floer complex in the fiber over $S$. We obtain $[\mu^{2}([L],w)]=[a_{V}^{W}]$. This means that $\eta_{V''}([a_{V}^{W}])=[w]$ where $\eta$ is the isomorphism between the 
complexes $CF(L,L')_{R}$ and $CF(L,L')_{P}$ induced by the cobordism $V''$ restricted to the segment $RP$, changing the data $\mathcal{D}_{L}$ to $\mathcal{D}'$ (as described in Lemma \ref{lem:data-change}).
 
 Denote now by $w'=\sum_{y}\#_{2} \mathcal{M}^{U}_{V'',c''}(\emptyset; y)y\in CF(L,L'; \mathcal{D}')$
 the cycle defined by counting marked tear-drops with an interior marked point being sent to $U$ (see Figure \ref{fig:cabling2}). 
 It follows, from an argument similar to the one above, that $\eta'_{V''}([a_{V'}^{W}])=[w']$ where $\eta'_{V''}$ is the identification using (the restriction of $V''$ to) the interval $S'P$.

 \ \\

 \underline{Step 3. Conclusion of the proof.} 
 
\noindent  We start by remarking that the construction of the cabling $V''$ at Step 1, 
 together with the coherent, regular perturbation data   $\mathcal{D}_{V''}$ and the marking $c''$
 is possible by the usual scheme. It is useful to note here that, as we are only using  
tear-drops with interior marked points we may pick coherent perturbations inductively, following 
the total valence $||u||$ of the curves $u$ involved. In particular, the cycles $w$, $w'$ are well defined
as in Step 2. Notice that these cycles do not depend on the part of the marking $c''$ given by
self-intersection points of $V''$ that belong to $\{P\}\times (L\cap L')$. We denote this part of the 
marking $c''$ by $c''_{P}$.

 We now want to remark that there is a choice of $c''_{P}$ making $V''$ unobstructed if and only if $[w]=[w']$. This shows the statement
 as it implies, by Step 2,  that $\Theta(V)=\Theta(V')$ if an only if $V$ and $V'$ are cabling equivalent. 
 
 In turn, the cobordism $V''$ (together with the data defined before) is unobstructed if and only if for any self intersection point $y$ the number (mod $2$) of marked tear-drops with output at $y$ is $0$. Given that both $V$ and $V'$ are unobstructed and in view of the presence of the bottleneck at the point $P$, the only points  $y$ where the number of tear-drops might not vanish are the points $y\in \{P\}\times (L\cap L')$. We now consider such a point $y$ and the moduli spaces of tear-drops that have $y$ as an output.  A marked tear-drop
 $u$ with output at $y$ has a projection onto the plane that is holomorphic in a neighbourhood of $P$. As a result this tear-drop could be either completely contained in $\{P\}\times M$, or if its projection $v=\pi\circ u$ onto the plane is not constant, then the image of $v$ is either to the left or to the right of  $P$ since otherwise the image of the curve $v$ would enter one of the unbounded quadrants at the point $P$. 
 Moreover, $v$ is simple in a neighbourhood of $P$. This is seen as follows.  First, because $u$ is a marked teardrop and in view of the action condition on the points belonging to the marking, 
 none of the entry puncture points of $u$
 can belong to $\{P\}\times M$ as otherwise the image of $v$ would again intersect one of the unbounded quadrants at $P$ (if an output lies over $P$, then the curve $v$ will ``get out'' of $P$ through one
 of the unbounded quadrants). Secondly, due to asymptotic convergence at the exit boundary point, the restriction of $v$ to some neighbourhood of the exit puncture is simple and, again because the image of $v$ can not intersect one of the unbounded quadrants at $P$, it follows that $v$ is simple over a possibly smaller neighbourhood of the exit.  Thus, we deduce that $u$ is simple in a neighbourhood of $y$.

 We conclude that either $\pi\circ u$ is constant equal to $P$ or $u$ belongs to one of $\mathcal{M}^{U}_{V'',c''}(\emptyset; y)$ or 
 $\mathcal{M}^{T}_{V'',c''}(\emptyset; y)$.  Denote by $\mathcal{M}_{L\cup L', c'''}(\emptyset;y)$ the 
 set of tear-drops $u$ with a constant projection. Here $c'''$ is the marking 
 $c'''=c_{L}\cup c_{L'}\cup c''_{P}$. Notice that we can write $\#_{2}\mathcal{M}_{L\cup L', c'''}(\emptyset;y)= <d_{L,L'}(c''_{P}),y>$
 where $d_{L,L'}$ is the differential of the Floer complex $CF(L,L')$. Of course, we work here - as before - with the Floer complex with a differential counting marked strips, with the marking $c_{L}$ for $L$ and $c_{L'}$ for $L'$. We conclude therefore that the total number (mod $2$) of marked tear-drops at the point $y$ is
 $n_{y}=<w,y>-<w',y>- <d_{L,L'}(c''_{P}),y>$. Unobstructedness of $V''$ is equivalent to the vanishing 
 of all these numbers and thus to the equality $w-w'= d_{L,L'}c''_{P}$. This means that if $V''$ is unobstructed, then $[w]=[w']$. Conversely, if $[w]=[w']$ we can pick $\eta$ to be such that $d_{L,L'}\eta=w-w'$ and then  the cabling $V''$ defined as above and with $c''_{P}=\eta$ is unobstructed. 
 \end{proof}

\begin{rem} a. The fact that we work over $\Z/2$ is important in these calculations. It is clear that to work 
over $\Z$ one needs not only to fix spin structures (so that the various moduli spaces are oriented) but also 
to allow $0$-size surgery with weights (essentially allowing integral multiples of the same intersection 
point). 
\end{rem}

Proposition \ref{prop:theta-cabl} together with the other properties of the functor $\Theta$ from \S\ref{subsubsec:funct-th} show that cabling is indeed an equivalence relation and that Axioms 1, 2, 3 
from \S\ref{subsec:surg-models} are satisfied.  

\subsection{From algebra to geometry.}\label{subsec:alg-to-geo}
\subsubsection{Surgery over cycles and geometrization of module morphisms.}\label{subsubsec:surg-mor}

The main purpose of this subsection is to show that Axiom 4 from \S\ref{subsec:surg-models} is satisfied in our 
context. In other words, we need to show that for any morphism $V\in \mor_{\mathsf{C}ob^{\ast}}(L,L')$
there exists a surgery morphism $S_{V}\in \mathcal{L}ag^{\ast}(M)$ that is cabling equivalent to $V$.
By Proposition \ref{prop:theta-cabl} this is equivalent to finding $S_{V}$ such that $\Theta(S_{V})=\Theta(V)$. 
This property is a consequence of the next result.

\begin{prop}\label{prop:surgery-mor}
Let $L,L'\in \mathcal{L}ag^{\ast}(M)$ and let $a\in CF(L,L')$ be a cycle. Then there exists a marked
cobordism $S_{L,L';a}:L\cobto L'$, $S_{L,L';a}\in \mathcal{L}ag^{\ast}_{0}(\C\times M)$
such that $$\Theta(S_{L,L'; a})=[a]\in HF(L,L')~.~$$
\end{prop}

\begin{proof} It is not hard to construct a cobordism with double points containing $1$-dimensional
clean intersections, that represents the trace of a $0$-size surgery: we can simply put 
 $S_{L,L'; a}=\gamma\times L\bigcup \gamma'\times L'$  where $\gamma$, $\gamma'$ are the curves
 from Figure \ref{fig:surgery-curves} (see also \S\ref{subsec:surg}) this cobordism has three ends 
 (indicated here with their markings): $(L,c)$, $(L',c')$ and $(L\cup L', c\cup c' \cup a)$.
 The difficulty is to construct a deformation $S_{a}$ of this $S_{L,L'; a}$ together with the associated
 marking of $S_{a}$ so that this marking restricts to the markings of the ends (as in Definition \ref{def:marked-cob}) and such that $S_{a}$ is unobstructed,
 so that, together with the relevant perturbation data, we get an element in  $\mathcal{L}ag^{\ast}_{0}(\C\times M)$.

We start with the two Lagrangians $L$ and $L'$ viewed as elements in $\mathcal{L}ag^{\ast}(M)$ and thus endowed with the additional data $L=(L, c, \mathcal{D}_{L}, h_{L})$, $L'=(L', c',\mathcal{D}_{L'}, h_{L'})$. Additionally, the primitives $f_{L}$ and $f_{L'}$ are also fixed.  In view of Corollary \ref{cor:inv-cob} and Lemma \ref{lem:data-change} we may assume that $L$ and $L'$ are intersecting
transversely at points different from the self-intersection points of both $L$ and $L'$ and that 
there is a choice of coherent data $\mathcal{D}''$ associated to the immersed Lagrangian $L\cup L'$ that extends both $\mathcal{D}_{L}$ and $\mathcal{D}_{L'}$. To fix ideas, we indicate from the start
that the cobordism that we will construct will consist of the union of two sufficiently small deformations $V$
of $\gamma\times L$ and $V'$ of $\gamma'\times L'$, with $\gamma$ and $\gamma'$ two curves to be described below in Figure \ref{fig:surgery-curves}, together with additional marking points associated to $a$. 
In other words $S_{a}$ is a Lagrangian immersion with domain the disjoint union $V\coprod V'$ and 
$j_{S_{a}}=j_{V}\cup j_{V'}$. The primitive on $S_{a}$ restricts to $f_{V}$ on the $V$ factor
and to $f_{V'}$ on the $V'$ factor.

Notice that the curves $\gamma$, $\gamma'$ themselves are tangent at the point $P$
and they coincide along the bottom leg in the picture, to the left of $P$. The points $Q$, $P$, $R$ and $T$
are regular bottlenecks and there are inverted bottlenecks $X,Y,Z$. 
The primitives $f_{V}$ and $f_{V'}$ restrict to the primitives $f_{L}$ at $Q$
and $f_{L'}$ at $R$ and, by possibly having $\gamma$ make an ``upper bump'' in between $Q$ and $P$ we may
assume that the primitive $f_{V}$ restricts at $P$ along $L\subset \{P\}\times M$ to a primitive on $L$,
$f_{V}|_{L,P}$, such that $\min(f_{V}|_{L,P}) >\max (f_{V'}|_{L',P})$ where $f_{V'}|_{L',P}$ is 
the corresponding restriction of $f_{V'}$.

There are four ``regimes'' governing the perturbations providing $S_{a}$ and we describe them now.
Denote by $D_{i}\subset L$ small disks around the self-intersection points of $L$, let $D'_{i}\subset L'$ be small 
disks around the self-intersections points of $L'$ (in each case we view an intersection point as a pair $(P_{-},P_{+})\in I_{L}\subset L\times L$ and similarly for $L'$). Let also $K_{i}$, $K'_{i}$ be small disks
around the intersection points of $x_{i}\in j_{L}(L)\cap j_{L'} (L')$ with $K_{i}\subset L$, $K'_{i}\subset L'$. 
The first regime is non-perturbed:  one portion of $S_{a}$, $(S_{a})_{1}$
is given as: $$(S_{a})_{1}=(\gamma\times (L\setminus\{\cup_{i} K_{i}\bigcup\cup_{j}D_{j}\}))\bigcup 
(\gamma'\times (L'\setminus\{\cup_{i} K'_{i}\bigcup\cup_{k} D'_{k}\}))~.~$$
The second portion $(S_{a})_{2}$ contains a perturbation around the self intersection points of $L$
 in the sense of the small perturbation of the cobordism 
 $\gamma\times L$ as in \S\ref{subsubsec:smallpert}. In other words, if $D_{i}, D_{j}$ are two
 disks with $P_{-}\in D_{i}$, $P_{+}\in D_{j}$, $(P_{-},P_{+}) \in I^{<0}_{L}$, then
 $$(S_{a})_{2}=\gamma_{-}\times D_{i}\bigcup \gamma_{+}\times D_{j}~.~$$
 Similarly, there is another portion $(S_{a})'_{2}$
 that corresponds to a similar small perturbation of $\gamma'\times L'$.
 
 Finally, there is yet one other portion, $(S_{a})_{3}$, which is given in terms of the disks $K_{i}$, $K'_{i}$
 by:
 $$(S_{a})_{3}= \gamma_{-}\times K_{i}\bigcup \gamma'_{+}\times K'_{i} ~.~$$
\begin{figure}[htbp]
   \begin{center}
      \includegraphics[scale=0.2]{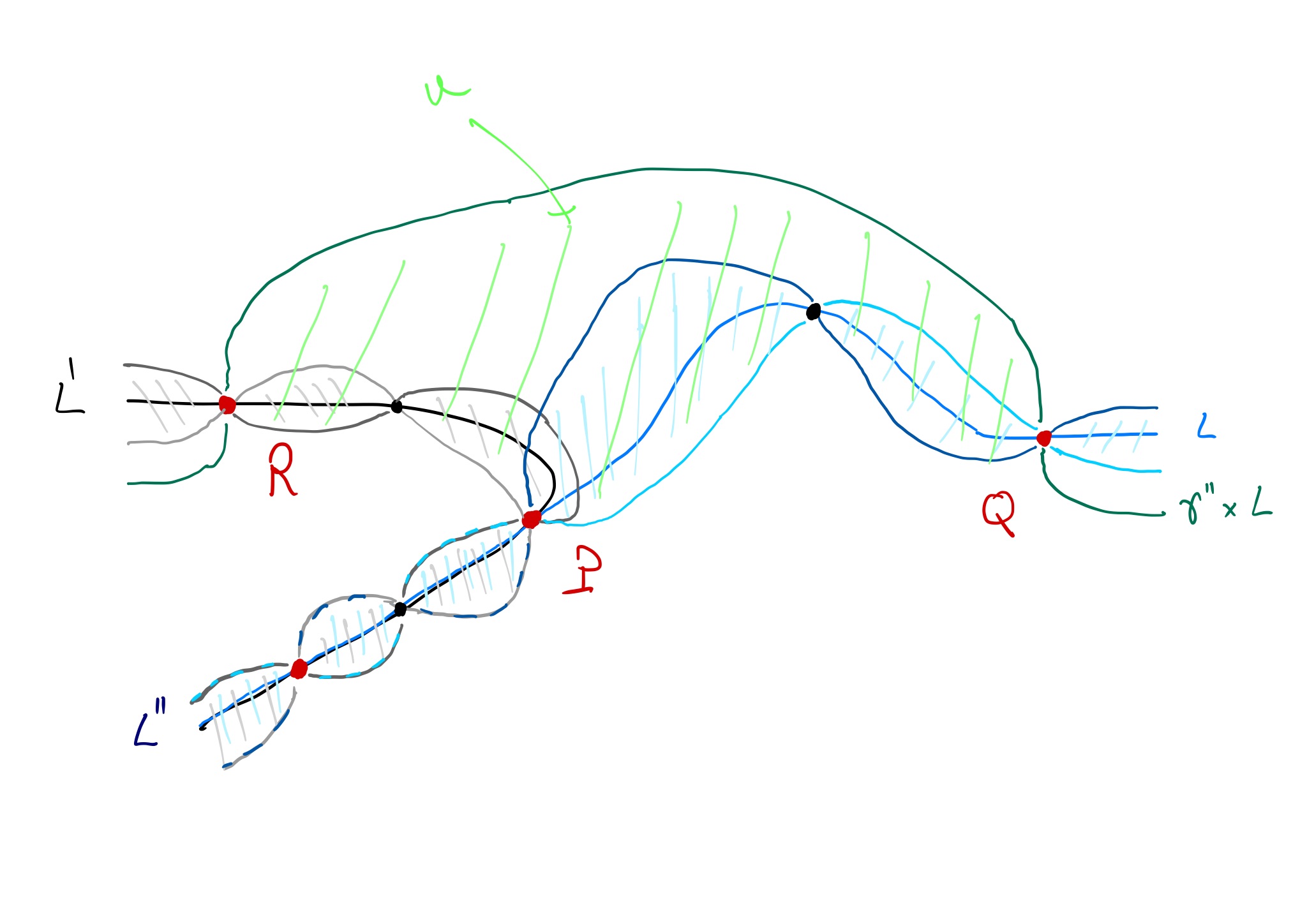}
   \end{center}
   \caption{\label{fig:surgery-curves} The two cobordisms $V$ - in blue - and $V'$ - in grey - together with the
   (regular) bottlenecks $P,Q,T,R$.}
\end{figure}
These four parts of $S_{a}$ are glued together by using perturbations on slighter larger disks $\tilde{D}_{i}$
such that $D_{i}\subset \tilde{D}_{i}$ (and similarly for $D'_{j}$, $K_{r}$ and $K'_{r}$). This is possible because the curves
$\gamma_{-}$, $\gamma_{+}$ are Hamiltonian isotopic in the plane to $\gamma$ (and similarly for the $(-)'$
curves). This Hamiltonian isotopy can be lifted to $\C\times M$ and the gluing uses an interpolation between
this Hamiltonian defined on  $\gamma \times D_{i}$  and the trivial one defined on exterior of 
$\gamma\times \tilde{D}_{i}$. 

The next point is to discuss the marking of $S_{a}$ and the relevant data $\mathcal{D}_{S_{a}}$.
The markings appear at the points $R$ where it is given by $c'$, at $Q$ where it is given by $c$, 
at $P$ where it is given by $c\cup c'\cup a$ and at $T$ where it is again $c\cup c'\cup a$. 
Notice that this makes sense as $a$ consists 
of pairs $(P_{-},P_{+})\in L\times L'$ such that $j_{L}(P_{-})=j_{L'}(P_{+})$. It follows that for the pairs as above we have $f_{S_{a}}(P_{-})>f_{S_{a}}(P_{+})$ (when all the relevant perturbations are sufficiently small). The data $\mathcal{D}_{S_{a}}$ restricts to the data $\mathcal{D}''$ in neighbourhoods of the four
regular bottlenecks $Q,R, P, T$. Away from the bottlenecks the data is a small deformation of $i\times \mathcal{D}''$. The deformations giving $V$, $V'$ as well as $\mathcal{D}_{S_{a}}$ are taken sufficiently
small such that with this data both $V$ and $V'$ are unobstructed. Notice that $S_{a}$ is a deformation of the immersed cobordism $S_{L,L';a}$, as desired. 

We now want to show that, after possibly diminishing these perturbations even more, 
 $L''=(L\cup L', c\cup c'\cup a, \mathcal{D}'')$ (which is viewed as the end over the bottleneck $T$) is unobstructed and that the cobordism $S_{a}$ is also unobstructed. We start with $L''$. Because both $L$ and $L'$ are unobstructed, it is immediate to see
 that the only marked tear-drops with boundary on $L''$ and that are not purely on $L$ or on $L'$,
 have to have some marked inputs $\in a$. Moreover, because the markings are action negative and because
 $\min(f_{V}|_{P,L})> \max(f_{V'}|_{P,L'})$  we also have the same type of inequality at $T$.
 This implies that a marked tear-drop with boundary on $L''$ can have at most one marked point  $\in a$ 
 as an input (to have more such points, the boundary of the tear-drop should cross back from $L'$ to $L$
 at some other input but that input can not belong to $a$ due to the action inequality). 
 In this case, obviously, the output is some intersection point $y\in L\cap L'$. We now consider the moduli
 space of such teardrops at $y$, $\mathcal{M}_{c\cup c'\cup a; L\cup L',}(\emptyset, y)$. It is immediate
 to see that this moduli space can be written as a union $\cup_{x\in a}\mathcal{M}_{\mathbf{c};L,L'}(x,y)$ where $\mathbf{c}$ are precisely the markings $c$ for $L$ and $c'$ for $L'$.
 In other words, $\mathcal{M}_{\mathbf{c};L,L'}(x,y)$ is the moduli space of marked Floer strips with boundaries on $L$ and $L'$, originating at $x$ and ending at $y$. Considering now only the $0$-dimensional
 such moduli spaces and recalling that the data $\mathcal{D}''$ ensures regularity, we 
 have $$\#_{2}\mathcal{M}_{c\cup c'\cup a; L\cup L'}(\emptyset, y)=\sum_{x\in a}\#_{2}\mathcal{M}_{\mathbf{c};L,L'}(x,y)=<da,y>=0$$ because, by hypothesis, $a$ is a cycle (here $d:CF(L,L')\to CF(L,L')$  is the (marked) Floer differential). Therefore, we conclude that  $L''$ is unobstructed.
 
 We pursue by analyzing $S_{a}$. We again start by using the fact that $V$ and $V'$ are each unobstructed.
 Therefore, the only type of tear-drops with boundary on $S_{a}$ that
  we need to worry about are those having a marked input $\in \{P\}\times a$ or  $\in \{T\}\times a$
  and output in $\{Y\}\times L''$.
 But we may assume that the cobordism $S_{a}$ restricted to the region in between $P$ and $T$ is in
 fact a sufficiently small deformation of $\gamma_{P,T}\times L''$ where $\gamma_{P,T}$ is the part
 of $\gamma $ in between $P$ and $T$ (it agrees there with $\gamma'$). This implies that that this region
 of $S_{a}$ is also unobstructed (by Lemma \ref{lem:ident}) and shows that $S_{a}$ itself is unobstructed.
 
 To conclude the proof of the proposition we still need to show that $\Theta(S_{a})= [a]$. 
 For this argument we will use Figure \ref{fig:surgery-curves2} and notice that, by considering a sufficiently small deformation of $\gamma''\times L$ (as in the figure) and the corresponding data, the morphism 
 $\phi_{S_{a}}^{W}$ is given by counting marked strips from $Q$ to $R$ in Figure \ref{fig:surgery-curves2}. 
\begin{figure}[htbp]
   \begin{center}
      \includegraphics[scale=0.17]{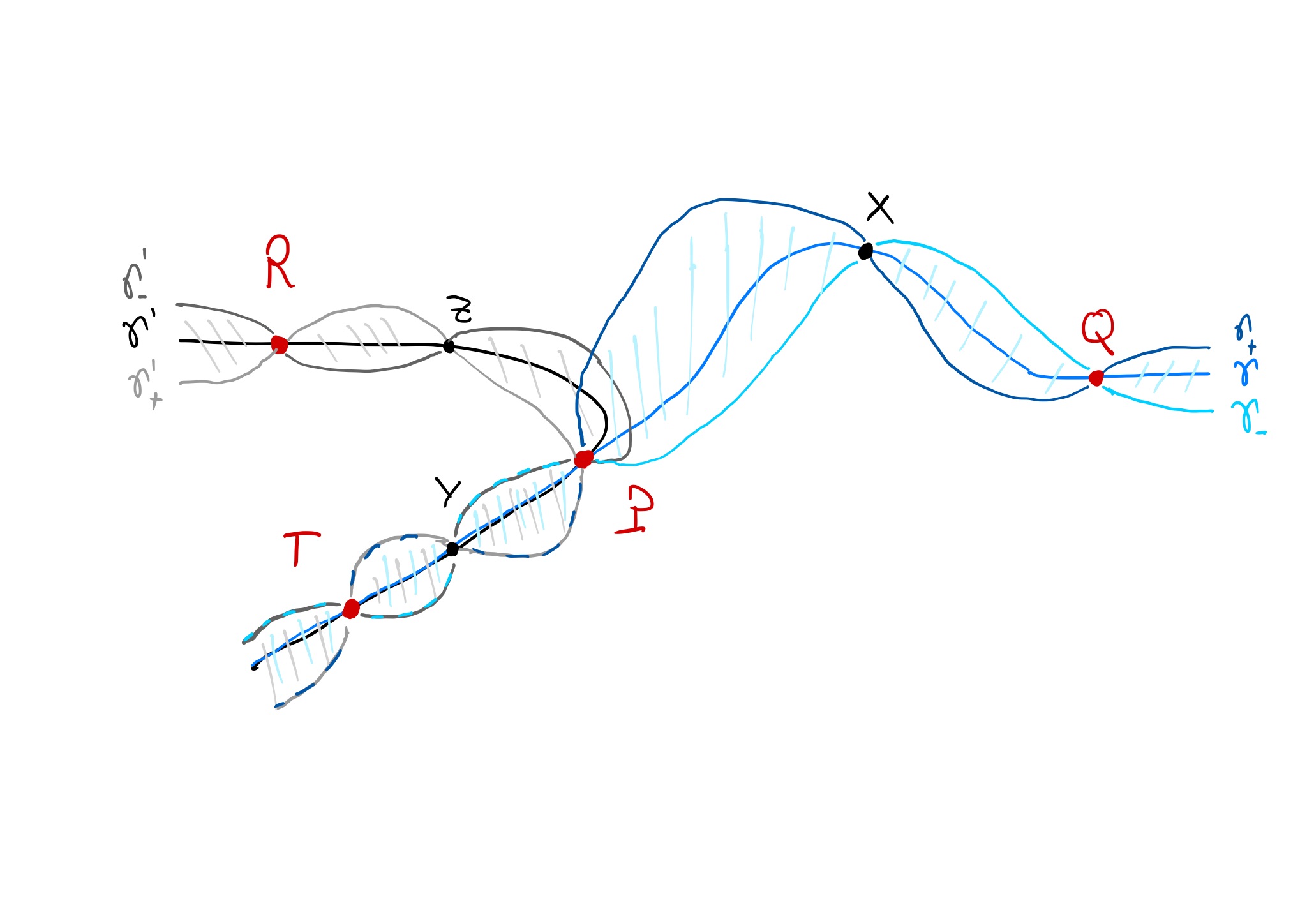}
   \end{center}
   \caption{\label{fig:surgery-curves2} The cobordisms $S_{a}$ together with the schematic representation of
   as sufficiently small deformation $W$ of $\gamma''\times L$ and a triangle $u$.}
\end{figure}
There has to be only one such marking passing from $V$ to $V'$ and the only possibility for such a
 marking is at the point $P$. Thus this marking has to belong to $\{P\}\times \{a\}$. In other words,
 we have $\phi_{S_{a}}^{W}([L])=\mu_{2}([L],a)$ where $\mu_{2}$ is the multiplication for the following
 three Lagrangian cobordisms $(W, V, V')$ given by counting triangles such a $u$ in the Figure. 
It is not difficult to see that $[L]$ is the unit (in homology) for this multiplication and thus $\Theta (S_{a})=[a]$ which concludes the proof.

\end{proof}

\subsubsection{$\mathsf{C}ob^{\ast}(M)$ has surgery models.}\label{subsubsec:surg-model}
In this subsection we intend to finish showing that the category $\mathsf{C}ob^{\ast}(M)$
satisfies the Axioms from \S\ref{subsec:surg-models}. Given that we already
constructed the functor $\Theta$ from Theorem \ref{thm:surg-models} and that we showed in Proposition
\ref{prop:theta-cabl} that $V$ and $V'$ are cabbling equivalent iff $\Theta(V)=\Theta(V')$, 
we already know that the quotient functor $\widehat{\Theta}$ is well defined and injective.
It follows that Axioms 1 and 2 are satisfied.  We showed in \S\ref{subsubsec:surg-mor} that Axiom
4 is also satisfied.

It remains show that Axiom 3, which was
verified for elements in $\mathcal{L}ag^{\ast}_{0}(\C\times M)$ in Corollary \ref{cor:inv-cob},
 is also true for  general $V$'s $\in \mathcal{L}ag^{\ast}(\C\times M)$.
 Moreover, we need to show that the functor $\widehat{\Theta}$ is triangulated. 
Indeed, if this is so, Axiom 5  - in the form made
explicit in Remark \ref{rem:axiom5} b - is also satisfied because the relevant naturality properties 
are already satisfied in the triangulated category  $H(mod(\fuk^{\ast}(M)))$ which is the image of the injective functor $\widehat{\Theta}$.  

Therefore, taken together, these two claims imply that $\mathsf{C}ob^{\ast}(M)$ has surgery models.

\begin{lem}\label{lem:inv-cob2} If $V$ is a simple cobordism $V \in \mathcal{L}ag^{\ast}(\C\times M)$,
possibly carrying (non marked) tear-drops, then $\Theta(V)\circ \Theta (\bar{V})=id$.
\end{lem}
\begin{proof} Given that $V$ may carry tear-drops, we need to deal with the possible presence of pivots.
The difficulty that these create is that the argument used in Corollary \ref{cor:inv-cob}, which is based on 
a Hamiltonian isotopy induced by a translation in the plane, is no longer directly applicable as the image of 
the moving cobordism intersects the pivots.
The argument makes use of Figure \ref{fig:simple-cob} below. The ends of the cobordism $V$ are $L$ and $L'$. We consider a Lagrangian $W\subset \C\times M$, $W=\gamma\times N$, where $\gamma$ is a circle intersecting the cylindrical region of $V$  at the point $P$ and $Q$ and  $N$ is in $\mathcal{L}ag^{\ast}(M)$ in generic position relative to $L$ and $L'$. For the moment, we assume $N$ is embedded.  The points $P$ and $Q$ are regular bottlenecks for $V$.
\begin{figure}[htbp]
   \begin{center}
      \includegraphics[scale=0.65]{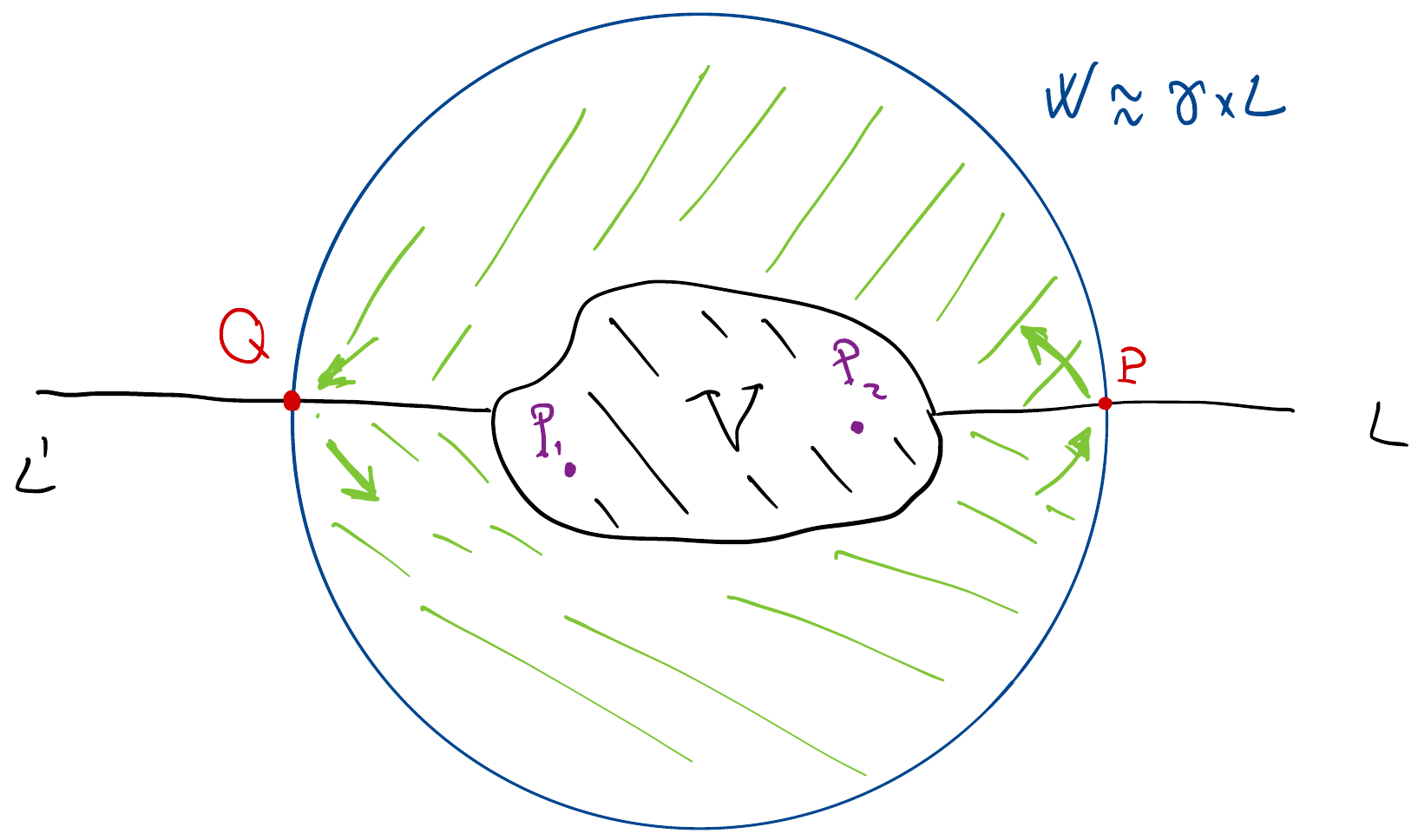}
   \end{center}
   \caption{\label{fig:simple-cob} Schematic representation of the cobordism $V$ of the Lagrangian $W=\gamma\times N$; $P_{1}$ and $P_{2}$ are pivots for $V$.}
\end{figure}
Following the definition of $\phi^{W}_{V}$ (see equation (\ref{eq:theta})) we see that this morphism is given by counting (marked strips) going from points $\{P\}\times \{x\}$ to points $\{Q\}\times \{y\}$ where $x\in N\cap L$, $y\in N\cap L'$. Similarly, $\phi^{W}_{\bar{V}}$ is given by counting the strips going from $Q$ to $P$. We now consider
the composition $\phi_{V}\circ \phi_{\bar{V}}$ and view it as providing a count of the points in a part of the boundary of the moduli space of marked strips going from $Q$ to $Q$. Another such part 
consists of disks of Maslov class $2$ with boundary on $W$, going through the points in $\{Q\}\times (N\cap L')$. If the deformations giving $W$ out of $\gamma\times N$ are
sufficiently small, we see that there is precisely one such disk for each point  in $\{Q\}\times (N\cap L')$ and we conclude that $\phi_{V}\circ \phi_{\bar{V}}$ is chain homotopic to the identity. 
The same argument actually shows that the composition $\phi_{V}\circ \phi_{\bar{V}}$ is chain homotopic to the identity as an $A_{\infty}$-module map and concludes the proof (alternatively,
we apply the same geometric argument as before to the case when $N$ is no longer embedded
and $W$ is a small enough deformation of $\gamma\times N$).
 \end{proof}

\begin{lem}\label{lem:theta-triang} The functor $\widehat{\Theta}$ is triangulated.
\end{lem}
\begin{proof}
Recall from \S\ref{subsubsec:tri} and Figure \ref{Fig:Extr} the definition of the distinguished triangles in 
$\mathsf{C}ob^{\ast}(M)$. It is easy to see that the distinguished triangle associated to a
 surgery cobordism $S_{a}$, such as constructed in \S\ref{subsubsec:surg-mor}, is mapped by $\Theta$ to 
 an exact triangle in the homological category of $A_{\infty}$-modules over $\fuk^{\ast}(M)$ (a simple 
proof in this case is perfectly similar to the embedded case in \cite{Bi-Co:lcob-fuk}).  However, we need to deal with the case of $V:L\cobto (L'',L')$, $V\in \mathcal{L}ag^{\ast}(\C\times M)$, when $V$ possibly carries (non-marked) tear-drops.  To do this 
we provide a different proof, along an approach similar to the one in Lemma \ref{lem:inv-cob2}. 
We consider Figure \ref{fig:triangle-disk} below (compare with Figure \ref{fig:simple-cob}). There 
are now three regular bottlenecks $P,Q,R$, each corresponding to one of the ends of $V$. The curve
$\gamma$ is a circle, as in Figure \ref{fig:simple-cob}. We consider a Lagrangian $W=\gamma\times N$
where $N$ is embedded, in generic position relative to all three of $L,L',L''$. 

\begin{figure}[htbp]
   \begin{center}
      \includegraphics[scale=0.65]{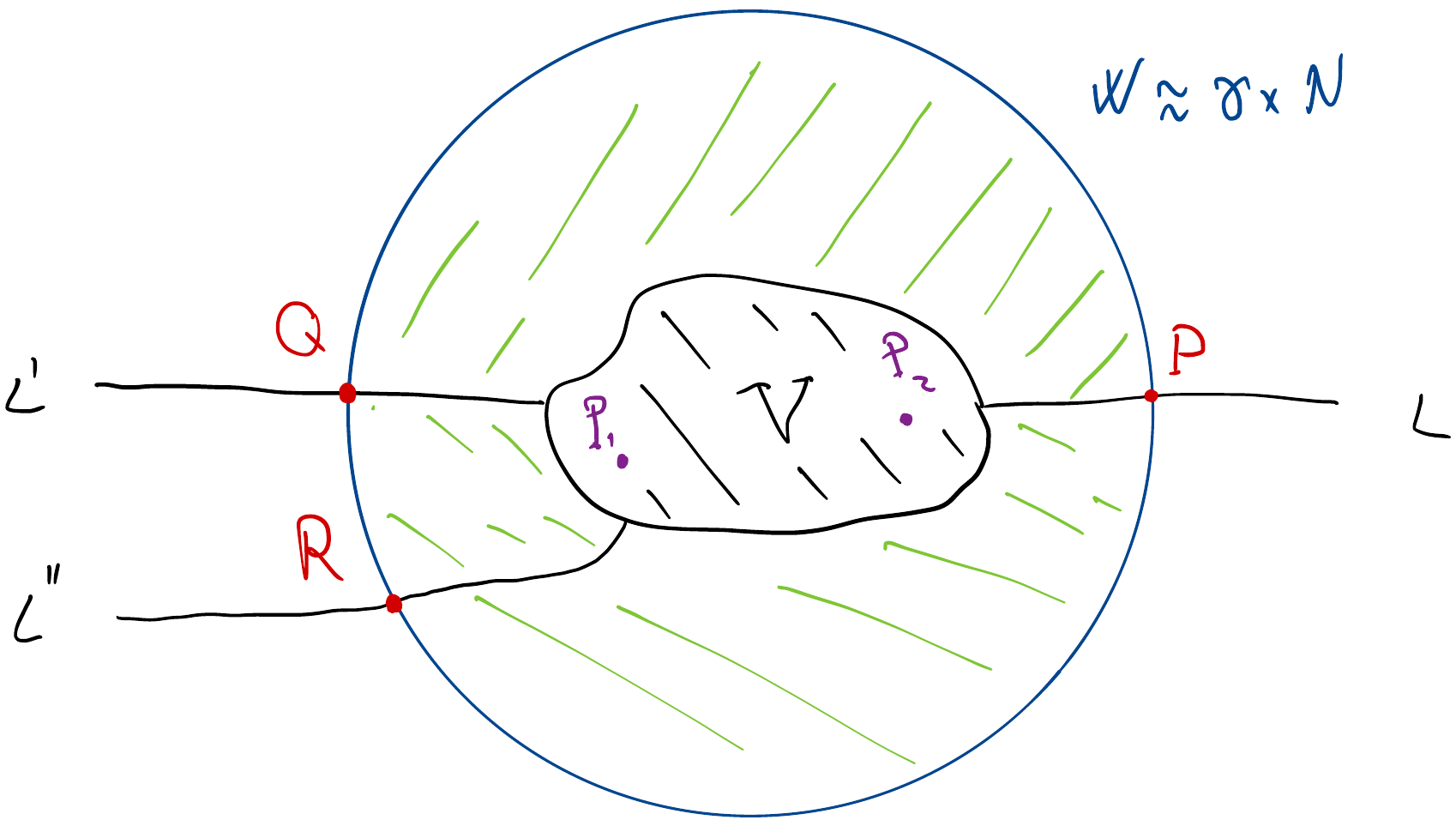}
   \end{center}
   \caption{\label{fig:triangle-disk} Schematic representation of the cobordism $V$ with three ends $L$, $L'$, $L''$, and $W=\gamma\times N$; $P_{1}$ and $P_{2}$ are pivots for $V$.}
\end{figure}

We again consider the (marked) Floer complex $C_{P}$ generated by the points $\in \{P\}\times (N\cap L)$
and similarly $C_{Q}$, $C_{R}$ for the other ends of $V$. There is a chain map that we denote
by $\psi^{P}_{Q}:C_{P}\to C_{Q}$ (and, similarly, $\psi^{Q}_{R}$, $\psi^{R}_{P}$) defined 
by counting Floer (marked) strips going from intersection points over $P$ to intersection points over $Q$. These three maps are respectively identified with the three  maps associated through the construction of $\Theta$ to the three maps in the distinguished triangle corresponding to $V$, as given in Figure \ref{Fig:Extr}. Thus we need to show that these three maps fit into an exact triangle of chain complexes (and, by extension, of 
$A_{\infty}$-modules). For this purpose, denote by $C_{P,Q}$ the cone of the map $\psi^{P}_{Q}$:
it is given as vector space by $C_{P}\oplus C_{Q}$ and has
 the differential $D_{P,Q}=(d_{P}+\psi^ {P}_{Q}, d_{Q})$.
There is a chain map $\psi^{P,Q}_{R}:C_{P,Q}\to C_{R}$ defined by counting strips from intersection points
over either $P$ or $Q$ to intersection points over $R$. Moreover, there is a chain map $\psi^{R}_{P,Q}:C_{R}\to C_{P,Q}$ given by counting strips from intersection points over $R$ to intersection points that are either 
over $P$ or $Q$. As in the proof of Lemma \ref{lem:inv-cob2} these two maps are homotopy inverses.
We notice that the projection $C_{P,Q}\to C_{P}$ composed with $\psi^{R}_{P,Q}$ agrees with $\psi^{R}_{P}$ and that $\psi^{P,Q}_{R}$ composed with the inclusion $C_{Q}\to C_{P,Q}$ 
 agrees with $\psi^{Q}_{R}$ which shows that the three maps $\psi^{P}_{Q}$,
  $\psi^{Q}_{R}$, $\psi^{R}_{P}$ fit into an exact triangle in the homotopy category.

 \end{proof}

\subsection{Proof of Theorem \ref{thm:surg-models}}\label{subsec:proof-surg-mod}
Most of Theorem \ref{thm:surg-models} has been already established. We need to only 
discuss here the metric statements in this theorem. 

We start by defining the shadow of an element $(V,c,\mathcal{D}_{V},h)\in \mathcal{L}ag^{\ast}(\C\times M)$. Recall that $V_{h}$ has regular bottlenecks along each of its ends and let $\hat{V}_{h}$ be
the part of $V_{h}$ that is left after removing the regions to the left of the positive
bottlenecks and to the right of the negative bottlenecks, as in Figure \ref{fig:reduce-bottle} (compare with
the cobordism $V_{h}$ from Figure \ref{fig:cob-id}). 
The shadow of $(V,c,\mathcal{D}_{V},h)$ is by definition, the shadow of
$\hat{V}_{h}$, in the sense of \S\ref{subsec:shadow}. 

\begin{figure}[htbp]
   \begin{center}
      \includegraphics[scale=0.78]{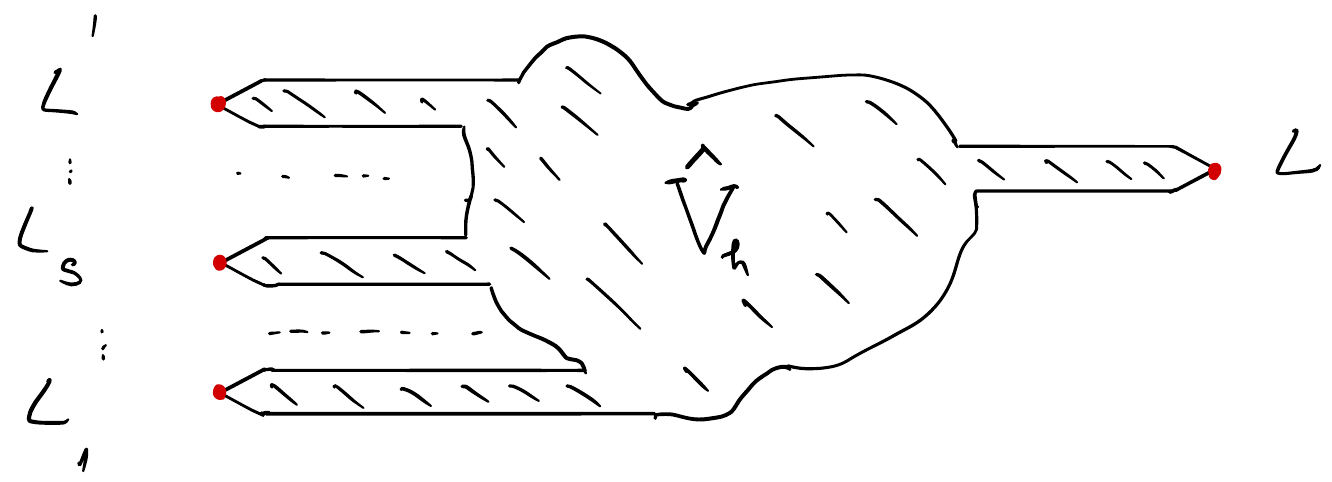}
   \end{center}
   \caption{\label{fig:reduce-bottle}The ``truncated'' cobordism $\hat{V}_{h}$ used to define the sahdow of $(V,c,\mathcal{D}_{V},h)$.}
\end{figure}

Notice that the band width parameter, $\epsilon$, appearing in the definition of the perturbation germs 
$h_{L}$ associated to an element  $L\in \mathcal{L}ag^{\ast}(M)$, as discussed in \S\ref{subsubsec:dec-bis},
can be taken arbitrarily small. Making use of this
and noticing also that in the construction of the surgery morphisms  in \S\ref{subsec:alg-to-geo} we may take 
the various relevant deformations with an arbitrary small area covered by their shadows - this applies to $h_{L},h_{L'}$ as well as the deformations giving the cobordisms $V$ and $V'$ - we deduce that 
the category $\mathsf{C}ob^{\ast}(M)$ has small surgery models in the sense of \S\ref{subsec:shadow}. 

The next step is to see that $\mathsf{C}ob^{\ast}(M)$ satisfies the rigidity Axiom 6 from \S\ref{subsec:shadow}. 
We will prove the stronger fact that $\mathsf{C}ob^{\ast}(M)$ is strongly rigid.

\begin{lem}\label{cor:rig}
The category $\mathsf{C}ob^{\ast}(M)$ is strongly rigid (see Definition \ref{def:rig}).
\end{lem}
\begin{proof}
The first part of the proof is to show that $\mathsf{C}ob^{\ast}(M)$ is rigid.
This means that we need to show that for any two different Lagrangians $L,L'\in \mathcal{L}ag^{\ast}(M)$ there exists a constant $\delta(L,L')>0$ such that for any simple cobordism 
$V:L\cobto L'$, $V\in\mathcal{L}ag^{\ast}(\C\times M)$  we have $\mathcal{S}(V)\geq \delta(L,L')$.

Recall that for two Lagrangians $L$ and $L'$, the Gromov width of $L$ relative to $L'$ is defined by 
$$w(L;L')=\sup\{ \frac{\pi r^{2}}{2} \ |\ \exists \ e: (B_{r},\omega_{o})\hookrightarrow (M,\omega)\ , \ e(B_{r})\cap L'=\emptyset \ , \ e^{-1}(L)=\R B_{r}\}$$ where $e$ is a symplectic embedding, 
$(B_{r},\omega_{0})$ is the standard ball of radius $r$ and $\R B_{r}=B_{r}\cap \R^{n}\times \{0\}$. 

Assume now that $V:L\cobto L'$, $V\in \mathcal{L}ag^{\ast}_{0}(\C\times M)$
 is a simple cobordism. Proceeding as in
\cite{Co-She:metric} we deduce that through the center of a ball $e(B_{r})$  as in the definition 
of $w(L;L')$, passes a perturbed marked $J_{0}$-holomorphic strip $u$ with boundary along $L$ and on $L'$  and of energy not bigger than $\mathcal{S}(V)$. 

The energy of $u$ is bounded from below by the energy of $Image(u)\cap e(B_{r})$.
By adjusting the data associated to the pair $L,L'$, there exists a constant $c_{L,L'}>0$ 
depending of the perturbations $\mathcal{D}_{L}$ and $\mathcal{D}_{L'}$ (and thus also of 
the base almost complex structure $J_{0}$) such that  the energy of $Image(u)\cap e(B_{r})$ is bounded from below by   $c_{L,L'}\frac{\pi r^{2}}{2}$ (if the almost complex structure used inside $B_{r}$  could be assumed to be the standard one and the perturbations vanish there, the energy in question would agree with
the symplectic area and we could take $c_{L,L'}=1$).

In summary, we have  
\begin{equation}\label{eq:ineq-width}\mathcal{S}(V)\geq c_{L,L'}w(L;L')
\end{equation} 
which implies that Axiom 6 is satisfied at least for simple cobordisms that do not carry non-marked tear-drops.
In case $V\in \mathcal{L}ag^{\ast}(\C\times M)$, the argument in \cite{Co-She:metric} no longer applies
because the argument there is based on a Hamiltonian isotopy such as the one appearing in Lemma \ref{lem:ham-isot} and this isotopy potentially intersects the poles associated to $V$. 
However, a quantitative adjustment of the argument in Lemma \ref{lem:inv-cob2} does imply 
the statement even in this case.

Consider now a cobordism
$V:L\cobto (F_{1},\ldots,F_{i}, L', F_{i+1}, F_{m})$, $V \in\mathcal{L}ag^{\ast}(\C\times M)$. Using the braiding in Remark \ref{rem:var-braiding}, we may replace it by a cobordism $V':L\cobto L''$ where 
$L''$ is obtained by iterated $0$-size surgeries involving the $F_{i}$'s and $L'$. Because $\mathsf{C}ob^{\ast}(M)$ has 
small surgergy models, the cabling used in the braiding can be assumed to add to the shadow of the initial
cobordism $V$ as little as wanted. In short, for any $\epsilon >0$ we can find such a $V'$ so  that 
$\mathcal{S}(V')\leq \mathcal{S}(V)+\epsilon$. Assume now that there are families $\mathcal{F}$ and $\mathcal{F}'$ with $\overline{\bigcup_{F\in\mathcal{F}}  F}\cap \overline{\bigcup_{F'\in\mathcal{F}'}  F'}$ totally discrete and $L\not=L'$ such that $d^{\mathcal{F},\mathcal{F}'}(L,L')=0$.
Focusing for now just on the family $\mathcal{F}$, it follows from the arguments 
 above that, for any $\epsilon$, there is a  cobordism $V'_{\epsilon}:L\cobto L''_{\epsilon}$
of shadow less than $\epsilon$ with $L''_{\epsilon}=F_{m}\#_{c_{m}}\ldots \#L'\#\ldots \# F_{1}$.
Consider now a point $x\in L\backslash L'$. As there exists a sufficiently small standard symplectic ball around $x$ 
(as in the definition of width) that also does not intersect $L'$, we deduce from (\ref{eq:ineq-width}) 
that $x\in \overline \bigcup_{F\in \mathcal{F}}F$. There is a small open set $U\subset (L\backslash L')$ containing $x$ and thus we deduce that all of $U\subset  \overline \bigcup_{F\in \mathcal{F}}F$. We apply the same argument to $\mathcal{F}'$ and we arrive at a contradiction because of our assumption on $\mathcal{F}$ and $\mathcal{F}'$.
\end{proof}

To conclude the proof of Theorem \ref{thm:surg-models} there are two more statements to justify.
The most important part is that the restriction $\widehat{\Theta}_{e}$ of the 
morphism $\widehat{\Theta}$ to the triangulated 
subcategory $\widehat{\mathsf{C}}ob^{\ast}_{e}(M)$ generated by the embedded Lagrangians 
gives a triangulated isomorphism to $D\fuk^{\ast}(M)$.   The functor
$\widehat{\Theta}$ is triangulated and all modules in $D\fuk^{\ast}(M)$ are iterated cones of Yoneda modules of embedded objects.  Therefore, by using Proposition \ref{prop:surgery-mor}, 
recurrence on the number of cones and equation (\ref{eq:yoneda-imm}) the result follows. 
Notice also that it also follows from Proposition \ref{prop:surgery-mor} and equation (\ref{eq:yoneda-imm})
that the category $\widehat{\mathsf{C}}ob^{\ast}(M)$ itself is the Donaldson category
$\mathcal{D}on(\mathcal{L}ag^{\ast}(M))$ associated to the Lagrangians in $\mathcal{L}ag^{\ast}(M)$.

\begin{rem}\label{rem:embed}
Notice that in this construction the base almost complex structure that one starts with can be
any almost complex structure $J_{0}$ on $M$ because exact, embedded Lagrangians are unobstructed with respect to any such structure. Moreover,  because we use in Proposition \ref{prop:surgery-mor} $0$-size
surgeries, all modules in $D\fuk^{\ast}(M)$ are represented by  marked immersed Lagrangians  $j_{L}:L\to M$ so that the immersion $j_{L}$ restricts to embeddings on the connected components of $L$. In other words,
in order to represent $D\fuk^{\ast}(M)$ the only immersions required are unions of embedded Lagrangians
together with appropriate choices of markings that are used to deform the usual Floer differential as well as 
the higher $A_{\infty}$ multiplications.
\end{rem}

The last part of  Theorem \ref{thm:surg-models} claims that $\widehat{\Theta}_{e}$ is non-decreasing 
with respect to the pseudo-metrics mentioned there. This follows through the same arguments as in 
\cite{Bi-Co-Sh:lshadows-long} but we will not give more details here.

\section{Some further remarks.}\label{sec:rem}

% !TEX root = ImmersedS.tex

\subsection{Related results.}

As mentioned before, the literature concerning Lagrangian cobordism starts with the work
of Arnold \cite{Ar:cob-1}, \cite{Ar:cob-2} who introduced the notion. Rigid behaviour first appears in this 
context in a remarkable paper due to Chekanov \cite{Chek:cob}.
All the consequences and particular statements that result from the existence of 
the functor $\Theta$ in \S\ref{subsubsec:funct-th}, 
when applied to only {\em embedded Lagrangians and cobordisms}, were
 established, mainly in \cite{Bi-Co:cob1} and \cite{Bi-Co:lcob-fuk}. 
A variety of additional facts are known and,  as they also contribute to various other aspects of the 
geometry - algebra dictionary discussed here, we will rapidly review the relevant literature here.  
\subsubsection{Simple cobordisms.} As discussed before, unobstructed simple cobordisms are inducing isomorphisms between the Floer theoretic invariants of the two ends. More is likely to be  true: for instance
simple exact cobordisms are conjectured to coincide with Lagrangian suspensions. Steps towards proving
this conjecture are due to Suarez \cite{Sua:s-cob} and Barraud-Suarez \cite{Bar-Su:rigid-cob}. They  prove, under 
constraints on the behaviour of the fundamental group inclusions of the ends,
that an exact simple cobordism is a pseudo-isotopy (in other words, the cobordism is smoothly trivial).

\subsubsection{Additional algebraic structures.}  There are a variety of results of this sort. Among them:
relations with Calabi-Yau structures due to Campling \cite{Campling}; to stability conditions,
discovered by Hensel \cite{Hensel}; $\Theta$ can be used \cite{Ch-Co:cob-Seidel} to provide a categorification of Seidel's representation that views this representation as an algebraic translation of Lagrangian suspension; cobordisms can be studied also in the total space of Lefschetz fibrations over the complex plane \cite{Bi-Co:lefcob-pub}; they also can be used to approach from a different angle a Seidel type exact sequence for generalized Dehn twists \cite{Mak-Wu:Dehn-twist}. Some connections to mirror 
symmetry appear in Smith-Sheridan \cite{Sm-Sh}.

\subsubsection{Cobordism groups.} Following earlier results due to Audin \cite{Aud:calc-cob} and Eliashberg \cite{El:cob} that compute in the flexible case cobordism  groups in euclidean spaces by means
of algebraic topology, there are essentially two recent calculations in the rigid case. Both apply to surfaces: Haug \cite{Haug:K-cob}
for the torus and Perrier \cite{Per2} (this last paper contains a number of gaps and errors, a corrected version
is in preparation by Rathel-Fourrier \cite{Ra-Fou})  for surfaces of genus at least two. 

\subsubsection{Variants of the ``surgery models'' formalism.} It is likely that the 
concept of a cobordism category with surgery models, or a variant, can be  
applied to some other contexts. There is some work in progress of Fontaine \cite{Fon} who
studies a related formalism for Morse functions. In a different direction, it is not clear for now whether
the five axioms governing categories with surgery models in the paper are in any way optimal.

\subsubsection{Measurements and shadows.} As mentioned earlier in the paper,
shadows of cobordisms have first appeared in \cite{Co-She:metric} for the case of simple
cobordisms and in the wider setting
 of general cobordisms in \cite{Bi-Co-Sh:lshadows-long} where are also introduced fragmentation metrics. 
A variety of interesting related results are due to Bisgaard \cite{Bis1}, \cite{Bis2}. In a different direction,
an interesting new approach is due to Chass\'e \cite{Cha} who is exploring the behaviour of  
shadow type metrics for classes of Lagrangians under uniform  curvature bounds.

\subsubsection{Other approaches/points of view.}
Lagrangian cobordism has also been considered from a different perspective by Nadler-Tanaka \cite{Na-Ta},
\cite{Tanaka}. There are  connections, only superficially explored at the moment, to micro-local sheaf theoretic techniques as evidenced in work of Viterbo \cite{Vi:sheaf} (page 55).

Finally, there is considerable work today on cobordisms between Lagrangians through Legendrians. This is certainly
a different point of view relative to the one that is central to this paper but there are also a variety 
of common aspects. For remarkable applications of this point of view see 
Chantraine-Rizell-Ghiggini-Golovko \cite{Chan-Ri-Ghi-Go}.

% !TEX root = ImmersedS.tex

\subsection{Some further questions} 

\subsubsection{Non-marked Lagrangians.} \label{subsubsec:nonmarked-Lag}
As mentioned before in the paper, in general, rigid cobordism categories with surgery 
models have to contain immersed Lagrangians among their objects. A natural question is 
whether one can have objects that are non-marked (or, in other words, with the empty markings). 

In some framework the answer to this question is expected to be in the affirmative (thus providing
a non-tautological solution to Kontsevich's conjecture). However, there are some difficulties
in achieving this, some conceptual and some technical.

Here are two examples of the problems on the conceptual side. 
\begin{itemize}
\item[a.] The $\epsilon$-surgery $L\#_{\mathbf{x},\epsilon} L'$ at some intersection points 
$\mathbf{x}=\{x_{i}\}_{i\in I}\subset L\cap L'$ with $\epsilon>0$ and
$L$, $L'$ exact is only exact (with suitable choices of handles) if $f_{L}(x_{i})-f_{L'}(x_{i})$ is the
same for all $i\in I$. As a result, the morphisms that can be obtained through an analogue of Proposition \ref{prop:surgery-mor} are associated to cycles of intersection points that have the same action. On the other
hand there is no such control of the action of the points representing $\Theta(V)$ with $\Theta$
the functor in \S\ref{subsubsec:funct-th} and thus it is not clear whether all the image of $\Theta(V)$, at least
as defined in \S\ref{subsubsec:funct-th}, can be represented geometrically.
\item[b.] we have not discussed here changing the base almost complex structure $J_{0}$ but it is certain that such changes in any case lead to a set-up involving bounding chains.
\end{itemize}

The technical difficulties involved to remove from the picture marked Lagrangians and replace them with immersed but non-marked ones, already appear at the following point. 
We need to compare $L\#_{c,\epsilon} L'$ for varying $\epsilon$ and fixed $c$ and 
for $L$ and $L'$ embedded. Assuming the marked Lagrangian $L\#_{c,0} L'$ is unobstructed 
(in the marked sense), the basic questions are whether it is true that:
\begin{itemize}
\item[i.] The non-marked Lagrangian $L\#_{c,\epsilon} L'$ is also unobstructed.
\item[ii.] The surgery cobordism from $L''=L\#_{c,0} L'$ to $L'''=L\#_{c,\epsilon}L'$ for small 
enough $\epsilon$ is also unobstructed and the relevant Yoneda modules of $L''$ and $L'''$ are
quasi-isomorphic.
\end{itemize}
A key step is to compare  moduli spaces $\mathcal{N}_{L'',c}$ of marked tear-drops with boundary on $L''$ to the moduli spaces $\mathcal{N}_{L'''}$ of non-marked
tear-drops with boundary on $L'''$, after rounding the corners at the marking $c$, as pictured in Figure \ref{fig:tear-strips} (assuming the dimension of these moduli spaces is $0$). 
The two moduli spaces need to be identified for $\epsilon$ small. While intuitively clear,
this identification is not obvious.  It depends on a gluing result that appears in \cite{FO3:book-chap-10}
but, beyond that, there are other subtleties. For instance, to ensure regularity of $\mathcal{N}_{L'',c}$
it is natural to use non-autonomous almost complex structures. On the other hand $\mathcal{N}_{L'''}$
consists, by default, of autonomous complex structures, thus these moduli spaces are not  
so easy to render regular. Additionally, the gluing result in \cite{FO3:book-chap-10} applies to autonomous 
curves. Regularity of curves defined with respect to autonomous structures and without perturbations
can potentially be approached  as in \cite{Laz:decomp}, as attempted in  \cite{Per1} (but this paper is not
complete at this time).
   
 \begin{figure}[htbp]
   \begin{center}
    \includegraphics[width=0.45\linewidth]{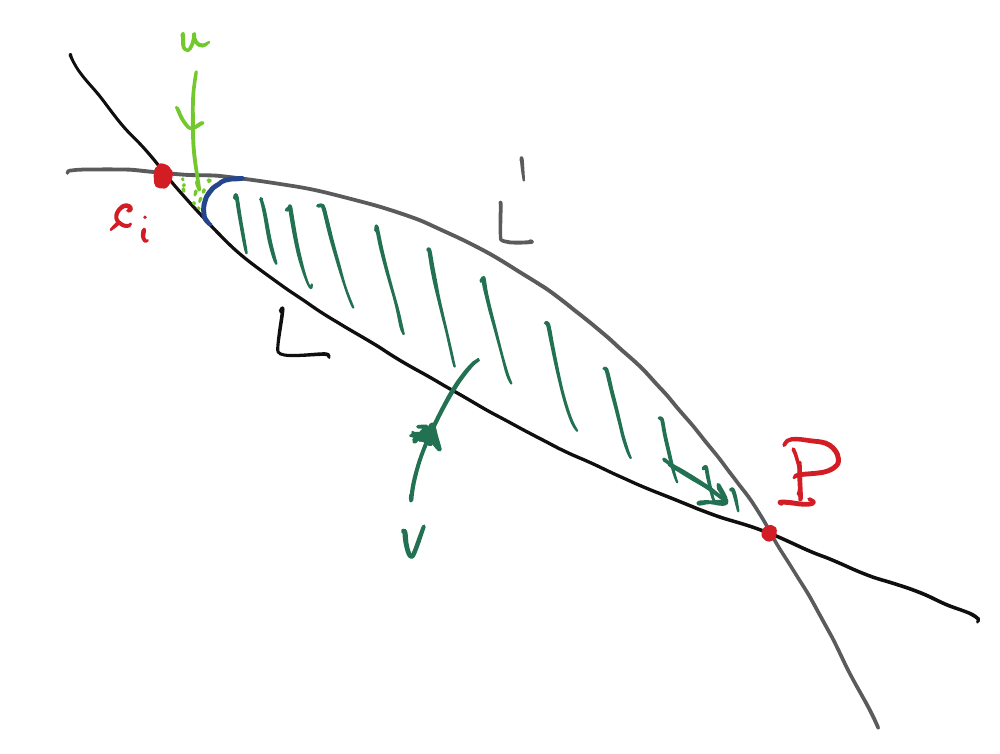}
    \end{center}
   \caption{\label{fig:tear-strips} Comparing teardrops to marked tear-drops (or Floer strips) after rounding a type corner.}
\end{figure}

Further, the same methods need also to be applied to deal with moduli spaces of 
$J$-holomorphic disks and with more complicated curves. This requires a more general gluing statement 
of the type of the one in \cite{FO3:book-chap-10}, that applies to polygons with more than a single corner having boundary on a generically immersed Lagrangian submanifold.  Moreover, one also needs a version of the Gromov compactness theorem that applies to curves with boundary on a sequence
of Lagrangians that converges to an immersed Lagrangian through a process that corresponds to reducing to zero the size of a handle attached through surgery. 

An additional issue, related to the application of the Gromov compactness argument in this setting, 
is that to achieve the comparison between the moduli spaces $\mathcal{N}_{L'',c}$ 
and $\mathcal{N}_{L'''}$ one needs uniform energy bounds
that apply for all $\epsilon$-surgeries for $\epsilon$ small enough. Otherwise,
the moduli spaces $\mathcal{N}_{L'',c}$ might contain curves with energy going to $\infty$ when $\epsilon\to 0$ and these curves obviously are not ``seen'' by the spaces $\mathcal{N}_{L'''}$. 

The gluing and the compactness ingredients mentioned are basically ``in reach'' today, the gluing statement is probably very close to \cite{FO3:book-chap-10} and the compactness one is likely to follow from some version of SFT compactness \cite{Bo-El-Ho-Wy-Ze}. For both, see also the recent work \cite{Pal-Wood-imm}.

\

\subsubsection{Bounding chains and invariance properties.} \label{subsubec:bound-inv}As mentioned before in \S\ref{subsubsec:bounding}, bounding chains and the markings of Lagrangian immersions used in this paper are closely related. Recalling from \cite{Akaho-Joyce} the Akaho-Joyce 
$A_{\infty}$ algebra associated to an immersed Lagrangian $L$, it is easy to see that a marking associated to $L$ provides a bounding chain for this algebra. Moreover, the marked Flower homology of two such marked Lagrangians $L$ and $L'$ is  the corresponding Floer homology associated to the respective bounding chains. Conversely, possibly under some additional conditions, it is expected that a bounding chain can be reduced to a marking.

Exploring in more detail the relation between markings and bounding chains is also of interest
because the notion of unobstructedness used in the paper is somewhat unsatisfying
as it depends on a specific choice of a base almost complex structure $J_{0}$. Certainly, it remains true that the resulting
category $\widehat{\mathsf{C}}ob^{\ast}_{e}(M)$ is a model for the derived Fukaya category of the embedded objects and thus it does not depend of $J_{0}$ up to triangulated isomorphism. Nonetheless, it is expected that  a more invariant construction would make explicit use of bounding chains to adjust the markings, data etc when transitioning from one base almost complex structure to another.  

In this context, an interesting related question is to understand the behaviour of 
bounding chains with respect to surgery. A first step in this direction appears in \cite{Pal-Wood-imm}.

\subsubsection{Novikov coefficients.} Going beyond $\Z/2$-coefficients seems 
possible and, in particular, using coefficients in the universal Novikov field. This extension is however
not completely immediate as it requires using surgery with coefficients that take into account the size
of the handles. In other words, where in the body of the paper we viewed surgery as an operation associated to a cycle $c= \sum c_{i}$ where $c_{i}\in  L\cap L'$ now we need to consider 
cycles of the form $c=\sum \alpha_{i}c_{i}$ with $\alpha_{i}\in\Lambda$ and the surgery needs to take
into account the size of the $\alpha_{i}$'s. 

\subsubsection{Idempotents.} With motivation in questions coming from mirror symmetry 
it is natural to try to understand how to take into account idempotents in the global picture of 
the cobordism categories considered here. This turns out not to be difficult (even if we leave the details
for a subsequent publication): one can view idempotents as an additional structure given by a class $q_{L}$ in the quantum homology $QH(L)$  defined for each one of the ends $L$ of a cobordism and
such that $q_{L}^{2}=q_{L}$. A similar structure also makes sense for a cobordism $V$, in this case the relevant class belongs to $QH(V,\partial V)$. There is a connectant
$QH(V,\partial V)\to QH(\partial V)$ and the requirement for a cobordism respecting the idempotents is for the classes on the boundary to descend from that on $V$. The case of standard cobordisms is recovered 
when all these classes are the units. For embedded Lagrangians, the additional structure provided by idempotents is not hard to control (compared to what was discussed before one also needs to show that the various morphisms induced by cobordisms also relate in the appropriate way the idempotents of the ends).
However,  we expect that the whole construction will be of increased interest when cobordisms 
and their ends are allowed to be immersed. The second author thanks Mohammed Abouzaid for 
a useful discussion on this topic.

\subsubsection{Small scale metric structure.}
 For an embedded Lagrangian $L\in \mathcal{L}ag^{\ast}_{e}(M)$,
 a basic question is whether there is some $\epsilon>0$ such that, if
 $d^{\mathcal{F},\mathcal{F}'}(L,L')\leq \epsilon$ for some $L'\in \mathcal{L}ag^{\ast}(M)$,
 then $L'$ is also embedded and, further, $L'$ is Hamiltonian isotopic to $L$. Could it be that even more 
 is true for $\epsilon$ possibly even smaller: there is a constant $k$ only depending on 
 $\mathcal{F},\mathcal{F}'$ such that $d^{\mathcal{F},\mathcal{F}'}(L,L')\geq k d_{H}(L,L')$
(where $d_{H}$ is the Hofer distance) ?

%\subsubsection{Diameter. } A typical question in this
%direction is to study examples of classes $\mathcal{L}ag^{\ast}(M)$ and families
%$\mathcal{F},\mathcal{F}'\subset \mathcal{L}ag^{\ast}(M)$ such that the diameter 
%of $\mathcal{L}ag^{\ast}(M)$ with respect to the metric $d^{\mathcal{F},\mathcal{F}'}$
%is finite. One such example is provided by the exact Lagrangians  in a cotangent disk bundle
%$D(T^{\ast}N)$ for some closed manifold $N$ \cite{Bi-Co:in-prep}. We consider here the compact Lagrangians to which we add the fibers of the disk bundle. These Lagrangians are viewed as pairs
%$(L,f_{L})$ where $f_{L}$ is a primitive of the pull-back of the Liouville form to $L$. The family $\mathcal{F}$ consists of a fibre $F$ together with all the possible primitives
%that it admits (obviously there is an $\R$-parametric family of such primitives). The family $\mathcal{F}'$ is given by a second fiber again together with all the primitives that it admits. 

\subsubsection{Embedded vs. immersed. }
There are two natural questions related to the formalism discussed here.
Are there criteria that identify objects of $D\fuk^{\ast}(M)$ that are not representable
by embedded Lagrangians? Similarly, are there criteria that identify un obstructed immersed, marked, Lagrangians that are not isomorphic to objects of $D\fuk^{\ast}(M)$ ?

\bibliography{bibliography}

\newcommand{\etalchar}[1]{$^{#1}$}
\def\cprime{$'$} \def\cprime{$'$}
\begin{thebibliography}{FOOO2}

\bibitem[AB1]{Alston-Bao:imm2}
G.~Alston and E.~Bao.
\newblock Immersed lagrangian floer cohomology via pearly trajectories.
\newblock Preprint (2019). Can be found at
  \url{https://arxiv.org/abs/1907.03072}.

\bibitem[AB2]{Alston-Bao}
G.~Alston and E.~Bao.
\newblock Exact, graded, immersed lagrangians and floer theory.
\newblock {\em J. Symplectic Geom.}, 16(2):357--438, 2018.

\bibitem[AJ]{Akaho-Joyce}
M.~Akaho and D.~Joyce.
\newblock Immersed lagrangian floer theory.
\newblock {\em J. Differential Geom.}, 86(3), 2010.

\bibitem[Aka]{Akaho}
M.~Akaho.
\newblock Intersection theory for lagrangian immersions.
\newblock 12(4):543--550, 2005.

\bibitem[Arn1]{Ar:cob-1}
V.~Arnol{\cprime}d.
\newblock Lagrange and {L}egendre cobordisms. {I}.
\newblock {\em Funktsional. Anal. i Prilozhen.}, 14(3):1--13, 96, 1980.

\bibitem[Arn2]{Ar:cob-2}
V.~Arnol{\cprime}d.
\newblock Lagrange and {L}egendre cobordisms. {II}.
\newblock {\em Funktsional. Anal. i Prilozhen.}, 14(4):8--17, 95, 1980.

\bibitem[Aud]{Aud:calc-cob}
M.~Audin.
\newblock Quelques calculs en cobordisme lagrangien.
\newblock {\em Ann. Inst. Fourier (Grenoble)}, 35(3):159--194, 1985.

\bibitem[BC1]{Bi-Co:cob1}
P.~Biran and O.~Cornea.
\newblock Lagrangian cobordism. {I}.
\newblock {\em J. Amer. Math. Soc.}, 26(2):295--340, 2013.

\bibitem[BC2]{Bi-Co:lcob-fuk}
P.~Biran and O.~Cornea.
\newblock Lagrangian cobordism and {F}ukaya categories.
\newblock {\em Geom. Funct. Anal.}, 24(6):1731--1830, 2014.

\bibitem[BC3]{Bi-Co:lefcob-pub}
P.~Biran and O.~Cornea.
\newblock Cone-decompositions of {L}agrangian cobordisms in {L}efschetz
  fibrations.
\newblock {\em Selecta Math. (N.S.)}, 23(4):2635--2704, 2017.

\bibitem[BCS]{Bi-Co-Sh:lshadows-long}
P.~Biran, O.~Cornea, and E.~Shelukhin.
\newblock {L}agrangian shadows and triangulated categories.
\newblock Preprint (2018). Can be found at
  \url{http://arxiv.org/pdf/1806.06630v1}.

\bibitem[BEH{\etalchar{+}}]{Bo-El-Ho-Wy-Ze}
F.~Bourgeois, Y.~Eliashberg, H.~Hofer, K.~Wysocki, and E.~Zehnder.
\newblock Compactness results in symplectic field theory.
\newblock {\em Geom. Topol.}, 7:799--888, 2003.

\bibitem[Bis1]{Bis2}
M.~R. Bisgaard.
\newblock A distance expanding flow on exact {L}agrangian cobordism classes.
\newblock Preprint 2016, Can be found at arXiv:1608.05821 [math.SG].

\bibitem[Bis2]{Bis1}
M.~R. Bisgaard.
\newblock Invariants of {L}agrangian cobordisms via spectral numbers.
\newblock Preprint 2016, Can be found at arXiv:1605.06144 [math.SG].

\bibitem[BS]{Bar-Su:rigid-cob}
J.-F. Barraud and L.S. Suarez.
\newblock The fundamental group of a rigid lagrangian cobordism.
\newblock {\em Annales Math. du Qu\`ebec}, 43:125--144, 2019.

\bibitem[Cam]{Campling}
E.~Campling.
\newblock Fukaya categories of lagrangian cobordisms and duality.
\newblock Preprint, can be found at arXiv:1902.00930 [math.SG].

\bibitem[CC]{Ch-Co:cob-Seidel}
F.~Charette and O.~Cornea.
\newblock Categorification of {S}eidel's representation.
\newblock Preprint (2013). To appear in {\em Israel Journal of Math.} Can be
  found at \url{http://arxiv.org/pdf/1307.7235v2}.

\bibitem[Cha]{Cha}
J.-P. Chass\'e.
\newblock {\em Ph{D} {T}hesis, {U}niversity of {M}ontreal, in progress}.
\newblock PhD thesis.

\bibitem[Che]{Chek:cob}
Y.~Chekanov.
\newblock Lagrangian embeddings and {L}agrangian cobordism.
\newblock In {\em Topics in singularity theory}, volume 180 of {\em Amer. Math.
  Soc. Transl. Ser. 2}, pages 13--23. Amer. Math. Soc., Providence, RI, 1997.

\bibitem[CL]{Cor-La:Cluster-2}
O.~Cornea and F.~Lalonde.
\newblock Cluster homology: an overview of the construction and results.
\newblock {\em Electron. Res. Announc. Amer. Math. Soc.}, 12:1--12, 2006.

\bibitem[CRGG]{Chan-Ri-Ghi-Go}
B.~Chantraine, G.~D. Rizell, P.~Ghiggini, and R.~Golovko.
\newblock Geometric generation of the wrapped fukaya category of weinstein
  manifolds and sectors.
\newblock Preprint, can be found at arXiv:1712.09126 [math.SG].

\bibitem[CS]{Co-She:metric}
O.~Cornea and E.~Shelukhin.
\newblock Lagrangian cobordism and metric invariants.
\newblock {\em J. Diff. Geom.}, 112:1--45, 2019.

\bibitem[Eli]{El:cob}
Y.~Eliashberg.
\newblock Cobordisme des solutions de relations diff\'erentielles.
\newblock In {\em South {R}hone seminar on geometry, {I} ({L}yon, 1983)},
  Travaux en Cours, pages 17--31. Hermann, Paris, 1984.

\bibitem[Flo]{Fl:Morse-theory}
A.~Floer.
\newblock Morse theory for {L}agrangian intersections.
\newblock {\em J. Differential Geom.}, 28(3):513--547, 1988.

\bibitem[Fon]{Fon}
P.~Fontaine.
\newblock M{A} {T}hesis, {U}niversity of {M}ontreal, in progress.

\bibitem[FOOO1]{FO3:book-chap-10}
K.~Fukaya, Y.-G. Oh, H.~Ohta, and K.~Ono.
\newblock Lagrangian intersection {F}loer theory - anomaly and obstruction,
  chapter 10.
\newblock Preprint, can be found at
  \url{http://www.math.kyoto-u.ac.jp/\~fukaya/Chapter10071117.pdf}.

\bibitem[FOOO2]{FO3:book-vol1}
K.~Fukaya, Y.-G. Oh, H.~Ohta, and K.~Ono.
\newblock {\em Lagrangian intersection {F}loer theory: anomaly and obstruction.
  {P}art {I}}, volume~46 of {\em AMS/IP Studies in Advanced Mathematics}.
\newblock American Mathematical Society, Providence, RI, 2009.

\bibitem[Fuk]{Fukaya-immersed}
K.~Fukaya.
\newblock Unobstructed immersed lagrangian correspondence and filtered a
  infinity functor.
\newblock Preprint, can be found at arXiv:1706.02131 [math.SG].

\bibitem[Hau]{Haug:K-cob}
L.~Haug.
\newblock The lagrangian cobordism group of $t^2$.
\newblock {\em Selecta Math. (N.S.)}, 21:1021--1069, 2015.

\bibitem[Hen]{Hensel}
F.~Hensel.
\newblock Stability conditions and lagrangian cobordisms.
\newblock Preprint, can be found at arXiv:1712.02252 [math.SG].

\bibitem[Laz]{Laz:decomp}
L.~Lazzarini.
\newblock Relative frames on {J}-holomorphic curves.
\newblock {\em Fixed Point Theory and Applications}, 9(2):213--256, 2011.

\bibitem[LS]{La-Si:Lag}
F.~Lalonde and J.-C. Sikorav.
\newblock Sous-vari\'{e}t\'{e}s {L}agrangiennes et {L}agrangiennes exactes des
  fibr\'{e}s cotangents.
\newblock {\em Comment. Math. Helv.}, 66(1):18--33, 1991.

\bibitem[MW]{Mak-Wu:Dehn-twist}
C.-Y. Mak and W.~Wu.
\newblock {D}ehn twists exact sequences through {L}agrangian cobordism.
\newblock Preprint (2015). Can be found at
  \url{http://arxiv.org/pdf/1509.08028}.

\bibitem[NT]{Na-Ta}
D.~Nadler and H.L. Tanaka.
\newblock A stable infinity-category of lagrangian cobordisms.
\newblock Preprint (2011). Can be found at
  \url{http://arxiv.org/pdf/1109.4835v1}.

\bibitem[Per1]{Per2}
A.~Perrier.
\newblock Lagrangian cobordism groups of higher genus surfaces.
\newblock Preprint, can be found at arXiv:1901.06002 [math.SG].

\bibitem[Per2]{Per1}
A.~Perrier.
\newblock Structure of {J}-holomorphic disks with immersed {L}agrangian
  boundary conditions.
\newblock Preprint, can be found at arXiv:1808.01849 [math.SG].

\bibitem[Pol]{Po:surgery}
L.~Polterovich.
\newblock The surgery of {L}agrange submanifolds.
\newblock {\em Geom. Funct. Anal.}, 1(2):198--210, 1991.

\bibitem[PW]{Pal-Wood-imm}
J.~Palmer and C.~Woodward.
\newblock Invariance of immersed {F}loer cohomology under {L}agrangian surgery.
\newblock Preprint, can be found at arXiv:1903.01943v1 [math.SG].

\bibitem[RF]{Ra-Fou}
D.~Rathel-Fourrier.
\newblock Ph{D} {T}hesis, {U}niversity of {M}ontreal, in progress.

\bibitem[Sch]{Schmaschke-imm}
F.~Schm\"aschke.
\newblock {F}loer homology of {L}agrnagians in clean intersction.
\newblock Preprint, can be found at arXiv:1606.05327 [math.SG].

\bibitem[Sei]{Se:book-fukaya-categ}
P.~Seidel.
\newblock {\em Fukaya categories and {P}icard-{L}efschetz theory}.
\newblock Zurich Lectures in Advanced Mathematics. European Mathematical
  Society (EMS), Z\"urich, 2008.

\bibitem[SL]{Sua:s-cob}
L.~S. Suarez-Lopez.
\newblock Exact {L}agrangian cobordism and pseudo-isotopy.
\newblock 2017.

\bibitem[SS]{Sm-Sh}
I.~Smith and N.~Sheridan.
\newblock Rational equivalence and {L}agrangian tori on {K}3 surfaces.
\newblock Preprint, can be found at arXiv:1809.03892 [math.SG].

\bibitem[Tan]{Tanaka}
H.~Tanaka.
\newblock Generation for {L}agrangian cobordisms in {W}einstein manifolds.
\newblock Preprint, can be found at arXiv:1810.10605 [math.SG].

\bibitem[Vit]{Vi:sheaf}
C.~Viterbo.
\newblock Sheaf quantization of lagrangians and floer cohomology.
\newblock Preprint, can be found at arXiv:1901.09440 [math.SG].

\end{thebibliography}

%\input{Add.tex}
%\bibliography{/home/biran/latex/general/bibliography}

%
%\begin{thebibliography}{10}
%
%\end{thebibliography}
%

\end{document}